\documentclass[a4paper,reqno]{amsart}
\usepackage{vmargin}
\usepackage[pdftex]{graphicx}
\usepackage{amsfonts}
\usepackage{amssymb}
\usepackage{amsrefs} 
\usepackage{mathrsfs}
\usepackage{stmaryrd}
\usepackage{amsmath}
\usepackage{amsthm}
\usepackage[shortlabels]{enumitem}
\usepackage{nomencl}
\usepackage[bookmarks=true]{hyperref}   
\usepackage{esint}
\usepackage{dsfont}
\usepackage{upgreek}
\usepackage{xcolor}
\usepackage{slashed}
\usepackage{scalerel}
\usepackage{comment}
\usepackage{todonotes}
\usepackage[utf8]{inputenc}
\usepackage{skull}
\usepackage{tikz-cd}
\usepackage{faktor}
\usepackage{mathtools}

\newcommand{\GProj}{\mathcal{P}}

\hypersetup{
    colorlinks,%
}

\author{Lashi Bandara}
\author{Magnus Goffeng}
\author{Hemanth Saratchandran}
\title[Realisations of elliptic operators on compact manifolds with boundary]
{Realisations of elliptic operators on compact manifolds with boundary}
\date{\today}

\address{Lashi Bandara, 
Institut für Mathematik,
Universität Potsdam, 
D-14476, Potsdam OT Golm, Germany
}
\urladdr{\href{http://www.math.uni-potsdam.de/~bandara}{http://www.math.uni-potsdam.de/~bandara}}
\email{\href{mailto:lashi.bandara@uni-potsdam.de}{lashi.bandara@uni-potsdam.de}}

\address{Magnus Goffeng,
Centre for Mathematical Sciences,
University of Lund,
Box 118, 221 00 LUND, Sweden
}
\urladdr{\href{https://www.lunduniversity.lu.se/lucat/user/e604494124f99d9bb6048e890306f7a4}{https://www.lunduniversity.lu.se/lucat/user/e604494124f99d9bb6048e890306f7a4}}
\email{\href{mailto:magnus.goffeng@math.lth.se}{magnus.goffeng@math.lth.se}}

\address{Hemanth Saratchandran 
The University of Adelaide
Adelaide, South Australia,
5005, Australia
}
\urladdr{\href{https://researchers.adelaide.edu.au/profile/hemanth.saratchandran}{https://researchers.adelaide.edu.au/profile/hemanth.saratchandran}}
\email{\href{mailto:hemanth.saratchandran@adelaide.edu.au}{hemanth.saratchandran@adelaide.edu.au}}


\def\colour{\color}

\newcommand{\Hinfty}{\ensuremath{\mathrm{H}^\infty}}

\newcommand{\mathhil}[1]{\mathscr{\uppercase{#1}}}

\def\colour{\colour}

\def\colour{\color}

\newtheorem{theorem}{Theorem}[section]
\newtheorem{corollary}[theorem]{Corollary}
\newtheorem{lemma}[theorem]{Lemma}
\newtheorem{proposition}[theorem]{Proposition}

\newtheorem{definition}[theorem]{Definition}
\newtheorem{remark}[theorem]{Remark}

\newtheorem{example}[theorem]{Example}

\newcommand{\cbrac}[1]{\left(#1\right)}

\newcommand{\dbrac}[1]{\left\{#1\right\}}
\newcommand{\modulus}[1]{|#1|}

\newcommand{\set}[1]{\dbrac{#1}}

\newcommand{\ran}{\mathrm{ran}}
\newcommand{\dom}{\mathrm{dom}}

\newcommand{\ad}{\ast} 

\newcommand{\comp}{\, \circ\, }

\newcommand{\R}{\mathbb{R}}
\newcommand{\C}{\mathbb{C}}

\renewcommand{\emptyset}{\varnothing}

\newcommand{\intersect}{\cap}

\newcommand{\rest}[1]{{{\lvert_{}}_{}}_{#1}}
\newcommand{\close}[1]{\overline{#1}}		

\renewcommand{\epsilon}{\varepsilon}

\renewcommand{\phi}{\varphi}

\newcommand{\embed}{\hookrightarrow}		

\newcommand{\isomorphic}{\cong}			
\newcommand{\isom}{\cong}			

\newcommand{\tensor}{\otimes}

\newcommand{\End}{\mathrm{End}}

\newcommand{\norm}[1]{\| #1 \|}			

\newcommand{\spt}[1]{{\rm spt} {\text{ }}#1}	

\newcommand{\interior}[1]{\mathring{#1}}	
\DeclareFontFamily{OT1}{restrictfont}{}
\DeclareFontShape{OT1}{restrictfont}{m}{n}{<-> fmvr8x}{}

\newcommand{\adj}[1]{{#1}^\ast}			

\newcommand{\inprod}[1]{\left\langle #1 \right\rangle}	

\newcommand{\sym}{\upsigma}
\newcommand{\spec}{\mathrm{spec}}				
\newcommand{\specpt}{\mathrm{spec}_{\mathrm{pt}}}		
\newcommand{\specres}{\mathrm{spec}_{\mathrm{res}}}		
\newcommand{\res}{\mathrm{res}}		
\newcommand{\Eig}{\mathrm{Eig}}		

\newcommand{\conj}{\mathrm{conj}}

\newcommand{\Lp}[2][{}]{{\rm L}^{#2}_{\rm #1}}		
\newcommand{\Ck}[2][{}]{{\rm C}^{#2}_{\rm #1}}		
\newcommand{\Hard}[2][{}]{{\rm H}^{#2}_{\rm #1}}		
\newcommand{\SobH}[2][{}]{\Hard[#1]{\rm #2}}

\newcommand{\cB}{\mathcal{B}}

\renewcommand{\epsilon}{\varepsilon} 
\DeclareMathOperator{\indx}{ind}

\newcommand{\Hil}{\mathscr{H}}

\newcommand{\checkH}{\check{\mathrm{H}}}
\newcommand{\hatH}{\hat{\mathrm{H}}}

\newcommand{\checkh}{\check{\mathhil{h}}}

\newcommand{\coker}{\mathrm{coker}}

\newcommand{\Proj}[1]{\mathrm{P}_{#1}}
\newcommand{\Ca}{\mathcal{C}}
\newcommand{\Cac}{\mathcal{C}^\mathrm{c}}

\newcommand{\op}{\mathrm{op}} 

\newcommand{\SobHH}[2][{}]{{\mathbb{H}}^{#2}_{\rm #1}}

\makeatletter
\def\Ddots{\mathinner{\mkern1mu\raise\p@
\vbox{\kern7\p@\hbox{.}}\mkern2mu
\raise4\p@\hbox{.}\mkern2mu\raise7\p@\hbox{.}\mkern1mu}}
\makeatother

\newtheorem{thm*}{Theorem}
\newtheorem{remark*}[thm*]{Remark}

\newtheoremstyle{QuestionStyle}
{3pt}
{3pt}
{}
{}
{\bf}
{.}
{.5em}
{}
\theoremstyle{QuestionStyle} 
\newtheorem*{question*}{Question}

\renewcommand{\Re}{\mathrm{Re}\ }

\begin{document}

\maketitle
\begin{abstract}
This paper investigates realisations of elliptic differential operators of general order on manifolds with boundary following the approach of Bär-Ballmann to first order elliptic operators. The space of possible boundary values of elements in the maximal domain is described as a Hilbert space densely sandwiched between two mixed order Sobolev spaces. The description uses Calderón projectors which, in the first order case, is equivalent to results of Bär-Bandara using spectral projectors of an adapted boundary operator. Boundary conditions that induce Fredholm as well as regular realisations, and those that admit higher order regularity, are characterised. In addition, results concerning spectral theory, homotopy invariance of the Fredholm index, and well-posedness for higher order elliptic boundary value problems are proven. 
\end{abstract} 
\tableofcontents

\parindent0cm
\setlength{\parskip}{\baselineskip}

\section{Introduction}

Elliptic boundary value problems have a long history emerging from classical problems in physics and engineering. Their mathematical description dates back to classical works such as \cite{grisvard, horIII, lionsmagenes,grubb68,grubb74, grubb77, schechter59,vishik}. Elliptic operators -- their index theory and spectral theory -- is broadly used in mathematics connecting seemingly disjoint areas of mathematics, for instance through noncommutative geometry and geometric analysis. The spectral theory of elliptic boundary value problems has primarily been focused on Laplace type operators (for an overview see \cite{ivrii16}) or Dirac type operators (see for instance \cite{bruninglesch99,bruninglesch01,bosswojc,Grubb03}), but has also been studied in larger generality (e.g. in \cite{geygru74,gerdsgreenbook,grubb74,grubb77a, grubb84}). The index theory of elliptic boundary value problems is also a well studied area. For instance, the index theory of Dirac operators with the spectral APS boundary condition was computed by Atiyah-Patodi-Singer \cite{APS}, and has since been well studied, see for instance \cite{G99,grubb92,melroseAPS,bruninglesch99,bosswojc}. The index theory for boundary value problems in the Boutet de Monvel calculus was computed by \cite{boutetdemonvel} relating to ideas of Agranovich-Dynin. This formula covers the case of local elliptic boundary conditions. Boutet de Monvel's results have been extensively studied, see \cite{elmaretal,fedosovindex,fedoind,rempelschulze}, and were generalised to an extended Boutet de Monvel calculus by Schulze-Seiler \cite{schulzeseiler}. 

In this paper we take a new approach to higher order elliptic operators on manifolds with boundary. The approach is based on work of Bär-Ballmann \cite{BB}. The paper \cite{BB} consists of a coherent overview of first order boundary value problems (admitting a self-adjoint adapted boundary operator) based on an analytic description of the Cauchy data space as well as a novel graphical decomposition of regular realisations of first order operators. The results were later extended to general first order problems by the first listed author and Bär in \cite{BBan}. Although this approach was first fully utilised and made explicit in \cite{BB}, the ideas have been around since the '60s and was implicitly contained in a paper by Seeley \cite[top of page 782]{seeley65}. They were visible already in Agranovich-Dynin's classical works on index theory for elliptic boundary value problems of general order. Similar approaches predating Bär-Ballmann's have been studied in the abstract in \cite{bossfury}, been applied to Dirac-Schrödinger operators in \cite{ballbruncarr} as well as having been used in the study of Maxwell's equations on Lipschitz domains \cite{HR2019,BuCoSh}. Related ideas can also be seen in for instance  \cite{APS,bruninglesch01,grubb68, grubb77,grubb92,G99}. Since its appearance, the coherence of  Bär-Ballmann's work \cite{BB} has paved the way for numerous applications \cite{MR4011805,MR4000837,MR3981455,MR3908762,MR3850258}. 

The aim of this paper is to present a perspective similar to \cite{BB}, providing an overview of the theory for general order elliptic boundary value problems from the vantage point of the Cauchy data space. As such, the main novelty lies in the framework and the methods. Indeed, several of the results in this paper are, as stand alone results, at best mild generalisations of known results and are unlikely to surprise experts in the field of boundary value problems. A key feature in several of the results is that they can treat all (or at least the regular extensions with the occasional restriction to the pseudo-local case) closed extensions of an elliptic differential operator on an equal footing with the special classes of boundary conditions studied classically. We have attempted to the best of our abilities to indicate where the results we are generalising can be found in the abundance of literature on boundary value problems.

The approach we take lies close in spirit to the ideas of noncommutative geometry \cite{connesbook} whose methods have proven their worth in index theory, relating to recent work on boundaries in noncommutative geometry \cite{FGMR}. 
It takes a more abstract perspective than the semiclassical methods ordinarily deployed to study boundary value problems, e.g. in \cite{grisvard, horIII, lionsmagenes,grubb77, schechter59,vishik}, and lies closer to methods for order one elliptic operators seen in \cite{G99,grubb92,bruninglesch99,bosswojc,bossfury,BB,BBan} rather than the abstract methods of, for instance, \cite{grubb68, grubb74, malamud,behrndtetal}. {\bf The level of abstraction makes the method feasible for further development of index theory and applications to a larger class of problems than elliptic differential operators on manifolds with boundary}.
We anticipate this includes elliptic operators on singular manifolds, hypoelliptic operators (cf. the subelliptic boundary conditions in \cite{epsteinsub}) or geometric operators in Lorentzian geometry (cf. \cite{MR4011805}). These facts raises optimism for demystifying the general index formulas appearing in the (extended) Boutet de Monvel calculus relying on abstract homotopies \cite{boutetdemonvel,rempelschulze}, non-explicit inverses to isomorphisms in $K$-theory \cite{elmaretal} or Chern characters of operator valued symbols \cite{fedoind,fedosovindex}.

The method of Bär-Ballmann for first order operators can be summarised as describing boundary value problems by means of the Cauchy data space, i.e. the range of the corresponding trace map onto a mixed Sobolev space on the boundary. For higher order operators we apply the same ideas to the full trace map (taking into account all boundary traces up to the order of the operator) into a mixed Sobolev space on the boundary. The range of the full trace map -- the Cauchy data space -- characterises the realisation defining the boundary value problem when topologising it so that the full trace map is a quotient. The analytic details involved in this procedure produces {\bf precise information about regularity and Fredholm properties of the realisations} in terms of properties of the boundary condition relative to the range of the full trace mapping on the maximal domain of the elliptic operator. The general theme in this framework is to reduce properties of realisations, and their proofs, to neat functional analysis arguments and Fredholm theory.

The description of the possible realisations of an elliptic operator as subspaces of a Hilbert space of functions on the boundary relies on a precise description of the Cauchy data space. Abstractly, the Cauchy data space is isomorphic to the quotient of the maximal domain by the minimal domain. By classical results of Lions-Magenes \cite{lionsmagenes63,lionsmagenes} the Cauchy data space is a concrete subspace of a mixed order Sobolev space on the boundary. For an elliptic first order differential operator this is described in detail in \cite{BBan} using an adapted boundary operator. To describe the Cauchy data space of a general order elliptic differential operator,  we employ the Calderón projection of Seeley \cite{seeley65}.  Although our setup is based on the work of Bär-Ballmann, similar ideas date back much further as discussed above.

We emphasise that this paper, despite allowing for elliptic differential operators of general order, by no means supersedes \cite{BB,BBan}.
A significant novelty distinguishing \cite{BBan} from this paper is that \cite{BBan} links analysis on the Cauchy data space to the \Hinfty-functional calculus, which in \cite{BBan} forms the key to obtaining the estimates required to topologise the Cauchy data space. In contrast, in this paper
we use Seeley's work on Calderón projectors to analyse the Cauchy data space.

\subsection{Main results and overview of paper}

Let $M$ be a  compact manifold with boundary $\Sigma:=\partial M$. We fix a smooth measure $\mu$ on $M$. 
In our convention, $\Sigma \subset M$ and the interior of $M$ is denoted by $\interior{M} = M \setminus \Sigma$. We also choose a smooth interior pointing vectorfield $\vec{T}$ transversal to $\Sigma=\partial M$. The smooth measure on $\Sigma$ induced by $\mu$ and $\vec{T}$ will be denoted by $\nu$.

In this paper, we are  mainly concerned with elliptic differential operators $D:\Ck{\infty}(M;E)\to \Ck{\infty}(M;F)$ acting between hermitian vector bundles $(E,h^E) \to M$ and $(F,h^F) \to M$. The hermitian metric will be implicit in the notations. We let $m$ denote the order of $D$. The  maximal realisation $D_{\rm max}$ of $D$ on $\Lp{2}$ is given by $\dom(D_{\max}):=\{f\in \Lp{2}(M;E): Df\in \Lp{2}(M;F)\}$ where $D$ acts in a distributional sense on $\Lp{2}(M;E)$  as a consequence of the presence of a unique formal adjoint $D^\dagger: \Ck{\infty}(M;F) \to \Ck{\infty}(M;E)$.  The minimal realisation $D_{\rm min}$ is defined from the domain obtained from closing $\Ck[c]{\infty}(\interior{M};E)$ in the graph norm of $D$; it is readily seen that $\dom(D_{\rm min})=\SobH[0]{m}(M;E)$ (cf. Proposition \ref{geneprop}, on page \pageref{geneprop}). Many of the general results we prove only rely on abstract properties of the minimal and maximal realisations using the machinery of Appendix \ref{bärballappsection} (see page \pageref{bärballappsection}). For $s\in \R$, we use the notation 
$$\SobHH{ s}(\Sigma;E\otimes \C^m):=\bigoplus_{j=0}^{m-1}\SobH{s-j}(\Sigma;E).$$
As described in Subsection \ref{subsec:calderonproj}, the space $\SobHH{ s}(\Sigma;E\otimes \C^m)$ is easiest thought of as a graded Sobolev space with respect to a certain grading on $E\otimes \C^m$.

A (closed) realisation of $D$ is a (closed) extension $D_e$ of $D_{\rm min}$ such that $D_e\subseteq D_{\rm max}$. In other words, a realisation of $D$ is an extension of $D$ acting on $\SobH[0]{m}(M;E)$ to a domain on $\Lp{2}(M;E)$ where it acts in the ordinary distributional sense. 

The following theorem summarises how the Bär-Ballmann machinery applies in the setting of higher order elliptic differential operators on manifolds with boundary.

\begin{thm*}
\label{firstofmsmsdmds}
Let $D$ be an elliptic differential operator of order $m>0$ as in the preceding paragraph. 
\begin{enumerate}[(i), itemsep=1em]
\item \label{firstof:1} The full trace mapping $\gamma:\SobH{m}(M;E)\to \SobHH{{m-{\frac12}}}(\Sigma;E\otimes \C^m)$, $u\mapsto (\partial_{x_n}^ju|_\Sigma)_{j=0}^{m-1}$ extends to a continuous trace mapping $\gamma:\dom(D_{\rm max})\to \SobHH{{-{\frac12}}}(\Sigma;E\otimes \C^m)$ with kernel $\dom(D_{\rm min})=\SobH[0]{m}(M;E)$ and range being the \emph{Cauchy data space} 
$$\checkH(D)=P_\mathcal{C}\SobHH{{-{\frac12}}}(\Sigma;E\otimes \C^m)\bigoplus (1-P_\mathcal{C}) \SobHH{{m-{\frac12}}}(\Sigma;E\otimes \C^m),$$
where $P_\mathcal{C}$ is a Calderón projection (for more details on $P_\mathcal{C}$, see Section \ref{sec:calderon}). Equipping $\checkH(D)$ with the Hilbert space topology making the projectors onto $\SobHH{{-{\frac12}}}(\Sigma;E\otimes \C^m)$ and $\SobHH{{m-{\frac12}}}(\Sigma;E\otimes \C^m)$, respectively, into partial isometries, the canonical isomorphism $\checkH(D)\cong \dom(D_{\rm max})/\dom(D_{\rm min})$ is a Banach space isomorphism. 
\item \label{firstof:2} There is a one-to-one correspondence between (closed) realisations of $D$ and (closed) subspaces $B\subseteq \checkH(D)$ as follows. A (closed) subspace $B\subseteq \checkH(D)$ uniquely determines a (closed) realisation $D_{\rm B}$ of $D$ defined by 
$$\dom(D_{\rm B})=\{f\in \dom(D_{\rm max}): \gamma(f)\in B\},$$
and any (closed) realisation $\hat{D}$ uniquely determines a (closed) subspace $B:=\gamma\dom(\hat{D})$ with $D_{\rm B}=\hat{D}$. Moreover, 
$$(D_{\rm B})^*=D^\dagger_{{\rm B}^*},$$
where $B^*$ is the adjoint boundary condition (for more details see Subsection \ref{subsec:adjointbcs}, starting on page \pageref{subsec:adjointbcs}, and Appendix \ref{app:bspaceandbc}, starting on page \pageref{app:bspaceandbc}).
\item \label{firstof:3} A realisation $D_{\rm B}$ of $D$ is semi-regular, that is we have the domain inclusion 
$$\dom(D_{\rm B})\subseteq \SobH{m}(M;E),$$ 
if and only if $B\subseteq \SobHH{{m-{\frac12}}}(\Sigma;E\otimes \C^m)$. A realisation $D_{\rm B}$ is regular (i.e. $D_{\rm B}$ and $(D_{\rm B})^*$ are semi-regular) if and only if $B\subseteq \SobHH{{m-{\frac12}}}(\Sigma;E\otimes \C^m)$ and $B^*\subseteq \SobHH{{m-{\frac12}}}(\Sigma;E\otimes \C^m)$ where $B^*$ is the adjoint boundary condition.
\end{enumerate}
\end{thm*}

The reader can find this theorem spread out in the bulk of the paper as follows. Part \ref{firstof:1} can be found in Theorem \ref{refinedseeley} (on page \pageref{refinedseeley}). Part \ref{firstof:2} is proven in much larger generality in Appendix \ref{bärballappsection}, more precisely in Proposition \ref{charbodundalda} on page \pageref{charbodundalda}. Part \ref{firstof:3} is proven in Proposition \ref{charsemiellintermsofboundary}, see page \pageref{charsemiellintermsofboundary}. The precise definition of regularity can be found in Definition \ref{def:elliptiiclcld} (page \pageref{def:elliptiiclcld}). In \cite{BB,BBan}, regular boundary conditions are called elliptic, we discuss our choice of terminology further in Remark \ref{ontheremineendl} (page \pageref{ontheremineendl}). We remark that Theorem \ref{firstofmsmsdmds} concerns local statements at the boundary and therefore\footnote{Using for instance Lemma \ref{approximcladldlad} on page \pageref{approximcladldlad} or Proposition \ref{localizingcheckspace} on page \pageref{localizingcheckspace}} extends to the case that $M$ is noncompact with compact boundary when $D$ is complete.

Based on Theorem \ref{firstofmsmsdmds}, {\bf a boundary condition for $D$ is defined to be a closed subspace $B\subseteq \checkH(D)$ and we write $D_{\rm B}$ for the associated realisation.} We say that a boundary condition $B$ is (semi-) regular if $D_{\rm B}$ is (semi-) regular.

Theorem \ref{firstofmsmsdmds} has direct consequences to the  well-posedness of PDEs. 
If $D$ is an elliptic differential operator of order $m>0$ with a semi-regular boundary condition $B$ such that $\ker(D_{\rm B})=0$, then the partial differential equation 
$$
Du=f,\qquad \gamma(u)\in B$$
is well-posed for $f\in \ran(D_{\rm B})=\ker(D_{{\rm B}^*}^\dagger)^\perp$. For more details, see Proposition \ref{Cor:WellPosed} on page \pageref{Cor:WellPosed}. In particular, if $D$ is a Dirac type operator (so $m=1$), the partial differential equation 
$$Du=f,\qquad \gamma(u)=0$$
is well-posed for $f\in \ran(D_{\rm min})$. See more in Corollary \ref{firstordodereowl} on page \pageref{firstordodereowl}.

The previous theorem has the following consequence on the spectral theory of the maximal and minimal realisations of an elliptic operator. We include this result because we have noticed some confusion surrounding it in the community. The result is known to experts in the field (cf. the discussion on \cite[page 60-61]{grubbdistop}).

\begin{thm*}
Let $D:\Ck{\infty}(M;E)\to \Ck{\infty}(M;E)$ be an elliptic differential operator of order $m>0$ acting between sections on a Hermitian vector bundle $E\to M$ over a compact manifold with boundary with $\dim(M) > 1$. The spectrum of $D_{\rm max}$ on $\Lp{2}(M;E)$ is
$$\spec(D_{\max}) = \specpt(D_{\max}) = \C.$$
That is, the spectrum of $D_{\max}$ is purely discrete and  each generalised eigenspace is of infinite dimension. 
Moreover, the spectrum of $D_{\rm min}$ on $\Lp{2}(M;E)$ is
$$\spec(D_{\min}) = \C.$$
If $m=1$ and $D$ is a Dirac type operator then $\spec(D_{\min})= \specres(D_{\min}) = \C$; i.e., it  consists solely of residual spectrum.
\end{thm*}

The reader can find the first part of this theorem on general order operators as Proposition \ref{Prop:SpecoHO} (see page \pageref{Prop:SpecoHO}) in the bulk of the text. The statement concerning $\spec(D_{\min})$ for Dirac-type operators can be found in Corollary \ref{Cor:SpecoFOres} (see page \pageref{Cor:SpecoFOres}). The property distinguishing Dirac type operators from general elliptic operators is the  so-called \emph{unique continuation property} which undermines the existence of eigenvalues for $D_{\rm min}$. See the discussion in Remark \ref{Rem:Residue} on page \pageref{Rem:Residue}.

We now turn to some results on the index theory of realisations of elliptic differential operators. We say that a boundary condition $B\subseteq \checkH(D)$ is Fredholm if the associated realisation $D_{\rm B}$ is Fredholm. We also introduce the notation $\mathcal{C}_D\subseteq \checkH(D)$ for the Hardy space -- the image of $\ker(D_{\rm max})$ under the full trace mapping. Note that $\mathcal{C}_D= P_\mathcal{C}\checkH(D)=P_\mathcal{C}\SobHH{{-{\frac12}}}(\Sigma;E\otimes \C^m)$ by Theorem \ref{firstofmsmsdmds}.

\begin{thm*}
\label{fredhoclcldmda}
Let $D$ be an elliptic differential operator of order $m>0$ as above. A boundary condition $B\subseteq \checkH(D)$ is Fredholm if and only if $(B,\mathcal{C}_D)$ is a Fredholm pair in $\checkH(D)$. In this case, it holds that 
$$\indx(D_{\rm B}) = \indx(B, \Ca_D) + \dim \ker (D_{\min}) - \dim \ker (D_{\min}^\dagger).$$
Moreover, $\indx(D_{\rm B})$ is a homotopy invariant of $(D,B)$ in the sense of Theorem \ref{rigidityformaula} on page \pageref{rigidityformaula} (see also Theorem \ref{cortorigidityformaula} on page \pageref{cortorigidityformaula}).
\end{thm*}

The characterisation of Fredholm boundary conditions and the index formula is proven by abstract principles and can be found in Theorem \ref{Thm:FredChar} (see page \pageref{Thm:FredChar}) in the body of the text. Results similar to Theorem \ref{fredhoclcldmda} for operators of order $m=1$ can be found in \cite{bossfury,bosswojc}.

\begin{thm*}
\label{elliptidcocoddn}
Let $D$ be an elliptic differential operator of order $m>0$. Assume that $P$ is a projection on $\SobHH{{m-{\frac12}}}(\Sigma;E\otimes \C^m)$ and set $A:=P_\mathcal{C}-(1-P)$. Consider the following statements:
\begin{enumerate}
\item \label{ell:1} The operator $A$ is Fredholm on $\SobHH{{m-{\frac12}}}(\Sigma;E\otimes \C^m)$.
\item \label{ell:2} The operator $P$ extends by continuity to $\checkH(D)$ and $\SobHH{{-{\frac12}}}(\Sigma;E\otimes \C^m)$, and $A$ defines a Fredholm operator on $\checkH(D)$.
\item \label{ell:3} The operator $P$ extends by continuity to $\SobHH{{-{\frac12}}}(\Sigma;E\otimes \C^m)$ and $A$ defines a Fredholm operator on $\SobHH{{-{\frac12}}}(\Sigma;E\otimes \C^m)$.
\end{enumerate}
The following holds:
\begin{enumerate}[(i)]
\item If \ref{ell:1} and \ref{ell:2}  holds, then $P$ is \emph{boundary decomposing} (see Definition \ref{boundafefidodo} on page \pageref{boundafefidodo}) and in particular
$$\norm{u}_{\checkH(D)} \simeq \norm{(1-P) u}_{\SobHH{{m-\frac12}}(\Sigma;E\otimes \C^m)} + \norm{P u}_{\SobHH{-{\frac12}}(\Sigma;E\otimes \C^m)}.$$ 
\item If \ref{ell:1} and \ref{ell:3} holds, then the space 
$$B_P:=(1-P)\SobHH{{m-{\frac12}}}(\Sigma;E\otimes \C^m),$$ 
defines a regular boundary condition for $D$. 
\end{enumerate}

Moreover, if $P$ is a zeroth order pseudodifferential projection in the Douglis-Nirenberg calculus (see more details in Subsection \ref{subsec:calderonproj}), then the following are equivalent:
\begin{enumerate}[a)]
\item $P$ is Shapiro-Lopatinskii elliptic with respect to $D$ (in the sense of Definition \ref{def:lsell}).
\item The operator $P_\mathcal{C}-(1-P)\in \Psi^{\pmb 0}_{\rm cl}(\Sigma; E\otimes \C^m)$ is elliptic in the Douglis-Nirenberg calculus. 
\item The space 
$$B_P:=(1-P)\SobHH{{m-\frac12}}(\Sigma;E\otimes \C^m),$$
is a regular boundary condition for $D$, so in particular, the realisation $D_{\rm B}$ defined from 
$$\dom(D_{\rm B}):=\{u\in \dom(D_{\rm max}): P\gamma u=0\}=\{u\in \SobH{m}(M;E): P\gamma u=0\},$$
is regular.
\item The space 
$$B_P:=(1-P)\SobHH{{m-\frac12}}(\Sigma;E\otimes \C^m),$$
is a Fredholm boundary condition for $D$, so in particular, the realisation $D_{\rm B}$ defined from 
$$\dom(D_{\rm B}):=\{u\in \dom(D_{\rm max}): P\gamma u=0\},$$
is a Fredholm operator.
\end{enumerate}
If any of the equivalent conditions a)-d) holds, and moreover $[P,P_\mathcal{C}]$ is order $-m$ in the Douglis-Nirenberg calculus, then $P$ is also boundary decomposing.
\end{thm*}

The reader can find item i) stated as Theorem \ref{thm:bodunarodoad} (see page \pageref{thm:bodunarodoad}) and item ii) stated as Theorem \ref{thm:sufficienfeidntofell} (see page \pageref{thm:sufficienfeidntofell}) in the body of the text. The equivalence of item a) and b) is found in Proposition \ref{lsellipticgiveselliptic} (see page \pageref{lsellipticgiveselliptic}), and the equivalence of item b), c) and d) is found in Theorem \ref{equivocndndforpseudolocal} (see page \pageref{equivocndndforpseudolocal}). The final conclusion of the theorem appears as Corollary \ref{pseudostahtdecompose} (see page \pageref{pseudostahtdecompose}). 

\begin{remark*}
The reader should be wary of the fact that the structure of the Cauchy data space $\checkH(D)$ is quite different from that of a Sobolev space. We give three instances of how this manifests:
\begin{itemize}
\item If $P$ is a zeroth order pseudodifferential projection in the Douglis-Nirenberg calculus and $P_\mathcal{C}-(1-P)$ is elliptic then $B_P=(1-P)\SobHH{{m-\frac12}}(\Sigma;E\otimes \C^m)$ is a regular boundary condition but if $(1-P_\mathcal{C})PP_\mathcal{C}$ is not of order $-m$ (in the Douglis-Nirenberg calculus) then $P$ fails to be boundary decomposing and even fails to act boundedly on $\checkH(D)$. This follows from Lemma \ref{boundednessofpseudodosnnchck}.
\item The obvious projectors on the Cauchy data space defining Dirichlet or Neumann conditions for a Laplacian are not bounded operators on the Cauchy data space by Example \ref{ex:dirichletandneumann}, but a more refined projection (projecting along the Hardy space) produces a bounded projector as in Subsection \ref{highregofdiricl}. The pseudodifferential operators in the Douglis-Nirenberg calculus that act boundedly on the Cauchy data space are characterised in Lemma \ref{boundednessofpseudodosnnchck}.
\item We give an example in Section \ref{subsec:symbcominorderone} (see page \pageref{subsec:symbcominorderone}) of a formally self-adjoint first order elliptic operator on the unit disc, with associated adapted boundary operator $A$ such that the classical pseudodifferential operator $P_\mathcal{C}-\chi^+(A)$ on the boundary (i.e. the unit circle) is of order $-1$ but does not act compactly on $\checkH(D)$. Therefore, the contrast between the approach in this paper to that in \cite{BBan} (where spectral projectors $\chi^\pm(A)$ topologise the Cauchy data space) goes beyond compact perturbations even in the first order case. The image of $\chi^+(A)$ differs substantially from the image of $P_\mathcal{C}$ -- the Hardy space -- as seen in Proposition \ref{compactnessfailure} containing an example where the two images' intersection is a finite-dimensional space of smooth functions.
\end{itemize}
\end{remark*}

Much of the work in \cite{BBan} concerning regular boundary conditions relied on graphical decompositions. The following theorem extends \cite[Theorem 2.9]{BBan} to elliptic differential operators of any order $>0$. The theorem can be found as Theorem \ref{Thm:Ell} (see page \pageref{Thm:Ell}) in the body of the text.

\begin{thm*}
Let $D$ be an elliptic differential operator of order $m>0$, $\GProj_+$ a boundary decomposing projection (see Definition \ref{boundafefidodo} on page \pageref{boundafefidodo}), and $B$ a boundary condition for $D$. The following are equivalent: 
\begin{enumerate}[(i)] 
\item \label{Thm:Ell:1}
	$B \subset \checkH(D)$ is an regular boundary condition,
\item \label{Thm:Ell:2} 
	$B$ is $\GProj_+$ graphically decomposable (see Definition \ref{def:elldecomspsHO} on page \pageref{def:elldecomspsHO}),
\item \label{Thm:Ell:3} 
	$B$ is $\GProj_+$ Fredholm decomposable in $\SobHH{{m-\frac12}}$ (see Definition \ref{Def:FP} on page \pageref{Def:FP}), 
\item \label{Thm:Ell:4} 
	$B$ is $\GProj_+$ Fredholm decomposable in $\checkH$ (see Definition \ref{Def:FPC} on page \pageref{Def:FPC}).
\end{enumerate}
\end{thm*} 

The ideas of Theorem \ref{firstofmsmsdmds} can also be applied to study higher boundary regularity. Inspired by \cite{BBan}, we say that a boundary condition $B\subseteq \checkH(D)$ is $s$-semiregular, for $s\geq 0$, if whenever $\xi\in B$ satisfies that $(1-P_\mathcal{C})\xi\in \bigoplus_{j=0}^{m-1} \SobHH{{s+m-{\frac12}}}(\Sigma;E\otimes \C^m)$ then $\xi\in \SobHH{{s+m-{\frac12}}}(\Sigma;E\otimes \C^m)$. For $s=0$, $0$-semiregularity is equivalent to semi-regularity. In \cite{BBan}, a notion of $s$-semiregular boundary conditions were introduced to prove higher boundary regularity for operators of order $m=1$ that implies our notion of $s-{\frac12}$-semiregularity for $s>{\frac12}$.  The following theorem can be found as Theorem \ref{Thm:BdyRegHigh} (see page \pageref{Thm:BdyRegHigh}) in the body of the text.

\begin{thm*}
\label{HObredogody}
Let $D$ be an elliptic operator of order $m>0$, $s\geq 0$ and $B$ a boundary condition. Then $B$ is $s$-semiregular if and only if $D_B$ is $s$-semiregular, i.e. that whenever $u\in \dom(D_{\rm B})$ satisfies $D_{\rm B}u\in \SobH{s}(M;F)$ it holds that $u\in \SobH{s+m}(M;E)$.
\end{thm*}

Theorem \ref{HObredogody} generalises results from \cite{BB,BBan} from first order to general order. The method of proof can be localised to the boundary, so in the case that $M$ is noncompact with compact boundary and $D$ is complete, it holds that $B$ is $s$-semiregular if and only if whenever $u\in \dom(D_{\rm B})$ satisfies that $D_{\rm B}u\in \SobH[\rm loc]{s}(M;F)$ it holds that $u\in \SobH[\rm loc]{s+m}(M;E)$ (where we use the notation $\SobH[\rm loc]{s}$ from \cite{BB,BBan}).

\begin{thm*}
\label{weylforgen}
Let $D$ be a formally self-adjoint elliptic operator of order $m>0$ acting on a vector bundle $E\to M$. Assume that $D$ has positive interior principal symbol $\sigma_D$ and define 
$$c_D:=\left(\frac{1}{n(2\pi)^n}\int_{S^*M} \mathrm{Tr}_E(\sigma_D(x,\xi)^{-\frac{n}{m}})\mathrm{d}x\mathrm{d}\xi\right)^{-\frac{m}{n}},$$ 
where $n$ is the dimension of $M$. Then any lower semi-bounded self-adjoint realisation $D_{\rm B}$ with $\dom(D_{\rm B})\subseteq \SobH{m}(M;E)$ is bounded from below with discrete real spectrum $\lambda_0(D_{\rm B})\leq\lambda_1(D_{\rm B})\leq\lambda_2(D_{\rm B})\leq\cdots$ satisfying that 
$$\lambda_k(D_{\rm B})=c_D k^{\frac{m}{n}}+o(k^{\frac{m}{n}}),\quad\mbox{as $k\to +\infty$}.$$
\end{thm*}

Theorem \ref{weylforgen} appears as Theorem \ref{generalwlelele} (see page \pageref{generalwlelele}) in the body of the text. The result might not be surprising as Weyl laws with error estimates are known under mild assumptions, for a brief historical overview see Section \ref{weylalallalla}. Our assumptions are even milder, and we include the result as it showcases yet another feature of regular realisations that only depend on abstract principles. The proof of Theorem \ref{weylforgen} was suggested to us by Gerd Grubb and relies on a precise description of the resolvent of $D_{\rm B}$, see Theorem \ref{resolventchar} on page \pageref{resolventchar}, and previous work of Grubb \cite{grubb68,grubb77a}. For historical note, we mention that for Laplace operators on manifolds with boundary (see for instance \cite{ivrii16} for a historical overview) and for elliptic pseudodifferential operators on closed manifolds (see for instance \cite[Chapter XXIX]{horIV} or \cite[Chapter III]{shubinsbook}), the remainder terms in the Weyl law are known more precisely.

\subsection{Overview of contents} 

Let us briefly describe the contents of the paper. 

In section \ref{section2onsetup} we set up the theory for elliptic differential operators, building heavily on an abstract viewpoint presented in Appendix \ref{bärballappsection}. The abstract theory allows us to set up boundary conditions (Subsection \ref{subsec:bcsearly}) as in \cite{BB,BBan} and characterise when they define Fredholm realisations. It allow us to make some straightforward applications to wellposedness in Subsection \ref{subsec:wellposed}. We also consider a number of examples: in Subsection \ref{examoaldldbc} we describe some classes of boundary conditions, in Subsection \ref{subsec:ellfirstorderfirst} we consider first order elliptic operators reconciling with results obtained in \cite{BBan}.
In Subsection \ref{subsec:scalprododo}, we consider scalar properly elliptic operators reconciling with more classical theory, e.g. \cite{grubb68,schechter59,lionsmagenes}.

In Section \ref{sec:calderon} we recall Seeley's work on Calderón projectors and study its implications for the Cauchy data space and its analytic structure. The required Douglis-Nirenberg calculus and Hörmander's description of the symbolic structure of a Calderón projection is described in Subsection \ref{subsec:calderonproj}. The Calderón projection from \cite{seeley65} and further consequences for boundary conditions are given in Subsection \ref{subsec:seeleyoncalderon}. In Subsection \ref{bounddecoddsp} we introduce the notion of boundary decomposing projections, the projections that characterise the analytic structure of the Cauchy data space.

In Section \ref{sec:cuchsosa} we further analyse the Cauchy data space. The results of Subsection \ref{subsec:cuchsosalocal} shows that the Cauchy data space can be locally defined and that the  Calderón projection could equally well be replaced with a projection constructed from a finite number of terms in Hörmander's construction in Subsection \ref{subsec:calderonproj}. In Subsection \ref{bpandcp} we give a precise description of the boundary pairing between the Cauchy data space of an elliptic operator and the Cauchy data space of its adjoint, allowing us to characterise adjoint boundary conditions and regularity of pseudolocal boundary conditions in Subsection \ref{subsec:adjointbcs}.

In Section \ref{subsec:charofellfirstorder} we characterise regular boundary conditions via graphical decompositions, in Section \ref{higherorderregsubsec} we characterise boundary conditions admitting higher regularity, in Section \ref{weylalallalla} we prove the Weyl law for elliptic differential operators with regular boundary conditions, and in Section \ref{sec:rigiddldlda} we prove a rigidity result for the index of Fredholm realisations of elliptic differential operators. Following this, Section \ref{subsec:symbcominorderone} consists of a computational exercise comparing a Calderón projection to the spectral projectors used in \cite{BBan}, quantifying their qualitative differences.

\subsection{Notation}
Throughout the paper, we use the analyst's inequality $a \lesssim b$ to mean that $a \leq C b$ for some constant $C$ where the dependence is apparent from context or explicitly specified.
By $a \simeq b$, we mean that $a \lesssim b$ and $b \lesssim a$. 

 The notation $+$ is used for internal sums of vector subspaces and $\oplus$ is used for direct sums. 
 For two subspaces $V_1,V_2\subseteq V$ we use the notation $V_1\oplus_WV_2=V_1+V_2$ to indicate $W:=V_1\cap V_2$, note that $V_1\oplus_WV_2\cong (V_1\oplus V_2)/d(W)$ for the diagonal embedding $d$.
In general, $\oplus$ does not imply that the sum is orthogonal. We use the notation $\oplus^\perp$ for orthogonal direct sums.
The term projection will be used for an idempotent operator on a Hilbert space; in general,  they will not be orthogonal.

All manifolds and their boundaries are assumed to be smooth. For a manifold $X$ without boundary, we write $\mathcal{D}'(X)$ for the Fréchet space of distributions, i.e. the topological dual of $\Ck[c]{\infty}(X)$. For a manifold with boundary $M$ we use the convention that the boundary is included in $M$.
We write $\Sigma:=\partial M$ and $\interior{M}:=M\setminus \Sigma$.

For an operator $T: \Hil \to \Hil$ over a Hilbert space $\Hil$, potentially unbounded, we denote its \emph{domain} by $\dom(T)$.
The \emph{kernel} and \emph{range} are then given by $\ker(T)$ and $\ran(T)$ respectively. 
The \emph{graph norm} of $T$ is $\norm{\cdot}_T = \sqrt{\norm{\cdot}^2+ \norm{T\cdot}^2}$, and the operator $T$ is said to be \emph{closed} if $(\dom(T), \norm{\cdot}_T)$ is a Banach space. Since the graph norm polarises, $(\dom(T), \norm{\cdot}_T)$ is a Hilbert space if $T$ is closed.
If $S, T$ are two operators, then we write $S \subset T$ if $\dom(S) \subset \dom(T)$ and $S = T$ on $\dom(S)$. If $S\subseteq T$ are two closed operators then $\dom(S) \subset \dom(T)$ is a closed subspace  with respect to $\norm{\cdot}_T$. 

An operator $S$ is said to be \emph{adjoint} to $T$ if $\inprod{Tu,v} = \inprod{u, Sv}$ for $u \in \dom(T)$ and $v \in \dom(S)$.
A densely-defined $T$ has a \emph{unique} adjoint $T^\ad$ with domain 
$$\dom(T^\ad) = \set{u \in \Hil: \exists C_{u, T}\quad \modulus{\inprod{u, Tv}} \leq C_{u,T} \norm{v}, \, \forall v\in \dom(T)}.$$
A closed operator $T$ is \emph{Fredholm} if it has closed range with finite dimensional $\ker(T)$ and $\ker(T^\ast)$.

The operator $T$ is said to be \emph{invertible} if it is injective with dense range and $T^{-1}: \ran(T) \to \Hil$ is a bounded map.
An invertible $T$ has a unique extension $T^{-1}: \Hil \to \Hil$.
The \emph{spectrum} of the operator $T$ are the points $\lambda \in \C$ for which $(\lambda - T)$ is not invertible.
The set of all such points are denoted by $\spec(T)$.
There are a number of non-equivalent definitions of \emph{essential spectrum}, but for us, this means the points $\lambda \in \spec(T)$ for which $(\lambda - T)$ fails to be Fredholm.
The \emph{resolvent set} is the complement of $\spec(T)$ in $\C$ and it is denoted by $\res(T)$.

A signification notion throughout this paper is that of a \emph{Formally Adjointed Pair (FAP)} (cf. Appendix \ref{bärballappsection}).
This notion was used, without being named, in \cite[Chapter II]{grubb68}. Given Hilbert spaces $\Hil_1$ and $\Hil_2$, these are a pair of densely-defined and closed  operators $(T_{\min}, T_{\min}^\dagger)$ with $T_{\min}: \Hil_1 \to \Hil_2$ and $T_{\min}^\dagger: \Hil_2 \to \Hil_1$ adjoint to each other (i.e. $T_{\min}\subseteq (T_{\min}^\dagger)^*$ and $T_{\min}^\dagger\subseteq T_{\min}^*$).
These yield the following important maximal extensions
$$T_{\max} := (T_{\min}^\dagger)^\ast\quad\mbox{and}\quad T_{\max}^\dagger := T_{\min}^\ast.$$
A realisation of a FAP is an extension $T_{\min}\subseteq \hat{T}\subseteq T_{\max}$. We will see FAPs arising from geometry and elliptic differential operators. 

It is of fundamental importance is to understand extensions, closed or otherwise, of $T_{\min}$ (and respectively $T_{\min}^\dagger$).
For this, the space 
$$\checkh_T := \faktor{\dom(T_{\max})}{\dom(T_{\min})}$$
is of fundamental importance. 
A \emph{Cauchy data space} is a pair $(\gamma, \checkH(T))$, where $\gamma: \dom(T_{\max}) \to \checkH(T)$ is a bounded surjection with $\ker \gamma = \dom(T_{\min})$.
It is clear from the open mapping theorem that $\gamma$ induces an isomorphism $\checkh(T) \cong \checkH(T)$.

\subsection{Acknowledgements}

LB was supported by SPP2026 from the German Research Foundation (DFG).
MG was supported by the Swedish Research Council Grant VR 2018-0350.
HS was supported by the Australian Research Council, through the Australian Laureate Fellowship FL170100020 held by V. Mathai.
The authors would also like to thank Christian Bär and Andreas Rosén for useful discussions.
The authors are most grateful to Gerd Grubb, who motivated and inspired us through a multitude of comments and helpful suggestions on  earlier versions of the paper and providing us with the method of proof used in Section \ref{weylalallalla}. 

\section{Elliptic differential operators and boundary conditions}
\label{section2onsetup}

\subsection{Setup}
\label{setupforellde}
Let $M$ be a compact manifold with boundary $\partial M = \Sigma$ carrying a smooth measure $\mu$. 
In our convention, $\Sigma \subset M$ and the interior of $M$ is denoted by $\interior{M} = M \setminus \Sigma$.

Let $(E,h^E) \to M$ be a hermitian vector bundle and let $\Ck{\infty}(M;E)$ denote the smooth sections over $E$. 
In particular, the support of such sections are allowed to touch the boundary.
The subspace $\Ck[c]{\infty}(\interior{M}; E)$ consists of $u \in \Ck{\infty}(M; E)$ such that $\spt u \intersect \Sigma = \emptyset$. By an abuse of notation, we write $\Ck{\infty}(\Sigma;E)$ for the space of smooth sections of $E|_\Sigma$, and similarly for other function spaces of sections.
We shall also fix a smooth interior vectorfield $\vec{T}$ transversal to $\partial M=\Sigma$ and let $\nu$ denote the smooth volume measure on $\Sigma$, induced by $\mu$ and $\vec{T}$.
In what is to follow, the coordinate and derivative, respectively, in this transversal direction will be denoted $x_n$ and $\partial_{x_n}$.
For $s \in \R$, by $\SobH{s}(M;E)$ and $\SobH{s}(\Sigma;E)$, we denote the $\Lp{2}$ Sobolev spaces of $s$ derivatives on $M$ and $\Sigma$, respectively. We let $\SobH[0]{s}(M;E)$ denote the closure of $\Ck[c]{\infty}(\interior{M};E)$ in $\SobH{s}(\hat{M};E)$ for some closed manifold $\hat{M}$ containing $M$ as an open subset with smooth boundary.
Note that $\SobH{s}(M;E)$ can be identified with the quotient $\faktor{\SobH{s}(\hat{M};E)}{\SobH[0]{s}(\hat{M}\setminus M;E)}$.

For $(F,h^F) \to M$  another hermitian vector bundle, we will be concerned with elliptic differential operators $D: \Ck{\infty}(M;E) \to \Ck{\infty}(M;F)$  of order $m>0$. 
Further, by  $D^\dagger: \Ck{\infty}(M;F) \to \Ck{\infty}(M;E)$, let us denote the formal adjoint, obtained via integration by parts (cf. \cite{BB}).
By $D_{c}$, we denote $D$ with domain $\dom(D_{c}) = \Ck[c]{\infty}(\interior{M};E)$, and analogously $D^\dagger_c$ is defined from $\dom(D_{c}^\dagger) = \Ck[c]{\infty}(\interior{M};F)$. 
The  maximal and minimal domains of $D$ are then obtained as follows: 
$$D_{\max} = (D^\dagger_{\mathrm{c}})^\ad\quad\text{and}\quad D_{\min} = \close{D_{c}}.$$
It is clear that both these operators are closed and that $D_{\min} \subset D_{\max}$.
Moreover, it holds that $D_{\rm min}^*=D_{\rm max}^\dagger$ and $D_{\rm max}^*=D_{\rm min}^\dagger$. In the language of Appendix \ref{bärballappsection}, $(D_{\rm min},D_{\rm min}^\dagger)$ is a formally adjointed pair (FAP), see Definition \ref{efonedoenfap}. The next result shows that the FAP associated with an elliptic operator of order $m>0$ is a Fredholm FAP (see Definition \ref{efonedoenfapfredholm}).

\begin{proposition}
\label{geneprop}
If $D:\Ck{\infty}(M;E)\to \Ck{\infty}(M;F)$ is an elliptic differential operator of order $m>0$ then the following holds:
\begin{enumerate}
\item \label{geneprop:1} $\dom(D_{\rm min})=\SobH[0]{m}(M,E)$;
\item \label{geneprop:2} $D_{\rm min}$ has ``compact resolvent'' in the sense that $(1+D_{\rm min}^*D_{\rm min})^{-{\frac12}}$ is a compact operator on $\Lp{2}(M,E)$;
\item \label{geneprop:3} $D_{\rm min}$ has finite-dimensional kernel, closed range and $D_{\rm max}$ has closed range. Moreover, 
$$\Lp{2}(M;E) = \ker(D_{\min}) \oplus^\perp \ran(D_{\max}^\dagger) \quad\mbox{and}\quad \Lp{2}(M;F)= \ker(D_{\min}^\dagger) \oplus^\perp \ran(D_{\max}).$$
\end{enumerate}
\end{proposition}

\begin{proof}
Item \ref{geneprop:1}  follows from the Gårding inequality showing that the graph norm of $D$ is equivalent to the $\SobH{m}(M;E)$-norm on  $\Ck[c]{\infty}(\interior{M};E)$. 
Item \ref{geneprop:2} follows from the fact that $(1+D_{\rm min}^*D_{\rm min})^{-{\frac12}}$ factors over the domain inclusion $\dom(D_{\rm min})\hookrightarrow \Lp{2}(M,E)$ which is compact by item \ref{geneprop:1} as $m > 0$. Item \ref{geneprop:3}  follows from item \ref{geneprop:2} by Proposition \ref{closedragneandofiffni}. 
\end{proof}

Using the transversal vectorfield $\vec{T}$ at the boundary, for $j=0,1,2,\ldots$, we can define a trace mapping 
$$\gamma_j:\Ck{\infty}(M,E)\to \Ck{\infty}(\Sigma,E), \quad u\mapsto \partial_{x_n}^j u|_\Sigma,$$
where $x_n$ denotes the variable transversal to the boundary defined from the vector field $\vec{T}$. 

\begin{theorem}
Let $D:\Ck{\infty}(M;E)\to \Ck{\infty}(M;F)$ be an elliptic differential operator of order $m>0$. The trace mapping 
$$\gamma:\Ck{\infty}(M;E)\to \Ck{\infty}(\Sigma;E\otimes \C^m), \quad u\mapsto (\gamma_j u)_{j=0}^{m-1},$$ 
extends to a continuous mapping 
$$\gamma:\dom(D_{\rm max})\to \bigoplus_{j=0}^{m-1} \SobH {-{\frac12}-j}(\Sigma; E),$$
with dense range and 
$$\ker\gamma= \SobH[0]{m}(M;E).$$
\end{theorem}

This result is classical and can be found in \cite{lionsmagenes63,lionsmagenes}. The result is also proven in \cite{seeley65}. 

Following the notation of Appendix \ref{bärballappsection} we use the notation 
$$\check{\mathhil{h}}_D:=\check{\mathhil{h}}_{(D_{\rm min},D_{\rm min}^\dagger)}\equiv \faktor{\dom(D_{\rm max})}{\dom(D_{\rm min})}.$$ 
We note that since $\dom(D_{\rm min})=\SobH[0]{m}(M;E)=\ker\gamma$, setting 
$$\checkH(D) := \gamma \dom(D_{\max}),$$ 
the map $\gamma$ induces a bounded inclusion
$$\check{\mathhil{h}}_D \hookrightarrow \bigoplus_{j=0}^{m-1} \SobH {-{\frac12}-j}(\Sigma, E),$$
with range $\checkH(D)$. The reader should be aware that this inclusion is not a Banach space isomorphism onto its image. 
However, by giving $\checkH(D)$ the induced topology from this inclusion, we obtain that 
$$\checkh_D \cong \checkH(D) \subseteq \bigoplus_{j=0}^{m-1} \SobH {-{\frac12}-j}(\Sigma, E).$$
Since $\checkH(D)$ is isomorphic to the Hilbert space quotient $\faktor{\dom(D_{\rm max})}{\dom(D_{\rm min})}$, the Banach space $\checkH(D)$ is in fact a Hilbert space (although we shall not make use of a specific choice of inner product). We also note that since $\SobH{m}(M;E)\subseteq \dom(D_{\max})$ we have continuous inclusions with dense ranges
$$\bigoplus_{j=0}^{m-1} \SobH {m-{\frac12}-j}(\Sigma, E)\subseteq \checkH(D)\subseteq \bigoplus_{j=0}^{m-1} \SobH {-{\frac12}-j}(\Sigma, E).$$
Similarly, we obtain $\checkh_{D^\dagger}$ and $\checkH(D^\dagger)$ and continuous inclusions with dense ranges
$$\bigoplus_{j=0}^{m-1} \SobH {m-{\frac12}-j}(\Sigma, F)\subseteq \checkH(D^\dagger )\subseteq \bigoplus_{j=0}^{m-1} \SobH {-{\frac12}-j}(\Sigma, F).$$

\subsection{Boundary conditions}
\label{subsec:bcsearly}

Any elliptic operator $D: \Ck{\infty}(M;E) \to \Ck{\infty}(M;F)$ as above defines the FAP $(D_{\rm min},D_{\rm min}^\dagger)$ in the terminology of Appendix \ref{bärballappsection}.
Therefore, realisations of elliptic operators are readily described from abstract principles as subspaces of the Cauchy data space $\checkH(D)$. 
We reiterate that realisations and boundary conditions of elliptic operators form, to say the least, a well understood topic. 
Our ambition is solely to provide a new perspective consistent with the perspective due to Bär-Ballmann in the first order case.

A \emph{generalised boundary condition} $B$ for $D$ is a subspace of $\checkH(D)$. 
It is simply called a \emph{boundary condition} if $B$ is further required to be closed.
We note that by Proposition \ref{charbodundalda}, there is a one-to-one correspondence between realisations of an elliptic operator and its generalised boundary conditions. This correspondence is set up by associating the realisation $D_{\rm B}$ defined from 
$$\dom(D_{\rm B}):=\{f\in \dom(D_{\rm max}): \gamma(f)\in B\},$$ 
with $B$, and conversely to associate the subspace $\gamma(\dom(D_e))\subseteq \checkH(D)$ with a realisation $D_e$ of $D$. 
We also note that by Proposition \ref{adjointofafap}, for any generalised boundary condition $B$ of $D$, it holds that 
$$(D_{\rm B})^*=D_{{\rm B}^*}^\dagger,$$
where $B^\ast \subset \checkH(D^\dagger)$ is the \emph{adjoint boundary condition} defined by 
$$B^*:=\{\eta\in \checkH(D^\dagger): \langle D^\dagger f,g\rangle=\langle f,Dg\rangle \;\quad \forall f\in \gamma^{-1}(\{\eta\}), \, g\in \dom(D_{\rm B})\}.$$ 
In Subsection \ref{subsec:adjointbcs} (see page \pageref{subsec:adjointbcs}) we further describe the adjoint boundary condition as an annihilator of $B$ with respect to an explicit continuous sesquilinear form $\checkH(D^\dagger)\times \checkH(D)\to \C$, compare also to Appendix \ref{app:bspaceandbc}, on page \pageref{app:bspaceandbc}.
Note, $B^\ast$ is necessarily closed.
For convenience, we shall simply refer to a realisation or boundary condition for $D$ when speaking of a realisation or boundary condition for $(D_{\rm min},D_{\rm min}^\dagger)$. 
This terminology is consistent with the terminology for boundary conditions for first order elliptic differential operators from \cite{BB,BBan}.

\begin{definition}
\label{def:elliptiiclcld}
Let $D$ be an elliptic differential operator of order $m>0$ and $D_{\rm B}$ a closed realisation.  We say that that $D_{\rm B}$ is: 
\begin{itemize}
\item \emph{semi-regular} if $\dom(D_{\rm B})\subseteq \SobH m(M;E)$, 
\item \emph{regular} if $D_{\rm B}$ and $(D_{\rm B})^*$ are semi-regular.
\end{itemize} 
In light of Proposition \ref{charbodundalda}, we say that $B$ is (semi-)~regular if $D_{\rm B}$ is (semi-)~regular.
\end{definition}

\begin{remark}
\label{ontheremineendl}
The literature on boundary value problems contains a broad range of refined notions for ellipticity and regularity. Historically, the term ellipticity has been used for a local condition at the leading symbol level (including possible boundary symbols) that implies regularity properties.  For the boundary conditions consider in this paper and in \cite{BB,BBan}, there is in general no notion of boundary symbols. In \cite{BB,BBan}, the regularity property $\dom(D_{\rm B})\subseteq \SobH m(M;E)$ was therefore termed semi-ellipticity. To better reflect the terminology in bulk of the literature and to capture the fact that $\dom(D_{\rm B})\subseteq \SobH m(M;E)$  is a regularity property, we use the term semi-regular boundary condition in this paper.
\end{remark}

\begin{proposition}
\label{charsemiellintermsofboundary}
Let $D$ be an elliptic differential operator of order $m>0$ and $B$ a boundary condition. Then $D_{\rm B}$ is semi-regular if and only if $B\subseteq \bigoplus_{j=0}^{m-1} \SobH {m-{\frac12}-j}(\Sigma; E)$.
\end{proposition}

\begin{proof}
If $D_{\rm B}$ is semi-regular, then $\gamma(\dom(D_{\rm B}))\subseteq \bigoplus_{j=0}^{m-1} \SobH {m-{\frac12}-j}(\Sigma; E)$ by the trace theorem. We conclude that $B=\gamma(\dom(D_{\rm B}))\subseteq \bigoplus_{j=0}^{m-1} \SobH {m-{\frac12}-j}(\Sigma; E)$. 

If $B\subseteq \bigoplus_{j=0}^{m-1} \SobH {m-{\frac12}-j}(\Sigma; E)$, then any $x\in \dom(D_{\rm B})$ can be written as $x=x_0+x_{\rm B}$ where $x_0\in \dom(D_{\rm min})=\SobH[0]{m}(M;E)$ and $x_{\rm B}\in \dom(D_{\rm max})$ is a pre-image of $\gamma(x)\in B$ under $\gamma$. By the trace theorem, we can take $x_{\rm B}\in \SobH m(M;E)$ since $\gamma(x)\in \bigoplus_{j=0}^{m-1} \SobH {m-{\frac12}-j}(\Sigma; E)$. We conclude that $\dom(D_{\rm B})\subseteq \dom(D_{\rm min})+\SobH m(M;E)=\SobH m(M;E)$.
\end{proof}

\begin{proposition}
\label{Prop:MaxClosed}
Let $D$ be an elliptic differential operator of order $m>0$ and $B$ a semi-regular boundary condition.
Then, $\ker(D_{\rm B})$ is finite dimensional and $\ran(D_{\rm B})$ and  $\ran(D_{\rm B}^\ast)$ are closed. 
Moreover, 
$$\Lp{2}(M;E) = \ker(D_{{\rm B}}) \oplus^\perp \ran(D_{{\rm B}}^\ast)\quad\mbox{and}\quad \Lp{2}(M;F) = \ker(D_{{\rm B}}^\ast) \oplus^\perp \ran(D_{{\rm B}}).$$
In particular, if $B$ is regular then $D_{\rm B}$ is Fredholm.
\end{proposition}

\begin{proof}
If $B$ is semi-regular, the pair $\mathcal{T}_{\rm B}:=(D_{\rm B},D^\dagger_{\rm min})$ is a FAP (indeed $D_{\rm B}\subseteq D_{\rm max}=(D^\dagger_{\rm min})^*$) and $D^\dagger_{\rm min}\subseteq D_{{\rm B}^*}^\dagger=(D_{\rm B})^*$) with compact domain inclusion $\dom(D_{\rm B})\hookrightarrow \Lp{2}(M;E)$. The claimed result now follows from Proposition \ref{closedragneandofiffni}.
\end{proof}

\begin{remark}
We return to the study of regular boundary conditions below in Subsection \ref{subsec:adjointbcs} and Subsection \ref{subsec:charofellfirstorder}.
\end{remark}

For an elliptic operator $D$, we use the notation 
$$\mathcal{C}_D:=\gamma\ker D_{\rm max}\subseteq \checkH(D)\subseteq \bigoplus_{j=0}^{m-1} \SobH {-{\frac12}-j}(\Sigma, E).$$
The space $\mathcal{C}_D$ is called the Hardy space of $D$. We shall later see that $\mathcal{C}_D\subseteq \bigoplus_{j=0}^{m-1} \SobH {-{\frac12}-j}(\Sigma, E)$ is,  in fact, closed in $\checkH(D)$. In the setting of Appendix \ref{bärballappsection}, more precisely Definition \ref{def:cauchyspacefap}, the space $\mathcal{C}_D$ is denoted by $\mathcal{C}_{(D_{\rm min},D_{\rm min}^\dagger)}$.

\begin{proposition}
\label{domainandkernelHO}
Let $D$ be an elliptic differential operator of order $m>0$ and $B\subseteq B'$ be boundary conditions. 
\begin{enumerate}[(i)]
\item It holds that $D_{{\rm B}}\subseteq D_{{\rm B}'}$ and there is a short exact sequence of Hilbert spaces
$$0\to \dom(D_{{\rm B}})\to \dom(D_{{\rm B}'})\to \faktor{B'}{B}\to 0.$$
In particular,
\begin{align*}
&\faktor{\dom(D_{\rm B})}{\dom(D_{\rm min})}\cong B,\quad\text{and}\quad \\
&\faktor{\dom(D_{\rm max})}{\dom(D_{\rm B})}\cong \faktor{\checkH(D)}{B}
\end{align*}
hold for a boundary condition $B$.
\item The inclusion and restriction mappings fit into a short exact sequence of Hilbert spaces
$$0\to \ker(D_{\rm min})\to \ker(D_{{\rm B}})\to \mathcal{C}_D\cap B\to 0.$$
In particular, we have the short exact sequence
$$0\to \ker(D_{\rm min})\to \ker(D_{\rm max})\xrightarrow{\gamma} \mathcal{C}_D\to 0.$$
\end{enumerate}
\end{proposition}

\begin{proof}
The result follows from the Propositions \ref{Prop:DomCompABS} and \ref{sesforkernelslsa}.
\end{proof}

\begin{theorem}
\label{Thm:ClosedRangeCharHO} 
Let $D$ be an elliptic differential operator of order $m>0$ and $B$ be a generalised boundary condition. Then the following holds:
\begin{enumerate}[(i)]
\item  \label{Thm:ClosedRangeCharHO1} $\ran(D_{\rm B}) = \ran(D_{{\rm B} + \Ca_D})$  and it is closed if and only if $B + \Ca_D$ is a boundary condition (i.e. closed in $\checkH(D)$). 
\item \label{Thm:ClosedRangeCharHO2} $\ker(D_{\rm B})$ is finite-dimensional if and only if $B\cap \Ca_D$ is finite-dimensional.
\item \label{Thm:ClosedRangeCharHO3} $\ran(D_{\rm B})$ has finite algebraic codimension if and only if $B + \Ca_D$ has finite algebraic codimension in $\checkH(D)$ if and only if $\ran(D_{\rm B})$ is closed and $\ran(D_{\rm B})^\perp$ is finite-dimensional.
\end{enumerate}
\end{theorem}
\begin{proof}
Item \ref{Thm:ClosedRangeCharHO1} follows from Theorem \ref{Thm:ClosedRangeCharABS},  Proposition \ref{geneprop}, and Lemma \ref{Lem:RanABS}.
Item \ref{Thm:ClosedRangeCharHO2} follows from Proposition \ref{domainandkernelHO}. 
Item \ref{Thm:ClosedRangeCharHO3} follows from Proposition \ref{Prop:FinCharABS}.
\end{proof}

\begin{corollary}
Let $D$ be an elliptic differential operator of order $m>0$ and $B\subseteq \bigoplus_{j=0}^{m-1} \SobH {m-{\frac12}-j}(\Sigma, E)$ be closed in $\checkH(D)$. Then $B+\Ca_D\subseteq \checkH(D)$ is closed.
\end{corollary}
\begin{proof}
The space $B$ defines a semi-regular boundary condition by Proposition \ref{charsemiellintermsofboundary}. By Proposition \ref{Prop:MaxClosed}, $D_{\rm B}$ has closed range so $B+\Ca_D\subseteq \checkH(D)$ is closed by Theorem \ref{Thm:ClosedRangeCharHO}.
\end{proof}

A significant consequence of Theorem \ref{Thm:ClosedRangeCharHO} is that it allows us to relate the closedness of the range of a realisation of a boundary condition to the closedness of an associated boundary condition. 
It may seem that this implicitly asserts that an operator needs to be closed in order to have closed range.
This is not the case, as highlighted by the following quintessential example.

\begin{example}
\label{funnyboudnarydad}
Consider the generalised boundary condition 
$$B_{\rm Sob} = \bigoplus_{j=0}^{m-1} \SobH {m-{\frac12}-j}(\Sigma; E),$$ 
for an elliptic operator $D$ of order $m>0$. The space $B_{\rm Sob}$ is not a closed subspace of $\checkH(D)$, so the associated operator $D_{{\rm B}_{\rm Sob}}$ is not a closed operator. One can in fact show that $\overline{D_{{\rm B}_{\rm Sob}}}=D_{{\rm B}_{\rm Sob}+\mathcal{C}}=D_{\max}$. We can prove that $D_{{\rm B}_{\rm Sob}}$ has closed range and as such, $B_{\rm Sob}+\mathcal{C}_D$ is closed in $\checkH(D)$. We return to this example below in in Example \ref{funnyboudnarydadreturn} below and  prove, amongst other things,  that $B_{\rm Sob}+\mathcal{C}_D=\checkH(D)$.

Let us prove that $D_{{\rm B}_{\rm Sob}}$ has closed range. By a similar argument as in Proposition \ref{charsemiellintermsofboundary}, we see that 
$$\dom(D_{{\rm B}_{\rm Sob}})=\SobH m(M;E).$$
We can assume that $M$ is a domain with smooth boundary in a compact manifold $\hat{M}$, that $E,F$ extend to Hermitian vector bundles $\hat{E},\hat{F}\to \hat{M}$ and that there exists an elliptic differential operator $\hat{D}$ of order $m$ extending $D$ to $\hat{M}$. Let $\hat{T}\in \Psi^{-m}_{\rm cl}(\hat{M};\hat{F},\hat{E})$ denote a parametrix of $\hat{D}$ such that $\hat{D}\hat{T}-1$ is a projection onto the finite-dimensional space $\ker(\hat{D}^\dagger)$ and $\hat{T}\hat{D}-1$ is a projection onto the finite-dimensional space $\ker(\hat{D})$. We define the continuous operator $T:=\hat{T}|:\Lp{2}(M;F)\to \SobH m(M;E)$ which by construction inverts $D: \SobH m(M;E)\to \Lp{2}(M;F)$ (viewed as a continuous operator) from the right up to smoothing operators. From the compactness of $(1-DT)$ and Fredholm's theorem, we conclude that 
$$\ran(D_{{\rm B}_{\rm Sob}})=\ran(D: \SobH m(M;E)\to \Lp{2}(M;F))\supseteq \ran(DT: \Lp{2}(M;F)\to \Lp{2}(M;F)),$$ 
has finite algebraic codimension and is closed in $\Lp{2}(M;F)$.
\end{example} 

By understanding closedness of the ranges of an operator, we are able to understand which boundary conditions yield Fredholm operators.
In applications to index theory, it is essential to understand which boundary conditions yield such realisations.
This is described in the following theorem.

\begin{theorem}
\label{Thm:FredChar}
Let $D$ be an elliptic operator of order $m>0$ and $B$ be a boundary condition. Then the following are equivalent:
\begin{enumerate}[(i)]
\item \label{Thm:FredChar1} 
$(B, \Ca_D)$ is a Fredholm pair in $\checkH(D)$; 
\item \label{Thm:FredChar2}  
$D_{\rm B}$ is a Fredholm operator.  
\end{enumerate}
If either of these equivalent conditions hold, we have that
$$ B^\ast \cap \Ca_{D^\dagger} \cong \faktor{\checkH(D)}{(B + \Ca_D)}.$$
Moreover, 
$$\indx(D_{\rm B}) = \indx(B, \Ca_D) + \dim \ker (D_{\min}) - \dim \ker (D_{\min}^\dagger).$$
\end{theorem}

\begin{proof}
Follows from Theorem \ref{Thm:FredCharABS}.
\end{proof}

\begin{remark}
\label{hominvariancremakd}
A direct corollary of Theorem \ref{Thm:FredChar} is that whenever $B$ and $B'$ are two Fredholm boundary condition that are norm-continuously homotopic via Fredholm boundary conditions\footnote{Norm-continuously homotopic in the sense that $B$ and $B'$ are images of norm-continuously homotopic projectors where all projections along the homotopy define Fredholm boundary conditions.} then it holds that 
$$\indx(D_{\rm B})=\indx(D_{{\rm B}'}).$$
However, Theorem \ref{rigidityformaula} below provides a more general result concerning homotopy invariance so we omit the details of this argument.
\end{remark}

\begin{corollary}
\label{ellipticimpliesfredholmatbc}
Let $D$ be an elliptic operator of order $m>0$. Assume that  $B\subseteq \bigoplus_{j=0}^{m-1} \SobH {m-{\frac12}-j}(\Sigma; E)$ is closed in $\checkH(D)$ and satisfies that $B^*\subseteq  \bigoplus_{j=0}^{m-1} \SobH {m-{\frac12}-j}(\Sigma; F)$. Then $(B, \Ca_D)$ is a Fredholm pair in $\checkH(D)$ and 
$$ B^\ast \cap \Ca_{D^\dagger} \cong \faktor{\checkH(D)}{(B + \Ca_D)}.$$
\end{corollary}

\begin{proof}
Follows from Theorem \ref{Thm:FredChar} by noting that our assumptions are equivalent to $B$ being regular which implies that it is Fredholm by Proposition \ref{Prop:MaxClosed}.
\end{proof}

\subsection{The boundary pairing}
\label{subsec:preliminaryboundary}

Consider $\checkh_D$ with the quotient map $\dom(D_{\max}) \to \faktor{\dom(D_{\max})}{\dom(D_{\min})} = \checkh_D$ as a Cauchy data space.
Then, by abstract nonsense, the boundary pairing
$$\omega_{D}:\check{\mathhil{h}}_{D^\dagger}\times \check{\mathhil{h}}_{D}\to \C,$$
defined for $f\in \dom (D^\dagger_{\rm max})$, $g\in \dom (D_{\rm max})$ as
$$\quad \omega_D([f],[g]):=\langle D^\dagger f,g\rangle_{\Lp{2}(M;E)}-\langle f,Dg\rangle_{\Lp{2}(M;F)},$$
is well-defined and continuous. By an abuse of notation, we consider the boundary pairing as a sesquilinear mapping $\omega_D:\checkH(D^\dagger)\times \checkH(D)\to \C$. 

To describe $\omega_D$ on $\checkH(D)$ and $\checkH(D^\dagger)$ in greater detail, we introduce the following $m\times m$-matrix
\begin{equation}
\label{taudeffed}
\tau=
\begin{pmatrix} 
0&0&\cdots &0& 0&1\\
0&0&\cdots &0& 1&0\\
0&0&\cdots &1& 0&\vdots \\
\vdots &\vdots&\Ddots &0&\vdots & \\
0&1&\cdots &0& 0&0\\
1&0&\cdots &0& 0&0
\end{pmatrix}
\end{equation}
We introduce the notation $A_l=(A_l(x_n))_{x_n\in [0,1]}$ for the family of differential operators $\Ck{\infty}(\Sigma,E)\to \Ck{\infty}(\Sigma,F)$ of order $l$ such that 
$$D=\sum_{l=0}^m A_lD_{x_n}^{m-l},$$
near $\Sigma$. 
The following proposition follows from the computations on \cite[page 794]{seeley65}. See also \cite[Proposition 1.3.2]{gerdsgreenbook}.

\begin{proposition}
\label{proponboundarypairingforell}
Let $D$ be an elliptic differential operator of order $m>0$. There exists an $m\times m$ matrix of differential operators 
$$\widetilde{\scalebox{1.5}{a}}=\begin{pmatrix} 
A_0&0&\cdots &0& 0&0\\
A_{1,m-1}&A_0&\cdots &0& 0&0\\
A_{2,m-2}&A_{1,m-2}&\ddots &0& \vdots&\vdots \\
\vdots &\vdots&\ddots &A_0&0 &0 \\
A_{m-2,2}&A_{m-3,2}&\cdots &A_{1,2}& A_0&0\\
A_{m-1,1}&A_{m-2,1}&\cdots &A_{2,1}& A_{1,1}&A_0
\end{pmatrix},$$
where $A_{j,k}$ is a differential operator of order $j$ with the same principal symbol as $A_j|_{x_n=0}$, such that
$$\scalebox{1.5}{a}:=-i\tau \widetilde{\scalebox{1.5}{a}}=-i\begin{pmatrix} 
A_{m-1,1}&A_{m-2,1}&\cdots &A_{2,1}& A_{1,1}&A_0\\
A_{m-2,2}&A_{m-3,2}&\cdots &A_{1,2}& A_0&0\\
A_{m-3,3}&\Ddots& &A_0& 0&\vdots \\
\vdots &&\Ddots &0&\vdots & \\
A_{1,m-2}&A_0&\Ddots && &\\
A_0&0&\cdots &0& 0&0
\end{pmatrix},$$ 
satisfies
\begin{equation*}
\langle D^\dagger u,v\rangle_{\Lp{2}(M;E)}-\langle u, Dv\rangle_{\Lp{2}(M,F)}=\langle \gamma u,\scalebox{1.5}{a} \gamma v\rangle_{\Lp{2}(\Sigma;F\otimes \C^m)},
\end{equation*}
for all $u\in \Ck{\infty}(M;F)$ and $v\in \Ck{\infty}(M;E)$. 
\end{proposition}

The fact that the order of the differential operators appearing in $\widetilde{\scalebox{1.5}{a}}$ remains constant along diagonals in the matrix shows that, in fact, $\widetilde{\scalebox{1.5}{a}}\in \Psi^{\pmb 0}_{\rm cl}(\Sigma;E\otimes \C^m,F\otimes \C^m)$ via the Douglis-Nirenberg calculus we review in Subsection \ref{subsec:calderonproj}. 
The reader should note that we abuse notation and identify $A_0$ with $A_0|_{x_m=0}$ which is of order $0$, i.e. $A_0\in \Ck{\infty}(\Sigma; \mathrm{Hom}(E,F))$. 
Since $A_0$ is the restriction of the principal symbol to $x_n=0$, $\xi'=0$ and $\xi_n=1$, ellipticity implies that $A_0:E|_\Sigma\to F|_\Sigma$ is a vector bundle isomorphism.

\begin{lemma}
\label{someexresiosnaoadwa}
Let $D$ be an elliptic differential operator of order $m>0$. Let $\widetilde{\scalebox{1.5}{a}}$ and $\widetilde{\scalebox{1.5}{a}}_\dagger$ denote the matrices of differential operators constructed from $D$ and $D^\dagger$, respectively, as in Proposition \ref{proponboundarypairingforell}.  Recalling that $\scalebox{1.5}{a} = -i \tau \widetilde{\scalebox{1.5}{a}}$, it holds that
$$\scalebox{1.5}{a}^*+\scalebox{1.5}{a}_\dagger=0.$$
\end{lemma}

\begin{proof}
The lemma follows from that $ \scalebox{1.5}{a}$ defines the boundary pairing as in Proposition \ref{proponboundarypairingforell} and the symmetry condition on the boundary pairing in Proposition \ref{boundarypairingABS}.
\end{proof}

\begin{lemma}
\label{somecontinuuisuforata}
The operators $\widetilde{\scalebox{1.5}{a}}$ (from Proposition \ref{proponboundarypairingforell}) and $\tau$ satisfy the following.
\begin{enumerate}[(i)]
\item \label{somectslem1} 
	For any $s\in \R$, $\tau$ defines a unitary isomorphism 
	$$\bigoplus_{j=0}^{m-1}\SobH {s+j}(\Sigma;F)\to \bigoplus_{j=0}^{m-1}\SobH {m-1+s-j}(\Sigma;F)$$
\item \label{somectslem2}  
	The matrix of differential operators $\widetilde{\scalebox{1.5}{a}}$ is invertible as a matrix of differential operators and 
	$$\widetilde{\scalebox{1.5}{a}}^{-1}=
	\begin{pmatrix} 
	A_0^{-1}&0&\cdots &0& 0&0\\
	B_{1,m-1}&A_0^{-1}&\cdots &0& 0&0\\
	B_{2,m-2}&B_{1,m-2}&\ddots &0& \vdots&\vdots \\
	\vdots &\vdots&\ddots &A_0^{-1}&0 &0 \\
	B_{m-2,2}&B_{m-3,2}&\cdots &B_{1,2}& A_0^{-1}&0\\
	B_{m-1,1}&B_{m-2,1}&\cdots &B_{2,1}& B_{1,1}&A_0^{-1}
	\end{pmatrix},$$
	where $B_{j,k}$ is a differential operator of order $j$.
\item \label{somectslem3} 
	For any $s\in \R$, $\widetilde{\scalebox{1.5}{a}}$ defines a Banach space isomorphism 
	$$\bigoplus_{j=0}^{m-1}\SobH {s-j}(\Sigma;E)\to \bigoplus_{j=0}^{m-1}\SobH {s-j}(\Sigma;F),$$
	and $\scalebox{1.5}{a} = -i \tau \widetilde{\scalebox{1.5}{a}} $ defines a Banach space isomorphism 
	$$\bigoplus_{j=0}^{m-1}\SobH {s+j}(\Sigma;E)\to \bigoplus_{j=0}^{m-1}\SobH {m-1+s-j}(\Sigma;F),$$
\end{enumerate}
\end{lemma}

\begin{proof}
Item \ref{somectslem1}  is immediate. 
Item \ref{somectslem2}  follows from that $A_0^{-1}\widetilde{\scalebox{1.5}{a}}$ is lower triangular with the identity operator on the diagonal, and as such, the inverse of $A_0^{-1}\widetilde{\scalebox{1.5}{a}}$ is computed from polynomial operations on its entries through Gaussian elimination. 
The fact that the differential operator $B_{j,k}$ is of order $j$ follows from a short inspection of the Gaussian elimination. 
To prove item \ref{somectslem3}, we note that any matrix $(C_{j,k})_{j,k=0}^{m-1}$ of differential operators with the order of $C_{j,k}$ being $j-k$, acts continuously on $\bigoplus_{j=0}^{m-1}\SobH {s-j}(\Sigma;F)$. 
Now item \ref{somectslem3} follows from item \ref{somectslem2}.
\end{proof}

Following Appendix \ref{bärballappsection}, we use the notation 
$$\omega_D:\checkH(D^\dagger)\times\checkH(D)\to \C,$$
for the sesquilinear boundary pairing
$$\omega_D(\eta,\xi):=\langle D^\dagger u,v\rangle_{\Lp{2}(M;E)}-\langle u, Dv\rangle_{\Lp{2}(M,F)},$$
for $u\in \gamma^{-1}(\{\eta\})$ and $v\in \gamma^{-1}(\{\xi\})$. This pairing is well defined and non-degenerate by Proposition \ref{boundarypairingABS}. Below in Theorem \ref{boundarypairing}, we shall see that it is perfect. For now we settle with the pairing being perfect on the Sobolev spaces.

\begin{proposition}
\label{boundarypairingforsoblov}
The boundary pairing $\omega_D:\checkH(D^\dagger)\times\checkH(D)\to \C$
extends by continuity to a perfect pairing 
$$\omega_{D,s} :\bigoplus_{j=0}^{m-1} \SobH {-s-j}(\Sigma; F)\times \bigoplus_{j=0}^{m-1} \SobH {m-1+s-j}(\Sigma; E)\to \C,$$
for any $s\in \R$ with 
$$\omega_{D,s}(\eta,\xi)=\langle \eta,\scalebox{1.5}{a} \xi\rangle_{\Lp{2}(\Sigma;F\otimes \C^m)},$$
for $\eta\in \bigoplus_{j=0}^{m-1} \SobH {-s-j}(\Sigma; F)$ and $\xi\in \bigoplus_{j=0}^{m-1} \SobH {m-1+s-j}(\Sigma; E)$.
\end{proposition}

\begin{proof}
By Proposition \ref{proponboundarypairingforell}, $\omega_D([\xi],[\xi'])=\langle \xi,\scalebox{1.5}{a} \xi'\rangle_{\Lp{2}(\Sigma;F\otimes \C^m)}$ for $\xi\in \Ck{\infty}(\Sigma;F)$ and $\xi'\in \Ck{\infty}(\Sigma;E)$. 
The statement now follows from Lemma \ref{somecontinuuisuforata}.
\end{proof}

\subsection{Well-posedness in the absence of kernels}
\label{subsec:wellposed}

In the situation of boundary conditions without kernels, we are able to understand some properties of associated PDEs.

\begin{proposition}
Let $D$ be an elliptic operator of order $m>0$. Suppose that $B$ is a semi-regular boundary condition such that $\ker(D_{{\rm B}}) = 0$.
Then, the spectrum of $D_{\rm B}^\ast D_{{\rm B}}$ consists solely of discrete spectrum in $(0,\infty)$ and letting $\lambda_1$ denote its smallest nonzero eigenvalue, we have the Poincaré inequality
$$\sqrt{\lambda_1}  \norm{u}_{\Lp{2}(M;E)} \leq  \norm{D_{{\rm B}} u}_{\Lp{2}(M;F)}$$ 
for all $u \in \dom(D_{{\rm B}})$. 
\end{proposition}

Under an additional hypothesis on $B$, we give a precise description of the spectral asymptotics of $D_{\rm B}^\ast D_{{\rm B}}$ in Corollary \ref{weylformodulelososo} below.

\begin{proof}
The operator $D_{\rm B}^\ast D_{{\rm B}}$ satisfies $\dom(D_{\rm B}^\ast D_{{\rm B}}) \subset \SobH{m}(M;E)$, and moreover, it is easy to see that it is non-negative self-adjoint.
By the Rellich embedding theorem, we obtain that $\spec(D_{\rm B}^\ast D_{{\rm B}})$ is discrete, real, non-negative.
Also, $\ker(D_{\rm B}^\ast D_{{\rm B}}) =  \ker(D_{{\rm B}}) = 0$ and therefore, $\spec(D_{\rm B}^\ast D_{{\rm B}}) = \set{0 < \lambda_1 \leq \lambda_2 \leq \dots}$.
Then, via numerical range considerations,
\begin{equation*}
\label{Eq:Poin1}
\inprod{D_{\rm B}^\ast D_{{\rm B}} u,u} \geq \lambda_1 \norm{u}^2.
\end{equation*}
for $u \in \dom(D_{\rm B}^\ast D_{{\rm B}})$.
However, $\dom(D_{\rm B}^\ast D_{\rm B})$ is dense in $\dom(\sqrt{ D_{\rm B}^\ast D_{\rm B}}) = \dom(D_{\rm B})$ and therefore, we obtain the desired inequality.
\end{proof}

\begin{proposition}
\label{Cor:WellPosed} 
Let $D$ be an elliptic operator of order $m>0$. Suppose that $B$ is a semi-regular boundary condition  and that $\ker(D_{\rm B}) = 0$.
Then,
\begin{equation*}
\begin{cases}
D_{\rm B} u = f\\ 
\gamma(u) \in B
\end{cases}
\end{equation*}
for $f \in \ran(D_{{\rm B}})$ is well-posed.
\end{proposition}

\begin{proof}
By Proposition \ref{Prop:MaxClosed}, we have that $D_{\rm B}$ and $D_{\rm B}^\ast$ has closed range.
But $\Lp{2}(M;E) = \ker(D_{\rm B}) \oplus \ran(D_{\rm B}^\ast) = \ran(D_{\rm B}^\ast)$, and so on applying   \cite[Corollary 1, Chapter VII, Section 5]{Yosida}, we obtain that 
$$D_{{\rm B}}: \dom(D_{\rm B}) \subset \Lp{2}(\Sigma;E) \to \ran(D_{\rm B})$$
has a bounded inverse.
Therefore, 
$$ \norm{u}_{D_{\rm B}} \simeq \norm{u} + \norm{Du} = \norm{{D_{{\rm B}}^{-1} f}} + \norm{f} \lesssim \norm{f}.$$
This is exactly that the problem in the statement of the corollary is well-posed.
\end{proof}

We apply this to understand the well-posedness of the Dirichlet problem.
An important and well known notion is  \emph{weak UCP (unique continuation property)} for an operator $D$. 
This means that for smooth $u$, when $Du = 0$ and $u = 0$ on a nonempty open subset $\Omega \subset M$, then $u = 0$. 
Similarly, \emph{weak inner UCP} is satisfied by $D$ if for smooth $u$, when  $Du = 0$ and $u\rest{\Sigma} = 0$, this implies $u = 0$.

\begin{corollary}
Let $D$ be an elliptic differential operator satisfying weak inner UCP.
Then,
$$\begin{cases}
D u = f\\ 
\gamma(u) =0
\end{cases}$$
is well-posed on $\ran(D_{\Ca_D})=\ran(D_{\min})$.
\end{corollary}

\begin{proof}
We obtain well-posedness on $\ran(D_{\min})$ from Proposition \ref{Cor:WellPosed} with the semi-regular boundary condition $B = 0$ upon showing that $\ker(D_{\min}) = 0$.
By elliptic regularity, $\ker(D_{\min}) \subset \Ck{\infty}(M;E)$ so the weak inner UCP property  yields that $\ker(D_{\min}) = 0$.
Lastly, we have that $\ran(D_{\min}) = \ran(D_{0 + \Ca_D}) = \ran(D_{\Ca_D})$ from Theorem \ref{Thm:ClosedRangeCharHO} item \ref{Thm:ClosedRangeCharHO1}. 
\end{proof} 

Now, let us focus on the first order case. In the first order case, $D_{\rm min}$ corresponds to the Dirichlet problem for $D$.
As discussed in \cite[Section 1.2]{BBL2009}, weak UCP implies weak inner UCP for elliptic first order differential operators.
However, focusing on Dirac-type operators, the following holds.

\begin{corollary}
\label{firstordodereowl}
Let $D$ be  first order and Dirac-type. 
Then,
$$\begin{cases}
D u = f\\ 
\gamma(u) =0
\end{cases}$$
is well-posed on $\ran(D_{\Ca_D})=\ran(D_{\min})$.
\end{corollary}

\begin{proof}
As mentioned in \cite[Remark 2.2]{BBL2009}, it is well known that all first order Dirac-type operators satisfy weak UCP. 
This, in turn, implies weak inner UCP, and therefore, the conclusion follows.
\end{proof}

\begin{remark}
It is a well known fact that well-posedness fails for the Dirichlet problem for first order Dirac-type operators on $\Lp{2}$.
However, the corollary shows that on a very large space, well-posedness actually holds.
In fact, by Proposition \ref{Prop:MaxClosed}, it is easily seen that the obstruction to well-posedness is indeed large. 
It is precisely $\ker(D^\dagger_{max})=\ker(D_{\min}^*)$.
On a manifold of dimension exceeding 1, we see from Theorem \ref{infdimeker} that this is, indeed, an infinite dimensional space.
That is, in general, the obstruction to well-posedness for the Dirichlet problem is infinite dimensional.
\end{remark}

\subsection{Local, pseudo-local and bundle-like boundary conditions}
\label{examoaldldbc}

Let us give some examples of boundary conditions.

\begin{definition}
\label{somedefinidodod}
Let $D$ be an elliptic operator of order $m>0$ and $B$ a boundary condition. 
\begin{itemize}
\item We say that $B$ is pseudo-local if there exists an $m\times m$ matrix $P$ of pseudo-differential operators on $E|_\Sigma$ with $P^2=P$ and 
$$B=\checkH(D) \cap \ker(P:\mathcal{D}'(\Sigma;E\otimes \C^m)\to \mathcal{D}'(\Sigma;E\otimes \C^m)).$$ 
We say that the pseudo-local boundary condition $B$ is defined from $P$,  and to indicate the dependence on $P$ we sometimes use the notation $B_P$ for $B$.
\item We say that $B$ is local if there exists a differential operator $b$ on $E|_\Sigma\otimes \C^m\to E'$ (for some vector bundle $E'\to \Sigma$) such that 
$$B=\checkH(D) \cap \ker(b:\mathcal{D}'(\Sigma;E\otimes \C^m)\to \mathcal{D}'(\Sigma;E')).$$ 
We say that the local boundary condition $B$ is defined from $b$.
\item If $B$ is a local boundary condition defined from $b$ on $E|_\Sigma\otimes \C^m\to E'$ where $E'\subseteq E|_\Sigma\otimes \C^m$ is a subbundle and $b$ is a projection onto a complement of $E'$, we say that the associated boundary condition is bundle-like and defined from $E'$.
\end{itemize}
\end{definition}

 In \cite{BBan}, the bundle-like boundary conditions were called local. However, in the higher order case more general local boundary conditions naturally arise and are well studied, e.g. in \cite{lionsmagenes,grubb68,vishik}. 

The reader should note that a bundle-like boundary condition is both local and pseudo-local. If $B$ is a local boundary condition defined from $b$, then 
$$\dom(D_{\rm B})=\{u\in \dom(D_{\max}): b\gamma u=0\}.$$
Here $b\gamma u$ is interpreted in a distributional sense. In particular, $B$ is indeed a boundary condition because the topology on $\checkH(D)$ induced from the space of distributions is weaker than the norm topology of $\checkH(D)$. A similar argument shows that a pseudo-local boundary condition indeed is a boundary condition.

\begin{example}
\label{kindofdirhcoell}
For $0\leq k\leq m$, we consider the order zero operator $b$ that projects onto the first $k$ coordinates of $E|_\Sigma\otimes \C^m$. Let $B_k$ be the associated bundle-like boundary condition. We have that 
\begin{align*}
B_k=&\{\xi\in \checkH(D): \xi_0=\xi_1=\ldots=\xi_{k-1}=0\},\quad\mbox{and}\\
\dom(D_{{\rm B}_k})=&\{u\in \dom(D_{\rm max}): \gamma_0 u=\gamma_1u=\ldots \gamma_{k-1}u=0\}=\dom(D_{\rm max})\cap \SobH[0]{k}(M).
\end{align*}
\end{example}

\begin{proposition}
\label{describinglocalcndldldaadjoint}
Let $D$ be an elliptic operator of order $m>0$ and $B_P$ a pseudo-local boundary condition defined from $P$. Then $(B_P)^*$ is the pseudo-local boundary condition $B_{P_\dagger}$ for $D^\dagger$ defined from $P_\dagger:=(\scalebox{1.5}{a}^*)^{-1}(1-P^*)\scalebox{1.5}{a}^*$ where $\scalebox{1.5}{a}$ is the differential operator from Proposition \ref{proponboundarypairingforell}.
\end{proposition}

\begin{proof}
We remark that $(\scalebox{1.5}{a}^*)^{-1}(1-P^*)\scalebox{1.5}{a}^*$ is again an idempotent matrix of pseudodifferential operators because $(\scalebox{1.5}{a}^*)^{-1}$ is a matrix of differential operators by Lemma \ref{somecontinuuisuforata}. The space $B^*$ is characterised by the property that $\eta\in B^*$ if and only if $\omega_D(\eta,\xi)=0$ for all $\xi\in B$. We need to prove that 
$$B^*=\checkH(D^\dagger) \cap \ker( (1-P^*)\scalebox{1.5}{a}^*:\mathcal{D}'(\Sigma;F\otimes \C^m)\to \mathcal{D}'(\Sigma;E\otimes \C^m)).$$

Since $B=\checkH(D) \cap \ker(P:\mathcal{D}'(\Sigma;E\otimes \C^m)\to \mathcal{D}'(\Sigma;E\otimes\C^m))$, any $\xi \in B$ satisfies $\xi=(1-P)\xi$. In particular, if $\eta\in \checkH(D^\dagger) \cap \ker( (1-P^*)\scalebox{1.5}{a}^*:\mathcal{D}'(\Sigma;F\otimes \C^m)\to \mathcal{D}'(\Sigma;E\otimes \C^m))$ and $\xi\in B$ then 
$$\omega_D(\eta,\xi)=\langle \eta,\scalebox{1.5}{a} (1-P)\xi\rangle_{\Lp{2}(\Sigma;F\otimes \C^m)}=\langle(1-P^*)\scalebox{1.5}{a}^*\eta,\xi\rangle_{\Lp{2}(\Sigma;F\otimes \C^m)}=0.$$
We conclude that $\checkH(D^\dagger) \cap \ker( (1-P^*)\scalebox{1.5}{a}^*:\mathcal{D}'(\Sigma;F\otimes \C^m)\to \mathcal{D}'(\Sigma;E\otimes \C^m))\subseteq B^*$. 

To prove the converse inclusion, if $\eta\in B^*$ then for any $\xi_0\in \Ck{\infty}(\Sigma;E\otimes \C^m)$ we have that 
$$\omega_D(\eta,(1-P)\xi_0)=\langle(1-P^*)\scalebox{1.5}{a}^*\eta,\xi_0\rangle_{\Lp{2}(\Sigma;F\otimes \C^m)}.$$
Since $\xi:=(1-P)\xi_0\in B$, we have that $\langle(1-P^*)\scalebox{1.5}{a}^*\eta,\xi_0\rangle_{\Lp{2}(\Sigma;F\otimes \C^m)}=0$ for any $\xi_0\in \Ck{\infty}(\Sigma;E\otimes \C^m)$. In particular, $(1-P^*)\scalebox{1.5}{a}^*\eta=0$ in distributional sense if $\eta\in B^*$. Therefore $\checkH(D^\dagger) \cap \ker( (1-P^*)\scalebox{1.5}{a}^*:\mathcal{D}'(\Sigma;F\otimes \C^m)\to \mathcal{D}'(\Sigma;E\otimes \C^m))\supseteq B^*$.

\end{proof}

\begin{example}
Returning to Example \ref{kindofdirhcoell}, assume for simplicity that $E=F$ and that the elliptic operator $D$ takes the form $D=D_{x_n}^m+A$ near the boundary for a differential operator $A$ on $\Sigma$ (with no dependence on $x_n$). A short computation with integration by parts show that $\scalebox{1.5}{a}=0^*\sigma(A)\otimes \tau$ where $0:\Sigma\to T^*\Sigma$ denotes the zero section and $\tau$ is as in Equation \eqref{taudeffed}. Using this it is readily verified via Proposition \ref{describinglocalcndldldaadjoint} that $B_k^*=B_{m-k-1}$ for $k=0,\ldots, m$ with the convention that $B_{-1}=0$. If the order of $D$ is even, then $B_{m/2}$ should be interpreted as a Dirichlet condition. We return to this condition below in Subsection \ref{subsec:scalprododo} for scalar properly elliptic operators.
\end{example}

\subsection{Elliptic first order operators}
\label{subsec:ellfirstorderfirst}

The situation on which we have modelled our description of boundary conditions is that of first order operators described in \cite{BB,BBan}. 
For simplicity, we continue to focus on compact manifolds with boundary, while \cite{BB,BBan} allowed for non-compact manifolds with compact boundary.
The theory of boundary value problems for first order elliptic operators is further described in \cite{G99} and the special case of Dirac operators is studied in \cite{bosswojc}.

We follow the setup of Subsection \ref{setupforellde} and, throughout this subsection, consider an elliptic differential operator $D$ order $m=1$. Associated to such an operator $D$, there are adapted boundary operators $A$ on the bundle $E|_\Sigma$ which are first order elliptic differential operators on $\Sigma$ whose principal symbols satisfy
$$ \sym(A)(x,\xi') = \sym(D)(x,\xi'=0,\xi_n=1)^{-1} \comp \sym(D)(x,\xi').$$
Alternatively, an adapted boundary operator $A$ is an operator on $\Sigma$ such that near $\Sigma$ we can write
$$D=\sigma\cdot (\partial_{x_n}+A+R_0),$$
for an $x_n$-dependent first order differential operator $R_0$ on $\Sigma$ such that $R_0|_{x_n=0}$ is of order zero. Here we have shortened the notation $\sigma=\sym(D)(x,\xi'=0,\xi_n=1)$. The reader should note that in the notation of Proposition \ref{proponboundarypairingforell}, 
$$\sigma=A_0=\scalebox{1.5}{a} \quad\mbox{and}\quad A_1= \sigma\cdot (A+R_0).$$

In \cite{BBan} by Bär and Bandara, the authors show that there exists an \emph{admissible cut}, which is an $r \in \R$ so that $A_r := A -r$ is invertible $\omega_r$-bisectorial.
This means that $\spec(A_r)$ is contained in a closed bisector $S_{\omega_r} = \set{\zeta \in \C: \pm\arg \zeta \leq \omega_r}$ of angle $\omega_r < \pi/2$ in the complex plane, and that for $\mu \in (\omega_r, \pi/2)$, there exists a constant $C_\mu > 0$ so that whenever $\zeta$ is outside of the closed bisector $S_{\mu}$ of angle $\mu$, we have that $\modulus{\zeta}\norm{(\zeta - A_r)^{-1}} \leq C_\mu$.
The framework in \cite{BBan} extends the work of Bär and Ballmann in \cite{BB} where they make the additional assumption that $A$ can be chosen self-adjoint.

Given that the spectrum of $A_r$ is contained in the open left and right half-planes, 
a consequence of the bisectoriality of $A_r$ is that we can consider spectral projectors $\chi^{\pm}(A_r)$, where $\chi^{\pm}(\zeta) = 1$ for $\pm \Re \zeta > 0$ and $0$ otherwise.
Not only do these  spectral projectors exist, they  are pseudo-differential operators of order zero (see \cite{grubbsec}).  
This means that they act boundedly across all Sobolev scales and we define the space
$$\checkH_A(D) := \chi^-(A_r)\SobH{\frac{1}{2}}(\Sigma;E) \oplus \chi^+(A_r) \SobH{-\frac{1}{2}}(\Sigma;E).$$
The following is an important property of this space that we will use later. We refer its proof to Appendix \ref{appondimensiosnsns}.

\begin{lemma}
\label{Lem:CheckNonZero} 
Let $A$ be an adapted boundary operator for a first order elliptic operator acting on a manifold with boundary of dimension $>1$. For any admissible cut $r \in \R$, the spaces $\chi^{\pm}(A_r)\Lp{2}(\Sigma;E)$ are infinite dimensional. 
Moreover, the space $\checkH_A(D) \neq \set{0}$ is an infinite-dimensional space with 
$$\SobH{{\frac12}}(\Sigma;E)\subseteq \checkH_A(D) \subseteq \SobH{-{\frac12}}(\Sigma;E),$$
and each inclusion is dense.
\end{lemma}

In \cite[Theorem 2.3]{BBan}, the trace map $\gamma:u \mapsto u\rest{\Sigma}$, initially defined on $\Ck{\infty}(\Sigma;E)$, is extended to a bounded surjection
$$u \mapsto u\rest{\Sigma} : \dom(D_{\max}) \to \checkH_A(D)$$
with kernel $\ker (u \mapsto u\rest{\Sigma}) = \dom(D_{\min})$.
That is, $\checkH(D) = \checkH_A(D)$ with $\gamma$ is a Cauchy data space.
The topology of $\checkH(D)$ is given purely in terms of an associated differential operator on the boundary, namely, the adapted boundary operator $A$. We give a direct argument for the equality $\checkH(D) = \checkH_A(D)$ below in Example \ref{firstorderbounddecom}.

A pseudo-differential projection $P$ of order zero, for which $1-P - \chi^+(A_r)$ is a Fredholm operator induces a regular boundary condition via $B_P = (1-P)\SobH{\frac{1}{2}}(\Sigma;E)$.
In particular, for any admissible cut $r \in \R$, $P_r = \chi^+(A_r)$ is such a projection.
That is, $B_r = \chi^{-}(A_r)\SobH{\frac{1}{2}}(\Sigma;E)$ is a regular boundary condition. In particular, the  (generalised)  APS-realisation $D_{\rm APS}$, defined from 
$$\dom(D_{\rm APS}):=\{u\in H^1(M;E): \chi^+(A_r)\gamma(u)=0\},$$
is regular. For notational convenience, we assume $r=0$ and drop it from the notation for the remainder of the paper.
The reader can find further examples of boundary conditions for first order elliptic operators in \cite{BB,BBan,leschgor} and \cite[Section 4]{G99}.

\subsection{Scalar properly elliptic operators}
\label{subsec:scalprododo}

Let us consider an example that historically has played an important role and has long been well understood, see for instance \cite{lionsmagenes,grubb68,vishik}. We consider an elliptic differential operator $D$ of order $2m$ acting between sections of $E$ and $F$. We say that the differential operator is scalar if $E$ and $F$ are line bundles; the terminology is justified by the fact that $\mathrm{Hom}(E,F)$ is a line bundle (trivialisable at $\Sigma$ due to the existence of $A_0\in C^\infty(\Sigma;\mathrm{Aut}(E,F))$), and as such the symbol calculus is completely scalar. We recall the notion of a properly elliptic scalar operator, see \cite{schechter59}.

\begin{definition}
Let $D$ be a scalar elliptic differential operator of order $2m$. We say that $D$ is properly elliptic if for each $(x',\xi')\in S^*\Sigma$, the polynomial equation 
\begin{equation}
\label{charcowkeddl}
\sigma(D)|_\Sigma(x',\xi',z)=0,
\end{equation}
has exactly $m$ complex solutions with positive imaginary part. 
\end{definition}

\begin{remark}
Let us consider some special cases where scalar elliptic differential operators are automatically properly elliptic. If $\sigma(D)|_\Sigma$ is real-valued, then $z$ is a solution of Equation \eqref{charcowkeddl} if and only if $\bar{z}$ is, so real valued principal symbol ensures that a scalar elliptic differential operator is properly elliptic. 

If $\dim(M)>2$, the fibres of the cosphere bundle of $\Sigma$ are connected. Since $D$ is elliptic, Equation \eqref{charcowkeddl} has no real solutions and as such the number of solutions with positive imaginary part is locally constant. We note that if $z$ is a solution to Equation \eqref{charcowkeddl} at $(x',\xi')$, then $-z$ is a solution to Equation \eqref{charcowkeddl} at $(x',-\xi')$ because of the symmetry condition $\sigma(D)|_\Sigma(x',\xi',z)=\sigma(D)|_\Sigma(x',-\xi',-z)$. This proves that all scalar elliptic differential operators are properly elliptic if $\dim(M)>2$. 
\end{remark}

For a scalar properly elliptic differential operator $D$ of order $2m$, we define the bundle-like boundary condition $B_{\rm Dir}$ from the subbundle $E':=E\otimes \C^{m}$, viewed as a subbundle of $E\otimes \C^{2m}$ by embedding it in the last $m$ coordinates. Compare to Example \ref{kindofdirhcoell}. We write $D_{\rm Dir}:=D_{{\rm B}_{\rm Dir}}$ and call it the Dirichlet realisation of $D$. More explicitly, we have that 
$$\dom(D_{\rm Dir})=\{u\in \dom(D_{\rm max}): \gamma_0 u=\gamma_1 u=\cdots=\gamma_{m-1}u=0\}.$$ 
The following theorem reformulates several results from \cite{grubb68} to our setting, the reader is referred to \cite{grubb68} for proofs.

\begin{theorem}
Let $D$ be a scalar properly elliptic differential  operator of order $2m$. 
\begin{enumerate}[(i)]
\item $B_{\rm Dir}$ is a regular boundary condition, and so the realisation $D_{\rm Dir}$ of $D$ with domain
\begin{align*}
\dom(D_{\rm Dir})=&\SobH{2m}(M;E)\cap \SobH[0]{m}(M;E)=\\
=&\{u\in \SobH{2m}(M;E): \gamma_0 u=\gamma_1 u=\cdots=\gamma_{m-1}u=0\},
\end{align*}
is regular.
\item Also $D^\dagger$ is properly elliptic and 
$$D_{\rm Dir}^*=D^\dagger_{\rm Dir},$$
is the realisation of $D^\dagger$ with domain $\SobH{2m}(M;F)\cap \SobH[0]{m}(M;F)$.
\item The map 
$$\Ca_D\to \bigoplus_{j=0}^{m-1}\SobH{-j-{\frac12}}(M;E),$$
induced from the projection $ \bigoplus_{j=0}^{2m-1}\SobH{-j-{\frac12}}(M;E)\to  \bigoplus_{j=0}^{m-1}\SobH{-j-{\frac12}}(M;E)$, is a Fredholm operator. 
\item We can write 
$$\dom(D_{\rm max})=\SobH{2m}(M;E)\cap \SobH[0]{m}(M;E)\oplus_{\ker D_{\rm Dir}}\ker(D_{\rm max}),$$
and 
$$\checkH(D)=\bigoplus_{j=m}^{2m-1}\SobH{2m-j-{\frac12}}(M;E)\oplus_{ \gamma \ker(D_{\rm Dir})}\Ca_D.$$
\end{enumerate}
\end{theorem}

The reader should note that the theorem above holds for more general local (regular) boundary conditions than Dirichlet conditions, and are studied using different methods in \cite{grubb68}. We remark that from these type of results, several Fredholm type results follow readily for properly elliptic boundary value problems with regular boundary conditions.

\section{Calderón projectors and the Douglis-Nirenberg calculus}
\label{sec:calderon}

In this section we will review results of Hörmander and Seeley concerning Calderón projectors. See more in \cite[Chapter XX]{horIII}, \cite{seeley65} and also in \cite[Chapter 11]{grubbdistop}. Calderón projections  are projections on function spaces on the boundary $\Sigma$ onto the space of boundary values of homogeneous solutions $Du=0$ in the interior, i.e. the Hardy space $\mathcal{C}_D$. For higher order operators, it is necessary to keep track of  different order traces, so even though the Calderón projection is a matrix of pseudodifferential operators, its entries will have different order. However, the orders remain constant along the diagonals in the Calderón projection which makes it amenable for study via the calculus of Douglis-Nirenberg.

\subsection{Approximate Calderón projectors}
\label{subsec:calderonproj}

First, we will in this subsection recall some constructions found in Hörmander's book \cite{horIII}. It will provide a way of constructing the full symbol of the Calderón projection.

As above, we write $x_n$ for the coordinate near the boundary defined from the transversal vector field $\vec{T}$. The coordinate $x_n$ identifies a neighbourhood of the boundary with a collar $\Sigma\times[0,1]$ with the boundary defined by the equation $x_n=0$. Consider the differential operator $D_{x_n}=-i\partial_{x_n}$ defined in the collar neighbourhood of the boundary. Near the boundary, we can write 
$$D=\sum_{j=0}^m A_jD_{x_n}^{m-j}.$$
For each $j$, $A_j=(A_j(x_n))_{x_n\in [0,1]}$ is a family of differential operators $\Ck{\infty}(\Sigma,E)\to \Ck{\infty}(\Sigma,F)$ of order $j$. Let $a$ denote the principal symbol of $D$ and $a_j$ the principal symbol of $A_j$. In coordinates $x=(x',x_n)$ on $\Sigma\times[0,1]$ with associated cotangent coordinates $\xi=(\xi',\xi_n)$ on $T^*M|_{\Sigma}$, near $\Sigma$ we can  write 
$$a(x',x_n,\xi',\xi_n)=\sum_{j=0}^m a_j(x',x_n,\xi')\xi_n^{m-j}.$$
We define the boundary symbol $\sigma_\partial(A)$ as the ordinary differential operator valued function on $T^*\Sigma$ given by
\begin{equation}
\label{boundarysymbol}
\sigma_\partial(D)(x',\xi'):=\sum_{j=0}^m a_j(x',0,\xi')D_{t}^{m-j}.
\end{equation}
We consider $\sigma_\partial(D)$ as an element of $\Ck{\infty}(T^*\Sigma,\mathrm{Hom}(E|_{\Sigma},F|_{\Sigma})\otimes \mathcal{L}(\Ck{\infty}[0,\infty)))$. The conormal symbol of $D$ is the symbol of the boundary symbol 
\begin{equation}
\label{conormalsymbol}
\sigma_{\rm cn}(D)(x',\xi',\xi_n):=\sum_{j=0}^m a_j(x',0,\xi')\xi_n^{m-j},
\end{equation}
so that $\sigma_\partial(D)(x',\xi')=\sigma_{\rm cn}(D)(x',\xi',D_t)$ and $\sigma_{\rm cn}(D)=\sigma(D)|_\Sigma$. For more details on the yoga of boundary symbols, see \cite{gerdsgreenbook,rempelschulze,elmarbdmover} or \cite{melroseAPS} for related notions.

The following result follows from \cite[Theorem XX.1.3]{horIII}, which we state below as Theorem \ref{horsats}.

\begin{proposition}
\label{structureofeplus}
Let $D$ be elliptic of order $m>0$. The following construction produces a well-defined vector bundle 
$$E_+(D)\to S^*\Sigma.$$ 
Define $E_+(D)$ as the subset $E_+(D)\subseteq \pi^*(E)|_{S^*\Sigma} \otimes \C^m$ whose fibre over $(x',\xi')\in S^*\Sigma$ consists of $(v_0,v_1, \ldots, v_{m-1})\in E_{x}\otimes \C^m$ such that the solution $v\in \Ck{\infty}([0,\infty),E_x)$ of the ordinary differential equation 
\begin{equation}
\label{boundaryeq}
\begin{cases}
\sigma_\partial(D)(x',\xi')v(t)=0,\; t>0,\\
D^j_tv(0)=v_j, \; j=0,\ldots, m-1,\end{cases}
\end{equation}
is exponentially decaying as $t\to +\infty$.

The subset $E_-(D)\subseteq \pi^*(E)|_{S^*\Sigma}$ given as the set of vectors whose solution to \eqref{boundaryeq} is exponentially decaying as $t\to -\infty$ is a well-defined vector bundle defining a complement to $E_+(D)$ in $\pi^*(E)|_{S^*\Sigma} \otimes \C^m$.
\end{proposition}

\begin{definition}\label{Def:p+}
Let $p_+(D)\in \Ck{\infty}(S^*\Sigma, \mathrm{Hom}(\pi^*(E)|_{S^*\Sigma} \otimes \C^m))$ denote the projection onto $E_+(D)$ along $E_-(D)$ and $p_-(D):=1-p_+(D)$ its complementary projection.
\end{definition}

\begin{proposition}
\label{pplussandpminussdagger}
Let $D$ be elliptic of order $m>0$. It holds that 
$$p_+(D)=\sigma_{\pmb 0}(\scalebox{1.5}{a}^{-1})p_-(D^\dagger)^*\sigma_{\pmb 0}(\scalebox{1.5}{a}),$$
where $\scalebox{1.5}{a}$ is the matrix of differential operators from Proposition \ref{proponboundarypairingforell} and $\sigma_{\pmb 0}(\scalebox{1.5}{a})$ is its principal symbol (in the Douglis-Nirenberg calculus we review below)
$$\sigma_{\pmb 0}(\scalebox{1.5}{a})=-i\begin{pmatrix} 
\sigma_{m-1}(A_{m-1})&\sigma_{m-2}(A_{m-2})&\cdots &\sigma_2(A_{2})& \sigma_1(A_{1})&A_0\\
\sigma_{m-2}(A_{m-2})&\sigma_{m-3}(A_{m-3})&\cdots &\sigma_1(A_{1})& A_0&0\\
\sigma_{m-3}(A_{m-3})&\Ddots& &A_0& 0&\vdots \\
\vdots &&\Ddots &0&\vdots & \\
\sigma_1(A_{1})&A_0&\Ddots && &\\
A_0&0&\cdots &0& 0&0
\end{pmatrix}.$$ 
In particular, $\sigma_{\pmb 0}(\scalebox{1.5}{a})$ induces isomorphisms  
$$E_\pm(D)\cong E_\mp(D^\dagger).$$
\end{proposition}

We postpone the proof of Proposition \ref{pplussandpminussdagger} until after the statement of Theorem \ref{horsats} below.

\begin{proposition}
\label{nontrivialityofeplussd}
Let $D$ be an elliptic differential operator of order $m>0$ acting on a manifold of $\dim(M)>1$. The vector bundles $E_+(D),E_-(D)\to S^*\Sigma$ are non-trivial over each component of $S^\ast \Sigma$.
\end{proposition}

\begin{proof}
We note that $\sigma_{\rm cn}(D)(x',\xi',z)=(-1)^{m}\sigma_{\rm cn}(D)(x',-\xi',-z)$, so $z\in \C$ solves $\det \sigma_{\rm cn}(D)(x',\xi',z)=0$ if and only if $\det \sigma_{\rm cn}(D)(x',-\xi',-z)=0$. Since $E\otimes \C^m=E_+(D)\oplus E_-(D)$ by Proposition \ref{structureofeplus}, we conclude that given $(x',\xi')\in S^*\Sigma$, either $E_+(D)_{(x',\xi')}\neq 0$ or $E_-(D)_{(x',\xi')}\neq 0$. Therefore, given $(x',\xi')\in S^*\Sigma$ then at least one of the vector spaces $E_+(D)_{(x',\xi')}$ and $E_+(D)_{(x',-\xi')}$ is non-zero and at least one of the vector spaces $E_-(D)_{(x',\xi')}$ and $E_-(D)_{(x',-\xi')}$ is non-zero. 
\end{proof}

\begin{remark}
We remark that if $\dim(M)>2$, then $\pi_0(S^*\Sigma)=\pi_0(\Sigma)$ and in this case $E_+(D),E_-(D)\to S^*\Sigma$ are non-trivial over each component of $S^*\Sigma$. If $\dim(M)=2$, the ``odd-to-even'' part of a spin$^c$ Dirac operator twisted by a line bundle provides an example where $E_+(D)$ and $E_-(D)$ satisfy $\mathrm{rk}(E_+(D))+\mathrm{rk}(E_-(D))=1$ but are non-vanishing over each connected component of the boundary; their fibrewise supports are on complementary components of $S^*\Sigma=S^1\times \Sigma\dot{\cup} S^1\times \Sigma$.
\end{remark}

To construct the Calderón projection, we use the Douglis-Nirenberg calculus. The reader is referred to \cite{agmon, grubb77}, see also \cite[Chapter XIX.5]{horIII}, for details. We introduce some notation. If $\pmb{E}\to \Sigma$ is a Hermitian vector bundle orthogonally graded by $\R$ (viewed as a discrete group),   for $s\in \R$ we define the graded Sobolev space
$$\SobHH{{s}}(\Sigma,\pmb{E}):=\oplus_{s'+t=s} \SobH{s'}(\Sigma,\pmb{E}[t]),$$
where $\pmb{E}=\oplus_t \pmb{E}[t]$ is the grading of $\pmb{E}$. If $\pmb{E},\pmb{F}\to \Sigma$ are two Hermitian vector bundles orthogonally graded by $\R$ (viewed as a discrete group), we can define the space $\Psi^{\pmb m}_{\rm cl}(\Sigma;\pmb{E},\pmb{F})$ of Douglis-Nirenberg operators from $\pmb{E}$ to $\pmb{F}$ of order $m\in\R$ as the subspace of classical pseudodifferential operators $\Psi_{\rm cl}^*(\Sigma;\pmb{E},\pmb{F})$ that act homogeneously jointly in the Sobolev degree and the gradings of the vector bundles. More precisely, writing $\pmb{E}=\oplus_{k=1}^{n_E} E[s_k]$ and $\pmb{F}=\oplus_{j=1}^{n_F}F[r_j]$ for some $s_1,\ldots, s_{n_E},r_1,\ldots, r_{n_F}\in \R$, a matrix $(A_{jk})_{j=1,\ldots, n_E,k=1,\ldots, n_F}\in \Psi_{\rm cl}^*(\Sigma;\pmb{E},\pmb{F})$ belongs to $\Psi^{\pmb \alpha}_{\rm cl}(\Sigma;\pmb{E},\pmb{F})$ if and only if the order $\alpha_{j,k}$ of $A_{j,k}$ satisfies $\alpha_{j,k}=\alpha-r_k+s_j$ for all $j$ and $k$. By construction, $\Psi^{-\pmb \infty}(\Sigma;\pmb{E},\pmb{F}):=\cap_{s\in \R}\Psi^{-\pmb s}(\Sigma;\pmb{E},\pmb{F}) =\Psi^{-\infty}(\Sigma;\pmb{E},\pmb{F})$ -- the ordinary smoothing operators. We write $\Psi^{\pmb \alpha}_{\rm cl}(\Sigma;\pmb{E}):=\Psi^{\pmb \alpha}_{\rm cl}(\Sigma;\pmb{E},\pmb{E})$.

We define the principal symbol $\sigma_{\pmb{\alpha}}(A)$ to be the matrix of principal symbols $(\sigma_{\alpha_{j,k}}(A_{j,k}))_{j,k}$ which defines an element of $\Ck{\infty}(S^*\Sigma,\mathrm{Hom}(\pi^*\pmb{E},\pi^*\pmb{F}))$. We say that $A\in \Psi^{\pmb \alpha}_{\rm cl}(\Sigma;\pmb{E},\pmb{F})$ is elliptic if $\sigma_{\pmb{\alpha}}(A)$ is invertible in all points, i.e. that $\sigma_{\pmb{\alpha}}(A)\in \Ck{\infty}(S^*\Sigma,\mathrm{Iso}(\pi^*\pmb{E},\pi^*\pmb{F}))$. By the standard construction, $A\in \Psi^{\pmb \alpha}_{\rm cl}(\Sigma;\pmb{E},\pmb{F})$ is elliptic if and only if there exists a parametrix $B\in \Psi^{-\pmb \alpha}_{\rm cl}(\Sigma;\pmb{F},\pmb{E})$.

By construction, any $A\in \Psi^{\pmb \alpha}_{\rm cl}(\Sigma;\pmb{E},\pmb{F})$ defines a continuous operator 
$$A:\SobHH{{s}}(\Sigma;\pmb{E})\to \SobHH{{s'}}(\Sigma;\pmb{F}),$$
for any $s,s'\in \R$ with $s'\leq s-\alpha$. If $\sigma_{\pmb{\alpha}}(A)\neq 0$, this operator is compact if and only if $s'< s-\alpha$. 
For $\sigma_{\pmb{\alpha}}(A)=0$, this operator is compact whenever $s'\leq s-\alpha$. This operator is Fredholm if and only if $A$ is elliptic (cf. \cite[Theorem XIX.5.1 and XIX.5.2]{horIII}). 

The graded vector bundles we shall be concerned with take the form $E\otimes \C^m$ for a trivially graded Hermitian vector bundle $E$. We shall tacitly grade this vector bundle as follows. We view $\C^m$ as a graded bundle by 
$$\C^m[s]=
\begin{cases}
\C, \; &\mbox{if $s=0,1,\ldots, m-1$},\\
0, \; &\mbox{otherwise}.\end{cases}.$$
We grade $E\otimes \C^m$ as a tensor product, i.e.
$$(E\otimes \C^m)[s]=
\begin{cases}
E, \; &\mbox{if $s=0,1,\ldots, m-1$},\\
0, \; &\mbox{otherwise}.\end{cases}$$
In particular, 
$$\SobHH{{s}}(\Sigma; E\otimes \C^m)=\bigoplus_{j=0}^{m-1}\SobH{s-j}(\Sigma;E).$$
We also introduce the graded bundle $\C^m_{\rm op}$ which is the trivial vector bundle of rank $m$, graded by 
$$\C^m_{\rm op}[s]=\C^m[-s]=
\begin{cases}
\C, \; &\mbox{if $s=0,-1,\ldots, -m+1$},\\
0, \; &\mbox{otherwise}.\end{cases}.$$
The reader should note that $\C^m_{\rm op}$ can be identified with the graded dual bundle of $\C^m$, and the Hermitian structure of $E$ allow us to identify $E\otimes \C^m_{\rm op}$ with the graded dual bundle of $E\otimes \C^m$. In particular, 
$$\SobHH{{s}}(\Sigma; E\otimes \C^m_{\rm op})=\bigoplus_{j=0}^{m-1}\SobH{s+j}(\Sigma; E).$$
We can conclude the following proposition.

\begin{proposition}
\label{ell2pairingsofgradedsob}
The $\Lp{2}$-pairing induces a perfect pairing
$$\SobHH{{s}}(\Sigma; E\otimes \C^m)\times \SobHH{-{s}}(\Sigma; E\otimes \C^m_{\rm op})\to \C,$$
for any $s\in \R$.
\end{proposition}

\begin{definition}
Let $D$ be an elliptic differential operator of order $m>0$ on a manifold $M$ with boundary $\Sigma$ acting between sections of the Hermitian vector bundles $E$ and $F$. An approximate Calder\'on projection for $D$ is a classical pseudo-differential operator $Q\in \Psi^{\pmb 0}_{\rm cl}(\Sigma;E\otimes \C^m)$ which is order zero in the Douglis-Nirenberg sense and satisfies the following
\begin{itemize}
\item $Q^2-Q$ is smoothing.
\item $Q$ is an approximate projection onto the Hardy space $\Ca_D$ in the following sense: 
\begin{enumerate}
\item If $u\in \ker(D_{\rm max})$ then $v:=\gamma u\in \SobHH{-{{\frac12}}}(\Sigma;E\otimes \C^m)$ satisfies that $v-Qv\in \Ck{\infty}(\Sigma;E\otimes \C^m)$.
\item For any $v\in \SobHH{-{{\frac12}}}(\Sigma;E\otimes \C^m)$, there is a $u\in \ker(D_{\rm max})+\Ck{\infty}(M,E)$ (smooth up to the boundary) such that $Qv=\gamma u$. 
\end{enumerate}
\end{itemize}
\end{definition}

\begin{theorem}[\cite{horIII}, Theorem XX.1.3]
\label{horsats}
Let $D$ be an elliptic differential operator of order $m>0$. There exists an approximate Calderón projection $Q\in  \Psi^{\pmb 0}_{\rm cl}(\Sigma;E\otimes \C^m)$ for $D$ with principal symbol given by 
$$\sigma_{\pmb 0}(Q)=p_+(D),$$
(as defined in Defintion~\ref{Def:p+}) 
and the full symbol of $Q=(Q_{j,k})_{j,k=0}^{m-1}$ computed from the formula \cite[Equation XX.1.7]{horIII}:
\begin{equation}
\label{formulaforaprprood}
Q_{j,k}f=\sum_{l=0}^{m-1-k}(-i)^{j+l+1}\gamma_j\circ T[A_{m-l-k-1}f\otimes \delta_{t=0}^{(l)}],
\end{equation}
where $T$ is a pseudodifferential parametrix to $D$ computed in a neighbourhood of $\Sigma$. 
\end{theorem}

\begin{remark}
\label{remarkonredisuydd}
We remark that the formula for $Q_{j,k}$ might seem difficult to parse at first. But the fact that the trace mappings are from inside of $M$ makes it possible to compute the full symbol by means of residue calculus and the Paley-Wiener theorem. For instance, \cite[Equation XX.1.8]{horIII} computes the principal symbol of $Q$ by residue calculus as
\begin{align}
\label{sigmaoss}
\sigma_{\pmb 0}(Q)_{j,k}(x',\xi')\equiv& \sigma_{k-j}(Q_{j,k})(x',\xi')=\\
\nonumber
=&\sum_{\mathrm{Im}(\xi_n)>0}\mathrm{Res}_{\xi_n}\left[\sum_{l=0}^{m-k-1}\xi_n^{j+l}a(x',0,\xi,\xi_n)^{-1}a_{m-k-l-1}(x',0,\xi')\right]
\end{align}
where $(x',\xi')$ are coordinates on $T^*\Sigma$ and $((x',t),(\xi',\xi_n))$ are coordinates on $T^*M$ near $\Sigma$. The proof of \cite[Theorem XX.1.3]{horIII} identifies this expression with $p_+(D)$. The lower order terms in $Q$ are, in general, hard to compute. We shall give some examples of such computations below for order 1 operators. We can condense the formula \eqref{sigmaoss} into 
\begin{align}
\sigma_{\pmb 0}(Q)(x',\xi')=\sum_{\mathrm{Im}(\xi_n)>0}\mathrm{Res}_{\xi_n}\left[a(x',0,\xi',\xi_n)^{-1}e(\xi_n)\sigma_{\pmb 0}(\scalebox{1.5}{a})(x',\xi')\right],
\end{align}
where $e(\xi_n):=v(\xi_n)\otimes v(\xi_n)^*$ and $v(\xi_n)=(\xi_n^j)_{j=0}^{m-1}$. Here, we abuse notation and view $a^{-1}|_{T^*M|_\Sigma\setminus\Sigma}$ and $e$ as endomorphisms of $E|_{\Sigma}\otimes \C^m$. Note that $a^{-1}|_{T^*M|_\Sigma\setminus\Sigma}$ and $e$ commute because they act on different factors in the tensor product $E|_{\Sigma}\otimes \C^m$.
\end{remark}

\begin{proof}[Proof of Proposition \ref{pplussandpminussdagger}]
In light of Theorem \ref{horsats} and Remark \ref{remarkonredisuydd}, the proof consists of an exercise with boundary symbols and residue calculus. We compute that 
\begin{align*}
p_-(D^\dagger)^*(x',\xi')=&\left(-\sum_{\mathrm{Im}(\xi_n)<0}\mathrm{Res}_{\xi_n}\left[a^*(x',0,\xi',\xi_n)^{-1}e(\xi_n)\sigma_{\pmb 0}(\scalebox{1.5}{a}_\dagger)(x',\xi')\right]\right)^*=\\
=&-\sum_{\mathrm{Im}(\xi_n)>0}\mathrm{Res}_{\xi_n}\left[\left(a^*(x',0,\xi',\overline{\xi_n})^{-1}e(\overline{\xi_n})\sigma_{\pmb 0}(\scalebox{1.5}{a}_\dagger)(x',\xi')\right)^*\right]=\\
=&-\sum_{\mathrm{Im}(\xi_n)>0}\mathrm{Res}_{\xi_n}\left[\sigma_{\pmb 0}(\scalebox{1.5}{a}^*_\dagger)(x',\xi')a(x',0,\xi',\xi_n)^{-1}e(\xi_n)\right]=\\
=&\sum_{\mathrm{Im}(\xi_n)>0}\mathrm{Res}_{\xi_n}\left[\sigma_{\pmb 0}(\scalebox{1.5}{a})(x',\xi')a(x',0,\xi',\xi_n)^{-1}e(\xi_n)\right].
\end{align*}
In the first equality we used that $p_-(D^\dagger)$ is the complementary projection to $p_+(D^\dagger)$ which is the sum of all residues in the lower half plane, with a minus sign due to the orientation. In the second equality we used that the adjoint is antilinear, and interchanges the upper and lower half plane. In the last equality we used Lemma \ref{someexresiosnaoadwa}. We arrive at 
\begin{align*}
\sigma_{\pmb 0}(\scalebox{1.5}{a})&(x',\xi')^{-1}p_-(D^\dagger)^*(x',\xi')\sigma_{\pmb 0}(\scalebox{1.5}{a})(x',\xi')=\\
=&\sigma_{\pmb 0}(\scalebox{1.5}{a})(x',\xi')^{-1}\sum_{\mathrm{Im}(\xi_n)>0}\mathrm{Res}_{\xi_n}\left[\sigma_{\pmb 0}(\scalebox{1.5}{a})(x',\xi')a(x',0,\xi',\xi_n)^{-1}e(\xi_n)\right]\sigma_{\pmb 0}(\scalebox{1.5}{a})(x',\xi')=\\
=&\sum_{\mathrm{Im}(\xi_n)>0}\mathrm{Res}_{\xi_n}\left[a(x',0,\xi',\xi_n)^{-1}e(\xi_n)\sigma_{\pmb 0}(\scalebox{1.5}{a})(x',\xi')\right]=p_+(D).
\qedhere
\end{align*}
\end{proof}

\begin{corollary}
\label{compactnesscor}
Let $D$ be an elliptic differential operator of order $m>0$ acting on a manifold of $\dim(M)>1$ and let $Q$ be the approximate Calder\'on projection from Theorem \ref{horsats}. For any $s\in \R$, the operator 
\begin{equation}
\label{qdeffofod}
Q:\SobHH{ s}(\Sigma;E\otimes \C^m)\to \SobHH{ s}(\Sigma;E\otimes \C^m),
\end{equation}
is continuous but neither $Q$ nor $1-Q$ is compact.
\end{corollary}

This result should be compared to Lemma \ref{Lem:CheckNonZeroalternative1} in the appendix.

\begin{proof}
Continuity of \eqref{qdeffofod} follows from that $Q$ is order zero in the Douglis-Nirenberg sense. By standard arguments in pseudodifferential calculus,  \eqref{qdeffofod} is compact for some $s$ if and only if it is compact for all $s$ if and only if the principal symbol vanishes. The principal symbol of $Q$ is $p_+(D)$ which is non-trivial by Proposition \ref{nontrivialityofeplussd}, so $Q$ is never compact on $\SobHH{ s}(\Sigma;E\otimes \C^m)$. By the same argument, the principal symbol of $1-Q$ is $p_-(D)$ which is non-trivial by Proposition \ref{nontrivialityofeplussd}, so $1-Q$ is never compact on $\SobHH{ s}(\Sigma;E\otimes \C^m)$. 
\end{proof}

\begin{theorem}
\label{infdimeker}
Let $D$ be an elliptic differential operator of order $m>0$ on a manifold of dimension $>1$. Then $\ker(D_{\rm max})$ is infinite-dimensional. 
\end{theorem}

Despite this result being well known and an immediate consequence of results of Seeley that we recall in the next subsection, let us provide a short proof that remains at a symbolic level for an approximate Calderón projection.

\begin{proof}
By Proposition \ref{geneprop}, $\ker(D_{\rm max})$  is infinite-dimensional if and only if the Hardy space $\Ca_D$ is infinite-dimensional. We will argue by contradiction.

Assume that $\Ca_D$ is finite-dimensional. Since the space 
$$\Ca_D\subseteq \SobHH{-{{\frac12}}}(\Sigma;E\otimes \C^m),$$
is finite-dimensional it is closed. Consider the operator $Q:\SobHH{-{{\frac12}}}(\Sigma;E\otimes \C^m)\to \SobHH{-{{\frac12}}}(\Sigma;E\otimes \C^m)$ as in Equation \eqref{qdeffofod}. By Theorem \ref{horsats}, for any $v\in \SobHH{-{{\frac12}}}(\Sigma;E\otimes \C^m)$ it holds that 
$$Qv\in \Ca_D+\Ck{\infty}(\Sigma;E\otimes \C^m).$$ 
We conclude that $Q:\SobHH{-{{\frac12}}}(\Sigma;E\otimes \C^m)\to \SobHH{-{{\frac12}}}(\Sigma;E\otimes \C^m)$ factors over the inclusion 
$$\Ca_D+\Ck{\infty}(\Sigma;E\otimes \C^m)\hookrightarrow \SobHH{-{{\frac12}}}(\Sigma;E\otimes \C^m),$$ 
which is compact if $\mathcal{C}_D$ is finite-dimensional. Therefore, if $\mathcal{C}_D$ is finite-dimensional then $Q$ is compact on $\SobHH{-{{\frac12}}}(\Sigma;E\otimes \C^m)$ which contradicts Corollary \ref{compactnesscor}.
\end{proof}

\begin{corollary} 
\label{Prop:SpecoHO}
Assume that $(F,h^F) = (E, h^E)$ and let $D:\Ck{\infty}(M;E)\to \Ck{\infty}(M;E)$ be an elliptic differential operator of order $m>0$ acting on a manifold of $\dim(M)>1$. We then have that $\spec(D_{\max}) = \specpt(D_{\max}) = \C$. 
Consequently, $\spec(D_{\min}) = \C$. 
\end{corollary}

\begin{proof}
Fix $\lambda \in \C$, and note that $D - \lambda$ is again an elliptic differential operator of order $m>0$.
Now, Theorem \ref{infdimeker} implies that $\dim \ker(D_{\max} - \lambda) = \infty$. 
But that is exactly that $\lambda \in \specpt(D_{\max})$ and so we conclude that $\spec(D_{\max}) = \specpt(D_{\max}) = \C$.
For the statement concerning $D_{\min}$, we note that by mirroring this argument, we obtain that $\spec(D_{\max}^\dagger) = \C$ and since $D_{\min} = (D_{\max}^\dagger)^\ad$,
we obtain that $\spec(D_{\min}) = \spec(D_{\max}^\dagger)^\conj = \C$. 
\end{proof}

\begin{remark}
\label{Rem:Residue}
For $D$ with $\ker(D_{\min}) = 0$ (i.e., first order Dirac-type operators), we have that $0 \not \in \specpt(D_{\min})$. 
However, from Corollary \ref{Prop:SpecoHO}, we know that $0 \in \spec(D_{\min})$.
In fact, $0 \in \specres(D_{\min})$, i.e., it is in the residue spectrum.
This is because we have that $\Lp{2}(M;E) = \ker(D_{\max}^\dagger) \oplus \ran(D_{\min})$ from Proposition \ref{Prop:MaxClosed} which shows that $D_{\min}: \dom(D_{\min}) \to \Lp{2}(M;E)$ does not have dense range. 
In fact, the range fails to be dense on the infinite dimensional subspace $\ker(D_{\max}^\dagger)$. 
However, despite these shortcomings, and a complete lack of spectral theory, Proposition \ref{Cor:WellPosed} shows that the Dirichlet problem is indeed well-posed when restricted to $\ran(D_{\min}) = \ran(D_{\Ca_D})$ (see more in Subsection \ref{subsec:wellposed}).
\end{remark}

\begin{corollary}
\label{Cor:SpecoFOres}
Assume that $(F, h^F) = (E, h^E)$ and that $D: \Ck{\infty}(M;E)\to \Ck{\infty}(M;E)$ is first order elliptic and Dirac-type.
Then, $\spec(D_{\min})= \specres(D_{\min}) = \C$. 
\end{corollary}
\begin{proof}
For $\lambda \in \C$, the operator  $D_{\lambda} = D - \lambda$ is again first order Dirac-type.
By Remark \ref{Rem:Residue}, we conclude that $0 \in  \specres((D_{\lambda})_{\min})$.
But $(D_{\lambda})_{\min} = D_{\min} - \lambda$ and so the conclusion holds.
\end{proof}

\subsection{Seeley's results on Calderón projectors}
\label{subsec:seeleyoncalderon}

\begin{definition}
Let $D$ be an elliptic differential operator of order $m>0$ on a manifold $M$ with boundary $\Sigma$ acting between sections of the Hermitian vector bundles $E$ and $F$. A Calder\'on projection for $D$ is an approximate Calderón projection $P_\mathcal{C}\in \Psi^{\pmb 0}_{\rm cl}(\Sigma;E\otimes \C^m)$ such that it is a projection onto 
$$\mathcal{C}_D\subseteq \SobHH{-{{\frac12}}}(\Sigma;E\otimes \C^m).$$
\end{definition}

\begin{remark}
Calderón projectors are not uniquely determined. For instance, if $P_\mathcal{C}$ is a Calderón projection, then for any self-adjoint smoothing operator $R$ which is ``small''\footnote{How small can be made precise by pointwise estimates on the kernel, but it is irrelevant for the sake of this argument.} and $R|_{\mathcal{C}_D}=0$ we have that $P_\mathcal{C}(1+R)^{-1}$ also satisfies the conditions of the theorem. \emph{We shall tacitly use only the Calderón projection constructed in the next theorem of Seeley.}
\end{remark}

Let us recall the following theorem of Seeley. We remark that Seeley proved this result in the larger generality of higher regularity (cf. Theorem \ref{refinedseeleysobolev} below) on $\Lp{p}$-spaces, $1<p<\infty$. For $p\neq 2$, the spaces on the boundary are Besov spaces and not Sobolev spaces.

\begin{theorem}
\label{Thm:CaldProj}[Seeley \cite{seeley65}]
Let $D$ be an elliptic differential operator of order $m>0$ on a manifold $M$ with boundary $\Sigma$ acting between sections of the Hermitian vector bundles $E$ and $F$. Then there exists a Calderón projection $P_\mathcal{C}\in \Psi^{\pmb 0}_{\rm cl}(\Sigma;E\otimes \C^m)$ whose full symbol is uniquely determined as that of the approximate Calderón projection in Theorem \ref{horsats}. In particular, 
$$\sigma_{\pmb 0}(P_\mathcal{C})=p_+(D).$$ 
\end{theorem}

In fact, Seeley obtained even stronger results in  \cite{seeley65} that we partly summarise in the following theorem.

\begin{theorem}
\label{refinedseeley}
Let $D$ be an elliptic differential operator of order $m>0$ on a manifold $M$ with boundary $\Sigma$ acting between sections of the Hermitian vector bundles $E$ and $F$. There exists a Calderón projection $P_\mathcal{C}$ such that 
\begin{enumerate}[(i)]
\item \label{refinedseeley0:1}  
There exists a continuous Poisson operator 
$$\mathcal{K}:\SobHH{-{{\frac12}}}(\Sigma;E\otimes \C^m)\to \ker(D_{\rm max}),$$
with $P_\mathcal{C}=\gamma\circ \mathcal{K}$ (and in particular, $\mathcal{K}=\mathcal{K}\circ P_\mathcal{C}$).
\item \label{refinedseeley0:2} 
The image of 
$$\gamma:\dom(D_{\rm max})\to \SobHH{-{{\frac12}}}(\Sigma;E\otimes \C^m),$$
is precisely the subspace 
$$P_\mathcal{C}\SobHH{-{{\frac12}}}(\Sigma;E\otimes \C^m)\oplus(1-P_\mathcal{C})\SobHH{ m-{{\frac12}}}(\Sigma;E\otimes \C^m).$$
\end{enumerate} 
In particular, the identity map induces an isomorphism of Banach spaces 
$$\checkH(D) \cong   P_\mathcal{C}\SobHH{-{{\frac12}}}(\Sigma;E\otimes \C^m)\oplus (1-P_\mathcal{C})\SobHH{ m-{{\frac12}}}(\Sigma;E\otimes \C^m).$$
\end{theorem}

\begin{proof}
The first item is proven in \cite{seeley65}. The second item is a reformulation of a statement on \cite[top of page 782]{seeley65} which was stated but not proven there. We include its proof for the convenience of the reader. 

To prove the second item, it suffices to prove that 
$$(1-P_\mathcal{C})\gamma\dom(D_{\rm max})\subseteq \SobHH{ m-{{\frac12}}}(\Sigma;E\otimes \C^m).$$ 
We can assume that $M$ is a domain with smooth boundary in a compact manifold $\hat{M}$, that $E$ and $F$ extend to Hermitian vector bundles $\hat{E},\hat{F}\to \hat{M}$ and that there exists an elliptic differential operator $\hat{D}$ of order $m$ extending $D$ to $\hat{M}$. Write $D'$ for the elliptic differential operator of order $m$ defined from restricting $\hat{D}$ to $M^c$. Take a $g\in (1-P_\mathcal{C})\gamma\dom(D_{\rm max})\subseteq \gamma \dom(D_{\rm max})$ and pick $f\in \dom(D_{\rm max})$ lifting $g$. Since $g\in (1-P_\mathcal{C})\gamma\dom(D_{\rm max})$, \cite[Lemma 5]{seeley65} implies that there exists an $f'\in \ker(D')\subseteq \Lp{2}(M^c,\hat{E}|_{M^c})$ lifting $g$. We define $F\in \Lp{2}(\hat{M},E)$ as $f$ on $M$ and as $f'$ on $M^c$. Since $\gamma(f)=g=\gamma(f')$, we have that $F\in \dom(\hat{D}_{\rm max})$. By elliptic regularity on the closed manifold $\hat{M}$, it holds that $\dom(\hat{D}_{\rm max})=\SobH{m}(\hat{M},E)$. We conclude that $f\in \SobH{m}(M,E)$ and therefore $g=\gamma(f)\in \SobHH{ m-{{\frac12}}}(\Sigma;E\otimes \C^m)$. 
\end{proof}

\begin{remark}
It is also proven in \cite{seeley65} that the complementary projection $(1-P_{\mathcal{C}})$ is (up to a finite rank projection) a projection onto the Hardy space of the complementary manifold with boundary $\hat{M}\setminus \interior{M}$. Note that Seeley's construction depends on the choice of the closed manifold $\hat{M}$ and the extension of $D$ to an elliptic operator $\hat{D}$ on $\hat{M}$. In other words, the Calderon projection is a projection onto the Hardy space along the Hardy space of the complement. For this reason, the Calderon projection was denoted by $P^+$ in \cite{seeley65} and $C^+$ in \cite{grubbdistop}. Our notation $P_\mathcal{C}$ is chosen to avoid confusion with the notation for spectral projectors from \cite{BB,BBan}. 

We remark that a Calderon projection $P_{\mathcal{C}}$ such that  $(1-P_{\mathcal{C}})$ is (up to a finite rank projection) a projection onto the Hardy space of the complementary manifold is uniquely determined modulo smoothing operator. {\bf We therefore implicitly will use the Calderon projection $P_\mathcal{C}$ from Theorem \ref{refinedseeley} throughout the paper.}
\end{remark}

\begin{corollary}
Let $D$ be an elliptic differential operator of order $m>0$. The short exact sequence of Hilbert spaces 
$$0\to \dom(D_{\rm min})\to \dom(D_{\rm max})\xrightarrow{\gamma} \checkH(D)\to 0,$$
is split by the continuous mapping 
$$E:\checkH(D)\to \dom(D_{\rm max}),\quad E=\mathcal{K}+E_0\circ (1-P_\Ca),$$
where $E_0:\SobHH{{m-{\frac12}}}(\Sigma;E\otimes \C^m)\to \SobH{m}(M;E)$ is any continuous splitting of the trace mapping $\SobH{m}(M;E)\to \SobHH{{m-{\frac12}}}(\Sigma;E\otimes \C^m)$.
\end{corollary}

\begin{proof}
We have that $E=EP_\Ca+E(1-P_\Ca)$. The operators $EP_\Ca=\mathcal{K}$ and $E(1-P_\Ca)=E_0(1-P_\Ca)$ are continuous by Theorem \ref{Thm:CaldProj}, so $E$ is continuous. Moreover, we have that $\gamma E=\gamma \mathcal{K}+\gamma E_0(1-P_\Ca)=P_\Ca+(1-P_\Ca)=1$ so $E$ is a splitting.
\end{proof}

\begin{example}
\label{funnyboudnarydadreturn}
Let us provide some further context for Example \ref{funnyboudnarydad}.
We again consider the generalised boundary condition $B_{\rm Sob} = \bigoplus_{j=0}^{m-1} \SobH{m-{\frac12}-j}(\Sigma; E)$ for an elliptic differential operator $D$ of order $m>0$. Letting $\Proj{\Cac} = (I - \Proj{\Ca})$, the complementary projection to $\Proj{\Ca}$, 
$$B_{\rm Sob} + \Ca_D = \Proj{\Cac} \SobHH{{m-\frac12}}(\Sigma;E\otimes \C^m) \oplus \Ca_D = \Cac \oplus \Ca = \checkH(D),$$
by Theorem \ref{refinedseeley}.
We can now use Theorem \ref{Thm:ClosedRangeCharABS} to give an alternate proof of the fact that $\ran(D_{{\rm B}_{\rm Sob}})$ is closed. We can also use Lemma \ref{Lem:RanABS} to conclude $\ran(D_{{\rm B}_{\rm Sob}})= \ran(D_{\max})$ which we formulate in the following corollary. 
\end{example} 

\begin{corollary}
\label{rangeofmax}
Let $D$ be an elliptic differential operator of order $m>0$ on a manifold $M$ with boundary $\Sigma$ acting between sections of the Hermitian vector bundles $E$ and $F$. Then it holds that 
$$\ran(D_{\max})=D\SobH{m}(M;E),$$
as closed subspaces of $\Lp{2}(M;F)$ whose codimension is $\dim\ker(D^\dagger_{\rm min})$.
\end{corollary}

The next result is another corollary of Theorem \ref{refinedseeley}.

\begin{corollary}
\label{decomposingcheckspacewithcalderon}
Let $D$ be an elliptic differential operator of order $m>0$. The boundary condition $\Ca^c:=(1-P_\Ca)\SobHH{{m-{\frac12}}}(\Sigma; E\otimes \C^m)$ is regular and satisfies: 
\begin{itemize}
\item $\Ca^c$ is closed in $\SobHH{{m-{\frac12}}}(\Sigma; E\otimes \C^m)$ and in $\checkH(D)$ and for $x\in \Ca^c$, 
$$\|x\|_{\SobHH{{m-{\frac12}}}(\Sigma; E\otimes \C^m)}\simeq \|x\|_{\checkH(D)}.$$
\item $\Ca^c\cap \Ca_D=0$ and $\Ca^c+\Ca_D=\checkH(D)$.
\item The inclusion maps defines a Banach space isomorphism
$$\Ca^c\oplus \Ca_D\cong \checkH(D),$$
where $\Ca_D$ is equipped with either the norm from $\checkH(D)$ or the norm from $\SobHH{{-{\frac12}}}(\Sigma; E\otimes \C^m)$ (the two are equivalent on $\Ca_D$).
\end{itemize}
\end{corollary}

\begin{remark}
It follows from the proof Theorem \ref{Thm:ClosedRangeCharHO} and Corollary \ref{decomposingcheckspacewithcalderon} that an elliptic differential operator $D$ admits an invertible realisation if and only if $\ker(D_{\min})=0$ and $\ker(D_{\min}^\dagger)=0$. Indeed, for any realisation $D_{\rm B}$ we have the inclusions $\ker(D_{\min})\subseteq \ker(D_{\rm B})$ and $\ker(D_{\min}^\dagger)=\ker(D_{\rm B}^*)$ so $\ker(D_{\min})=0$ and $\ker(D_{\min}^\dagger)=0$ is a necessary condition. The converse follows from that $\ker(D_{\min})=0$ and $\ker(D_{\min}^\dagger)=0$ if and only if $D_{\Ca^c}$ is invertible by the proof Theorem \ref{Thm:ClosedRangeCharHO} and Corollary \ref{decomposingcheckspacewithcalderon}.
\end{remark}

For general regular boundary conditions, we have a statement similar to Corollary \ref{decomposingcheckspacewithcalderon}:

\begin{proposition}
\label{Prop:FredEll} 
Let $D$ be an elliptic differential operator of order $m>0$ and $B$ be a regular boundary condition. Then the following holds:\begin{enumerate}[(i)]
\item \label{Prop:FredEll1} 
	$(B, \Ca_D)$ is a Fredholm pair in $\checkH(D)$ and 
	$$ B^\ast \cap \Ca_{D^\dagger} \cong \faktor{\checkH(D)}{(B + \Ca_D)}.$$
\item \label{Prop:FredEll2} 
	There is a finite dimensional space $F\subseteq \SobHH{{m-{\frac12}}}(\Sigma; E\otimes \C^m)$  
	and a subspace $B_0\subseteq B+F$ of finite codimension defining a regular boundary condition with $\ker D_{\rm B_0}=\ker(D_{\min})$ for which we can write 
	$$\dom(D_{\rm max})=\dom(D_{{\rm B}_0})\oplus_{\ker D_{\min}}\ker D_{\rm max},$$
	and 
	$$\checkH(D)=B_0\oplus\Ca_D.$$
\end{enumerate}
\end{proposition}

The contents of item ii) in Theorem \ref{Prop:FredEll} can be interpreted as the statement that any regular boundary condition is up to finite dimensional perturbations a complement to the Hardy space in the Cauchy data space. In light of Theorem \ref{Thm:FredChar}, or rather Corollary \ref{Cor:FredInd}, there are constraints on the dimension of $F$ and the codimension of $B_0\subseteq B+F$ given by the identity: 
$$\dim(F)-\dim\left(\faktor{B+F}{B_0}\right)=\dim \ker (D_{\min}) - \dim \ker (D_{\min}^\dagger)-\indx(D_{\rm B}).$$

\begin{proof}
Item \ref{Prop:FredEll1} follows from Theorem \ref{Thm:FredChar}. 
If $B$ satisfies that $\ker(D_{\rm B}^*)=\ker(D_{\rm min}^\dagger)$ and $\ker D_{\rm B}=\ker(D_{\min})$, then with $B_0=B$ and $F=0$, item \ref{Prop:FredEll2} follows from Corollary \ref{charcomplements} applied to the FAP with closed range $\mathcal{T}_{\rm B}:=(D_{\rm B},D^\dagger_{\rm min})$. 

To prove item \ref{Prop:FredEll2}, we first prove that for any regular boundary condition $B$ there is a finite dimensional space $F\subseteq \SobHH{{m-{\frac12}}}(\Sigma; E\otimes \C^m)$ with $(B+F)\cap \Ca_D=B\cap \Ca_D$ and $B+F+\Ca_D= \checkH(D)$. 
Indeed, $B+F$ is clearly semi-regular and will be regular since $(B+F)^*\subseteq B^*$.
To construct such an $F$, we note that item \ref{Prop:FredEll1} implies that the inclusion $\faktor{B}{B\cap \Ca_D}\to \faktor{\checkH(D)}{\Ca_D}$ has finite codimension. Indeed, we have that 
$$\faktor{(\checkH(D)/\Ca_D)}{(B/B\cap \Ca_D)}=\faktor{\checkH(D)}{(B + \Ca_D)}.$$
As such, there are plenty of finite-dimensional subspaces $F\subseteq  \checkH(D)$ with $(B+F)\cap \Ca_D=B\cap \Ca_D$ and $B+F+\Ca_D= \checkH(D)$. 
That we can take $F\subseteq \checkH(D)\cap \SobHH{{m-{\frac12}}}(\Sigma; E\otimes \C^m)$ follows from Example \ref{funnyboudnarydadreturn} which implies that the inclusion $\SobHH{{m-{\frac12}}}(\Sigma; E\otimes \C^m)\hookrightarrow \faktor{\checkH(D)}{\Ca_D}$ induces an isomorphism
\begin{equation*}
\faktor{\SobHH{{m-{\frac12}}}(\Sigma; E\otimes \C^m)}{\left(\SobHH{{m-{\frac12}}}(\Sigma; E\otimes \C^m)\cap \Ca_D\right)}\cong \faktor{\checkH(D)}{\Ca_D}.
\end{equation*}

To finish the proof, we need to construct $B_0$. Upon replacing $B$ with $B+F$, the argument in the preceding paragraph allow us to assume that the regular boundary condition $B$ satisfies $\ker(D_{\rm B}^*)=\ker(D_{\rm min}^\dagger)$. We need to construct a regular boundary condition $B_0\subseteq B$ of finite codimension such that $B_0 + \Ca_D=B + \Ca_D$ and $B_0\cap \Ca_D=0$. Consider any $B_0\subseteq \SobHH{{m-{\frac12}}}(\Sigma; E\otimes \C^m)$ which forms a complement of $B\cap \Ca_D$ in $B$. Since $B$ is regular, $B\cap \Ca_D\subseteq \SobHH{{m-{\frac12}}}(\Sigma; E\otimes \C^m)$ is finite-dimensional so $B_0\subseteq B$ has finite codimension. Moreover, $B_0$ is closed in $\checkH(D)$ and by construction $B_0 + \Ca_D=B + \Ca_D$ and $B_0\cap \Ca_D=0$. It remains to prove that we can take the complement $B_0$ to satisfy that $B_0^*\subseteq \SobHH{{m-{\frac12}}}(\Sigma; F\otimes \C^m)$. We have that $B^*\subseteq B_0^*$ has finite codimension and using Proposition \ref{boundarypairingforsoblov} we deduce that we can choose $B_0$ such that $\faktor{B_0^*}{B^*}\subseteq \faktor{\SobHH{{m-{\frac12}}}(\Sigma; F\otimes \C^m)}{B^*}$ which proves the theorem.
\end{proof}

\subsection{Boundary decomposing projectors}
\label{bounddecoddsp}

With Theorem \ref{Thm:CaldProj}, we are able to prove results of use when studying graphical decompositions of regular boundary conditions, see more in \cite{BBan} and Subsection \ref{subsec:charofellfirstorder} below. We make the following definition that we shall make use of in the first order case below in Subsection \ref{subsec:charofellfirstorder}

\begin{definition}
\label{boundafefidodo}
Let $D$ be an elliptic differential operator of order $m>0$. We say that a projection $\GProj_+$ is boundary decomposing for $D$ if the following conditions are satisfied 
\begin{enumerate}[({P}1)]
\item \label{Def:EStartHO} \label{Def:E1HO} 
	$\GProj_+: \SobHH{ \alpha}(\Sigma; E\otimes \C^m) \to \SobHH{ \alpha}(\Sigma;E\otimes \C^m)$ is a bounded projection for $\alpha \in \left\{-\frac12, m-\frac12\right\}$, 
\item \label{Def:E2HO} 
	$\GProj_+: \checkH(D) \to \checkH(D)$, and $\GProj_- := (I - \GProj_+):  \checkH(D) \to \SobHH{{m-\frac12}}(\Sigma;E\otimes \C^m)$,  and
\item \label{Def:EEndHO} \label{Def:E3HO} 
	$\norm{u}_{\checkH(D)} \simeq \norm{\GProj_- u}_{\SobHH{{m-\frac12}}(\Sigma;E\otimes \C^m)} + \norm{\GProj_+ u}_{\SobHH{-{\frac12}}(\Sigma;E\otimes \C^m)}.$
\end{enumerate} 
\end{definition}

The next proposition follows from the open mapping theorem.

\begin{proposition}
\label{prop:bcimpliessppsls}
Let $D$ be an elliptic differential operator of order $m>0$ and $\GProj_+$ a boundary decomposing projection for $D$. Then 
$$\GProj_+ \checkH(D) = \GProj_+ \SobHH{{-\frac12}}(\Sigma;E)\quad\mbox{and}\quad \GProj_- \checkH(D) = \GProj_- \SobHH{{m-\frac12}}(\Sigma;E),$$
with equivalent norms.
\end{proposition}

Let us state yet another corollary of Theorem \ref{Thm:CaldProj}.

\begin{corollary}
\label{Cor:CalCheck}
The Calderón projection $P_\mathcal{C}$ of an elliptic differential operator $D$ is boundary decomposing for $D$. Moreover, the complementary subspace $\Cac := (I - \Proj{\Ca})\checkH(D)=(1-P_\mathcal{C})\SobHH{ m-{{\frac12}}}(\Sigma;E\otimes \C^m)$ defines a regular boundary condition with $\ker(D_{\Cac}) = \ker(D_{\min})$ and characterises $\dom(D_{\rm max})$ via the following continuously split short exact sequence of Hilbert spaces
$$0\to \dom(D_{\Cac})\to \dom(D_{\rm max})\xrightarrow{P_\mathcal{C}\circ \gamma} \mathcal{C}\to 0.$$
\end{corollary} 

\begin{proof}
By the fact that $\Proj{\Ca}$ is a pseudo-differential operator of order zero in the Douglis-Nirenberg calculus, it automatically satisfies \ref{Def:E1HO}. The properties \ref{Def:E2HO} and \ref{Def:E3HO} follow from Theorem \ref{refinedseeley}. It follows that $\Cac$ is a regular boundary condition. Since the restricted trace map is a surjection $\ker(D_{\Cac}) \to \Cac\cap \Ca=0$ with kernel $\ker(D_{\min})$, we conclude that $\ker(D_{\Cac}) = \ker(D_{\min})$. 
\end{proof}

\begin{example}
The Hardy space itself is a boundary condition, so we may consider the closed operator $D_{\Ca_D}$. The realisation $D_{\Ca_D}$ is called the soft extension, or the Krein extension, and is self-adjoint if $D$ is formally self-adjoint. Its properties were further studied in \cite{grubbsoft}.
It is clear that $\ker(D_{\Ca_D}) = \ker(D_{\max})$, which we have noted is infinite dimensional if $\dim(M)>1$.
By Proposition \ref{Prop:DomCompABS}, setting $B_2 = \checkH(D)$ and $B_1 = \Ca_D$, we obtain that 
$$ \faktor{\dom(D_{\max})}{\dom(D_{\Ca_D})} \cong \faktor{\checkH(D)}{\Ca_D} \cong \Cac,$$
where $\Cac$ is the kernel of $\Proj{\Ca}$.
By Corollary \ref{Cor:CalCheck}, we have that $\Cac$ is infinite dimensional.
That is, $D_{\Ca}$ is an operator which differs from $D_{\max}$ on an infinite dimensional subspace but still has infinite dimensional kernel.

This can be repeated on replacing $\Ca$ by any infinite dimensional closed subspace of $\Ca$.
Therefore, there's an abundance of boundary conditions with infinite dimensional kernel which is different from $D_{\max}$ on an infinite dimensional subspace.
\end{example} 

We now state a sufficient condition for a projection to be boundary decomposing. 

\begin{theorem}
\label{thm:bodunarodoad}
Let $D$ be an elliptic differential operator of order $m>0$ and $P$ a continuous projection on $\SobHH{{m-\frac12}}(\Sigma;E\otimes \C^m)$. Assume that $P$ satisfies the following: 
\begin{itemize} 
\item $P$ is continuous also in the $\checkH(D)$-norm and the $\SobHH{{-\frac12}}(\Sigma;E\otimes \C^m)$-norm.
\item The operator $A:=P_\mathcal{C}-(1-P)$ defines a Fredholm operator on $\SobHH{{m-\frac12}}(\Sigma;E\otimes \C^m)$ and by continuity extends to a Fredholm operator on $\checkH(D)$. 
\end{itemize} 
Then $P$ is boundary decomposing, and in particular 
$$\norm{u}_{\checkH(D)} \simeq \norm{(1-P) u}_{\SobHH{{m-\frac12}}(\Sigma;E\otimes \C^m)} + \norm{P u}_{\SobHH{-{\frac12}}(\Sigma;E\otimes \C^m)}.$$
\end{theorem}

\begin{proof}
By the open mapping theorem and density it suffices to prove that the assumptions on $P$ ensure that 
$$\norm{(1-P)u}_{\checkH(D)} \simeq \norm{(1-P) u}_{\SobHH{{m-\frac12}}(\Sigma;E\otimes \C^m)} \quad\mbox{and}\quad \norm{Pu}_{\checkH(D)} \simeq \norm{P u}_{\SobHH{-{\frac12}}(\Sigma;E\otimes \C^m)},$$
for $u\in \SobHH{{m-\frac12}}(\Sigma;E\otimes \C^m)$.

Since we have a continuous embedding $\SobHH{{m-\frac12}}(\Sigma;E\otimes \C^m)\hookrightarrow \checkH(D)$, we conclude that there is an estimate   
$\norm{(1-P)u}_{\checkH(D)} \lesssim \norm{(1-P) u}_{\SobHH{{m-\frac12}}(\Sigma;E\otimes \C^m)}$ for $u\in \SobHH{{m-\frac12}}(\Sigma;E\otimes \C^m)$. We now prove the converse estimate. Since $A$ is Fredholm on $\SobHH{{m-\frac12}}(\Sigma;E\otimes \C^m)$, we conclude that there is a finite rank operator $F$ on $\SobHH{{m-\frac12}}(\Sigma;E\otimes \C^m)$ such that 
$$\|u\|_{\SobHH{{m-\frac12}}(\Sigma;E\otimes \C^m)} \simeq \|Au\|_{\SobHH{{m-\frac12}}(\Sigma;E\otimes \C^m)}+\|Fu\|_{\SobHH{{m-\frac12}}(\Sigma;E\otimes \C^m)}.$$
By density, we can assume that $F$ extends to a finite rank operator on $\checkH(D)$. We then have that 
\begin{equation}
\label{firsteeksksdldl}
\norm{(1-P) u}_{\SobHH{{m-\frac12}}(\Sigma;E\otimes \C^m)}\lesssim \|A(1-P)u\|_{\SobHH{{m-\frac12}}(\Sigma;E\otimes \C^m)}+\|F(1-P)u\|_{\SobHH{{m-\frac12}}(\Sigma;E\otimes \C^m)}.
\end{equation}
But $A(1-P)=(P_\mathcal{C}-(1-P))(1-P)=(P_\mathcal{C}-1)(1-P)$ so 
\begin{align}
\label{secondeeksksdldl}
\|A(1-P)u\|&_{\SobHH{{m-\frac12}}(\Sigma;E\otimes \C^m)}+\|F(1-P)u\|_{\SobHH{{m-\frac12}}(\Sigma;E\otimes \C^m)}=\\
\nonumber 
=&\|(1-P_\mathcal{C})(1-P)u\|_{\SobHH{{m-\frac12}}(\Sigma;E\otimes \C^m)}+\|F(1-P)u\|_{\SobHH{{m-\frac12}}(\Sigma;E\otimes \C^m)}=\\
\nonumber 
=&\|(1-P_\mathcal{C})(1-P)u\|_{\checkH(D)}+\|F(1-P)u\|_{\SobHH{{m-\frac12}}(\Sigma;E\otimes \C^m)}\lesssim \\
\nonumber 
\lesssim &\|(1-P)u\|_{\checkH(D)}+\|F(1-P)u\|_{\checkH(D)}\lesssim \|(1-P)u\|_{\checkH(D)}.
\end{align}
In the second last estimate we used that all norms on finite-dimensional spaces are equivalent and in the third estimate we used that $F$ extends to a finite rank operator on $\checkH(D)$. Combining the estimates \eqref{firsteeksksdldl} and \eqref{secondeeksksdldl} we arrive at the estimate $\norm{(1-P) u}_{\SobHH{{m-\frac12}}(\Sigma;E\otimes \C^m)}\lesssim \norm{(1-P)u}_{\checkH(D)}$.

The proof that $\norm{Pu}_{\checkH(D)} \simeq \norm{P u}_{\SobHH{-{\frac12}}(\Sigma;E\otimes \C^m)}$ goes ad verbatim with $\checkH(D)$ instead of $\SobHH{{m-\frac12}}(\Sigma;E\otimes \C^m)$  and $\SobHH{{-\frac12}}(\Sigma;E\otimes \C^m)$ instead of $\checkH(D)$ and using that $A$ is Fredholm on $\checkH(D)$. 
\end{proof}

\begin{lemma}
\label{boundednessofpseudodosnnchck}
Let $D$ be an elliptic differential operator of order $m>0$ and $T\in \Psi^{\pmb 0}_{\rm cl}(\Sigma; E\otimes \C^m)$. Then $T$ defines a bounded operator on $\checkH(D)$ if and only if $(1-P_\mathcal{C})TP_\mathcal{C}\in \Psi^{\pmb{-m}}_{\rm cl}(\Sigma; E\otimes \C^m)$.
\end{lemma}

\begin{proof}
Assume first that $(1-P_\mathcal{C})TP_\mathcal{C}\in \Psi^{\pmb{-m}}_{\rm cl}(\Sigma; E\otimes \C^m)$ holds. The property $(1-P_\mathcal{C})TP_\mathcal{C}\in \Psi^{\pmb{-m}}_{\rm cl}(\Sigma; E\otimes \C^m)$ implies that 
\begin{align*}
T=&P_\mathcal{C}TP_\mathcal{C}+(1-P_\mathcal{C})T(1-P_\mathcal{C})+P_\mathcal{C}T(1-P_\mathcal{C})+(1-P_\mathcal{C})TP_\mathcal{C}\\
=&P_\mathcal{C}TP_\mathcal{C}+(1-P_\mathcal{C})T(1-P_\mathcal{C})+P_\mathcal{C}T(1-P_\mathcal{C})+\Psi^{\pmb{-m}}_{\rm cl}(\Sigma; E\otimes \C^m).
\end{align*}
It now follows that $T$ acts boundedly on $\checkH(D)$ from the facts that $\Psi^{\pmb{-m}}_{\rm cl}(\Sigma; E\otimes \C^m)$ acts boundedly on $\checkH(D)$ and $\Psi^{\pmb 0}_{\rm cl}(\Sigma; E\otimes \C^m)$ acts boundedly $\SobHH{{s}}(\Sigma;E\otimes \C^m)\to \SobHH{{t}}(\Sigma;E\otimes \C^m)$ for all $s,t\in \R$ with $t\leq s$.

Conversely, assume that $T$ acts boundedly on $\checkH(D)$. Since $T\in \Psi^{\pmb 0}_{\rm cl}(\Sigma; E\otimes \C^m)$, the argument above implies that the pseudodifferential operator $(1-P_\mathcal{C})TP_\mathcal{C}$ act boundedly on $\checkH(D)$. In particular, $(1-P_\mathcal{C})TP_\mathcal{C}$ will have to extend to a continuous operator $\SobHH{{-\frac12}}(\Sigma;E\otimes \C^m)\to \SobHH{{m-\frac12}}(\Sigma;E\otimes \C^m)$ and so it must be an operator of order $-m$, i.e. $(1-P_\mathcal{C})TP_\mathcal{C}\in \Psi^{\pmb{-m}}_{\rm cl}(\Sigma; E\otimes \C^m)$.
\end{proof}

\begin{corollary}
\label{pseudostahtdecompose}
Let $D$ be an elliptic differential operator of order $m>0$ and $P\in \Psi^{\pmb 0}_{\rm cl}(\Sigma; E\otimes \C^m)$ a projection. Assume that $P$ satisfies the following: 
\begin{itemize} 
\item $[P_\mathcal{C},P]\in \Psi^{\pmb{-m}}_{\rm cl}(\Sigma; E\otimes \C^m)$.
\item The operator $A:=P_\mathcal{C}-(1-P)\in \Psi^{\pmb 0}_{\rm cl}(\Sigma; E\otimes \C^m)$ is elliptic. 
\end{itemize} 
Then $P$ is boundary decomposing, and in particular 
$$\norm{u}_{\checkH(D)} \simeq \norm{(1-P) u}_{\SobHH{{m-\frac12}}(\Sigma;E\otimes \C^m)} + \norm{P u}_{\SobHH{-{\frac12}}(\Sigma;E\otimes \C^m)}.$$
\end{corollary} 

\begin{proof}
Since $(1-P_\mathcal{C})PP_\mathcal{C}=-[P_\mathcal{C},P]P_\mathcal{C}\in \Psi^{\pmb{-m}}_{\rm cl}(\Sigma; E\otimes \C^m)$, the projection $P$ extends to a continuous projection on $\checkH(D)$ by Lemma \ref{boundednessofpseudodosnnchck}. Therefore $P$ defines a continuous projection on $\SobHH{{m-{\frac12}}}(\Sigma;E\otimes \C^m)$, $\SobHH{{-{\frac12}}}(\Sigma;E\otimes \C^m)$ and $\checkH(D)$. Since $A:=P_\mathcal{C}-(1-P)$ is elliptic, it defines a Fredholm operator on $\SobHH{{m-{\frac12}}}(\Sigma;E\otimes \C^m)$. 

Let us now prove that $A$ extends to a Fredholm operator on $\checkH(D)$. Let $R$ denote a parametrix of $A$. We proceed by computing in the formal symbol algebra $\Psi^{\pmb 0}_{\rm cl}(\Sigma; E\otimes \C^m)/\Psi^{-\infty}(\Sigma; E\otimes \C^m)$. For a zeroth order pseudodifferential operator $T$ we write $\lfloor T\rfloor$ for its associated full symbol. In this notation, $\lfloor R\rfloor=\lfloor A\rfloor^{-1}$. We compute that 
\begin{align*}
\lfloor [P_\mathcal{C},R]\rfloor=&[\lfloor P_\mathcal{C}\rfloor,\lfloor A\rfloor^{-1}]=-\lfloor A\rfloor^{-1}[\lfloor P_\mathcal{C}\rfloor,\lfloor A\rfloor]\lfloor A\rfloor^{-1}=\\
=&-\lfloor A\rfloor^{-1}[\lfloor P_\mathcal{C}\rfloor,\lfloor P\rfloor]\lfloor A\rfloor^{-1}=-\lfloor R[P_\mathcal{C},P]R\rfloor.
\end{align*}
Therefore it holds that $[P_\mathcal{C},R]=-R[P_\mathcal{C},P]R+\Psi^{-\infty}(\Sigma;E\otimes \C^m)$. Since $[P_\mathcal{C},P]\in \Psi^{\pmb{-m}}_{\rm cl}(\Sigma; E\otimes \C^m)$, we conclude that $(1-P_\mathcal{C})RP_\mathcal{C}=-[P_\mathcal{C},R]P_\mathcal{C}\in \Psi^{\pmb{-m}}_{\rm cl}(\Sigma; E\otimes \C^m)$ so $R$ acts boundedly on $\checkH(D)$. Therefore, density and the fact that smoothing operators act compactly on $\checkH(D)$ implies that $A$ extends to a Fredholm operator on $\checkH(D)$.

The corollary now follows from Theorem \ref{thm:bodunarodoad}.
\end{proof}

Below in Theorem \ref{equivocndndforpseudolocal} we shall see that under assumptions on $P$ similar to those in Corollary \ref{pseudostahtdecompose}, $B_P:=(1-P)\SobHH{{m-{\frac12}}}(\Sigma;E\otimes \C^m)$ is a regular boundary condition. For the precise relations between the assumptions in Theorem \ref{equivocndndforpseudolocal} and Corollary \ref{pseudostahtdecompose}, see the formulation of Theorem \ref{elliptidcocoddn}.

\begin{example}
\label{firstorderbounddecom}
Let us briefly comment on the case of a first order elliptic differential operator $D$. In this case, Corollary \ref{pseudostahtdecompose} shows that any $P\in \Psi^0(\Sigma;E)$ with $P$ and $P_\mathcal{C}$ commuting at principal symbol level (as elements of $C^\infty(S^*M;\End(\pi^*E))$) and with $p_+(D)-(1-\sigma_0(P))\in C^\infty(S^*M;\mathrm{Aut}(\pi^*E))$, defines a boundary decomposing projection. Below in Lemma \ref{specprojvscalderon}, we prove that $p_+(D)$ coincides with the principal symbol of the positive spectral projection $\chi^+(A)$ for an adapted boundary operator. In particular, $P:=\chi^+(A)$ defines a boundary decomposing projection by Corollary \ref{pseudostahtdecompose}. This proves the claimed equality $\checkH(D)=\checkH_A(D)$ from Subsection \ref{subsec:ellfirstorderfirst}, thus reconciling the topological considerations of $\checkH(D)$ from Subsection \ref{subsec:seeleyoncalderon} with that of \cite{BBan} in the first order case.
\end{example}

\begin{example}
\label{ex:dirichletandneumann}
Consider the case of a second order elliptic differential operator $D$ acting on a vector bundle $E$ with scalar positive principal symbol (e.g. a Laplace type operator on a vector bundle). In this case, it is readily verified from solving Equation \eqref{boundaryeq} that 
$$\sigma_{\pmb 0}(P_\mathcal{C})(x',\xi')=
\frac{1}{2}
\begin{pmatrix}
1_E& |\xi'|^{-1}1_E\\
|\xi'|1_E& 1_E\end{pmatrix},$$
for the Riemannian metric on $\Sigma$ induced from the Riemannian metric that $\sigma_2(D)$ defines on $M$. We can consider the projectors defining Dirichlet or Neumann conditions (from the rule $B_P:=(1-P)\SobHH{m-{\frac12}}(\Sigma;E\otimes \C^2)$ or the explicit condition $P\gamma u=0$) are given by 
$$P_D:=\begin{pmatrix}1&0\\0&0\end{pmatrix}\quad\mbox{and}\quad P_N:=\begin{pmatrix}0&0\\0&1\end{pmatrix}.$$
We then have that 
$$[P_D,P_\mathcal{C}]=\frac{1}{2}
\begin{pmatrix}
0& |\xi'|^{-1}1_E\\
-|\xi'|1_E& 0\end{pmatrix} 
\quad\mbox{and}\quad 
[P_N,P_\mathcal{C}]=\frac{1}{2}
\begin{pmatrix}
0& -|\xi'|^{-1}1_E\\
|\xi'|1_E&0 \end{pmatrix}.$$
Therefore, $(1-P_\mathcal{C})P_NP_\mathcal{C}, (1-P_\mathcal{C})P_DP_\mathcal{C}\in \Psi^{\pmb{0}}_{\rm cl}(\Sigma; E\otimes \C^2)$ with non-vanishing principal symbol. By Lemma \ref{boundednessofpseudodosnnchck}, we conclude that neither $P_D$ nor $P_N$ act boundedly on $\checkH(D)$. In particular, these projectors are not boundary decomposing. A more appropriate projector is considered in Subsection \ref{highregofdiricl}.
\end{example}

\section{The Cauchy data space of an elliptic differential operator}
\label{sec:cuchsosa}

In this section we will refine our analysis of the Cauchy data space and boundary conditions further. 

\subsection{Constructing regular pseudodifferential boundary conditions}
\label{subsec:cuchsosalocal}

Theorem \ref{refinedseeley} implies that we can describe the Cauchy data space $\checkH(D)$ in terms of the Calderón projection. Unfortunately, the construction of the Calderón projection requires knowledge from the interior of $M$ and not just symbolic information of $D$ at the boundary. In this subsection, we give a construction of projectors from the Douglis-Nirenberg calculus that rely only on the germ of $D$ at the boundary yet still describes the Cauchy data space as in Corollary \ref{decomposingcheckspacewithcalderon}. The considerations of this subsection implies that many structural properties of $\checkH(D)$ extends to the case of noncompact $M$ with compact boundary. Recall the definition of a boundary decomposing projection from Definition \ref{boundafefidodo}.

\begin{lemma}
\label{approximcladldlad}
Let $D$ be an elliptic differential operator of order $m>0$. Then the following scheme produces a boundary decomposing pseudodifferential projection $P\in \Psi^{\pmb 0}_{\rm cl}(\Sigma; E\otimes \C^m)$ with $P-P_\mathcal{C}\in \Psi^{-\pmb m}_{\rm cl}(\Sigma; E\otimes \C^m)$.
\begin{enumerate}
\item Choose $P_1\in \Psi^{\pmb 0}_{\rm cl}(\Sigma; E\otimes \C^m)$ with $\sigma_{\pmb 0}(P_1)=p_+(D)$.
\item Construct $P_2\in \Psi^{\pmb 0}_{\rm cl}(\Sigma; E\otimes \C^m)$ by choosing a $T_2\in \Psi^{-\pmb 1}_{\rm cl}(\Sigma; E\otimes \C^m)$ with $\sigma_{-\pmb 1}(T_2)=\sigma_{-\pmb 1}(P_1-P_\mathcal{C})$ and set $P_2:=P_1-T_2$.
\item Proceed iteratively and define $P_{k+1}\in \Psi^{\pmb 0}_{\rm cl}(\Sigma; E\otimes \C^m)$ from $P_k$ by choosing a $T_{k+1}\in \Psi^{-\pmb {k}}_{\rm cl}(\Sigma; E\otimes \C^m)$ with $\sigma_{-\pmb{k}}(T_{k+1})=\sigma_{-\pmb{k}}(P_k-P_\mathcal{C})$ and set $P_{k+1}:=P_k-T_{k+1}$. Stop at $k=m$.
\item Choose a compact complex contour $\Gamma$ in the open right half-plane containing $1$ in its interior domain and not intersecting the spectrum of $P_m:\SobHH{{-\frac12}}(\Sigma;E\otimes \C^m)\to \SobHH{{-\frac12}}(\Sigma;E\otimes \C^m)$, and define $P$ as the Riesz projection
$$P:=\frac{1}{2\pi i}\int_\Gamma (\lambda-P_m)^{-1}\mathrm{d}\lambda.$$
\end{enumerate}
Moreover, $B_P:=(1-P)\SobHH{{m-\frac12}}(\Sigma;E\otimes \C^m)$ is a regular boundary condition for $D$ that characterises $\dom(D_{\rm max})$ via the following continuously split short exact sequence of Hilbert spaces
$$0\to \dom(D_{{\rm B}})\to \dom(D_{\rm max})\xrightarrow{P\circ \gamma} P\SobHH{{-\frac12}}(\Sigma;E\otimes \C^m)\to 0.$$

\end{lemma}

\begin{proof}
We need to verify that the stated scheme is possible to carry out and that it produces a boundary decomposing pseudodifferential projection $P\in \Psi^{\pmb 0}_{\rm cl}(\Sigma; E\otimes \C^m)$ that defines a regular boundary condition. For the proof, it is beneficial to bear in mind that if $e$ is an idempotent in a unital algebra, then the spectrum of $e$ is $\{0,1\}$ and for $\lambda\neq 0,1$ we have that 
$$(\lambda-e)^{-1}=\lambda^{-1}(1-e)+(\lambda-1)^{-1}e.$$
In particular, if $e$ is an idempotent in a unital algebra its Riesz projectors are well defined and that onto the component $\{1\}$ coincides with $e$.

Surjectivity of the symbol mappings ensures that the construction of $P_m\in \Psi^{\pmb 0}_{\rm cl}(\Sigma; E\otimes \C^m)$ is possible. Indeed, $P_m$ will in local charts have a full symbol that up to degree $-m$ is determined by the formula \eqref{formulaforaprprood}. We note that the spectrum of $P_m:\SobHH{{s}}(\Sigma;E\otimes \C^m)\to \SobHH{{s}}(\Sigma;E\otimes \C^m)$ is by elliptic regularity independent of $s\in \R$. Since $\sigma_{\pmb 0}(P_m)$ by construction is a projection, the essential spectrum of $P_m:\SobHH{{s}}(\Sigma;E\otimes \C^m)\to \SobHH{{s}}(\Sigma;E\otimes \C^m)$ is precisely the set $\{0,1\}\subseteq \C$. This ensures the existence of the complex contour $\Gamma$. By construction, $P$ will be a projection and belong to $\Psi^{\pmb 0}_{\rm cl}(\Sigma; E\otimes \C^m)$ (since $\Gamma$ is compact). 

To prove that $P-P_\mathcal{C}\in \Psi^{-\pmb m}_{\rm cl}(\Sigma; E\otimes \C^m)$ it suffices to prove that $(\lambda-P_m)^{-1}-(\lambda-P_\mathcal{C})^{-1}\in \Psi^{-\pmb m}_{\rm cl}(\Sigma; E\otimes \C^m)$ for $\lambda\in \Gamma$. By construction, $\sigma_{-\pmb{m-1}}(P_m-P_\mathcal{C})=0$, so $P_m-P_\mathcal{C}\in \Psi^{-\pmb m}_{\rm cl}(\Sigma; E\otimes \C^m)$ and $(\lambda-P_m)^{-1}-(\lambda-P_\mathcal{C})^{-1}\in \Psi^{-\pmb m}_{\rm cl}(\Sigma; E\otimes \C^m)$ follows from the identity 
$$(\lambda-P_m)^{-1}-(\lambda-P_\mathcal{C})^{-1}=(\lambda-P_m)^{-1}(P_\mathcal{C}-P_m)(\lambda-P_\mathcal{C})^{-1}.$$
Since $P-P_\mathcal{C}\in \Psi^{-\pmb m}_{\rm cl}(\Sigma; E\otimes \C^m)$, $P$ is boundary decomposing because of Corollary \ref{pseudostahtdecompose}. The proof that $P$ defines a regular boundary condition is postponed until Theorem \ref{thm:sufficienfeidntofell} below. 
\end{proof}

\begin{remark}
The reader should be aware that the condition $P-P_\mathcal{C}\in \Psi^{-\pmb m}_{\rm cl}(\Sigma; E\otimes \C^m)$ is far from sufficient to guarantee that $P-P_\Ca$ is compact on $\checkH(D)$. This happens despite $P-P_\Ca$ being compact on $\SobHH{ s}(\Sigma;E\otimes \C^m)$ for all $s\in \R$ (it even has singular values behaving like $O(k^{-m/\dim(M)})$ as $k\to \infty$). We give an example thereof in Proposition \ref{compactnessfailure} below in order $m=1$.

We also note that the condition $P-P_\mathcal{C}\in \Psi^{-\pmb m}_{\rm cl}(\Sigma; E\otimes \C^m)$ is far from necessary for $P$ to be boundary decomposing (cf. Corollary \ref{pseudostahtdecompose}) or even for $B_P:=(1-P)\SobHH{{m-\frac12}}(\Sigma;E\otimes \C^m)$ to be a regular boundary condition for $D$ (cf. Theorem \ref{thm:sufficienfeidntofell} below). 
\end{remark}

\begin{proposition}[Localizability of $\checkH(D)$]
\label{localizingcheckspace}
Let $D$ be an elliptic differential operator of order $m>0$ on a manifold with boundary $M$. Pick a point $x_0\in \Sigma$ and coordinates $\kappa:U\to U_0\subseteq M$ around $x_0$. Let $D_0$ be an elliptic differential operator of order $m$ on $U$ such that $D_0$ and $\kappa^*D$ coincide on $\kappa^{-1}(U_0\cap\Sigma)=U\cap \partial \R^n_+$ to order $m$. Let $Q_U\in \Psi^{\pmb 0}_{\rm cl}(U;\kappa^*E)$ differ from an approximate Calderón projection for $D_0$ by a term in $ \Psi^{\pmb{-m}}_{\rm cl}(U;\kappa^*E)$. 

Then for any $u\in \mathcal{D}'(\Sigma;E)$, it holds that $\chi_0 u\in \checkH(D)$ for all $\chi_0\in \Ck[c]{\infty}(U_0)$ if and only if $(1-Q_U)[\chi \kappa^*u]\in \SobHH[\rm loc]{{m-{\frac12}}}(U;\kappa^*E\otimes \C^m)$ and $Q_U[\chi \kappa^*u]\in \SobHH[\rm loc]{{-{\frac12}}}(U;\kappa^*E\otimes \C^m)$ for all $\chi\in \Ck[c]{\infty}(U)$.
\end{proposition}

\begin{proof}
It follows from Theorem \ref{horsats} that for any $\chi,\chi'\in \Ck[c]{\infty}(U)$, $\chi(\kappa^*P_\Ca-Q_U)\chi'\in \Psi^{-\pmb m}_{\rm cl}(U;\kappa^*E)$. Therefore, $(1-P_\Ca) \chi_0 u\in \SobHH{{m-\frac12}}(\Sigma;E\otimes \C^m)$ for all $\chi_0\in \Ck[c]{\infty}(U_0)$ if and only if $(1-Q_U)[\chi \kappa^*u]\in \SobHH[\rm loc]{{m-{\frac12}}}(U;\kappa^*E\otimes \C^m)$ for all $\chi\in \Ck[c]{\infty}(U)$ and $P_\Ca \chi_0 u\in \SobHH{{-\frac12}}(\Sigma;E\otimes \C^m)$ for all $\chi_0\in \Ck[c]{\infty}(U_0)$ if and only if $Q_U[\chi \kappa^*u]\in \SobHH[\rm loc]{{-{\frac12}}}(U;\kappa^*E\otimes \C^m)$ for all $\chi\in \Ck[c]{\infty}(U)$. The proposition follows.
\end{proof}

\subsection{The boundary pairing and Calderón projectors of the formal adjoint}
\label{bpandcp}

 In this subsection, we shall study the relation between a Calderón projection of $D$ and a Calderón projection of its formal adjoint $D^\dagger$. 
We apply this to the boundary pairing.  First, we shall revisit the operators $\tau$ and $\scalebox{1.5}{a}$ from Subsection \ref{subsec:preliminaryboundary}. The next proposition follows from Lemma \ref{somecontinuuisuforata} and the following observations.

For a matrix $T=(T_{i,j})_{i,j=0}^{m-1}$ of pseudo-differential operators on a closed manifold, we write $T^*=(T_{j,i}^*)_{i,j}$ for its $\Lp{2}$-adjoint and 
$$T^\#:=\tau T^*\tau=(T_{m-i-1,m-j-1}^*)_{i,j},$$
where $\tau$ is defined in \eqref{taudeffed}. 
A short computation shows that if $T\in \Psi^{\pmb \alpha}_{\rm cl}(\Sigma; E\otimes \C^m,F\otimes \C^m)$ then $T^{\#}\in \Psi^{\pmb \alpha}_{\rm cl}(\Sigma; F\otimes \C^m,E\otimes \C^m)$ and $T^{*}\in \Psi^{\pmb \alpha}_{\rm cl}(\Sigma; F\otimes \C^m_{\rm op},E\otimes \C^m_{\rm op})$.

\begin{proposition}
\label{returnofsomecontinuuisuforata}
The operators $\scalebox{1.5}{a}$, $\widetilde{\scalebox{1.5}{a}}$, and $\tau$ from Subsection \ref{subsec:preliminaryboundary} satisfy the following.
\begin{enumerate}[(i)] 
\item For any $s\in \R$, $\tau$ defines a unitary isomorphism 
$$\SobHH{ s}(\Sigma;F\otimes \C^m_{\rm op})\to \SobHH{{s+m-1}}(\Sigma;F\otimes \C^m),$$
and for any $\alpha\in \R$, conjugation by $\tau$ defines a product preserving isomorphism 
$$\Psi^{\pmb \alpha}_{\rm cl}(\Sigma; E\otimes \C^m,F\otimes \C^m)\to  \Psi^{\pmb \alpha}_{\rm cl}(\Sigma; E\otimes \C^m_{\rm op},F\otimes \C^m_{\rm op}).$$
\item For any $s\in \R$, $\widetilde{\scalebox{1.5}{a}}$ defines a Banach space isomorphism 
$$\SobHH{ s}(\Sigma;E\otimes \C^m)\to \SobHH{ s}(\Sigma;F\otimes \C^m),$$
and $\scalebox{1.5}{a}$ defines a Banach space isomorphism 
$$\SobHH{ s}(\Sigma;E\otimes \C^m)\to \SobHH{{s-m+1}}(\Sigma;F\otimes \C^m_{\rm op}),$$
and for any $\alpha\in \R$, conjugation by $\scalebox{1.5}{a}$ (i.e. the mapping $T\mapsto \scalebox{1.5}{a}T\scalebox{1.5}{a}^{-1}$) defines a product preserving isomorphism 
$$\Psi^{\pmb \alpha}_{\rm cl}(\Sigma; E\otimes \C^m)\to  \Psi^{\pmb \alpha}_{\rm cl}(\Sigma;F\otimes \C^m_{\rm op}).$$
\end{enumerate}
\end{proposition}

\begin{proposition}
\label{someexresiosnaoadwatwo}
Let $D$ be an elliptic differential operator of order $m>0$. Let $\widetilde{\scalebox{1.5}{a}}$ and $\widetilde{\scalebox{1.5}{a}}_\dagger$ denote the matrices of differential operators constructed from $D$ and $D^\dagger$, respectively, as in Proposition \ref{proponboundarypairingforell}. It holds that $\widetilde{\scalebox{1.5}{a}}\in \Psi^{\pmb 0}_{\rm cl}(\Sigma; E\otimes \C^m,F\otimes \C^m)$ and $\widetilde{\scalebox{1.5}{a}}_\dagger\in \Psi^{\pmb 0}_{\rm cl}(\Sigma; F\otimes \C^m,E\otimes \C^m)$ are elliptic and invertible in the Douglis-Nirenberg calculus and satisfy 
$$\widetilde{\scalebox{1.5}{a}}^\#=\widetilde{\scalebox{1.5}{a}}_\dagger.$$
\end{proposition}

\begin{proof}
The proposition is a consequence of Proposition \ref{proponboundarypairingforell} and Lemma \ref{someexresiosnaoadwa}.
\end{proof}

The following theorem refines Proposition \ref{pplussandpminussdagger} to the operator level.

\begin{theorem}
\label{calderonforadjoint}
Let $D$ be an elliptic differential operator of order $m>0$. We assume that $P_\Ca\in \Psi^{\pmb 0}_{\rm cl}(\Sigma;E\otimes \C^m)$ is a Calderón projection for $D$ and $P_{\Ca^\dagger}\in \Psi^{\pmb 0}_{\rm cl}(\Sigma;F\otimes \C^m)$ is a Calderón projection for $D^\dagger$. Then it holds that 
\begin{align*}
P_{\Ca^\dagger}=&\scalebox{1.5}{a}_\dagger^{-1}(1-P_\Ca^*)\scalebox{1.5}{a}_\dagger+\Psi^{-\infty}_{\rm cl}(\Sigma;F\otimes \C^m)=\\
=&[\scalebox{1.5}{a}(1-P_\Ca)\scalebox{1.5}{a}^{-1}]^*+\Psi^{-\infty}_{\rm cl}(\Sigma;F\otimes \C^m).
\end{align*}
\end{theorem}

\begin{proof}
Before proving the main identity of the theorem, we remark that $\scalebox{1.5}{a}_\dagger^{-1}(1-P_\Ca^*)\scalebox{1.5}{a}_\dagger=[\scalebox{1.5}{a}(1-P_\Ca)\scalebox{1.5}{a}^{-1}]^*$ by Lemma \ref{someexresiosnaoadwa} and $\scalebox{1.5}{a}(1-P_\Ca)\scalebox{1.5}{a}^{-1}\in \Psi^{\pmb 0}_{\rm cl}(\Sigma;F\otimes \C^m_{\rm op})$ by Proposition \ref{returnofsomecontinuuisuforata}.

Let us proceed with the proof. It follows from \cite[Chapter II, Lemma 1.2]{grubb68} that, up to a smoothing operator, we have  
$$\omega_D(\xi,\eta)=\omega_D(P_{\Ca^\dagger}\xi,\eta)-\omega_D(\xi,P_\Ca \eta),$$
for $\xi\in \checkH(D^\dagger)$ and $\eta\in \checkH(D)$. By Proposition \ref{proponboundarypairingforell} and Lemma \ref{someexresiosnaoadwa}, this is equivalent to the equality (valid up to a smoothing operator)
$$\langle \scalebox{1.5}{a}_\dagger \xi,\eta\rangle_{\Lp{2}(\Sigma;E\otimes \C^m)}=\langle \scalebox{1.5}{a}_\dagger \xi,P_\Ca\eta\rangle_{\Lp{2}(\Sigma;E\otimes \C^m)}-\langle P_{\Ca^\dagger} \xi,\scalebox{1.5}{a}\eta\rangle_{\Lp{2}(\Sigma;F\otimes \C^m)}.$$
In particular, we have that 
$$\langle \scalebox{1.5}{a}_\dagger \xi,(1-P_\Ca)\eta\rangle_{\Lp{2}(\Sigma;E\otimes \C^m)}=-\langle P_{\Ca^\dagger} \xi,\scalebox{1.5}{a}\eta\rangle_{\Lp{2}(\Sigma;F\otimes \C^m)},$$
up to a smoothing operator. We deduce from Lemma \ref{someexresiosnaoadwa} that $(1-P_\Ca^*)\scalebox{1.5}{a}_\dagger=\scalebox{1.5}{a}_\dagger P_{\Ca^\dagger}$ up to a smoothing operator.
\end{proof}

\begin{remark}
The reader should compare Theorem \ref{calderonforadjoint} to Proposition \ref{describinglocalcndldldaadjoint} describing adjoints of boundary conditions defined from projectors.
\end{remark}

\begin{corollary}
\label{structureofa}
Let $D$ be an elliptic differential operator of order $m>0$. We assume that $P_\Ca\in \Psi^{\pmb 0}_{\rm cl}(\Sigma;E\otimes \C^m)$ is a Calderón projection for $D$ and $P_{\Ca^\dagger}\in \Psi^{\pmb 0}_{\rm cl}(\Sigma;F\otimes \C^m)$ is a Calderón projection for $D^\dagger$. Then $\scalebox{1.5}{a}$ extends to a continuous and invertible operator 
$$\scalebox{1.5}{a}:
\begin{matrix}
(1-P_\Ca)\SobHH{{m-{\frac12}}}(\Sigma; E\otimes \C^m)\\
\bigoplus\\
P_\Ca \SobHH{{-{\frac12}}}(\Sigma; E\otimes \C^m)
\end{matrix}
\longrightarrow
\begin{matrix}
(1-P_{\Ca^\dagger}^*)\SobHH{{{\frac12}-m}}(\Sigma; F\otimes \C^m_{\rm op})\\
\bigoplus\\
P_{\Ca^\dagger}^* \SobHH{{{\frac12}}}(\Sigma; F\otimes \C^m_{\rm op})
\end{matrix}.$$
Moreover, in this decomposition
$$\scalebox{1.5}{a}=\begin{pmatrix}
0& \scalebox{1.5}{a}^+\\
\scalebox{1.5}{a}^-& 0\end{pmatrix} +\Psi^{-\infty}(\Sigma; E\otimes \C^m,F\otimes \C^m),$$
where 
$$\scalebox{1.5}{a}^+=(1-P_{\Ca^\dagger}^*)\scalebox{1.5}{a}P_\Ca\quad \mbox{and}\quad \scalebox{1.5}{a}^-=P_{\Ca^\dagger}^*\scalebox{1.5}{a}(1-P_\Ca).$$
\end{corollary}

\begin{proof}
This follows directly from Proposition \ref{returnofsomecontinuuisuforata} and the relation between $P_\Ca$ and $P_{\Ca^\dagger}$ from Theorem \ref{calderonforadjoint}.
\end{proof}

\begin{definition}
Let $D$ be an elliptic differential operator of order $m>0$. We define the space
$$\hatH(D)=\begin{matrix}
(1-P_\Ca^*)\SobHH{{{\frac12}-m}}(\Sigma; E\otimes \C^m_{\rm op})\\
\bigoplus\\
P_\Ca^* \SobHH{{{\frac12}}}(\Sigma; E\otimes \C^m_{\rm op})
\end{matrix}$$
\end{definition}

\begin{lemma}
\label{relationtohatspace}
The $\Lp{2}$-pairing induces a perfect pairing $\inprod{\cdot, \cdot}: \hatH(D) \times \checkH(D)^\ast$ and hence a Banach space isomorphism 
$$\hatH(D)\cong \checkH(D)^*.$$
Moreover, the operator $\scalebox{1.5}{a}$ defines an isomorphism 
$$\scalebox{1.5}{a}:\checkH(D)\to \hatH(D^\dagger).$$
\end{lemma}

\begin{proof}
By Theorem \ref{refinedseeley} we have that 
$$\checkH(D)=\begin{matrix}
(1-P_\Ca)\SobHH{{m-{\frac12}}}(\Sigma; E\otimes \C^m)\\
\bigoplus\\
P_\Ca \SobHH{{-{\frac12}}}(\Sigma; E\otimes \C^m)
\end{matrix}.$$
Proposition \ref{ell2pairingsofgradedsob} implies that the $\Lp{2}$-pairing induces a Banach space isomorphism $\hatH(D)\cong \checkH(D)^*$. The fact that $\scalebox{1.5}{a}:\checkH(D)\to \hatH(D^\dagger)$ is an isomorphism follows from Corollary \ref{structureofa}.
\end{proof}

\begin{remark}
If $D$ is an elliptic differential operator of order $m=1$ with adapted boundary operator $A$, the results of \cite{BBan} shows that  \
$$\hatH(D)=\chi^-(A^\ast)\SobH{-{\frac12}}(\Sigma;E)\oplus \chi^+(A^\ast)\SobH{{\frac12}}(\Sigma;E).$$ 
We use the notation $\hatH_A(D)$ for the right hand side of this equation. The reader should be aware that in \cite{BBan}, this space was denoted by $\hatH(A^\ast)$. 
\end{remark}

\begin{theorem}
\label{boundarypairing}
Let $D$ be an elliptic differential operator of order $m>0$. The boundary pairing
$$\omega_{D}:\checkH(D^\dagger)\times \checkH(D)\to \C.$$
takes the form 
$$\omega_{D}(\xi,\xi')=\langle \xi,\scalebox{1.5}{a} \xi'\rangle_{\Lp{2}(\Sigma;F\otimes \C^m)},\quad \xi\in \checkH(D^\dagger),\; \xi'\in \checkH(D),$$
and is a perfect pairing. 
\end{theorem}

\begin{proof}
The formula for the boundary pairing for smooth sections follows from Proposition \ref{proponboundarypairingforell} and that it can be written as the stated $\Lp{2}$-pairing  follows from Lemma \ref{relationtohatspace}.
\end{proof}

\subsection{Adjoint boundary conditions and a sufficient condition for regularity}
\label{subsec:adjointbcs}

In light of the results of last subsection, we can further describe adjoint boundary conditions. 
In particular, we can derive a sufficient condition for certain semi-regular boundary conditions to be regular.

\begin{proposition}
\label{cordescofadjointbc}
Let $D$ be an elliptic differential operator of order $m>0$ and $B\subseteq \checkH(D)$ a generalised boundary condition. The adjoint boundary condition $B^*\subseteq \checkH(D^\dagger)$ takes the form 
$$B^*=(\scalebox{1.5}{a}B)^\perp=\scalebox{1.5}{a}_\dagger^{-1}B^\perp,$$
where $\perp$ denotes the induced  annihilators with respect to  $\Lp{2}$-pairing.
\end{proposition}

\begin{proof}
Using the description in Lemma \ref{relationtohatspace}, the space $(\scalebox{1.5}{a}B)^\perp\subseteq \checkH(D^\dagger)$ is the annihilator of $\scalebox{1.5}{a}B\subseteq \hatH(D^\dagger)$ in the perfect $\Lp{2}$-pairing $\hatH(D^\dagger) \times \checkH(D^\dagger)\to \C$. By Theorem \ref{boundarypairing}, $(\scalebox{1.5}{a}B)^\perp$ is the annihilator of $B$ in the boundary pairing $\omega_D$, so $B^*=(\scalebox{1.5}{a}B)^\perp$ follows. By a similar argument, $\scalebox{1.5}{a}_\dagger B^*\subseteq \hatH(D)$ is the annihilator of $B\subseteq \checkH(D)$ in the $\Lp{2}$-pairing, so $\scalebox{1.5}{a}_\dagger B^*=B^\perp$ which implies $B^*=\scalebox{1.5}{a}_\dagger^{-1}B^\perp$.
\end{proof}

The reader should compare Proposition \ref{cordescofadjointbc} to Proposition \ref{describinglocalcndldldaadjoint} describing the adjoint boundary condition of a pseudo-local boundary condition. We shall now state a theorem providing a rather general sufficient condition for regularity.

\begin{theorem}
\label{thm:sufficienfeidntofell}
Let $D$ be an elliptic differential operator of order $m>0$ and $P$ a continuous projection on $\SobHH{{m-\frac12}}(\Sigma;E\otimes \C^m)$ that extends by continuity to a projection on $\SobHH{{-\frac12}}(\Sigma;E\otimes \C^m)$. Assume that the operator $A:=P_\mathcal{C}-(1-P)$ defines a Fredholm operator on $\SobHH{{m-\frac12}}(\Sigma;E\otimes \C^m)$ that extends by continuity to a Fredholm operator on $\SobHH{{-\frac12}}(\Sigma;E\otimes \C^m)$. Then $P$ defines an regular boundary condition 
$$B_P:=(1-P)\SobHH{{m-\frac12}}(\Sigma;E\otimes \C^m),$$
for $D$. In other words, we have that 
$$\{u\in \dom(D_{\rm max}): P\gamma u=0\}=\{u\in \SobH{m}(M;E): P\gamma u=0\},$$
and this space defines the domain of an regular realisation of $D$.
\end{theorem}

\begin{proof}
By using the same argument as in the proof of Theorem \ref{thm:bodunarodoad}, we see that if $A:=P_\mathcal{C}-(1-P)$ is Fredholm on $\SobHH{{m-\frac12}}(\Sigma;E\otimes \C^m)$, then 
$$\norm{(1-P)u}_{\checkH(D)} \simeq \norm{(1-P) u}_{\SobHH{{m-\frac12}}(\Sigma;E\otimes \C^m)},$$
for $u\in \SobHH{{m-\frac12}}(\Sigma;E\otimes \C^m)$. In particular, $B$ is closed in $\checkH(D)$. This proves that $B$ is a semi-regular boundary condition. It remains to prove that $B^*$ is a semi-regular boundary condition. 

We claim that $B^*=(1-P_\dagger)\SobHH{{m-\frac12}}(\Sigma;F\otimes \C^m)$ and that this is a semi-regular boundary condition. Here we use the notation $P_\dagger:=\scalebox{1.5}{a}_\dagger^{-1}(1-P^*)\scalebox{1.5}{a}_\dagger$, where $P^*:\SobHH{{\frac12}}(\Sigma;E\otimes \C^m_{\rm op})\to \SobHH{{\frac12}}(\Sigma;E\otimes \C^m_{\rm op})$ is the continuous projection defined from the $\Lp{2}$-pairing $\SobHH{{-\frac12}}(\Sigma;E\otimes \C^m)\times \SobHH{{\frac12}}(\Sigma;E\otimes \C^m_{\rm op})\to \C$. By the arguments in the preceding paragraph and Proposition \ref{describinglocalcndldldaadjoint}, this claim follows if $A_\dagger:=P_{\mathcal{C}^\dagger}-(1-P_\dagger)$ is Fredholm on $\SobHH{{m-\frac12}}(\Sigma;F\otimes \C^m)$. However, using Theorem \ref{calderonforadjoint} we have that 
$$A_\dagger=\scalebox{1.5}{a}_\dagger^{-1}A^*\scalebox{1.5}{a}_\dagger,$$
where we by an abuse of notation write $A$ for its Fredholm extension to $\SobHH{{-\frac12}}(\Sigma;E\otimes \C^m)$ and $A^*$ for its $\Lp{2}$-adjoint on $\SobHH{{\frac12}}(\Sigma;E\otimes \C^m_{\rm op})$. The fact that $A_\dagger$ is Fredholm now follows from that $A$ is Fredholm on $\SobHH{{-\frac12}}(\Sigma;E\otimes \C^m)$.
\end{proof}

The next result specialises Theorem \ref{thm:sufficienfeidntofell} to pseudolocal boundary conditions and characterises those that are regular. We encourage the reader to compare the assumptions on $P$ in the next corollary to those in Corollary \ref{pseudostahtdecompose}.

\begin{theorem}
\label{equivocndndforpseudolocal}
Let $D$ be an elliptic differential operator of order $m>0$ and $P\in \Psi^{\pmb 0}_{\rm cl}(\Sigma; E\otimes \C^m)$ a projection. Define the generalised boundary condition 
$$B_P:=(1-P)\SobHH{{m-\frac12}}(\Sigma;E\otimes \C^m).$$
Then the following are equivalent:
\begin{enumerate}[i)]
\item The operator $P_\mathcal{C}-(1-P)\in \Psi^{\pmb 0}_{\rm cl}(\Sigma; E\otimes \C^m)$ is elliptic. 
\item $B_P$ is an regular boundary condition for $D$.
\item $B_P$ is a Fredholm boundary condition for $D$. 
\end{enumerate}
\end{theorem}

We remark that the space $B_P$ defined in the theorem above is a priori different from the pseudo-local boundary condition defined from $P$ as in Definition \ref{somedefinidodod}. A posteriori, they produce the same boundary condition under any of the assumptions i)-iii) in Theorem \ref{equivocndndforpseudolocal}.

\begin{proof}
Item i) implies item ii) by Theorem \ref{thm:sufficienfeidntofell}. Item ii) implies item iii) by Proposition \ref{Prop:MaxClosed}. Using \cite[Theorem XIX.5.1 and XIX.5.2]{horIII}, to prove that iii) implies i) it suffices to prove that $P_\mathcal{C}-(1-P)$ is Fredholm as soon as $B_P$ is Fredholm boundary condition. We note that $B_P$ is Fredholm if and only if $B_P^*$ is Fredholm if and only if $(P_\mathcal{C}\checkH(D),B_P)$ and $((1-P_\mathcal{C}^*)\hatH(D),B_P^\perp)$ are Fredholm pairs. As such, $P_\mathcal{C}-(1-P)$ is Fredholm whenever $B_P$ is a Fredholm boundary condition by combining \cite[Lemma A.9]{BBan} with Atkinson's theorem.
\end{proof}

To add some more flavor to the regularity condition on a pseudo-local boundary condition as in Theorem \ref{equivocndndforpseudolocal}, let us compare this to the well studied notion of Shapiro-Lopatinskii ellipticity. 

\begin{definition}
\label{def:lsell}
Let $D$ be an elliptic differential operator of order $m>0$ and $B$ a pseudo-local boundary condition defined from a pseudo-differential projection $P\in \Psi^{\pmb 0}_{\rm cl}(\Sigma;E\otimes \C^m)$. We say that $P$ is Shapiro-Lopatinskii elliptic if for any $(x',\xi')\in T^*\Sigma$, when the differential equation
\begin{equation}
\label{boundaryeqlsdef}
\begin{cases}
\sigma_\partial(D)(x',\xi')v(t)=0,\quad t>0,\\
\sigma_{\pmb 0}(P)(x',\xi')\left(D^j_tv(0)\right)_{j=0}^{m-1}=h,\end{cases}
\end{equation}
has a solution it is uniquely determined and exponentially decaying.
\end{definition}

Compare the definition of Shapiro-Lopatinskii elliptic pseudo-differential projectors the the construction of $E_+(D)$ in Proposition \ref{structureofeplus}. The reader can find more on Shapiro-Lopatinskii ellipticity of pseudo-local boundary conditions in \cite{schulzeseiler}. Similarly, for more general boundary conditions, if for any $(x',\xi')\in T^*\Sigma$, the differential equation
$$
\begin{cases}
\sigma_\partial(D)(x',\xi')v(t)=0,\quad t>0,\\
\sigma(b)(x',\xi')\left(D^j_tv(0)\right)_{j=0}^{m-1}=h,\end{cases}
$$
has a unique exponentially decaying solution (for a suitably defined symbol $\sigma(b)$) for any right hand side $h$, the problem is said to be Shapiro-Lopatinskii elliptic. The reader can find more on Shapiro-Lopatinskii ellipticity of a local boundary condition in \cite{agranovichencyclopedia}. The constructions in \cite{agranovichencyclopedia} show how local considerations at the boundary prove to be a highly efficient tool for determining regularity of a local boundary condition. Indeed for the purpose of verifying regularity of local boundary conditions, local semiclassical methods are more direct than the global perspective on a boundary condition as a subspace $B\subseteq \checkH(D)$ that dominates this paper.

\begin{example}
We return to the study of Dirichlet and Neumann conditions on a second order elliptic differential operator $D$ acting on a vector bundle $E$ with scalar positive principal symbol as in Example \ref{ex:dirichletandneumann}. We have that 
\begin{align*}
\sigma_{\pmb 0}(P_\mathcal{C}-(1-P_D))(x',\xi')&=
\frac{1}{2}
\begin{pmatrix}
1_E& |\xi|'^{-1}1_E\\
|\xi'|1_E& -1_E\end{pmatrix}
\quad\mbox{and}\\
\sigma_{\pmb 0}(P_\mathcal{C}-(1-P_N))(x',\xi')&=
\frac{1}{2}
\begin{pmatrix}
-1_E& |\xi|'^{-1}1_E\\
|\xi'|1_E& 1_E\end{pmatrix}.
\end{align*}
It follows that $P_\mathcal{C}-(1-P_D)$ and $P_\mathcal{C}-(1-P_N)$ are elliptic. Therefore Theorem \ref{equivocndndforpseudolocal} reproves the well known fact that Dirichlet and Neumann conditions are regular. A short computation (or using Proposition \ref{lsellipticgiveselliptic} below) reproves the well known fact that Dirichlet and Neumann conditions are Shapiro-Lopatinskii elliptic. By Example \ref{ex:dirichletandneumann}, the projectors $P_D$ and $P_N$ are not boundary decomposing which shows that regularity of a boundary condition defined from a projection is not equivalent to the boundary decomposing property of a projector.
\end{example}

We now return to the general case of pseudo-local boundary conditions defined from a projection in the Douglis-Nirenberg calculus.

\begin{proposition}
\label{lsellipticgiveselliptic}
Let $D$ be an elliptic differential operator of order $m>0$ and $P\in \Psi^{\pmb 0}_{\rm cl}(\Sigma;E\otimes \C^m)$. Then $P$ is Shapiro-Lopatinskii elliptic if and only if $P_\mathcal{C}-(1-P)\in \Psi^{\pmb 0}_{\rm cl}(\Sigma; E\otimes \C^m)$ is elliptic. 

In particular, a Shapiro-Lopatinskii elliptic pseudo-local boundary condition is regular in the sense of Subsection \ref{subsec:bcsearly}.
\end{proposition}

\begin{proof}
The condition of Shapiro-Lopatinskii ellipticity is equivalent to $\sigma_0(P)$ defining an isomorphism between $E_+(D)$ and $\sigma_{\pmb 0}(P)(E\otimes \C^m)$. By \cite[Proof of Theorem 2.15]{BBan}, this is equivalent to $p_+(D)-\sigma_{\pmb 0}(1-P)$ being an isomorphism in all points, or in other words that $P_\mathcal{C}-(1-P)$ is elliptic.
\end{proof}

\section{Characterisations of regularity via graphical decompositions}
\label{subsec:charofellfirstorder}

A graphical characterisation for regular boundary conditions was first given in \cite{BB}, motivated by the desire to capture regular boundary conditions beyond those that are given by pseudodifferential projectors. 
The main application in mind was a remarkably conceptual proof of the relative index theorem, originally due to Gromov-Lawson \cite{GL}. 
In \cite{BB}, the boundary decomposing projectors were assumed to be self-adjoint spectral projectors, with the assumption that the adapted operator on the boundary was self-adjoint.
A particular luxury in this setting was that the obtained decomposition could be equated with orthogonal complements in $\Lp{2}$. 
This was subsequently generalised \cite[Theorem 2.9]{BBan} for general first order operators, where orthogonal complements were no longer available and the decomposition had to be obtained separately for the boundary condition and its adjoint.
In particular, this meant that there were double the number of spaces in the decomposition and we consider the most significant step towards our presentation here for the general order case. 

\subsection{Preliminaries}
As above, we assume that $D$ is an elliptic differential operator and $\GProj_+$ is a boundary decomposing projection for $D$ (in the sense of Definition \ref{boundafefidodo} on page \pageref{boundafefidodo}). We use the notation $\GProj_-=1-\GProj_+$ and $\GProj_+^*$ for the appropriate dual operator to $\GProj_+$ with respect to the duality induced by  $\Lp{2}$ inner product.

\begin{lemma}
\label{Lem:Pairing}
The pairing 
$$\inprod{\cdot, \cdot}: \adj{\GProj}_{\pm} \SobHH{{-\alpha}}(\Sigma;E\otimes \C^n_{\rm op}) \times \GProj_{\pm} \SobHH{{\alpha}}(\Sigma;E\otimes \C^m)\to \C,$$ 
induced from the $\Lp{2}(\Sigma;E)$ inner product is a perfect pairing for $\alpha \in \{-\frac12,m- \frac12\}$. 
\end{lemma} 

\begin{proof}
The $\Lp{2}$-inner product $\inprod{\cdot,\cdot}$ induces a perfect pairing between $\SobHH{{-\alpha}}(\Sigma;E\otimes \C^n_{\rm op})$ and $\SobHH{{\alpha}}(\Sigma;E\otimes \C^m)$ by Proposition \ref{ell2pairingsofgradedsob}. 
The desired conclusion follows simply on making the right identifications to the spaces $\cB$, $\cB^\ast$, $X$, $Y$ in Lemma \ref{Lem:SubPair}. 
\end{proof}

\begin{proposition}
\label{Prop:ProjSob}
For any boundary decomposing projection $\GProj_+$ for the elliptic differential operator $D$, we have that: 
\begin{enumerate}[(i)]
\item  \label{Prop:ProjSob:1} 
$\GProj_+ \checkH(D) = \GProj_+ \SobHH{{-\frac12}}(\Sigma;E\otimes \C^m)$ and $\GProj_- \checkH(D) = \GProj_- \SobHH{{m-\frac12}}(\Sigma;E\otimes \C^m)$, and 
\item  \label{Prop:ProjSob:2} 
$\hatH(D) = \GProj_+^\ast \SobHH{{\frac12}}(\Sigma;E\otimes \C^m_{\rm op}) \oplus \GProj_-^\ast \SobHH{{\frac12-m}}(\Sigma;E\otimes \C^m_{\rm op})$.
\end{enumerate}
\end{proposition}
\begin{proof}
Item \ref{Prop:ProjSob:1} follows from Proposition \ref{prop:bcimpliessppsls}. The proof of \ref{Prop:ProjSob:2} follows from \ref{Prop:ProjSob:1}  since $\checkH(D) = \GProj_- \SobHH{{m-\frac12}}(\Sigma;E\otimes \C^m)\oplus \GProj_+ \SobHH{{-\frac12}}(\Sigma;E\otimes \C^m)$, Lemma \ref{Lem:Pairing}, and \cite[Theorem 4.8 (4.12) in Chapter 4, Section 4]{Kato}.
\end{proof}

%

\begin{definition}[Graphical decomposition]
\label{def:elldecomspsHO}
Let $B$ be a boundary condition and $\GProj_+$ a boundary decomposing projection for an elliptic differential operator $D$ of order $m>0$. Following \cite{BBan}, we say that a graphical decomposition of $B$ (with respect to $\GProj_+$) is a decomposition of the following form:
\begin{enumerate}[(i)]
\item
	there exist mutually complementary subspaces $W_{\pm}$ and $V_{\pm}$ of $\SobHH{{m-{\frac12}}}(\Sigma;E\otimes \C^m)$ satisfying: 
	$$W_{\pm} \oplus V_{\pm} = \GProj_\pm \SobHH{{m-{\frac12}}}(\Sigma;E\otimes \C^m),$$
\item
	it holds that $\adj{W_-} \subset \SobHH{{{\frac12}}}(\Sigma;E\otimes \C^m_{\rm op})$, and $\adj{W_\pm} \subset \SobHH{{{\frac12} - m}}(\Sigma; E\otimes \C^m_{\rm op})$ and the subspaces $W_\pm\subset \SobHH{{m-{\frac12}}}(\Sigma;E\otimes \C^m)$ are finite dimensional,
\item
	there exists a continuous map $g: V_- \to V_+$ such that  
	$$\adj{g} (\adj{V}_+\cap \SobHH{{{\frac12}}}(\Sigma;E\otimes \C^m_{\rm op})) \subset \adj{V}_-\cap \SobHH{{{\frac12}}}(\Sigma;E\otimes \C^m_{\rm op}),$$ 
	where $\adj{g}$ denotes the adjoint in the induced $\Lp{2}$-pairing between $ \SobHH{{m-{\frac12}}}(\Sigma;E\otimes \C^m)$ and $\SobHH{{\frac12-m}}(\Sigma;E\otimes \C^m_{\op})$
	$$ B = \set{v + gv: v \in V_-} \oplus W_+.$$
\end{enumerate} 
If a boundary condition admits a graphical decomposition (with respect to $\GProj_+$), we say it is $\GProj_+$-graphically decomposable. 
\end{definition}

We remark that spaces $V_{\pm}^*,W_{\pm}^\ast$, denotes subspaces of  $\SobHH{{{\frac12}-m}}(\Sigma;E\otimes \C^m_{\rm op})$ obtained via the induced pairing from the $\Lp{2}$-inner product, formed with respect to the direct sum decomposition $ \SobHH{{m-{\frac12}}}(\Sigma;E\otimes \C^m)=V_+\oplus W_+\oplus V_-\oplus W_-$.

\begin{definition}[$\GProj_+$ Fredholm pair decomposed in $\SobHH{{m-\frac12}}$]
\label{Def:FP}
Let $B \subset \SobHH{{m-\frac12}}(\Sigma;E\otimes \C^m)$. Suppose: 
\begin{enumerate}[(i)] 
\item 
\label{Def:FP1} 
$B$ is closed subspace of $\SobHH{{m-\frac12}}(\Sigma;E\otimes \C^m)$,

\item 
\label{Def:FP2}
$(\GProj_+\SobHH{{m-\frac12}}(\Sigma;E\otimes \C^m), B)$ and $(\GProj_-^*\SobHH{{\frac12}}(\Sigma;E\otimes \C^m_{\rm op})), B^\perp\cap \SobHH{{\frac12}}(\Sigma;E\otimes \C^m_{\rm op}))$ are Fredholm pairs in $\SobHH{{m-\frac12}}(\Sigma;E\otimes \C^m)$ and $\SobHH{{\frac12}}(\Sigma;E\otimes \C^m_{\rm op}))$, respectively, and 
\item 
\label{Def:FP3} $\indx(\GProj_+\SobHH{{m-\frac12}}(\Sigma;E), B) = -\indx(\GProj_-^*\SobHH{{\frac12}}(\Sigma;E\otimes \C^m_{\rm op})), B^\perp\cap \SobHH{{\frac12}}(\Sigma;E\otimes \C^m_{\rm op})).$
\end{enumerate}
Then, we say that $B$ is \emph{Fredholm-pair decomposed with respect to $\GProj_+$ in $\SobHH{{m-\frac12}}$}.
\end{definition}

\begin{definition}[$\GProj_+$ Fredholm pair decomposed in $\checkH(D)$]
\label{Def:FPC}
Let $B \subset \SobHH{{m-\frac12}}(\Sigma;E\otimes \C^m)$. Suppose: 
\begin{enumerate}[(i)] 
\item 
\label{Def:FPC1} 
$B$ is closed subspace of $\checkH(D)$ and $B^\perp\cap \SobHH{{\frac12}}(\Sigma;E\otimes \C^m_{\rm op})$ is a closed subspace of $\hatH(D)$,

\item 
\label{Def:FPC2}
$(\GProj_+\checkH(D), B)$ and $(\GProj_-^*\hatH(D), B^\perp\cap \SobHH{{\frac12}}(\Sigma;E\otimes \C^m_{\rm op}))$ are Fredholm pairs in $\checkH(D)$ and $\hatH(D)$ respectively, 
\item 
\label{Def:FPC3} $\indx (\GProj_+\checkH(D), B)= -\indx (\GProj_-^*\hatH(D), B^\perp\cap \SobHH{{\frac12}}(\Sigma;E\otimes \C^m_{\rm op})).$
\end{enumerate}
Then, we say that $B$ is \emph{Fredholm-pair decomposed with respect to $\GProj_+$ in $\checkH(D)$}.
\end{definition}

Note that the $\perp$ appearing in these definitions  taken with respect to the induced pairing from $\Lp{2}$ between $\SobHH{\frac12-m}(\Sigma;E \otimes \C^m_{\rm op})$ and $\SobHH{m - \frac12 }(\Sigma;E \otimes \C^m)$.

\subsection{Decomposability}

The main theorem we prove in this section is the following.
\begin{theorem}
\label{Thm:Ell}
Let $D$ be an elliptic differential operator of order $m>0$. For  $\GProj_+$ being a boundary decomposing projection for $D$, the following are equivalent: 
\begin{enumerate}[(i)] 
\item \label{Thm:Ell:1}
	$B \subset \checkH(D)$ is a regular boundary condition,
\item \label{Thm:Ell:2} 
	$B$ is $\GProj_+$ graphically decomposable,
\item \label{Thm:Ell:3} 
	$B$ is $\GProj_+$ Fredholm decomposable in $\SobHH{{m-\frac12}}$ 
\item \label{Thm:Ell:4} 
	$B$ is $\GProj_+$ Fredholm decomposable in $\checkH(D)$.
\end{enumerate}
If one of these equivalent conditions are satisfied, then
\begin{equation}
\label{Eqn:Ell:Adj} 
\scalebox{1.5}{a}_\dagger\adj{B} = B^{\perp}= \set{u - \adj{g}u: u \in \adj{V}_+\cap \SobHH{{{\frac12}}}(\Sigma;E\otimes \C^m_{\rm op})} \oplus \adj{W}_-.
\end{equation}
and 
\begin{equation}
\label{Eqn:Ell:FP}
\begin{aligned}
\indx(\GProj_+\SobHH{{m-\frac12}}(\Sigma;E\otimes \C^m), B) &= \indx (\GProj_+\checkH(D), B), \\
\indx(\GProj_-^*\SobHH{{\frac12}}(\Sigma;E\otimes \C^m_{\rm op})), B^\perp\cap \SobHH{{\frac12}}(\Sigma;E\otimes \C^m_{\rm op})) &= \indx (\GProj_-^*\hatH(D), B^\perp\cap \SobHH{{\frac12}}(\Sigma;E\otimes \C^m_{\rm op})).
\end{aligned}
\end{equation}
In particular, the properties of being graphically decomposable, Fredholm decomposable in $\SobHH{{m-\frac12}}$ and Fredholm decomposable in $\checkH(D)$ are independent of the choice of boundary decomposing projection.
\end{theorem}

The proof of this theorem is aided by the structure of \cite[Theorem 2.9]{BBan} which deals with the first order case.
In a sense, the essential ingredients are contained in that proof, although the subtleties of this decomposition requires significant care.
For that purpose, we give full details of the argument for the general order case. Before entering the proof of Theorem \ref{Thm:Ell}, let us give an index theoretical consequence. The following corollary follows directly from Theorem \ref{Thm:FredChar} and Equation \eqref{Eqn:Ell:FP}

\begin{corollary}
\label{simplexindex}
Let $D$ be an elliptic differential operator of order $m>0$ and $B\subseteq \SobHH{{m-\frac12}}(\Sigma;E\otimes \C^m)$ a regular boundary condition for $D$. Write $\Ca_D^{m-\frac12}:=P_\Ca\SobHH{{m-\frac12}}(\Sigma;E\otimes \C^m)=\mathcal{C}_D\cap \SobHH{{m-\frac12}}(\Sigma;E\otimes \C^m)$. The index of $D_B$ is given by 
$$\indx(D_B)=\indx(B, \Ca_D^{m-\frac12}) + \dim \ker (D_{\min}) - \dim \ker (D_{\min}^\dagger).$$
in terms of an index of the Fredholm pair $(B, \Ca_D^{m-\frac12})$ in $\SobHH{{m-\frac12}}(\Sigma;E\otimes \C^m)$.
\end{corollary}

\begin{proof}[Proof that \ref{Thm:Ell:2}$\implies$ \ref{Thm:Ell:1} in Theorem \ref{Thm:Ell} ]
We shall prove that graphical decomposability implies regularity. First, note that a simple calculation yields 
\begin{equation} 
\label{Eq:Char1} 
\set{v + gv: v \in V_-}^{\perp, \hatH(D)} = W_-^\ast  \oplus W_+^\ast \oplus \set{u - g^\ast u: u \in V_+^\ast \cap \SobHH{ {\frac12}}(\Sigma; E \tensor \C^m_{\op})}.
\end{equation} 
Moreover, we have that $\SobHH{ m - \frac12}(\Sigma; E \tensor \C^m) = W_- \oplus V_- \oplus W_+ \oplus V_+$, and by Lemma \ref{Lem:Pairing}, and using Lemma \ref{Lem:SubPair}, we have that 
$$W_+^{\perp, \SobHH{ {\frac12} - m}(\Sigma;E \tensor \C^m_{\op})} = V_-^\ast \oplus W_-^\ast \oplus V_+^\ast.$$
Therefore, 
\begin{equation}
\label{Eq:Char2} 
W_+^{\perp, \hatH(D)} = W_+^{\perp, \SobHH{ {\frac12} - m}(\Sigma;E \tensor \C^m_{\op})}  \cap \hatH(D) = V_-^\ast \oplus W_-^\ast \oplus (V_+^\ast  \cap \GProj_+^\ast \SobHH{ {\frac 12}}(\Sigma; E \tensor \C^m_{\op})).
\end{equation} 
By \cite[Theorem 4.8 (4.12) in Chapter 4, Section 4]{Kato}, we have that 
\begin{align*}
B^\perp &= (W_+ \oplus \set{v + gv: v \in V_-})^{\perp, \hatH(D)}  \\
	&= W_+^{\perp, \hatH(D)} \cap \set{v + gv: v \in V_-}^{\perp, \hatH(D)} \\
	&= W_-^\ast \oplus \set{u - g^\ast u: u \in V_-^\ast \cap \SobHH{ {\frac12}}(\Sigma; E \tensor \C^m_{\op})},
\end{align*}
where the ultimate equality follows from combining \eqref{Eq:Char1} and \eqref{Eq:Char2}.
Note that by Definition \ref{def:elldecomspsHO}, we have that  $B^\perp \subset \SobHH{ {\frac12}}(\Sigma; E \tensor \C^m_{\op})$.  
Proposition \ref{cordescofadjointbc} yields that
$$\scalebox{1.5}{a}_\dagger\adj{B} = B^{\perp}= \set{u - \adj{g}u: u \in \adj{V}_+\cap \SobHH{{{\frac12}}}(\Sigma;E\otimes \C^m_{\rm op})} \oplus \adj{W}_- \subset \SobHH{ {\frac12}}(\Sigma; E \tensor \C^m_{\op}), $$
and therefore, we have that $\adj{B} \subset  \SobHH{{m - {\frac12}}}(\Sigma;F\otimes \C^m)$  so $B$ is regular.
\end{proof}

The proof of \ref{Thm:Ell:1}$\implies$ \ref{Thm:Ell:2}  is considerably more involved. In order to proceed with the proof \ref{Thm:Ell:1}$\implies$ \ref{Thm:Ell:2}, the spaces in the Definition \ref{def:elldecomspsHO} need to be defined.
Taking inspiration from \cite{BBan}, we define the following spaces:
\begin{equation}
\label{Eq:Spaces} 
\begin{aligned}
&W_-^\ast := \adj{\GProj}_- \SobHH{{\frac12}}(\Sigma;E\otimes \C^m_{\rm op}) \intersect B^{\perp, \SobHH{ {\frac12} - m}(\Sigma; E\tensor \C^m_{\op})}  \\
&W_+ := \GProj_+\SobHH{{m-\frac12}}(\Sigma;E\otimes \C^m) \intersect B \\
&V_-^\ast := \adj{\GProj}_-\SobHH{{\frac12-m}}(\Sigma;E\otimes \C^m_{\rm op})\intersect (W_-^\ast)^{\perp,\SobHH{- {\frac12}}(\Sigma; E\tensor \C^m)} \\
&V_+ := \GProj_+\SobHH{{m-\frac12}}(\Sigma;E\otimes \C^m) \intersect W_+^{\perp,\SobHH{{\frac12} -m}(\Sigma;E\otimes \C^m_{\op})}
\end{aligned}
\end{equation}

Note that  by definition $W_-^\ast \subset \adj{\GProj}_- \SobHH{{\frac12}}(\Sigma;E\otimes \C^m_{\rm op})$. 
Therefore, we are able to consider its annihilator  $(W_-^\ast)^{\perp,\SobHH{- {\frac12}}(\Sigma; E\tensor \C^m)}$ in $\SobHH{- {\frac12}}(\Sigma; E\tensor \C^m)$. 
Recalling the definition of the hybrid spaces 
$$
\SobHH{- {\frac12}}(\Sigma; E\tensor \C^m) = \oplus_{j=0}^{m - 1} \SobH{-\frac12 - j}(\Sigma;E) \supset \oplus_{j=0}^{m -1} \SobH{\frac12 - m + j}(\Sigma;E) = \SobHH{ {\frac12}-m}(\Sigma; E\tensor \C^m_{\op}).$$
Therefore, $V_-^\ast$ is well-defined. 
Similarly, it is readily verified that 
$$ \SobHH{{m-\frac12}}(\Sigma;E\otimes \C^m) \subset  \SobHH{{\frac12} -m}(\Sigma;E\otimes \C^m_{\op}),$$
and therefore, we have that  $V_+$ is well defined.

This then allows us to define the remaining spaces as follows simply as the ranges of dual projectors or, as given by Lemma \ref{Lem:SubPair}, annihilators with respect to the corresponding duality.

\begin{equation}
\label{Eq:Spaces2} 
\begin{aligned}
&W_+^\ast := \adj{\GProj}_+ \SobHH{{\frac12}-m}(\Sigma;E\otimes \C^m_{\rm op}) \intersect V_+^{\perp, \SobHH{ {\frac12} - m}(\Sigma; E\tensor \C^m_{\op})}  \\
&W_- := \GProj_-\SobHH{{m-\frac12}}(\Sigma;E\otimes \C^m) \intersect (V_-^\ast)^{\perp, \SobHH{{m-\frac12}}(\Sigma;E\otimes \C^m)} \\
&V_+^\ast := \adj{\GProj}_+\SobHH{{\frac12-m}}(\Sigma;E\otimes \C^m_{\rm op})\intersect W_+^{\perp,\SobHH{ {\frac12}-m}(\Sigma; E\tensor \C^m_{\op})} \\
&V_- := \GProj_-\SobHH{{m-\frac12}}(\Sigma;E\otimes \C^m) \intersect (W_-^\ast)^{\perp, \SobHH{{m-\frac12}}(\Sigma;E\otimes \C^m)}
\end{aligned}
\end{equation}

\begin{lemma}
\label{Lem:SubProp}
Assume that $B$ is regular and that $V_+$, $V_-$, $W_+$ and $W_-$ are defined as in the paragraphs above. 
The following hold:
\begin{enumerate}[(i)]
\item \label{Lem:SubProp:3}
	$W_\pm\subseteq \SobHH{{m-\frac12}}(\Sigma;E\otimes \C^m)$ and $W_\pm^\ast \subset \SobHH{{\frac12}-m}(\Sigma;E\otimes \C^m_{\rm op})$ are finite dimensional.
	In particular, $W_-^\ast = \adj{\GProj}_- \SobHH{{\frac12}-m}(\Sigma;E\otimes \C^m_{\rm op}) \intersect B^{\perp, \SobHH{ {\frac12} - m}(\Sigma; E\tensor \C^m_{\op})}$.
\item \label{Lem:SubProp:4}
	$\GProj_- B\subseteq \SobHH{{m-\frac12}}(\Sigma;E\otimes \C^m) $ and $\GProj_+^\ast B^\perp\subseteq \SobHH{{\frac12}}(\Sigma;E\otimes \C^m_{\rm op})$ are closed subspaces.
\item \label{Lem:SubProp:1} 
	All of the subspaces $V_+$, $V_-$, $W_+$ and $W_-$ as well as $V_+^\ast$, $V_-^\ast $, $W_+^\ast$ and $W_-^\ast$ listed in \eqref{Eq:Spaces} and  \eqref{Eq:Spaces2}  are closed in $\SobHH{{m-\frac12}}(\Sigma;E\otimes \C^m)$ and $\SobHH{{\frac12}-m}(\Sigma;E\otimes \C^m_{\op})$ respectively. 
\item  \label{Lem:SubProp:2}
	$\GProj_{\pm}\SobHH{{m-\frac12}}(\Sigma;E\otimes \C^m) = V_{\pm} \oplus W_{\pm}$ and $\GProj_{\pm}^\ast\SobHH{{\frac12-m}}(\Sigma;E\otimes \C^m_{\rm op})= V_{\pm}^\ast \oplus W_{\pm}^\ast$.
\end{enumerate} 
\end{lemma}

\begin{proof}
Consider $\GProj_+^\ast\rest{B^\perp}: B^\perp \to \SobHH{{\frac12}}(\Sigma;E\otimes \C^m_{\op})$.
By definition, we have that $B^\perp$ is closed in $\hatH(D)$. 
Set $X = (B^\perp, \norm{\cdot}_{\hatH(D)})$, $Y = \SobHH{{\frac12}}(\Sigma;E\otimes \C^m_{\op})$, $Z =  \SobHH{{\frac12 - m}}(\Sigma;E\otimes \C^m_{\op})$ and the map $K = \GProj_-^\ast\rest{B^\perp}: X \to Z$.
For for $b^\perp \in B^\perp$
$$\norm{b^\perp}_{\hatH(D)} \simeq \norm{\GProj_+^\ast\rest{B^\perp}b^\perp}_{\SobHH{{\frac12}}(\Sigma;E\otimes \C^m_{\op})} + \norm{\GProj_-^\ast\rest{B^\perp}b^\perp}_{\SobHH{{\frac12 - m}}(\Sigma;E\otimes \C^m_{\op})}.$$
Note that 
$$\GProj_-^\ast\rest{B^\perp}: B^\perp \subset  \SobHH{{\frac12}}(\Sigma;E\otimes \C^m_{\op}) \to \SobHH{{\frac12}}(\Sigma;E\otimes \C^m_{\op}) \xhookrightarrow{\mathrm{compact}} \SobHH{{\frac12 -m}}(\Sigma;E\otimes \C^m_{\op})$$
and so
on invoking \cite[Lemma A.1]{BBan} with this choice of $X,Y,Z$ and $K$ yields that $\ker \GProj_+^\ast \rest{B^\perp}$ is finite dimensional and $P_+^\ast B^\perp$ is closed in $\SobHH{{\frac12}}(\Sigma;E\otimes \C^m_{\op})$.
Moreover, 
\begin{align*}
\ker \GProj_+^\ast \rest{B^\perp} 
	&= \set{b^\perp \in B^\perp: \GProj_+^\ast b^\perp = 0}  
	= \set{b^\perp \in B^\perp: b^\perp \in \GProj_-^\ast \hatH(D) = 0} \\
	&\hspace{7em}= B^\perp \cap  \GProj_-^\ast \SobHH{{\frac12 - m}}(\Sigma;E\otimes \C^m_{\op}) 
	= B^\perp \cap \GProj_-^\ast \SobHH{{\frac12}}(\Sigma;E\otimes \C^m_{\op}),
\end{align*}
where the penultimate equality follows from Proposition \ref{Prop:ProjSob} \ref{Prop:ProjSob:1} and the ultimate equality from the fact that $B^\perp \subset \SobHH{{\frac12}}(\Sigma;E\otimes \C^m_{\op})$.

Similarly, consider the map $\GProj_-\rest{B}: B \to \SobHH{{m-\frac12}}(\Sigma;E\otimes \C^m)$.
Since $B \subset \SobHH{{m-\frac12}}(\Sigma;E\otimes \C^m)$, we have that whenever $b \in B$, 
$$ \norm{b}_{\SobHH{{m-\frac12}}(\Sigma;E\otimes \C^m)} \simeq \norm{b}_{\checkH(D)} \simeq \norm{\GProj_-\rest{B}b}_{\SobHH{{m-\frac12}}(\Sigma;E\otimes \C^m)} + \norm{\GProj_+\rest{B}b}_{\SobHH{{-\frac12}}(\Sigma;E\otimes \C^m)}.$$
However, 
$$ \GProj_+\rest{B}: B \to \SobHH{{m-\frac12}}(\Sigma;E\otimes \C^m) \xhookrightarrow{\mathrm{compact}} \SobHH{{-\frac12}}(\Sigma;E\otimes \C^m),$$
and therefore,  on invoking \cite[Lemma A.1]{BBan} yields that $W_+ = \ker(\GProj_+\rest{B})$ is finite dimensional and $\GProj_-B$ is closed in $\SobHH{{m-\frac12}}(\Sigma;E\otimes \C^m)$.

Therefore, we have proven the property  \ref{Lem:SubProp:4} as well as a part of the property \ref{Lem:SubProp:3}, namely that $W_-^\ast$ and $W_+$  are finite dimensional and that the asserted equality for $W_-^\ast$ holds.
We prove that $W_-$ and $W_-^\ast$ are finite dimensional after proving \ref{Lem:SubProp:1} and \ref{Lem:SubProp:2}.

For that, we first prove that $\GProj_-^\ast \SobHH{{\frac12 - m}}(\Sigma;E\otimes \C^m_{\op}) = V_-^\ast \oplus W_-^\ast$.
From the formula for $W_-^\ast$ that we have already proven, we have that $W_-^\ast \subset \GProj_-^\ast \SobHH{{\frac12 - m}}(\Sigma;E\otimes \C^m_{\op})$ and it is closed since it is finite dimensional.
By construction $V_-^\ast \subset \GProj_-^\ast \SobHH{{\frac12 - m}}(\Sigma;E\otimes \C^m_{\op})$.
Therefore, we have that $V_-^\ast + W_-^\ast \subset  \GProj_-^\ast \SobHH{{\frac12 - m}}(\Sigma;E\otimes \C^m_{\op})$.

Next, we show that $V_-^\ast \cap  W_-^\ast = 0$.
Fix $w \in V_-^\ast \cap  W_-^\ast =\GProj_-^\ast \SobHH{{\frac12 - m}}(\Sigma;E\otimes \C^m_{\op}) \cap (W_-^\ast)^{\perp, \GProj_-^\ast \SobHH{{-\frac12}}(\Sigma;E\otimes \C^m)} \cap W_-^\ast$.
In particular, we have that $w \in (W_-^\ast)^{\perp, \GProj_-^\ast \SobHH{-{\frac12}}(\Sigma;E\otimes \C^m)} \cap W_-^\ast$.
Since $w \in W_-^\ast$ we also have that $w \in \SobHH{{\frac12}}(\Sigma;E\otimes \C^m_{\op})$ which is in a perfect pairing with $\SobHH{-{\frac12}}(\Sigma;E\otimes \C^m)$.
Therefore, we have that $\inprod{w, w} = 0$.
Moreover,  it is readily verified that $\SobHH{{\frac12}}(\Sigma;E\otimes \C^m_{\op}) \subset \oplus_{j=0}^{m-1} \Lp{2}(\Sigma;E)$ and therefore,
$$ \norm{w}_{\Lp{2}(\Sigma;E)}^2 = \inprod{w,w}_{\Lp{2}(\Sigma;E)} = \inprod{w,w}_{\SobHH{{\frac12}}(\Sigma;E\otimes \C^m_{\op}) \times \SobHH{-{\frac12}}(\Sigma;E\otimes \C^m)} = 0,$$
from which we conclude that $w = 0$.
This proves that $V_-^\perp \oplus W_-^\perp \subset \GProj_-^\ast \SobHH{{\frac12 - m}}(\Sigma;E\otimes \C^m_{\op})$. 

We now prove the reverse containment.
For that, let $k = \dim W_-^\ast$ and fix a basis $e_1, \dots,  e_k$.
By considering each $e_i \in \oplus_{j=0}^{m-1} \Lp{2}(\Sigma;E)$, we can run the Gram-Schmidt process to obtain an $\Lp{2}$-orthonormal basis $\tilde{e}_i$. 
Therefore, without the loss of generality, assume that $\set{e_i}$ are orthonormal in $\Lp{2}$.
For $u \in  \SobHH{-{\frac12}}(\Sigma;E\otimes \C^m)$, define 
$$ Pu = \sum_{i=1}^k \inprod{u, e_i} e_i.$$
It is readily checked that this is a projector and in fact, $P: \SobHH{-{\frac12}}(\Sigma;E\otimes \C^m) \to \SobHH{-{\frac12}}(\Sigma;E\otimes \C^m)$ with $P \SobHH{-{\frac12}}(\Sigma;E\otimes \C^m) =W_-^\ast$ and $ \ker P =  (W_-^\ast)^{\perp, \SobHH{-{\frac12}}(\Sigma;E\otimes \C^m)}$. 
Therefore, restricting to $u \in \GProj_-^\ast \SobHH{{\frac12 - m}}(\Sigma;E\otimes \C^m_{\op}) \subset \SobHH{-{\frac12}}(\Sigma;E\otimes \C^m)$, we obtain that $u = Pu - (1 - P)u$ and $Pu \in W_-^\ast$. 
This means that $(1 - P)u \in (W_-^\ast)^{\perp, \SobHH{-{\frac12}}(\Sigma;E\otimes \C^m)} \cap \GProj_-^\ast \SobHH{{\frac12 - m}}(\Sigma;E\otimes \C^m_{\op})$ since $W_-^\ast \subset \GProj_-^\ast \SobHH{{\frac12 - m}}(\Sigma;E\otimes \C^m_{\op})$ and therefore, $(1 - P)u \in V_-^\ast$.
Hence, we conclude that $V_-^\perp \oplus W_-^\perp = \GProj_-^\ast \SobHH{{\frac12 - m}}(\Sigma;E\otimes \C^m_{\op})$.
By exploiting the duality yielded in Lemma \ref{Lem:Pairing}, we conclude that $V_-, W_-$ are closed in $\GProj_- \SobHH{{m-\frac12}}(\Sigma;E\otimes \C^m)$. 
Via Lemma \ref{Lem:SubPair}, we obtain  $\GProj_- \SobHH{{m-\frac12}}(\Sigma;E\otimes \C^m) = V_- \oplus W_-$. 

A similar argument then holds to show that $\GProj_+ \SobHH{{m-\frac12}}(\Sigma;E\otimes \C^m) = V_+ \oplus W_+$.
Showing $V_+ + W_+ \subset \GProj_+ \SobHH{{m-\frac12}}(\Sigma;E\otimes \C^m)$ and that $V_+ \cap W_+ = 0$ is an easy calculation.
The most important point to show the reverse containment is note  that $\SobHH{{m-\frac12}}(\Sigma;E\otimes \C^m) \subset  \SobHH{{\frac12-m}}(\Sigma;E\otimes \C^m_{\op})$. 
This allows for the construction of a projector $Q:\SobHH{{\frac12-m}}(\Sigma;E\otimes \C^m_{\op}) \to \SobHH{{\frac12-m}}(\Sigma;E\otimes \C^m_{\op})$ with $Q \SobHH{{\frac12-m}}(\Sigma;E\otimes \C^m_{\op}) = W_+$ and $\ker Q = W_+^{\perp,\SobHH{{m -\frac12}}(\Sigma;E\otimes \C^m)}$ and consequently, write $\GProj_+ \SobHH{{m-\frac12}}(\Sigma;E\otimes \C^m) \subset V_+ \oplus W_+$.
As before, this yields that $V_+^\ast$ and $W_+^\ast$   are closed subspaces of $\GProj_+^\ast \SobHH{{\frac12-m}}(\Sigma;E\otimes \C^m_{\op})$ and Lemma \ref{Lem:SubPair} then yields $\GProj_+^\ast \SobHH{{\frac12-m}}(\Sigma;E\otimes \C^m_{\op})  = V_+^\ast \oplus W_+^\ast$. 

This proves \ref{Lem:SubProp:1} and \ref{Lem:SubProp:2}. 
It remains to show that $W_+^\ast$ and $W_-$ are finite dimensional.
For this, note that by using Lemma \ref{Lem:SubPair}, we have induced perfect pairings $\inprod{\cdot,\cdot}: W_+ \times W_+^\ast \to \C$ and $\inprod{\cdot,\cdot}: W_- \times W_-^\ast \to \C$.
But that means that $W_+' \cong W_+^\ast$ and $W_-' \cong W_-^\ast$, where $W_+'$ and $W_-'$ are dual spaces of $W_+$ and $W_-$ respectively.
By finite dimensionality we have that $W_+' \cong W_+$ and $W_-' \cong W_-$ and this finishes the proof. 
\end{proof}

\begin{lemma}
\label{Lem:GraphSpace} 
Under the same assumptions as in Lemma \ref{Lem:SubProp}, the following hold: 
\begin{enumerate}[(i)] 
\item \label{Lem:GraphSpaces:1}
$V_+^\ast \cap \SobHH{{\frac12}}(\Sigma;E\otimes \C^m_{\op}) = \GProj_+^\ast B^\perp$, and 

\item \label{Lem:GraphSpaces:2}
$V_- = \GProj_- B$.
\end{enumerate}
\end{lemma}
\begin{proof}
First, we prove that 
\begin{equation} 
\label{Eq:W+}
W_+ = (\GProj_+^\ast B^\perp)^{\perp, \SobHH{-{\frac12}}(\Sigma;E\otimes \C^m)} \cap \GProj_+^\ast \SobHH{-{\frac12}}(\Sigma;E\otimes \C^m).
\end{equation}
Then, 
\begin{align*}
x \in (\GProj_+^\ast B^\perp)^{\perp, \SobHH{-{\frac12}}(\Sigma;E\otimes \C^m)} 
	&\iff 0 = \inprod{x, \GProj_+^\ast b^\perp}\ \forall b^\perp \in B^\perp\\
	&\iff 0 = \inprod{\GProj_+ x,  b^\perp}\  \forall b^\perp \in B^\perp \\
	&\implies \GProj_+ x \in (B^\perp)^{\perp, \SobHH{-{\frac12}}(\Sigma;E\otimes \C^m)}.
\end{align*}
Therefore, $x \in (\GProj_+^\ast B^\perp)^{\perp, \SobHH{-{\frac12}}(\Sigma;E\otimes \C^m)} \cap \GProj_+^\ast \SobHH{-{\frac12}}(\Sigma;E\otimes \C^m)$ implies that $x = \GProj_+ x$.
Moreover, since  by Proposition \ref{Prop:ProjSob} $\checkH(D) \subset \SobHH{-{\frac12}}(\Sigma;E\otimes \C^m)$ and $B^\perp = B^{\perp, \hatH(D)} \subset \SobHH{{\frac12}}(\Sigma;E\otimes \C^m_{\op})$, we obtain that 
$B = (B^{\perp, \hatH(D)})^{\checkH(D)} = (B^\perp)^{\SobHH{-{\frac12}}(\Sigma;E\otimes \C^m)} \cap \checkH(D)$.
Therefore, $x \in (\GProj_+^\ast B^\perp)^{\perp, \SobHH{-{\frac12}}(\Sigma;E\otimes \C^m)} \cap \GProj_+^\ast \SobHH{-{\frac12}}(\Sigma;E\otimes \C^m)$ then yields that $x \in \checkH(D)$ and hence, $x \in B$.
This shows , 
$$(\GProj_+^\ast B^\perp)^{\perp, \SobHH{-{\frac12}}(\Sigma;E\otimes \C^m)} \cap \GProj_+^\ast \SobHH{-{\frac12}}(\Sigma;E\otimes \C^m) \subset   B \cap \GProj_+ \SobHH{-{\frac12}}(\Sigma;E\otimes \C^m) = W_+.$$

For the reverse containment, fix $w \in W_+ =   B \cap \GProj_+ \SobHH{-{\frac12}}(\Sigma;E\otimes \C^m)$. 
Therefore, for all $b^\perp \in B^\perp$, we have that
\begin{align*} 
0 = \inprod{w, b^\perp}_{ \SobHH{-{\frac12}}(\Sigma;E\otimes \C^m) \times  \SobHH{{\frac12}}(\Sigma;E\otimes \C^m_{\op})} 
	&= \inprod{\GProj_+ w, b^\perp}_{ \SobHH{-{\frac12}}(\Sigma;E\otimes \C^m) \times  \SobHH{{\frac12}}(\Sigma;E\otimes \C^m_{\op})} \\
	&= \inprod{w,\GProj_+^\ast  b^\perp}_{ \SobHH{-{\frac12}}(\Sigma;E\otimes \C^m) \times  \SobHH{{\frac12}}(\Sigma;E\otimes \C^m_{\op})}.
\end{align*}
Therefore, we have $w \in (\GProj_+^\ast B^\perp)^{\perp, \SobHH{-{\frac12}}(\Sigma;E\otimes \C^m)}$ and by assumption $w \in  \GProj_+ \SobHH{-{\frac12}}(\Sigma;E\otimes \C^m)$. 
That is precisely that $w \in (\GProj_+^\ast B^\perp)^{\perp, \SobHH{-{\frac12}}(\Sigma;E\otimes \C^m)} \cap \GProj_+^\ast \SobHH{-{\frac12}}(\Sigma;E\otimes \C^m)$.

Now, by the definition of $V_+^\ast$ in \eqref{Eq:Spaces2}, we have that
\begin{align*}
V_+^\ast \cap \SobHH{{\frac12}}(\Sigma;E\otimes \C^m_{\op}) 
	&= \GProj_+^\ast \SobHH{{\frac12-m}}(\Sigma;E\otimes \C^m_{\op}) \cap W_+^{\perp,  \SobHH{{\frac12-m}}(\Sigma;E\otimes \C^m_{\op})} \cap  \SobHH{{\frac12}}(\Sigma;E\otimes \C^m_{\op}) \\
	&= \GProj_+^\ast \SobHH{{\frac12}}(\Sigma;E\otimes \C^m_{\op}) \cap W_+^{\perp,  \SobHH{{\frac12}}(\Sigma;E\otimes \C^m_{\op})}
\end{align*}
By the description of $W_+$ in \eqref{Eq:W+} and on using \cite[Theorem 4.8 (4.12) in Chapter 4, Section 4]{Kato}, we obtain
$$ W_+^{\perp,  \SobHH{{\frac12}}(\Sigma;E\otimes \C^m_{\op})}  
	=  \GProj_+^\ast B^\perp + (\GProj_+ \SobHH{-{\frac12}}(\Sigma;E\otimes \C^m))^{\perp,\SobHH{{\frac12}}(\Sigma;E\otimes \C^m_{\op})} 
	= \GProj_+^\ast B^\perp + \GProj_-^\ast \SobHH{{\frac12}}(\Sigma;E\otimes \C^m_{\op}).$$
Therefore, 
\begin{align*}
V_+^\ast \cap \SobHH{{\frac12}}(\Sigma;E\otimes \C^m_{\op}) 
	&= \GProj_+^\ast \SobHH{{\frac12}}(\Sigma;E\otimes \C^m_{\op}) \cap W_+^{\perp,  \SobHH{{\frac12}}(\Sigma;E\otimes \C^m_{\op})} \\
	&= \GProj_+^\ast \SobHH{{\frac12}}(\Sigma;E\otimes \C^m_{\op}) \cap (\GProj_+^\ast B^\perp + \GProj_-^\ast \SobHH{{\frac12}}(\Sigma;E\otimes \C^m_{\op}) \\ 
	&= \GProj_+^\ast \SobHH{{\frac12}}(\Sigma;E\otimes \C^m_{\op}) \cap \GProj_+^\ast B^\perp \\
	&=  \GProj_+^\ast B^\perp.
\end{align*}
This proves \ref{Lem:GraphSpaces:1}. 

The proof of \ref{Lem:GraphSpaces:2} is similar. 
First, we prove that 
\begin{equation}
\label{Eq:W_-*} 
W_-^\ast = (\GProj_- B)^{\perp, \SobHH{{\frac12-m}}(\Sigma;E\otimes \C^m_{\op})} \cap \GProj_-\SobHH{{\frac12-m}}(\Sigma;E\otimes \C^m_{\op}).
\end{equation}
Recall that by Lemma \ref{Lem:SubProp} \ref{Lem:SubProp:3}, $W_-^\ast =  \adj{\GProj}_- \SobHH{{\frac12}-m}(\Sigma;E\otimes \C^m_{\rm op}) \intersect B^{\perp, \SobHH{ {\frac12} - m}(\Sigma; E\tensor \C^m_{\op})}$.
and so for $w \in W_-^\ast$, we have that  for all $b \in B$, 
\begin{align*}
0 &=  \inprod{w,b}_{ \SobHH{{\frac12-m}}(\Sigma;E\otimes \C^m_{\op}) \times \SobHH{{m-\frac12}}(\Sigma;E\otimes \C^m)} \\
	&= \inprod{\GProj_-^\ast w, b} = \inprod{w, \GProj_- b}_{ \SobHH{{\frac12-m}}(\Sigma;E\otimes \C^m_{\op}) \times \SobHH{{m-\frac12}}(\Sigma;E\otimes \C^m)}
\end{align*}
This shows the containment ``$\subset$''. 
For the reverse containment, let $x \in (\GProj_- B)^{\perp, \SobHH{{\frac12-m}}(\Sigma;E\otimes \C^m_{\op})} \cap \GProj_-\SobHH{{\frac12-m}}(\Sigma;E\otimes \C^m_{\op})$.
Arguing similarly,  for all $b \in B$, 
$$ 
0 = \inprod{x, \GProj_-b} = \inprod{\GProj_-^\ast x, b}$$
which yields that $x = \GProj_-^\ast x \in B^{\perp, \SobHH{{\frac12-m}}(\Sigma;E\otimes \C^m_{\op})}$.

With this, we now note that
\begin{align*}
V_- 	&= \GProj_- \SobHH{{m-\frac12}}(\Sigma;E\otimes \C^m) \cap (W_-^\ast)^{\perp, \SobHH{{m-\frac12}}(\Sigma;E\otimes \C^m)} \\
	&= \GProj_- \SobHH{{m-\frac12}}(\Sigma;E\otimes \C^m) \cap (\GProj_-B + \GProj_+ \SobHH{{m-\frac12}}(\Sigma;E\otimes \C^m)) \\ 
	&=  \GProj_- \SobHH{{m-\frac12}}(\Sigma;E\otimes \C^m) \cap \GProj_-B \\ 
	&= \GProj_-B.
\qedhere
\end{align*}
\end{proof}

\begin{proof}[Proof that \ref{Thm:Ell:1}$\implies$ \ref{Thm:Ell:2} in Theorem \ref{Thm:Ell} ]
We shall prove that regularity implies graphical decomposability. We use the spaces $V_+$, $V_-$, $W_+$ and $W_-$ constructed on the preceding pages. 
Since $B \subset \SobHH{{m-\frac12}}(\Sigma;E\otimes \C^m) \subset \SobHH{{\frac12-m}}(\Sigma;E\otimes \C^m_{\op})$, we have that $W_+^{\perp,\SobHH{{\frac12-m}}(\Sigma;E\otimes \C^m_{\op})}  \cap B \neq 0$.
Define 
$$
X_- := \GProj_-\rest{B \cap W_+^{\perp}}: B \cap W_+^{\perp} \to \GProj_-B.
$$

We prove that this map is a bijection and hence, by the open mapping theorem, it is a Banach space isomorphism between $(B, \norm{\cdot}_{\SobHH{{\frac12-m}}(\Sigma;E\otimes \C^m_{\op})})$ and $(\GProj_-B, \norm{\cdot}_{\SobHH{{\frac12-m}}(\Sigma;E\otimes \C^m_{\op})})$.
First, we establish surjectivity of $X_-$. 
For that,  note that 
\begin{equation}
\label{Eq:ConsProj} 
\SobHH{{\frac12-m}}(\Sigma;E\otimes \C^m_{\op}) = W_+ \oplus W_+^{\perp,\SobHH{{\frac12-m}}(\Sigma;E\otimes \C^m_{\op})},
\end{equation} 
which follows from a construction of a projector as in the proof of \ref{Lem:SubProp:1} in Lemma \ref{Lem:SubProp}.
Let $P_{W_+, W_+^{\perp}}$ denote the projection to  $W_+$ along $W_+^{\perp, \SobHH{{\frac12-m}}(\Sigma;E\otimes \C^m_{\op})}$ and the complementary projector $P_{W_+^\perp, W_+}$.
Fix $x = P_-b \in \GProj_-B$ and note
$$ x = P_- ( P_{W_+^\perp,W_+} b + P_{W_+,W_+^\perp} b) = P_- P_{W_+^\perp,W_+}b.$$
Since $W_+ \subset B$, $P_{W_+^\perp,W_+}b = b - P_{W_+^\perp,W_+} b \in B$ and hence $P_{W_+^\perp,W_+}b \in B \cap W_+^{\perp, \SobHH{{\frac12-m}}(\Sigma;E\otimes \C^m_{\op})}$.
This shows surjectivity. 
Injectivity follows easily by taking  $0 = X_-b = P_- b$ which allows us to conclude 
$$b \in \GProj_+ \SobHH{{m-\frac12}}(\Sigma;E\otimes \C^m) \cap B \cap W_+^{\perp, \SobHH{{\frac12-m}}(\Sigma;E\otimes \C^m_{\op})} = W_+ \cap W_+^{\perp, \SobHH{{\frac12-m}}(\Sigma;E\otimes \C^m_{\op})} = 0.$$

Define projectors $\SobHH{{m-\frac12}}(\Sigma;E\otimes \C^m)$ as:
$$P_+ := P_{V_+,W_+ \oplus \GProj_- \SobHH{{m-\frac12}}(\Sigma;E\otimes \C^m) }\quad\text{and}\quad Q_+ := P_{W_+,  V_+ \oplus \GProj_- \SobHH{{m-\frac12}}(\Sigma;E\otimes \C^m) }.$$
Furthermore, note that since $V_-$ has an induced Hilbert structure and $\GProj_-B \subset V_-$ is closed, we obtain a closed complementary subspace $V_-'$ such that $V_- = \GProj_-B \oplus V_-'$.
Using the fact that $X_-: (B, \norm{\cdot}_{\SobHH{{m-\frac12}}(\Sigma;E\otimes \C^m)}) \to (\GProj_-B, \norm{\cdot}_{\SobHH{{m-\frac12}}(\Sigma;E\otimes \C^m)})$ is an isomorphism, and the fact that $V_- = \GProj_- B$ from Lemma \ref{Lem:GraphSpace} \ref{Lem:GraphSpaces:2}, define  the $\SobHH{{m-\frac12}}(\Sigma;E\otimes \C^m)$-bounded map $g_0: V_- \to V_+$ by setting  
$$g_0(x)  = P_+ \comp (X_-)^{-1}.$$

First, we prove that $W_+ \oplus \set{u + g_0 u: u \in \GProj_-B} \subset B$.
By construction, we have that $W_+ \subset B$, so it remains to show $\set{u + g_0 u: u \in V_-} \subset B$. 
For that, we compute
\begin{align*} 
u + g_0 u &= u + P_+ \comp (X_-)^{-1} u \\ 
	&= X_- v + P_+v \\
	&= X_-v + P_+v + Q_+ v - Q_+ v \\ 
	&= \GProj_- v + \GProj_+ v - Q_+ v \\
	&= v - Q_+ v \in B \cap W_+^{\perp, \SobHH{{\frac12-m}}(\Sigma;E\otimes \C^m_{\op})} + W_+ \subset B,
\end{align*} 
where we have set $v := (X_-)^{-1}u$.

We now prove the reverse containment that $B \subset W_+ \oplus \set{u + g_0 u: u \in V_-}$.
Via the decomposition \eqref{Eq:ConsProj}, we have a induced decomposition 
$$
\SobHH{{\frac12-m}}(\Sigma;E\otimes \C^m_{\op}) = W_+ \oplus (W_+^{\perp,\SobHH{{\frac12-m}}(\Sigma;E\otimes \C^m_{\op})} \cap \SobHH{{\frac12-m}}(\Sigma;E\otimes \C^m_{\op})).$$
This decomposition holds  because $W_+ \subset \SobHH{{\frac12-m}}(\Sigma;E\otimes \C^m_{\op})$ and for the same reason, the projectors for this decomposition are the projectors $P_{W_+, W_+^{\perp}}$ and $P_{W_+^\perp, W_+}$ restricted to $\SobHH{{\frac12-m}}(\Sigma;E\otimes \C^m_{\op})$.
For $b \in B \subset \SobHH{{\frac12-m}}(\Sigma;E\otimes \C^m_{\op})$, let  $b^\perp = P_{W_+^\perp, W_+}b$  and $b_+ = P_{W_+, W_+^{\perp}}b$, so that $b = b^\perp + b_+$.
It is easy to see that, since $W_+ \subset B$, $b^\perp \in B \cap W_+^{\perp, \SobHH{{\frac12-m}}(\Sigma;E\otimes \C^m_{\op})}$.
Let $v := X_- b^\perp$. 
Then, 
\begin{align*} 
b^\perp = (X_-)^{-1}v  
	&= \GProj_- (X_-)^{-1} v + \GProj_+ (X_-)^{-1}v 
	= v + P_+  (X_-)^{-1}v + Q_+ (X_-)^{-1}v    \\
	&\hspace{10em}= v + g_0 v + Q_+ (X_-)^{-1}v 
	\in \set{u + g_0u: u \in \GProj_- B} \oplus W_+.
\end{align*}

Let 
$$X_+^\ast := \GProj_+^\ast\rest{B^\perp \cap (W_-^\ast)^{\perp, \SobHH{{-\frac12}}(\Sigma;E\otimes \C^m)}}:B^\perp \cap (W_-^\ast)^{\perp, \SobHH{{-\frac12}}(\Sigma;E\otimes \C^m)} \to \GProj_+^\ast B^\perp.$$
Note that, by using a similar projector construction as in the proof of Lemma \ref{Lem:SubProp} \ref{Lem:SubProp:4}, we obtain 
\begin{equation}
\label{Eq:ConsProj2}
\SobHH{{-\frac12}}(\Sigma;E\otimes \C^m) = W_-^\ast \oplus (W_-^\ast)^{\perp, \SobHH{{-\frac12}}(\Sigma;E\otimes \C^m)}.
\end{equation}
As before,  $w \in W_-^\ast \cap (W_-^\ast)^{\perp, \SobHH{{-\frac12}}(\Sigma;E\otimes \C^m)}$ yields $w = 0$ since, in particular, $w \in \oplus_{j=0}^{m -1} \Lp{2}(\Sigma;E)$ and therefore, 
$\inprod{w,w} = \norm{w}^2_{\Lp{2}}$.
For $x = \GProj_+^\ast b^\perp \in \GProj_+^\ast B^\perp$, we obtain $x = \GProj_+^\ast P_{W_-^\ast, (W_-^\ast)^{\perp}}b^\perp$. 
As before, since $W_-^\ast \subset B^\perp$, we have that $P_{W_-^\ast, (W_-^\ast)^{\perp}}b^\perp \in B^\perp \cap (W_-^\ast)^{\perp, \SobHH{{-\frac12}}(\Sigma;E\otimes \C^m)}$ which establishes the surjectivity of $X_+^\ast$.
Injectivity follows simply on noting that $X_+^\ast x = 0$ implies 
$$ x \in \GProj_-^\ast \SobHH{{-\frac12}}(\Sigma;E\otimes \C^m) \cap B^\perp \cap  (W_-^\ast)^{\perp, \SobHH{{-\frac12}}(\Sigma;E\otimes \C^m)} = W_-^\ast \cap (W_-^\ast)^{\perp, \SobHH{{-\frac12}}(\Sigma;E\otimes \C^m)} = 0.$$
This shows that $X_+^\ast$ is a surjection and by the open mapping theorem, it is an isomorphism between $(B^\perp \cap  (W_-^\ast)^{\perp, \SobHH{{-\frac12}}(\Sigma;E\otimes \C^m)}, \norm{\cdot}_{\SobHH{{\frac12}}(\Sigma;E\otimes \C^m_{\op})})$ and  $(\GProj_+^\ast B^\perp, \norm{\cdot}_{\SobHH{{\frac12}}(\Sigma;E\otimes \C^m_{\op})})$, where the former space was established as a Banach space in  Lemma \ref{Lem:SubProp} \ref{Lem:SubProp:4}.

Next, define the following projectors on $\SobHH{{\frac12-m}}(\Sigma;E\otimes \C^m_{\op})$: 
$$P_-^\ast := P_{V_-^\ast,W_-^\ast \oplus \GProj_+^\ast \SobHH{{\frac12-m}}(\Sigma;E\otimes \C^m_{\op}) }\quad\text{and}\quad Q_-^\ast := P_{W_-^\ast,  V_-^\ast \oplus \GProj_+^\ast \SobHH{{\frac12-m}}(\Sigma;E\otimes \C^m_{\op}) }.$$
We have that $P_-^\ast\rest{\SobHH{{\frac12}}(\Sigma;E\otimes \C^m_{\op})}$ is a bounded map from $\SobHH{{\frac12}}(\Sigma;E\otimes \C^m_{\op})$ to $\SobHH{{\frac12}}(\Sigma;E\otimes \C^m_{\op})$ since
$$P_-^\ast u = Q_-^\ast u - \GProj_+^\ast u$$
and $\GProj_+^\ast\rest{\SobHH{{\frac12}}(\Sigma;E\otimes \C^m_{\op})}$ is a bounded map on $\SobHH{{\frac12}}(\Sigma;E\otimes \C^m_{\op})$ and $Q_+^\ast \SobHH{{\frac12-m}}(\Sigma;E\otimes \C^m_{\op}) = W_-^\ast \subset \SobHH{{\frac12}}(\Sigma;E\otimes \C^m_{\op})$.
Therefore, using Lemma \ref{Lem:GraphSpace} \ref{Lem:GraphSpaces:1} to equate $\GProj_+^\ast B^{\perp} = V_+^\ast \cap \SobHH{{\frac12}}(\Sigma;E\otimes \C^m_{\op})$, 
$$ h_0 := P_-^\ast \comp (X_+^\ast)^{-1}: V_+^\ast \cap \SobHH{{\frac12}}(\Sigma;E\otimes \C^m_{\op})  \to V_-^\ast \cap \SobHH{{\frac12}}(\Sigma;E\otimes \C^m_{\op}).$$
Since $P_-^\ast$ is $\SobHH{{\frac12}}(\Sigma;E\otimes \C^m_{\op})$  bounded, the map $X_+^\ast$ was established to be an isomorphism of the respective Banach spaces it maps between, we have that $h_0$ is bounded in the $\SobHH{{\frac12}}(\Sigma;E\otimes \C^m_{\op})$-norm.

To show $W_-^\ast \oplus \set{u + h_0 u: u \in \GProj_+^\ast B^\perp} \subset B^{\perp}$, we note that  $W_-^\ast \subset B^{\perp}$ by construction and that, setting $w = (X_+^\ast)^{-1} u$ for $u \in P_+^\ast B^\perp$,
$$ u+ h_0 u = w - Q_-\ast w \in B^\perp + W_-^\ast \subset B^\perp.$$
For the reverse containment, using \eqref{Eq:ConsProj2}, we write $x = x^\perp + x_- \in B^{\perp} \cap (W_-^\ast)^{\perp, \SobHH{{-\frac12}}(\Sigma;E\otimes \C^m)} \oplus W_-^\ast$ and letting 
 $w = (X_+^\ast)x^\perp$, we  deduce that
$$x^\perp = w + h_0w + Q_-^\ast (X_+^\ast)^{-1}w \in \set{u + h_0 u: v \in P_+^\ast B^{\perp}} \oplus W_-^\ast.$$

Lastly, we show that $h_0 = -g_0^\ast\rest{V_+^\ast \cap \SobHH{{{\frac12}}}(\Sigma;E\otimes \C^m_{\rm op})}$. 
For that, we note that $\inprod{B, B^\perp} = 0$ in the pairing between $\checkH(D)$ and $\hatH(D)$ and so therefore, 
$$ 0 = \inprod{u + g_0 u, v + h_0v} = \inprod{u,v} + \inprod{u, h_0v} + \inprod{g_0 u, v} + \inprod{g_0u, h_0v}.$$
However $u \in V_-$ and $v \in V_+^\ast$ and therefore, 
$$ 0 = \inprod{u,v} =   \inprod{g_0u, h_0v}.$$
From this, we obtain that 
$$ \inprod{u, h_0v} = \inprod{u, -g_0^\ast v}$$
for all $u \in V_-$ and $v \in V_+^\ast \cap \SobHH{{\frac12}}(\Sigma;E\otimes \C^m_{\op})$.
This shows us that $h_0v = -g_0^\ast v$ for all $v \in V_+^\ast \cap \SobHH{{\frac12}}(\Sigma;E\otimes \C^m_{\op})$ and moreover, since $h_0 v \in V_-^\ast \cap \SobHH{{\frac12}}(\Sigma;E\otimes \C^m_{\op})$, by construction of $h_0$, we obtain that 
$$ g_0(V_+^\ast \cap \SobHH{{\frac12}}(\Sigma;E\otimes \C^m_{\op})) \subset V_-^\ast \cap \SobHH{{\frac12}}(\Sigma;E\otimes \C^m_{\op}).$$
This finishes the proof.
\end{proof}

Next, we move onto proving the assertions regarding Fredholm pairs characterisations of regular boundary conditions.

\begin{proof}[Proof that \ref{Thm:Ell:2}$\implies$ \ref{Thm:Ell:3} in Theorem \ref{Thm:Ell} ]
We shall prove that graphical decomposability implies Fredholm decomposability in $\SobHH{{m-{\frac12}}}$. 
The argument is exactly as in the proof of \cite[Theorem 2.9]{BBan}.
The idea is to use \ref{Thm:Ell:2} to write $B = \set{ v + gv: v \in V_{-}} \oplus W_+$ and 
$$\scalebox{1.5}{a}_\dagger\adj{B} = B^{\perp}= \set{u - \adj{g}u: u \in \adj{V}_+\cap \SobHH{{{\frac12}}}(\Sigma;E\otimes \C^m_{\rm op})} \oplus \adj{W}_-.$$
This, along with Proposition \ref{Prop:ProjSob} yields
\begin{align*}
&B \cap \GProj_+\SobHH{{m-\frac12}}(\Sigma;E\otimes \C^m) = W_+,\quad (B^\perp\cap \SobHH{{\frac12}}(\Sigma;E\otimes \C^m_{\rm op})) \cap \hatH(D) = W_-^\ast,\\
&\faktor{\SobHH{{m-\frac12}}(\Sigma;E\otimes \C^m)}{B + \GProj_+\SobHH{{m-\frac12}}(\Sigma;E\otimes \C^m)} \isom \faktor{\checkH(D)}{B + \GProj_+\checkH(D)} \isom W_-,\text{ and}\\ 
&\faktor{\SobHH{{\frac12}}(\Sigma;E\otimes \C^m_{\rm op})}{B^{\perp} + \GProj_-^\ast \hatH(D)} \isom  \faktor{\hatH(D)}{B^{\perp} + \GProj_-^\ast \SobHH{{\frac12}}(\Sigma;E\otimes \C^m_{\rm op})} \isom W_+^\ast. 
\end{align*}
Since $W_\pm \isom W_\pm^\ast$, the conclusion follows.
\end{proof}

Now, we prove the rest of the theorem in two parts.

\begin{proof}[Proof that \ref{Thm:Ell:3}$\implies$ \ref{Thm:Ell:4} in Theorem \ref{Thm:Ell} ]
We shall prove that Fredholm decomposability in $\SobHH{{m-\frac12}}$ implies Fredholm decomposability in $\hatH(D)$. 
The argument is divided into three steps.
\begin{enumerate}[(a)] 
\item 
\label{Pf:a}
$B$ is closed in $\checkH(D)$.

Set $W_+ := B \intersect \GProj_+ \SobHH{{m-\frac12}}(\Sigma;E\otimes \C^m)= B \intersect \GProj_+ \checkH(D)$, since $B \subset \SobHH{{m-\frac12}}(\Sigma;E\otimes \C^m)$. 
Therefore, using Proposition \ref{Prop:ProjSob}, 
$$ B = W_+ \oplus (B \cap W_+^{\perp})\quad\text{and}\quad\GProj_+\checkH(D) = W_+ \oplus (\GProj_+ \checkH(D) \intersect W_+^{\perp}).$$ 
Also, $B \intersect \GProj_+ \SobHH{{m-\frac12}}(\Sigma;E\otimes \C^m)\intersect W_+^{\perp} = W_+ \intersect W_+^{\perp} = 0$, and therefore,
$$  B + \GProj_+ \checkH(D)  = B \oplus ( B \cap W_+^{\perp}).$$

Next, we claim that this space is closed.
For that, we use \cite[Lemma A.4]{BBan} to see that $B + \GProj_+ \checkH(D) = \GProj_-B \oplus \GProj_+ \checkH(D)$ and note that $\GProj_-B \subset \GProj_-\SobHH{{m-\frac12}}(\Sigma;E\otimes \C^m) = \GProj_-\checkH(D)$ is a closed subspace in $\checkH(D)$.

Now, using that $(\GProj_-^*\SobHH{{\frac12}}(\Sigma;E\otimes \C^m_{\rm op})), B^\perp\cap \SobHH{{\frac12}}(\Sigma;E\otimes \C^m_{\rm op}))$ is a Fredholm pair in $\SobHH{{\frac12}}(\Sigma;E\otimes \C^m_{\rm op}))$, we obtain that $(\GProj_+\SobHH{{-\frac12}}(\Sigma;E\otimes \C^m_{\rm op})), (B^\perp\cap \SobHH{{\frac12}}(\Sigma;E\otimes \C^m_{\rm op}))^\perp)$ is a Fredholm pair in $\SobHH{{-\frac12}}(\Sigma;E\otimes \C^m)$ with opposite index.
Using Proposition \ref{Prop:ProjSob} and \cite[Lemma A.5]{BBan}, we obtain that $(\GProj_+\checkH(D), \check{B})$ is a Fredholm pair in $\checkH(D)$, where $\check{B} = (B^\perp\cap \SobHH{{\frac12}}(\Sigma;E\otimes \C^m_{\rm op}))^\perp$.
This yields $B \oplus (\check{B} \intersect \GProj_+ \checkH(D) \intersect W_+^\perp)$ is closed  in $\checkH(D)$. 

To prove the desired claim, we show that $\check{B} \intersect \GProj_+ \checkH(D) \intersect W_+^{\perp} = 0$.
For that, it suffices to show that $ \check{B} \intersect \GProj_+ \checkH(D) = B \intersect \GProj_+ \checkH(D)$.
Since $B \subset \check{B}$, and the dimension of $B \intersect \GProj_+ \checkH(D)$ is finite, it suffices to show that $\dim ( \check{B} \intersect \GProj_+ \checkH(D))  \leq  \dim( B \intersect \GProj_+ \checkH(D))$.
For that, 
\begin{align*}
\dim ( \check{B} \intersect \GProj_+ \checkH(D)) 
	&= \dim\cbrac{\faktor{\hatH(D)}{ B^{\perp} + \GProj_-^\ast \hatH(D)}}\\
	&= \dim \cbrac{ \faktor{\GProj_+^\ast \SobHH{{\frac12}}(\Sigma;E\otimes\C^m_{\rm op})}{\GProj_+^\ast B^{\perp}}} \\ 
	&\leq \dim \cbrac{ \faktor{\GProj_+^\ast \SobHH{{\frac12}}(\Sigma;E\otimes\C^m_{\rm op})}{\GProj_+^\ast B^\perp}} \\ 
	&= \dim(B \intersect \GProj_+ \SobHH{{m-\frac12}}(\Sigma;E\otimes\C^m))=\dim( B \intersect \GProj_+ \checkH(D)).
\end{align*}

\item $(\GProj_+\checkH(D), B)$ is a Fredholm pair in $\checkH(D)$, with  
$$\indx(\GProj_+\checkH(D), B) = \indx(\GProj_+\SobHH{{m-\frac12}}(\Sigma;E\otimes\C^m), B).$$
\label{Pf:b} 

Contained in the proof of \ref{Pf:a} was that $B + \GProj_+\checkH(D)$ is a closed subspace of $\checkH(D)$. 
It is clear that $B \intersect   \GProj_+\checkH(D)$ is a closed subspace of $\checkH(D)$.

Using \cite[Lemma A.4]{BBan}, we obtain that 
$$ \faktor{\checkH(D)}{B + \GProj_+ \checkH(D)} \isom \faktor{\GProj_- \SobHH{{m-\frac12}}(\Sigma;E\otimes\C^m)}{\GProj_- B}$$
as well as
$$ \faktor{\checkH(D)}{B + \GProj_+ \SobHH{{m-\frac12}}(\Sigma;E\otimes\C^m)} \isom \faktor{\GProj_-\SobHH{{m-\frac12}}(\Sigma;E\otimes\C^m)}{\GProj_- B}.$$
This along with  the fact that $B \intersect   \GProj_+\checkH(D) = B \intersect   \GProj_+\SobHH{{m-\frac12}}(\Sigma;E\otimes\C^m)$  then yields the conclusion.

\item $B^{\perp}\cap\SobHH{{\frac12}}(\Sigma;E\otimes\C^m_{\rm op})$ is closed in $\hatH(D)$ and $(\GProj_-^*\hatH(D), B^\perp\cap \SobHH{{\frac12}}(\Sigma;E\otimes \C^m_{\rm op}))$  is a Fredholm pair with   
$$\indx (\GProj_-^*\hatH(D), B^\perp\cap \SobHH{{\frac12}}(\Sigma;E\otimes \C^m_{\rm op}))= \indx(\GProj_-^*\SobHH{{\frac12}}(\Sigma;E\otimes \C^m_{\rm op})), B^\perp\cap \SobHH{{\frac12}}(\Sigma;E\otimes \C^m_{\rm op})).$$ 

That  $B^{\perp}\cap\SobHH{{\frac12}}(\Sigma;E\otimes\C^m_{\rm op})$ is closed in $\hatH(D)$ is obtained via a repetition of the argument in \ref{Pf:a} with the appropriate changes.
Similarly, the remaining conclusion is obtained on mirroring the argument \ref{Pf:b} with the necessary changes.
\qedhere
\end{enumerate} 
\end{proof}

The last remaining  part to establish Theorem \ref{Thm:Ell} is given below. 

\begin{proof}[Proof that \ref{Thm:Ell:4}$\implies$ \ref{Thm:Ell:1} in Theorem \ref{Thm:Ell} ]
We shall prove that Fredholm decomposability in $\hatH(D)$ implies regularity. . 
Since $B \subset \SobHH{{m-\frac12}}(\Sigma;E\otimes \C^m)$ and  closed in $\checkH(D)$, we only need to show that $B^{\perp} \subseteq \SobHH{{\frac12}}(\Sigma;E\otimes \C^m_{\rm op})$.

By hypothesis,  we have that $\indx (\GProj_+\checkH(D), B)= -\indx (\GProj_-^*\hatH(D), B^\perp\cap \SobHH{{\frac12}}(\Sigma;E\otimes \C^m_{\rm op}))$, and therefore $\indx (\GProj_-^*\hatH(D), B^\perp\cap \SobHH{{\frac12}}(\Sigma;E\otimes \C^m_{\rm op}))= \indx (\GProj_-^*\hatH(D), B^\perp)$.
By \cite[Lemma A.7]{BBan}, we obtain that $B^{\perp} =B^\perp\cap \SobHH{{\frac12}}(\Sigma;E\otimes \C^m_{\rm op})$.
\end{proof}

\subsection{The first order case}

For general boundary decomposing projectors $\GProj_{\pm}$ in the first order case, a decomposition more in the spirit of \cite{BBan} can be obtained. We remark that if $\GProj_+$ is boundary decomposing for a first order elliptic operator, then $\GProj_+$ extends by interpolation to $\SobH{s}(\Sigma;E)$ for any $s\in [-{\frac12},{\frac12}]$.

\begin{definition}[Graphical $\Lp{2}$-decomposition]
\label{def:foelldecomspsHO}
Let $B$ be a boundary condition and $\GProj_+$ a boundary decomposing projection for an elliptic differential operator $D$ of order $1$. We say that a \emph{graphical $\Lp{2}$-decomposition} (with respect to $\GProj_+$) of $B$ is the following decomposition: 
\begin{enumerate}[(i)]
\item
	there exist mutually complementary subspaces $W_{\pm}$ and $V_{\pm}$ of $\Lp{2}(\Sigma;E)$ satisfying: 
	$$W_{\pm} \oplus V_{\pm} = \GProj_\pm \Lp{2}(\Sigma;E),$$
\item
	it holds that $W_{\pm}$ are finite dimensional with  $W_{\pm}, \adj{W_\pm} \subset \SobH{\frac12}(\Sigma; E)$, and 
\item
	there exists a bounded linear map $g: V_- \to V_+$ such that  
	$$\adj{g} (\adj{V}_+\cap \SobH{\frac12}(\Sigma;E)) \subset \adj{V}_-\cap \SobH{\frac12}(\Sigma;E)\quad\mbox{and}\quad g(V_-\cap \SobH{\frac12}(\Sigma;E)) \subset V_+\cap \SobH{\frac12}(\Sigma;E),$$ 
	where $\adj{g}$ denotes the $\Lp{2}$-adjoint and 
		$$ B = \set{v + gv: v \in V_-\cap \SobH{\frac12}(\Sigma;E)} \oplus W_+.$$
\end{enumerate}
A boundary condition admitting a graphical $\Lp{2}$-decomposition is said to be graphically $\Lp{2}$-decomposable.
\end{definition}

As before, the spaces $W_{\pm}^\ast$ that appear in the definition, as before, are simply the ranges of the dual projectors.
However, our computations are always in $\Lp{2}$, and therefore, this is simply the ranges of the adjoint projectors in $\Lp{2}$ with respect to the $\Lp{2}$ inner product. The special case $\GProj_+=\chi^+(A)$, for an adapted boundary operator $A$, of Definition \eqref{def:foelldecomspsHO} was stated in \cite[Definition 2.6]{BBan}.

We emphasise that in Definition \ref{def:foelldecomspsHO}, unlike for the graphical decomposition in Definition \ref{def:elldecomspsHO}, the decomposition obtained is in $\Lp{2}$.
Indeed, the following theorem asserts that in the first order case, these decompositions are equivalent. The precise equivalence is set up as follows: a graphical $\Lp{2}$-decomposition produces a graphical decomposition as in Definition \ref{def:elldecomspsHO} by intersecting with $\SobH{{\frac12}}$ and the converse is constructed from $\Lp{2}$-closures.
We present this first order definition separately as it can be useful to compute in $\Lp{2}$ in the first order case instead of the more complicated spaces that appear in the general order decomposition.

\begin{theorem}
\label{Thm:FirstOrderDecomp}
For a first order elliptic differential operator $D$, a boundary condition $B$ is regular if and only if it is graphically $\Lp{2}$-decomposable.
In this case, we have that 
$$B^{\perp}= \set{u - \adj{g}u: u \in \adj{V}_+\cap \SobH{\frac12}(\Sigma;E)} \oplus \adj{W}_-.$$
\end{theorem}

The proof of this theorem runs as in the proof of \cite[Theorem 2.9]{BBan}, with the projectors $\chi^{\pm}(A_r)$ replaced by $\GProj_{\pm}$.

\section{Higher order boundary regularity}
\label{higherorderregsubsec}

We will now prove results on higher order boundary regularity. These results rely on the fact that Theorem \ref{refinedseeley} due to Seeley was, in fact, stated in \cite{seeley65}  also for higher order Sobolev spaces. 

\subsection{Preliminaries}
For an elliptic differential operator $D$, we let $D_{\rm min,s}$ and $D_{\rm max, s}$ denote the minimal and maximal, respectively, closures of $D$ acting densely defined on $\SobH{s}(M;E)\to \SobH{s}(M;F)$. We let $\gamma_s:\dom (D_{\rm max, s})\to \SobHH{{s-{\frac12}}}(\Sigma;E\otimes \C^m)$ denote the trace mapping.

\begin{theorem}
\label{refinedseeleysobolev}
Let $D$ be an elliptic differential operator of order $m>0$. 
The Calderón projection $P_\mathcal{C}\in \Psi^{\pmb 0}_{\rm cl}(\Sigma;E\otimes \C^m)$ from Theorem \ref{refinedseeley} satisfies the following.
\begin{enumerate}[(i)] 
\item \label{refinedseeley:1} 
The operator $\mathcal{K}$ from Theorem \ref{refinedseeley} defines a continuous Poisson operator 
$$\mathcal{K}_s:\SobHH{{s-{\frac12}}}(\Sigma;E\otimes \C^m)\to \dom(D_{\rm max,s}),$$
such that $P_\mathcal{C}=\gamma_s\circ \mathcal{K}_s$ on $\SobHH{{s-{\frac12}}}(\Sigma;E\otimes \C^m)$.
\item  \label{refinedseeley:2} 
The image of 
$$\gamma_s:\dom(D_{\rm max,s})\to \SobHH{{s-{\frac12}}}(\Sigma;E\otimes \C^m),$$
is precisely the subspace 
$$P_\mathcal{C}\SobHH{{s-{\frac12}}}(\Sigma;E\otimes \C^m)\oplus(1-P_\mathcal{C})\SobHH{{s+m-{\frac12}}}(\Sigma;E\otimes \C^m).$$
\end{enumerate} 
In particular, the trace map $\gamma_s$ induces an isomorphism of Banach spaces 
$$\faktor{\dom(D_{\rm max,s})}{\dom(D_{\rm min,s})}\cong P_\mathcal{C}\SobHH{{s-{\frac12}}}(\Sigma;E\otimes \C^m)\oplus (1-P_\mathcal{C})\SobHH{{s+m-{\frac12}}}(\Sigma;E\otimes \C^m).$$
\end{theorem}

While the Calderon projection is in general hard to compute, we shall formulate a condition on a boundary decomposing projection in order for it to implement higher semiregularity properties. For $s\geq 0$, we introduce the notation 
$$\checkH^s(D):=P_\mathcal{C}\SobHH{{s-{\frac12}}}(\Sigma;E\otimes \C^m)\oplus (1-P_\mathcal{C})\SobHH{{s+m-{\frac12}}}(\Sigma;E\otimes \C^m)
\cong 
\faktor{\dom(D_{\rm max,s})}{\dom(D_{\rm min,s})},$$
where the last isomorphism is implemented by $\gamma$ by Theorem \ref{refinedseeleysobolev}.

\begin{definition}
\label{boundafefidodohigher}
Let $D$ be an elliptic differential operator of order $m>0$ and $s \geq 0$ a positive number. We say that a projection $\GProj_+$ is $s$-boundary decomposing for $D$ if $\GProj_+$ is boundary decomposing for $D$ and additionally satisfies the following conditions 
\begin{enumerate}[({P}1)]
\item \label{Def:EStartHO} \label{Def:E1HO} 
	$\GProj_+: \SobHH{s+ \alpha}(\Sigma; E\otimes \C^m) \to \SobHH{s+ \alpha}(\Sigma;E\otimes \C^m)$ is a bounded projection for $\alpha \in \left\{-\frac12, m-\frac12\right\}$, 
\item \label{Def:E2HO} 
	$\GProj_+: \checkH^s(D) \to \checkH^s(D)$, and $\GProj_- := (I - \GProj_+):  \checkH^s(D) \to \SobHH{{s+m-\frac12}}(\Sigma;E\otimes \C^m)$,  and
\item \label{Def:EEndHO} \label{Def:E3HO} 
	$\norm{u}_{\checkH^s(D)} \simeq \norm{\GProj_- u}_{\SobHH{{s+m-\frac12}}(\Sigma;E\otimes \C^m)} + \norm{\GProj_+ u}_{\SobHH{s-{\frac12}}(\Sigma;E\otimes \C^m)}.$
\end{enumerate} 
\end{definition}

Similarly to Theorem \ref{thm:bodunarodoad} and Corollary \ref{pseudostahtdecompose}, we have the following result providing a sufficient condition for a projection being $s$-boundary decomposing. 

\begin{theorem}
\label{thm:bodunarodoadhigher}
Let $D$ be an elliptic differential operator of order $m>0$, $s \geq 0$ a positive number and $P$ a continuous projection on $\SobHH{{s+m-\frac12}}(\Sigma;E\otimes \C^m)$. Assume that $P$ satisfies the following: 
\begin{itemize} 
\item $P$ is continuous also in the norms on $\checkH(D)$, $\checkH^s(D)$, $\SobHH{{-\frac12}}(\Sigma;E\otimes \C^m)$, $\SobHH{{s-\frac12}}(\Sigma;E\otimes \C^m)$  and $\SobHH{{m-\frac12}}(\Sigma;E\otimes \C^m)$.
\item The operator $A:=P_\mathcal{C}-(1-P)$ defines a Fredholm operator on $\SobHH{{s+m-\frac12}}(\Sigma;E\otimes \C^m)$ and by continuity extends to Fredholm operators on $\checkH(D)$, $\checkH^s(D)$ and $\SobHH{{m-\frac12}}(\Sigma;E\otimes \C^m)$. 
\end{itemize} 
Then $P$ is $s$-boundary decomposing, and in particular 
$$\norm{u}_{\checkH^s(D)} \simeq \norm{(1-P) u}_{\SobHH{{s+m-\frac12}}(\Sigma;E\otimes \C^m)} + \norm{P u}_{\SobHH{s-{\frac12}}(\Sigma;E\otimes \C^m)}.$$
In the special case that $P\in \Psi^{\pmb 0}_{\rm cl}(\Sigma; E\otimes \C^m)$ is a projection such that
\begin{itemize} 
\item $[P_\mathcal{C},P]\in \Psi^{\pmb{-m}}_{\rm cl}(\Sigma; E\otimes \C^m)$; and
\item the operator $A:=P_\mathcal{C}-(1-P)\in \Psi^{\pmb 0}_{\rm cl}(\Sigma; E\otimes \C^m)$ is elliptic,
\end{itemize} 
then $P$ is $s$-boundary decomposing.
\end{theorem}

Let us consider a corollary of Theorem \ref{refinedseeleysobolev}.

\begin{corollary}
\label{cor:highfhdowo}
Let $D$ be an elliptic differential operator of order $m>0$ and $s\geq 0$. Then it holds that 
\begin{align*}
\dom(D_{\rm max})\cap &\SobH{s+m}(M;E)=\\
=&\{u\in \dom(D_{\rm max}): D_{\rm max}u\in\SobH{s}(M;F),\; P_\mathcal{C}\gamma u\in \SobHH{{s+m-{\frac12}}}(\Sigma;E\otimes \C^m)\}.
\end{align*}
\end{corollary}

\begin{proof}
The inclusion $\subseteq$ is clear, we need to prove that if $u\in \dom(D_{\rm max})$ satisfies that $ D_{\rm max}u\in\SobH{s}(M;F)$ and $P_\mathcal{C}\gamma u\in \SobHH{{s+m-{\frac12}}}(\Sigma;E\otimes \C^m)$ then $u\in \SobH{s+m}(M;E)$. After replacing $u$ with $u-\mathcal{K}_{s+m}P_\mathcal{C}\gamma u$ we can assume that $P_\mathcal{C}\gamma u=0$. In other words, we need to prove that if $u\in \dom(D_{\mathcal{C}^c})$ and $D_{\mathcal{C}^c}u\in \SobH{s}(M;F)$ then $u\in\SobH{s+m}(M;E)$. 

Define $\mathcal{C}^c_s:=(1-P_{\mathcal{C}}) \SobHH{{s+m-{\frac12}}}(\Sigma;E\otimes \C^m)$ and consider the associated realisation $D_{\mathcal{C}^c_s,s}$ of $D$ on $\SobH{s}(M;E)$. We note that $D_{\mathcal{C}^c_s,s}$ is semi-regular in the sense that $\dom(D_{\mathcal{C}^c_s,s})\subseteq \SobH{s+m}(M;E)$. 

Theorem \ref{refinedseeleysobolev} implies that $\Ck{\infty}(M;E) \cap \dom(D_{\Ca^c})$ is a core for $D_{\Ca^c_s,s}$ for any $s$. Therefore, we have dense inclusions
\begin{align*}
D_{\Ca^c} (\Ck{\infty}(M;E) \cap \dom(D_{\Ca^c})) &\subset \ran(D_{\Ca^c_s,s})\quad\mbox{and}\\
D_{\Ca^c} (\Ck{\infty}(M;E) \cap \dom(D_{\Ca^c})) &\subset  \ran(D_{\Ca^c})\cap \SobH{s}(M;E).
\end{align*}
However, the spaces $\ran(D_{\mathcal{C}^c_s,s})=\ran(D_{\max,s})=D\SobH{s+m}(M;E)$ and $\ran(D_{\Ca^c})\cap \SobH{s}(M;E)$ are closed in $\SobH{s}(M;E)$, and therefore, $\ran(D_{ \Ca^c_s,s}) = \ran(D_{\Ca^c})\cap  \SobH{s}(M;E)$.
The condition that $u \in \dom(D_{\Ca^c})$ with $D_{\Ca^c}u \in \SobH{s}(M;F)$ and $u \in \SobH{s}(M;E)$ therefore implies that there exists $v \in \dom(D_{\Ca^c_s,s})\subseteq \SobH{s+m}(M;E)$ such that $Dv = D_{\Ca^c} u$. 
We then have that $D(v - u) = 0$, so $\gamma(v - u) \in \Ca\cap \Ca^c= 0$, and therefore $v - u \in \ker(D_{\min}) \subset \Ck{\infty}(M;E)$.
Hence, $u = v + (u - v) \in \SobH{s+m}(M;E) + \Ck{\infty}(M;E) = \SobH{s+m}(M;E)$ which proves the required claim. 
\end{proof}

\subsection{Higher order regularity of boundary conditions}

In \cite{BB, BBan}, higher order regularity of boundary conditions was introduced under the name of regular boundary conditions. The regularity was stated in terms of graphical decompositions of regular boundary conditions. We shall formulate a slightly more general condition. 

\begin{definition}
\label{higherorderregsemiregular}
Let $D$ be an elliptic differential operator, $B\subseteq \checkH(D)$ a semi-regular boundary condition (in $\Lp{2}$-sense), $s\geq 0$ and $P$ an $s$-boundary decomposing projection. 
\begin{itemize}
\item We say that $B$ is $s$-semiregular (w.r.t $P$) if for any $\xi\in B$ it holds that $\xi\in \SobHH{{s+m-{\frac12}}}(\Sigma;E\otimes \C^m)$ if and only if $(1-P)\xi\in \SobHH{{s+m-{\frac12}}}(\Sigma;E\otimes \C^m)$.
\item We say that the realisation $D_B$ is $s$-semiregular if whenever $u\in \dom(D_{\rm B})$ satisfies that $D_{\rm B}u\in \SobH{s}(M;F)$ then $u\in \SobH{s+m}(M;E)$.
\end{itemize}
\end{definition}

The reader should note that the proof of Corollary \ref{cor:highfhdowo} essentially boils down to proving that $\mathcal{C}^c$ is $s$-semiregular for any $s\geq 0$ (w.r.t $P_{\mathcal{C}}$). 
Therefore, Corollary \ref{cor:highfhdowo} partially proves the main result of this section (Theorem \ref{Thm:BdyRegHigh} below) in the special case of the boundary condition $\mathcal{C}^c$.

\begin{theorem}[Higher order boundary regularity]
\label{Thm:BdyRegHigh}
Let $D$ be an elliptic differential operator of order $m>0$, $s\geq 0$ and $B\subseteq \checkH(D)$ a boundary condition.  Then the following are equivalent:
\begin{enumerate}
\item $D_B$ is $s$-semiregular.
\item $B$ is $s$-semiregular with respect to $P_\mathcal{C}$ 
\item $B$ is $s$-semiregular with respect to any $s$-boundary decomposing projection.
\end{enumerate}
\end{theorem}

\begin{proof}
We start by showing that if $P$ is an $s$-boundary decomposing projection and there is a $\xi\in B$ with $(1-P)\xi\in \SobHH{{s+m-{\frac12}}}(M;E\otimes \C^m)$ but  $\xi\notin \SobHH{{s+m-{\frac12}}}(M;E\otimes \C^m)$, then there is a $u\in \dom(D_{\rm B})$ satisfying that $D_{\rm B}u\in \SobH{s}(M;F)$ and $u\notin \SobH{s+m}(M;E)$. Indeed, we take a preimage $u\in \gamma^{-1}(\xi)$ written as $u=u_0+u_1\in \dom(D_{\max,s})$ where $\gamma(u_0)=(1-P)\xi$ and $\gamma(u_1)=P\xi$. Since $(1-P)\xi\in \SobHH{{s+m-{\frac12}}}(M;E\otimes \C^m)$ we can assume that $u_0\in \SobH{s+m}(M;E)$. However, if $(1-P)\xi\in \SobHH{{s+m-{\frac12}}}(M;E\otimes \C^m)$ but  $\xi\notin \SobHH{{s+m-{\frac12}}}(M;E\otimes \C^m)$ then $P\xi \notin \SobHH{{s+m-{\frac12}}}(M;E\otimes \C^m)$ so $u_1\notin \SobH{s+m}(M;E)$. This proves that 1) implies 3).

It is clear that 3) implies 2). It remains to prove that 2) implies 1), i.e that if $B$ is $s$-semiregular with respect to $P_\mathcal{C}$ then whenever $u\in \dom(D_{\rm B})$ satisfies that $D_{\rm B}u\in \SobH{s}(M;F)$ then $u\in \SobH{s+m}(M;E)$. Take a $u\in \dom(D_{\rm B})$ satisfying that $D_{\rm B}u\in \SobH{s}(M;F)$. By applying the same argument as in the proof of Corollary \ref{cor:highfhdowo}, we have that $w:=u-\mathcal{K}\gamma u\in \dom(D_{\mathcal{C}^c_s,s})\subseteq \SobH{s+m}(M;E)$. On the other hand $\gamma(w)=\gamma(u)-P_\mathcal{C}\gamma(u)=(1-P_\mathcal{C})\gamma(u)\in \SobHH{{s+m-{\frac12}}}(M;E\otimes \C^m)$ so by weak $s$-semiregularity $\gamma(u)\in \SobHH{{s+m-{\frac12}}}(M;E\otimes \C^m)$. Therefore $P_\mathcal{C}\gamma(u)\in \SobHH{{s+m-{\frac12}}}(M;E\otimes \C^m)$ and can conclude that $\mathcal{K}\gamma u=\mathcal{K}_{s+m}\gamma u\in \SobH{s+m}(M;E)$ by Theorem \ref{refinedseeleysobolev}. We conclude that $u=w+\mathcal{K}\gamma u\in \SobH{s+m}(M;E)$. We conclude that 1) and 2) are equivalent.
\end{proof} 

Based on the graphical decompositions of Section \ref{subsec:charofellfirstorder}, let us provide a method of producing semiregular boundary conditions. 

\begin{proposition}
\label{charintermsofchar}
Let $B$ be a boundary condition for an elliptic differential operator $D$ of order $m>0$, $s\geq 0$ and $P$ an $s$-boundary decomposing projection. Assume that $B$ is graphically decomposable with respect to $P$ as in Definition \ref{def:elldecomspsHO}. Then $B$ is $s$-semiregular if and only if 
\begin{align*}
W_+&\subseteq \SobHH{{s+m-{\frac12}}}(\Sigma;E\otimes \C^m)\quad\mbox{and}\\ 
g(V_+&\cap \SobHH{{s+m-{\frac12}}}(\Sigma;E\otimes \C^m)) \subseteq V_-\cap \SobHH{{s+m-{\frac12}}}(\Sigma;E\otimes \C^m).
\end{align*}
\end{proposition}

The proof goes ad verbatim to \cite[Lemma 8.9]{BBan}.

\begin{corollary}
\label{highregformodulelososo}
Let $D$ be an elliptic differential operator of order $m>0$ and $B\subseteq \checkH(D)$ a boundary condition. Then the realisation $D_{\rm B}^*D_{\rm B}$ of $D^\dagger D$ is regular if $B$ is regular and $m$-semiregular. 
\end{corollary}

\begin{proof}
Before getting down to brass tacks, we note that the operator $D_{\rm B}^*D_{\rm B}$ is self-adjoint, so $D_{\rm B}^*D_{\rm B}$ is a regular realisation if and only if it is a semi-regular realisation. If $B$ is regular and $m$-semiregular then any $u\in \dom(D_{\rm B}^*D_{\rm B})$ satisfies $D_{\rm B}u\in \dom(D_{\rm B}^*)\subseteq \SobH{m}(M;F)$, because $B$ is regular. Since $B$ is $m$-semiregular, Theorem \ref{Thm:BdyRegHigh} implies that $u\in \SobH{2m}(M;E)$ when $D_{\rm B}u\in \SobH{m}(M;F)$. We conclude that $\dom(D_{\rm B}^*D_{\rm B})\subseteq \SobH{2m}(M;E)$ if $B$ is regular and $m$-semiregular.
\end{proof}

\begin{remark}
We note that an interpolation argument shows that if $D$ is an elliptic differential operator of order $m>0$ and $B\subseteq \checkH(D)$ a boundary condition such that $D_{\rm B}^*D_{\rm B}$ is regular, then $B$ is semi-regular. Since the proof of Corollary \ref{highregformodulelososo} only used semi-regularity of $B^*$ and $m$-semiregularity of $B$, we conclude that any $m$-semiregular boundary condition $B$ with $B^*$ semi-regular is regular.
\end{remark}

\subsection{Higher regularity of the Dirichlet problem}
\label{highregofdiricl}

As an example of our theory, we consider elliptic operators with a uniquely solvable Dirichlet problem. Such differential operators play a prominent role in \cite{grubb68,grubb71,grubb74}. The results of this subsection are well known. Along the way, we shall exemplify the results of this paper for operators with positive principal symbols.

For an elliptic differential operator $D$ of even order $m$, we define the Dirichlet condition as the space
$$B_{\rm Dir}:=\{\xi=(\xi_j)_{j=0}^{m-1}\in \checkH(D) \subset \SobHH{\frac12 - m}(\Sigma;E \tensor C^m): \xi_j=0, \;j=0,\ldots, \frac{m}{2}-1\}.$$
We say that $D$ has a uniquely solvable Dirichlet problem if $B_{\rm Dir}$ is a regular boundary condition and $D_{\rm Dir}:=D_{B_{\rm Dir}}$ is an invertible realisation. We note the following important example.

\begin{lemma}
\label{dirrealzlem}
Let $D=D^\dagger$ be a formally selfadjoint elliptic operator of order $m$ acting on the sections of a hermitian vector bundle $E$ on an $n$-dimensional manifold with boundary. Assume that $D$ has positive interior principal symbol $\sigma_D$. Then $D_{\rm Dir}$ is the realisation with domain $\dom(D_{\rm Dir}):=\SobH{m}(M;E)\cap \SobH[0]{m/2}(M;E)$. Moreover, $D_{\rm Dir}$ is regular, self-adjoint and lower semibounded. In particular, for large enough $\lambda\in \R$, $D+\lambda$ has a uniquely solvable Dirichlet problem.
\end{lemma}

We remark that if $D$ has positive interior symbol, the symmetry property $\sigma_D(x,-\xi)=(-1)^m\sigma_D(x,\xi)$ implies that $m$ is even. 

\begin{proof}
By extending $D$ to a self-adjoint elliptic differential operator with positive principal symbol on a closed manifold containing $M$ as a smooth subdomain, we see that the Gårding inequality implies that for $\lambda>0$ large enough it holds that 
$$\langle (D+\lambda)\varphi,\varphi\rangle_{\Lp{2}(M;E)}\sim \|\varphi\|_{\SobH[0]{m/2}(M;E)}^2,$$
for $\varphi\in \Ck[c]{\infty}(M;E)$. Therefore $D_c$ is lower semi-bounded and symmetric, and $D_{\rm Dir}$ is indeed the Friedrichs extension of $D_c$. Therefore, $D_{\rm Dir}$ is self-adjoint and lower semibounded. Since $D_{\rm Dir}$ is semi-regular by construction, self-adjointness implies regularity.
\end{proof}

We now describe a certain projection used in \cite{grubb68}. By Lemma \ref{dirrealzlem} unique solvability of the Dirichlet problem is no serious restriction for operators with positive principal symbol. Under this assumption, $\gamma$ restricts to an isomorphism $\ker(D_{\rm max})\cong \mathcal{C}_{D}$. Since $D_{\rm Dir}$ is invertible and regular, the operator 
$$P_\zeta:=1-D_{\rm Dir}^{-1}D_{\rm max},$$ 
is a continuous projection on $\dom(D_{\max})$. The projection $P_\zeta$ is the projection onto $\ker(D_{\rm max})$ along $\dom(D_{\rm Dir})$. As such, $P_\zeta$ defines a Banach space isomorphism 
$$(1-P_\zeta)\oplus P_\zeta:\dom(D_{\max})\xrightarrow{\sim} \dom(D_{\rm Dir})\oplus \ker(D_{\rm max}).$$
For more details, see \cite[Chapter II]{grubb68}. Under $\gamma$, we arrive at a decomposition
$$\checkH(D)=B_{\rm Dir}\oplus \mathcal{C}_{D},$$
as Banach spaces when topologising $B_{\rm Dir}$ as a subspace of $\SobHH{m-1/2}(\Sigma;E\otimes \C^m)$. We let $P_{\zeta,\partial}$ denote the continuous projection in $\checkH(D)$ onto $\mathcal{C}_{D}$ along $B_{\rm Dir}$; note that $P_{\zeta,\partial}$ is induced from $P_\zeta$.

\begin{lemma}
\label{bdecodladlda}
Let $D$ be an elliptic differential operator of even order with uniquely solvable Dirichlet problem. Let $s\geq 0$. The projection $P_{\zeta,\partial}$ is $s$-boundary decomposing for $D$ and is of order zero in the Douglis-Nirenberg calculus $P_{\zeta,\partial}\in \Psi^{\pmb 0}_{\rm cl}(\Sigma;E\otimes \C^m)$. 
\end{lemma} 

\begin{proof}
We write $\gamma=(\gamma_D,\gamma_N)$ where 
\begin{align*}
\gamma_D:&\dom(D_{\max})\to \SobHH{-1/2}(\Sigma;E\otimes \C^{m/2}), \quad u\mapsto (\partial_{x_n}^k u|_\Sigma)_{k=0}^{m/2-1},\mbox{ and}\\
\gamma_N:&\dom(D_{\max})\to \SobHH{-m/2-1/2}(\Sigma;E\otimes \C^{m/2}), \quad u\mapsto (\partial_{x_n}^k u|_\Sigma)_{k=m/2}^{m-1}.
\end{align*}
 By \cite{grubb71}, there is a Dirichlet-to-Neumann operator $\Lambda_{DN}\in \Psi^{\pmb m}_{\rm cl}(\Sigma;E\otimes \C^{m/2})$ defined from $\Lambda_{DN}\xi:=\gamma_{N}u$ where $u\in \ker(D_{\max})$ satisfies $\gamma_Du=\xi$. A short computation shows that 
\begin{equation}
\label{compekakaldal}
P_{\zeta,\partial}=\begin{pmatrix} 1& 0\\ \Lambda_{DN}&0\end{pmatrix},
\end{equation}
in the decomposition 
$$\checkH(D)\subseteq \SobHH{-1/2}(\Sigma;E\otimes \C^m)=\SobHH{-1/2}(\Sigma;E\otimes \C^{m/2})\oplus \SobHH{-m/2-1/2}(\Sigma;E\otimes \C^{m/2}).$$
It follows that $P_{\zeta,\partial}\in \Psi^{\pmb 0}_{\rm cl}(\Sigma;E\otimes \C^m)$. 

By construction, $1-P_{\zeta,\partial}$ maps $\checkH(D)$ into $B_{\rm Dir}\subseteq \SobHH{m-1/2}(\Sigma;E\otimes \C^m)$ and defines a decomposition $\checkH(D)=B_{\rm Dir}\oplus \mathcal{C}_{D+\lambda}$ of Banach spaces. We conclude that $P_{\zeta,\partial}$ is boundary decomposing for $D$. The fact that $P_{\zeta,\partial}$ is $s$-boundary decomposing for all $s\geq 0$ follows from Theorem \ref{thm:bodunarodoadhigher}.
\end{proof}

\begin{example}
An immediate consequence of Lemma \ref{bdecodladlda}  and Theorem \ref{Thm:BdyRegHigh} is the well known fact that the Dirichlet problem is $s$-semiregular for all $s\geq 0$ if it is uniquely solvable in $\Lp{2}$. With respect to $P_{\zeta,\partial}$, we can graphically decompose the Dirichlet condition as 
$$B_{\rm Dir}= \set{v + gv: v \in V_-} \oplus W_+,$$
for 
\begin{align*}
V_-=&(1-P_{\zeta,\partial})\checkH(D)=(1-P_{\zeta,\partial})\SobHH{m-1/2}(\Sigma;E\otimes \C^m),\\ 
V_+=&P_{\zeta,\partial}\checkH(D)=P_{\zeta,\partial}\SobHH{m-1/2}(\Sigma;E\otimes \C^m),\\
g=&0\quad \mbox{and} \quad W_+=W_-=0.
\end{align*} 
\end{example}

\begin{example}
\label{lapapallald}
Another consequence of Lemma \ref{bdecodladlda}  and Theorem \ref{Thm:BdyRegHigh} concerns Robin-type problems with claims concerning $s$-semiregularity for all $s\geq 0$ if the Dirichlet problem is uniquely solvable in $\Lp{2}$. This is studied in more detail and generality in \cite{grubb74}. 

We give an example in the case that $D$ is second order with positive principal symbol, e.g. a Laplace type operator. For an element $a\in C^\infty(\Sigma;\End(E))$, consider the boundary condition
$$B_a:=\{(\xi_0,\xi_1)\in \checkH(D): a \xi_0+\xi_1=0\}.$$
This boundary condition corresponds to the Robin problem for $D$, i.e. for $f\in \Lp{2}(M;E)$, the equation $D_{B_{a}}u=f$ is equivalent to the Robin problem 
$$\begin{cases} 
Du=f, \quad &\mbox{in $M$}\\
au|_\Sigma+\partial_nu|_{\Sigma}=0,  \quad &\mbox{on $\Sigma$}.
\end{cases}$$
The boundary condition $B_a$ is a pseudolocal boundary condition $B_a=B_{P_a}$ where 
$$P_a=\begin{pmatrix}0& 0\\ a& 1\end{pmatrix}.$$
A computation with Proposition \ref{describinglocalcndldldaadjoint} (cf. \cite{grubb74}) shows that $B_a$ is regular.

We compute from Equation \eqref{compekakaldal} that 
\begin{equation}
\label{compfofdooad}
(1-P_{\zeta,\partial}(\lambda))\xi=\begin{pmatrix} 0\\ -\Lambda_{DN} \xi_0+\xi_1\end{pmatrix}.
\end{equation}
Take an $s\geq 0$, and assume that $\xi\in B_a$ satisfies $(1-P_{\zeta,\partial}(\lambda))\xi\in \SobHH{s+m-1/2}(\Sigma;E\otimes \C^2)$, or equivalently (by Equation \eqref{compfofdooad}) that $-\Lambda_{DN} \xi_0+\xi_1\in \SobH{s+m-3/2}(\Sigma;E)$. If $(\xi_0,\xi_1)\in B_a$, we therefore have that $(\Lambda_{DN} +a)\xi_0\in \SobH{s+m-3/2}(\Sigma;E)$. Since $\Lambda_{DN} \in \Psi^1_{\rm cl}(\Sigma;E)$ is elliptic, we conclude from elliptic regularity that if $\xi=(\xi_0,\xi_1)\in B_a$ satisfies $(1-P_{\zeta,\partial}(\lambda))\xi\in \SobHH{s+m-1/2}(\Sigma;E\otimes \C^2)$ then $\xi_0\in \SobH{s+m-1/2}(\Sigma;E)$ so $\xi\in \SobHH{s+m-1/2}(\Sigma;E\otimes \C^2)$. We conclude that $B_a$ is $s$-semiregular for any $s\geq 0$ when $a\in C^\infty(\Sigma;\End(E))$. In particular, if $f\in \SobH{s}(M;E)$, any solution $u\in \Lp{2}(M;E)$  to the Robin problem 
$$\begin{cases} 
Du=f, \quad &\mbox{in $M$}\\
au|_\Sigma+\partial_nu|_{\Sigma}=0,  \quad &\mbox{on $\Sigma$},
\end{cases}$$
satisfies $u\in \SobH{s+m}(M;E)$.

Let us graphically decompose the Robin condition $B_a$ with respect to $P_{\zeta,\partial}$. For $\xi=(\xi_0,\xi_1)=(\xi_0,-a\xi_0)\in B_a$, we can write 
$$\xi=
\underbrace{\begin{pmatrix}
\xi_0\\ \Lambda_{DN}\xi_0
\end{pmatrix}}_{\in P_{\zeta,\partial}\checkH(D)}+
\underbrace{\begin{pmatrix}
0\\ -\Lambda_{DN}\xi_0-a\xi_0
\end{pmatrix}}_{\in (1-P_{\zeta,\partial})\checkH(D)}.$$
The operator $\Lambda_{DN}$ is first order elliptic, and so is $\Lambda_{DN}+a$. So there exists an elliptic $T\in \Psi^{-1}_{\rm cl}(\Sigma;E)$ such that $1-(\Lambda_{DN}+a)T$ and $1-T(\Lambda_{DN}+a)$ are smoothing finite rank projectors, $\ran((\Lambda_{DN}+a)T)=\ran(\Lambda_{DN}+a)$ and $\ker(T(\Lambda_{DN}+a))=\ker(\Lambda_{DN}+a)$. Define the space 
$$V_-:=\left\{\begin{pmatrix} 0\\\eta_1\end{pmatrix}\in (1-P_{\zeta,\partial})\SobHH{3/2}(\Sigma;E\otimes \C^2): (\Lambda_{DN}+a)T\eta_1=\eta_1\right\}.$$
The space 
$$W_-:=\left\{\begin{pmatrix} 0\\\eta_1\end{pmatrix}\in (1-P_{\zeta,\partial})\SobHH{3/2}(\Sigma;E\otimes \C^2): (\Lambda_{DN}+a)T\eta_1=0\right\},$$
is by the construction of $T$ a finite-dimensional subspace of $\Ck{\infty}(\Sigma;E\otimes \C^2)$ such that 
$$V_-\oplus W_-=(1-P_{\zeta,\partial})\SobHH{3/2}(\Sigma;E\otimes \C^2).$$
Define the space 
$$V_+:=\left\{\begin{pmatrix} \eta_0\\\Lambda_{DN}\eta_0\end{pmatrix}\in P_{\zeta,\partial}\SobHH{3/2}(\Sigma;E\otimes \C^2): T(\Lambda_{DN}+a)\eta_0=\eta_0\right\}.$$
The space 
$$W_+:=\left\{\begin{pmatrix} \eta_0\\\Lambda_{DN}\eta_0\end{pmatrix}\in P_{\zeta,\partial}\SobHH{3/2}(\Sigma;E\otimes \C^2): T(\Lambda_{DN}+a)\eta_0=0\right\},$$
is by the construction of $T$ a finite-dimensional subspace of $\Ck{\infty}(\Sigma;E\otimes \C^2)$ such that 
$$V_+\oplus W_+=P_{\zeta,\partial}\SobHH{3/2}(\Sigma;E\otimes \C^2).$$
Define the map $g:V_-\to V_+$ by 
$$g\begin{pmatrix} 0\\\eta_1\end{pmatrix}:=\begin{pmatrix} -T\eta_1\\-\Lambda_{DN}T\eta_1\end{pmatrix}.$$
A short computation shows that 
$$\adj{g} (\adj{V}_+\cap \SobHH{{{\frac12}}}(\Sigma;E\otimes \C^2_{\rm op})) \subset \adj{V}_-\cap \SobHH{{{\frac12}}}(\Sigma;E\otimes \C^2_{\rm op}),$$ 
and that 
$$B_{a}= \set{v + gv: v \in V_-} \oplus W_+.$$
This gives a graphical decomposition of $B_a$. The existence of a graphical decomposition will by Theorem \ref{Thm:Ell} give an alternative proof of the fact that $B_a$ is elliptic. The construction of $g$ in terms of a pseudodifferential operator, combined with Proposition \ref{charintermsofchar}, gives an alternative explanation for the fact that $B_a$ is $s$-semiregular for any $s\geq 0$.
\end{example}

\section{Differential operators with positive principal symbols and their Weyl laws}
\label{weylalallalla}

In this section we shall prove a Weyl law for general lower semi-bounded self-adjoint regular realisations of an elliptic differential operator $D$ of order $m>0$ with positive interior principal symbol. Weyl laws with error estimates are known under mild assumptions from \cite{agmonasy,horasy}, see also \cite[Theorem 1]{grubb77a}. The authors are grateful to Gerd Grubb who showed us how to prove the Weyl law for general boundary conditions by combining the spectral asymptotics of the Dirichlet realisation with \cite{grubb68}. 

The Weyl formula for the asymptotic behaviour of the eigenvalues of Dirichlet realisations of elliptic operators  dates back to Weyl \cite{weyl12}, where his focus was on the Laplacian. Weyl's result was extended to more general operators in many contributions to the literature, as summed up e.g. in Agmon's paper \cite{agmonasy} (including the treatment by Gårding \cite{garding53} of variable-coefficient operators). In \cite{agmonasy} and subsequent treatments by Hörmander \cite{horasy}, Duistermaat-Guillemin \cite{MR405514}, Seeley \cite{MR506893} and Ivrii \cite{MR575202} the focus has been on the remainder term. See more in the survey of Ivrii \cite{ivrii16}. We shall here just study the principal asymptotics, however for realisations that are more general than those defined by pseudodifferential boundary conditions.

We now use results of \cite{grubb68,gerdsgreenbook} to describe the resolvent of $D_{\rm B}$ under the assumptions stated above. The method of proof will be to reduce the problem to a well behaved boundary condition by studying the resolvent difference
$$(D_{\rm B}+\lambda)^{-1}-(D_{\rm Dir}+\lambda)^{-1},$$
for the Dirichlet extension $D_{\rm Dir}$. The idea of studying this resolvent difference and its localisation to the boundary is due to Birman \cite{MR0142896,MR0177311}. The precise description that we use follows Grubb \cite{grubb68}. Recall the results of Subsection \ref{highregofdiricl} for the Dirichlet realisation. We write $P_{\zeta,\partial}(\lambda)$ for the boundary decomposing projection on $\checkH(D)$ induced by the projection $P_\zeta(\lambda):=1-(D_{\rm Dir}+\lambda)^{-1}(D_{\rm max}+\lambda)$ on $\dom(D_{\rm max})$ (cf. Lemma \ref{bdecodladlda}).

\begin{theorem}
\label{resolventchar}
Let $D=D^\dagger$ be formally selfadjoint elliptic operator of order $m$ acting on the sections of a hermitian vector bundle $E$ on an $n$-dimensional manifold with boundary. Assume that $D$ has positive interior principal symbol $\sigma_D$ and that $B$ is a regular and self-adjoint boundary condition. Then $D_{\rm B}$ has discrete spectrum and for some $\lambda>0$ large enough then we can write 
$$(D_{\rm B}+\lambda)^{-1}=Q_++G+\mathfrak{G}_{\rm B},$$
where 
\begin{enumerate}[(i)]
\item $Q_+$ is the truncation of a pseudo-differential operator $\hat{Q}\in \Psi^{-m}_{\rm cl}(\hat{M};\hat{E})$ that acts on a hermitian vector bundle $\hat{E}$ extending $E$ to a closed manifold $\hat{M}$ containing $M$ as a smooth domain, and satisfies that 
$$\sigma_{-m}(\hat{Q})|_{T^*M\setminus M}=\sigma_D^{-1}.$$
\item $G$ is a singular Green operator of order $-m$ and class $0$ (cf. \cite{gerdsgreenbook}).
\item $\mathfrak{G}_{\rm B}$ is an operator that factors over continuous mappings 
\begin{align*}
\Lp{2}(M;E)\xrightarrow{P_{\ker(D_{\rm max}+\lambda)}} &\ker(D_{\rm max}+\lambda)\xrightarrow{\gamma}\mathcal{C}_{D+\lambda}\xrightarrow{P_{\overline{P_{\zeta,\partial}(\lambda)B}^{\checkH(D)}}} \overline{P_{\zeta,\partial}(\lambda)B}^{\checkH(D)}\xrightarrow{\mathfrak{T}_{\rm B}^{-1}}\\
 &\xrightarrow{\mathfrak{T}_{\rm B}^{-1}}P_{\zeta,\partial}(\lambda)B\hookrightarrow \SobHH{{ m-{\frac12}}}(\Sigma;E\otimes \C^m)\hookrightarrow \SobHH{{ -{\frac12}}}(\Sigma;E\otimes \C^m)\xrightarrow{\mathcal{K}}\Lp{2}(M;E),
 \end{align*}
where $P_{\ker(D_{\rm max}+\lambda)}$ and $P_{\overline{P_{\zeta,\partial}(\lambda)B}^{\checkH(D)}}$ denotes the orthogonal projectors and $\mathfrak{T}_{\rm B}:P_{\zeta,\partial}(\lambda)B\to \overline{P_{\zeta,\partial}(\lambda)B}^{\checkH(D)}$ is the operator constructed in \cite[Chapter II]{grubb68} (and there denoted by the densely defined $T:V\to W$ for $V=W=\overline{P_{\zeta,\partial}(\lambda)B}^{\checkH(D)}$ with $\dom(T)=P_{\zeta,\partial}(\lambda)B$).
\end{enumerate}
\end{theorem}

We remark that all three operators $Q_+$, $G$ and $\mathfrak{G}_{\rm B}$ depend on the choice of $\lambda$. However, $Q_+$ and $G$ are independent of the boundary condition $B$ (a priori there is an implicit dependence via the choice of $\lambda$, but it is irrelevant for the structural statements we are interested in).

\begin{proof}
It follows from \cite[Section 3.3]{gerdsgreenbook} that we, for $\lambda>0$ outside a discrete set of real numbers, can write
$$(D_{\rm Dir}+\lambda)^{-1}=Q_++G,$$
where $Q_+$ and $G$ are as in the statement of the theorem. The operator $D_{\rm B}$ is self-adjoint because $D$ is formally self-adjoint and $B$ is Lagrangian. Since $\dom(D_{\rm B})\subseteq \SobH{m}(M;E)$, the inclusion $\dom(D_{\rm B})\hookrightarrow \Lp{2}(M;E)$ is compact. Therefore $D_{\rm B}$ has only discrete spectrum. For $\lambda$ outside a discrete subset so that $D_{\rm B}+\lambda$ and $D_{\rm Dir}+\lambda$ are invertible, it follows from \cite[Chapter II, Theorem 1.4]{grubb68} that 
$$(D_{\rm B}+\lambda)^{-1}-(D_{\rm Dir}+\lambda)^{-1}=\mathfrak{G}_{\rm B},$$
where $\mathfrak{G}_{\rm B}$ is as in the statement of the theorem. This proves the theorem.
\end{proof}

\begin{lemma}
\label{lemmamamsoso}
The singular values of the operators appearing in Theorem \ref{resolventchar} have the following  asymptotics:
\begin{enumerate}[(i)]
\item $\mu_k(Q_+)=c_Qk^{-\frac{m}{n}}+o(k^{-\frac{m}{n}})$ where 
$$c_Q=\left(\frac{1}{n(2\pi)^n}\int_{S^*M} \mathrm{Tr}_E(\sigma_m(Q)(x,\xi)^{\frac{n}{m}})\mathrm{d}x\mathrm{d}\xi\right)^{\frac{m}{n}}.$$
\item $\mu_k(G)=O(k^{-\frac{m}{n-1}})$
\item $\mu_k(\mathfrak{G}_{\rm B})=O(k^{-\frac{m}{n-1}})$
\end{enumerate}
\end{lemma}

\begin{proof}
Item i) follows from \cite[Proposition 4.5.3]{gerdsgreenbook}. One can also prove item i) by hand by applying the Weyl law on a closed manifold, see e.g. \cite[Chapter XXIX]{horIV} or \cite[Chapter III]{shubinsbook}, to suitable perturbations of $G_+$. The statement of item ii) follows from \cite{grubb84}. The statement of item iii) follows from that $\mathfrak{G}_{\rm B}$ factors over the inclusion $\SobHH{{ m-{\frac12}}}(\Sigma;E\otimes \C^m)\hookrightarrow \SobHH{{ -{\frac12}}}(\Sigma;E\otimes \C^m)$ whose singular value sequence is $O(k^{-m/(n-1)})$ as $k\to \infty$. Indeed, using a Weyl law on the $n-1$-dimensional compact manifold $\Sigma$ shows that the singular values of $\SobH{s+m}(\Sigma;E)\hookrightarrow \SobH{s}(\Sigma;E)$ is $O(k^{-m/(n-1)})$ as $k\to \infty$ for any $s\in \R$.
\end{proof}

\begin{remark}
The statement that $\mu_k(Q_++G)=\mu_k(D_{\rm Dir}^{-1})=c_Qk^{-\frac{m}{n}}+o(k^{-\frac{m}{n}})$ is a direct consequence of the spectral asymptotics of the Dirichlet realisation, see \cite{agmonasy}. Compare to \cite{garding53} and \cite[Theorem 1]{grubb77a}. Therefore $\mu_k(Q_+)=c_Qk^{-\frac{m}{n}}+o(k^{-\frac{m}{n}})$ follows from that $\mu_k(G)=o(k^{-\frac{m}{n}})$ by \cite{grubb84} and the Weyl-Fan inequality (see for instance \cite[Theorem 1.7]{simontrace}). The proof above is included since it shows how the different components in the resolvent of $D_B$ contributes to the spectral asymptotics.
\end{remark}

\begin{theorem}
\label{generalwlelele}
Let $D=D^\dagger$ be formally selfadjoint elliptic operator of order $m$ acting on the sections of a hermitian vector bundle $E$ on an $n$-dimensional manifold with boundary. Assume that $D$ has positive interior principal symbol $\sigma_D$ and that $B$ is a regular and self-adjoint boundary condition defining a lower semi-bounded realisation $D_B$. Define the constant 
$$c_D=\left(\frac{1}{n(2\pi)^n}\int_{S^*M} \mathrm{Tr}_E(\sigma_D(x,\xi)^{-n/m})\mathrm{d}x\mathrm{d}\xi\right)^{-m/n}.$$ 
The operator $D_{\rm B}$ has discrete real spectrum, and its spectrum $\lambda_0(D_{\rm B})\leq \lambda_1(D_{\rm B})\leq \lambda_2(D_{\rm B})\leq\cdots$ satisfies that 
$$\lambda_k(D_{\rm B})=c_D k^{m/n}+o(k^{\frac{m}{n}}),\quad\mbox{as $k\to +\infty$}.$$
\end{theorem}

\begin{proof}
Using Theorem \ref{resolventchar}, Lemma \ref{lemmamamsoso} and the Weyl-Fan inequality (see for instance \cite[Theorem 1.7]{simontrace}) we conclude that 
$$\mu_k(D_{\rm B})=c_D k^{m/n}+o(k^{\frac{m}{n}}),\quad\mbox{as $k\to +\infty$}.$$
Here $(\mu_k(D_{\rm B}))_{k\in \mathbb{N}}$ denotes the singular value sequence of the unbounded operator $D_B$, so by definition $\mu_k(D_{\rm B})=\sqrt{\mu_k((1+D_{\rm B}^*D_{\rm B})^{-1})^{-1}-1}$. Since $D_B$ is self-adjoint and lower semi-bounded, $\mu_k(D_{\rm B})=\lambda_k(D_{\rm B})$ for $k$ large enough and the theorem follows. 
\end{proof}

\begin{remark}
The operator $D_B$ is lower semi-bounded if and only if the operator $\mathfrak{T}_{\rm B}$ from Theorem \ref{resolventchar}, item iii), is lower semi-bounded. For a certain class of elliptic boundary conditions, the operator $\mathfrak{T}_{\rm B}$ is realised as a pseudodifferential operator in \cite[Section 8]{grubb74} and in this case the asymptotics of the negative eigenvalues was studied.
\end{remark}

\begin{remark}
Theorem \ref{generalwlelele} was proven in \cite{agmonasy} under the additional hypothesis that there is a $k>n/m$ such that $D_B^k$ is regular, i.e. $\dom(D_B^k)\subseteq \SobH{km}(M;E)$. By complex interpolation, using that $D_B$ is self-adjoint, this condition implies $\dom(D_B^l)\subseteq \SobH{lm}(M;E)$ for $l=1,\ldots, k$. Note that $\dom(D_B^k)\subseteq \SobH{km}(M;E)$ if and only if $B$ is $m(k-1)$-semiregular. Let us provide a simple example where the assumption appearing in \cite{agmonasy} fails, but Theorem \ref{generalwlelele} holds. Our example is quite similar to Example \ref{lapapallald} and provided solely for the purpose of giving intuition for the difference in assumptions between Theorem \ref{generalwlelele} and \cite{agmonasy}.

Assume that $M\subseteq \R^n$ is a compact domain with smooth boundary and consider a shift of the standard Laplacian 
$$D=\Delta+1=1-\sum_{j=1}^n \partial_{x_j}^2.$$ 
In particular, both the Dirichlet problem and the Neumann problem for  $D$ are uniquely solvable. For a real-valued function $a\in L^\infty(\partial \Omega,\R)$, consider the Robin boundary condition
$$B_a:=\{(\xi_0,\xi_1)\in \checkH(D): a \xi_0+\xi_1=0\}.$$
A short computation with Proposition \ref{describinglocalcndldldaadjoint} (or in other words using Green's formula) shows that $B_a$ is self-adjoint. Since the operator $D_{B_a}$ is defined from the closed lower semibounded quadratic form
$$q_a(u):=\int_M (|u|^2+|\nabla u|^2)\mathrm{d} V(x)+\int_\Sigma a|u|^2\mathrm{d} S(x),\quad u\in \SobH{1}(M),$$
the operator $D_{B_a}$ is lower semi-bounded and so satisfies the assumptions of Theorem \ref{generalwlelele} as soon as $\dom(D_{B_a})\subseteq \SobH{m}(M;E)$, i.e. $B_a$ is semi-regular (or equivalently regular, by self-adjointness).

By the same argument as in Example \ref{lapapallald}, $B_a$ is $s$-semiregular if and only if for any $(\xi_0,\xi_1)\in B_a$, the property $(\Lambda_{DN} +a)\xi_0\in \SobH{s+1/2}(\Sigma)$ implies that $\xi_0\in \SobH{s+3/2}(\Sigma)$. If $a$ is such that the multiplication operator $a:\SobH{s+3/2}(\Sigma)\to \SobH{s+1/2}(\Sigma)$ is compact, then for a generic\footnote{In the sense that for $\epsilon\in (-1,1)$ outside of a discrete set of points $\Lambda_{DN} +a+\epsilon:\SobH{s+3/2}(\Sigma)\to \SobH{s+1/2}(\Sigma)$ is invertible.} such $a$, the boundary condition $B_a$ is $s$-semiregular. Conversely, if $B_a$ is $s$-semi-regular, then elliptic regularity of $\Lambda_{DN}$ and the fact that $(\Lambda_{DN} +a)\xi_0\in \SobH{s+1/2}(\Sigma)\; \Rightarrow\xi_0\in \SobH{s+3/2}(\Sigma)$ implies that $a:\SobH{s+3/2}(\Sigma)\to \SobH{s+1/2}(\Sigma)$ is continuous. 

Let us now give a condition on $a$ such that $\dom(D_{B_a}^k)\subseteq \SobH{km}(M;E)$ holds for $k=1$ but not for $k=2$, and therefore for no $k>1$. 
We note that $u\in \dom(D_{B_a}^2)$ if and only if $D_{B_a}u\in\dom(D_{B_a})$. Therefore, the inclusion $\dom(D_{B_a}^2)\subseteq \SobH{4}(M)$ holds if and only if for $u\in \dom(D_{B_a})$ such that $D_{B_a}u\in \SobH{2}(M)$ it holds that $u\in \SobH{4}(M)$. From Theorem \ref{Thm:BdyRegHigh} we see that $\dom(D_{B_a}^2)\subseteq \SobH{4}(M)$ if and only if $B_a$ is $2$-semiregular. By the argument above, this occurs if the function  $a$ makes the multiplication operator $a:\SobH{3/2}(\Sigma)\to \SobH{1/2}(\Sigma)$ compact but the multiplication operator $a:\SobH{7/2}(\Sigma)\to \SobH{5/2}(\Sigma)$ is not even continuous. This occurs when $a\in \Ck{1}(\Sigma)$ but $a\notin \SobH{5/2}(\Sigma)$. Indeed, continuity of $a:\SobH{7/2}(\Sigma)\to \SobH{5/2}(\Sigma)$ fails if $a=a1\notin \SobH{5/2}(\Sigma)$ and $\Ck{1}(\Sigma)$ acts continuously $\SobH{1/2}(\Sigma)$ (by interpolation of the continuous actions $\SobH{1}(\Sigma)$ and $\Lp{2}(\Sigma)$) and $\SobH{3/2}(\Sigma)\to \SobH{1/2}(\Sigma)$ is compact by Rellich's theorem. For an example of an $a\in \Ck{1}(\Sigma)$ with $a\notin \SobH{5/2}(\Sigma)$, take $M$ to be the unit disc, so $\Sigma=S^1$, and define $a$ from the lacunary Fourier series
\begin{equation}
\label{laclausnsd}
a(z):=\sum_{k=0}^\infty 2^{-2k}(z^{2^k}+z^{-2^k})\in \Ck{1}(S^1)\setminus \SobH{5/2}(S^1).
\end{equation}
For $a$ as in Equation \eqref{laclausnsd}, the regular realisation $(1+\Delta)_{B_a}$ on the unit disc provides a concrete example where the Weyl law holds by Theorem \ref{generalwlelele} but the assumption $\dom(D_B^k)\subseteq \SobH{km}(M;E)$ for some $k>n/m$ from \cite{agmonasy} fails.
\end{remark}

\begin{corollary}
\label{weylformodulelososo}
Let $D:\Ck{\infty}(M;E)\to \Ck{\infty}(M;F)$ be an elliptic differential operator of order $m>0$ and $B\subseteq \checkH(D)$ a regular and $m$-semiregular boundary condition. Then the sequence of singular values of the operator $D_{\rm B}:\Lp{2}(M;E)\dashrightarrow \Lp{2}(M;F)$ satisfies that 
$$\mu_k(D_{\rm B})=\sqrt{c_{D^\dagger D}} k^{\frac{m}{n}}+o(k^{\frac{m}{n}}),\quad\mbox{as $k\to +\infty$},$$
where  
$$c_{D^\dagger D}=\left(\frac{1}{n(2\pi)^n}\int_{S^*M} \mathrm{Tr}_E((\sigma_D(x,\xi)^*\sigma_D(x,\xi))^{-n/2m})\mathrm{d}x\mathrm{d}\xi\right)^{-2m/n},$$
and $\sigma_D$ denotes the interior principal symbol of $D$. 
\end{corollary}

\begin{proof}
We have that $\mu_k(D_{\rm B})^2=\lambda_k(D_{\rm B}^*D_{\rm B})$ and by Corollary \ref{highregformodulelososo}, $D_{\rm B}^*D_{\rm B}$ is regular if $B$ is regular and $m$-semiregular. The corollary follows from Theorem \ref{generalwlelele} since $D_{\rm B}^*D_{\rm B}$ is non-negative.
\end{proof}

\begin{remark}
A consequence of Corollary \ref{weylformodulelososo} is that if $D$ is a formally self-adjoint, elliptic differential operator of order $m>0$ and $B\subseteq \checkH(D)$ is a Lagrangian, regular and $m$-semiregular boundary condition, then $D_{\rm B}$ has discrete spectrum consisting of real eigenvalues that asymptotically behaves as  
$$|\lambda_k(D_{\rm B})|=c_{|D|} k^{\frac{m}{n}}+o(k^{\frac{m}{n}}),\quad\mbox{as $k\to +\infty$},$$
where  
$$c_{|D|}=\left(\frac{1}{n(2\pi)^n}\int_{S^*M} \mathrm{Tr}_E(|\sigma_D(x,\xi)|^{-n/m})\mathrm{d}x\mathrm{d}\xi\right)^{-m/n}.$$
For more details on the spectral asymptotics of the sequences of negative and positive eigenvalues, see \cite[Chapter 4.5]{gerdsgreenbook}. 
\end{remark}

\begin{remark}
The assumption that $B$ is regular in Theorem \ref{generalwlelele} can in general not be dropped. To give an example thereof, we assume for simplicity that $D$ is as in Theorem \ref{generalwlelele} and $D_{\rm Dir}$ is invertible. Following the notations of \cite{grubb68,grubb74}, for a small $m'<m$ we take a regular $T_0\in \Psi^{\pmb{m'}}(\Sigma;E\otimes \C^m)$ with $[T_0,P_\mathcal{C}]\in \Psi^{\pmb{m'-m}}(\Sigma;E\otimes \C^m)$  and define $T:=P_\mathcal{C}T_0P_\mathcal{C}$. The operator $T$ is a densely defined Fredholm operator $\mathcal{C}_D\to \mathcal{C}_D$ with $\dom(T)=P_\mathcal{C}\SobHH{{m'-{\frac12}}}(\Sigma;E\otimes \C^m)$. Assume for simplicity that $T$ is invertible. Let $D_T$ be the realisation of $D$ constructed as in \cite[Chapter II, Proposition 1.2]{grubb68}. The realisation $D_T$ is not regular, but $\gamma(\dom(D_T))\subseteq \SobHH{{m'-{\frac12}}}(\Sigma;E\otimes \C^m)$, cf. \cite[Corollary 6.10]{grubb74}. By \cite[Chapter II, Theorem 1.4]{grubb68}, we have that 
$$D_T^{-1}-D_{\rm Dir}=T^{-1},$$
where $T^{-1}$ is ``extended by zero'' to an operator on $\Lp{2}(M;E)=\mathcal{C}_D\oplus \ran(D_{\rm min})$. If $m'<m(n-1)/n$, the proof of Theorem \ref{generalwlelele} breaks down as it only produces the lower estimate 
$$\mu_k(D_T)\gtrsim k^{\frac{m'}{n-1}},\quad\mbox{as $k\to \infty$}.$$
\end{remark}

\section{Rigidity results for the index of elliptic differential operators}
\label{sec:rigiddldlda}

In this section we prove a rigidity result for the Fredholm index of Fredholm realisations $D_{\rm B}$. As noted in Remark \ref{hominvariancremakd}, the index formula of Theorem \ref{Thm:FredChar} implies homotopy invariance for the choice of Fredholm boundary condition when $D$ is fixed. We now extend this observation to more general homotopies. 

We topologise the space of pseudodifferential operators by declaring the symbol mappings and its local splittings defined from Kohn-Nirenberg quantisation, as well as the inclusion of the smoothing operators, to be continuous. The space of order $m$ differential operators $\Ck{\infty}(M;E)\to \Ck{\infty}(M;F)$ will be topologised as linear operators between the Fréchet topologies of $\Ck{\infty}(M;E)$ and $\Ck{\infty}(M;F)$. Equivalently, we can topologise the space of order $m$ differential operators $\Ck{\infty}(M;E)\to \Ck{\infty}(M;F)$ in terms of the $\Ck{\infty}$-topology on the coefficients of the operator.

\begin{theorem}
\label{rigidityformaula}
Assume that $(D_t)_{t\in [0,1]}$ is a family of elliptic differential operators of order $m>0$ such that 
\begin{enumerate}[(i)]
\item \label{rigid:1}
The family is continuous in the topology on the space of order $m$ differential operators $\Ck{\infty}(M;E)\to \Ck{\infty}(M;F)$.
\item \label{rigid:2} 
$\dom (D_{t,{\rm max}})$ is independent of $t$.
\end{enumerate}
Let $[0,1]\ni t\mapsto P_t\in \mathbb{B}(\checkH(D))$ be a norm continuous family of projectors such that $B_t:=(1-P_t)\checkH(D)$ satisfies that $(B_t,\Ca_{D_t})$ is a Fredholm pair in $\checkH(D)$ for all $t\in [0,1]$. Then it holds that 
$$\indx(D_{0,B_0})=\indx(D_{1,B_1}).$$
\end{theorem}

In order to prove this theorem, we first prove the following two lemmas.  

\begin{lemma}
\label{continingaoaod}
\label{rigidlem1} 
Assume that $(D_t)_{t\in [0,1]}$ is a family of elliptic differential operators of order $m>0$ as in Theorem \ref{rigidityformaula}. Then viewing $D_{t,\max}$ as a continuous linear mapping $\dom(D_{0,\max})\to \Lp{2}(M;F)$ it holds that $t\mapsto D_{t,\max}$ defines a norm-continuous mapping $[0,1]\to \mathbb{B}(\dom(D_{0,\max}),\Lp{2}(M;F))$.
\end{lemma}

\begin{proof}
The proof of this lemma is slightly technical and relies on the construction of approximate Calderón projectors following Theorem \ref{horsats} and \ref{refinedseeley}. The family $(D_t)_{t\in [0,1]}$ extends to a continuous family of elliptic differential operators on a  closed manifold containing $M$ as a domain, this is as in the proof of Theorem \ref{refinedseeley} above. We follow the notation in the proof of Theorem \ref{refinedseeley}. Since the family $(D_t)_{t\in [0,1]}$ is continuous in the topology on the space of order $m$ differential operators, we can construct a continuous family $(\hat{T}_t)_{t\in [0,1]}\subseteq \Psi^{-m}_{\rm cl}(M;F,E)$ such that $\hat{T}_t$ is a parametrix of $\hat{D}_t$ for all $t\in [0,1]$. Following Equation \eqref{formulaforaprprood}, we define an approximate Poisson operator 
$$
\tilde{\mathcal{K}}_t\left(\xi_j\right)_{j=0}^{m-1}=\sum_{j=0}^{m-1}\sum_{l=0}^{m-1-j}(-i)^{l+1}\hat{T}_t[A_{m-l-j-1}\xi_j\otimes \delta_{t=0}^{(l)}].
$$
Using the same arguments as in \cite{seeley65}, it holds that $\tilde{\mathcal{K}}_t:\checkH(D)\to \dom(D_{t,\rm max})$ continuously for all $t\in [0,1]$. Indeed, letting $\mathcal{K}_{D_t}$ denote the Poisson operator of $D_t$ (see \cite{seeley65} or Theorem \ref{refinedseeley}) above, we have that $\tilde{\mathcal{K}}_t-\mathcal{K}_{D_t}:\checkH(D)\to \Ck{\infty}(M;E)$. We therefore also have that $\gamma\circ\tilde{\mathcal{K}}_t- P_{\mathcal{C}_t}\in \Psi^{-\infty}(\Sigma;E\otimes \C^m)$ is a smoothing operator, and in particular $\gamma\circ \tilde{\mathcal{K}}_t$ is an approximate Calderón projection as in Theorem \ref{horsats}.

Consider the family of bounded operators $R_t:=\tilde{\mathcal{K}}_t\circ \gamma:\dom(D_{\max})\to \dom(D_{\max})$. We note that for $f\in C^\infty(M;E)$,
$$D_tR_tf=\sum_{j=0}^{m-1}\sum_{l=0}^{m-1-j}(-i)^{l+1}(D_t\hat{T}_t-1)[A_{m-l-j-1}\gamma_j(f)\otimes \delta_{t=0}^{(l)}],$$
because distributions supported on $\Sigma$ vanish when restricted to $M$. However, since $\hat{T}_t$ is a continuous family of parametrices, the mapping $[0,1]\ni t\mapsto (\hat{D}_t\hat{T}_t-1)\in \Psi^{-\infty}(\hat{M};\hat{F})$ is continuous. We conclude that $t\mapsto D_tR_t$ factors over $\scalebox{1.5}{a}\gamma:\dom(D_{\max})\to \hatH(D^\dagger)$, the continuous inclusion $\hatH(D^\dagger)\hookrightarrow \mathcal{D}'(\Sigma;F\otimes \C^m)$ and a continuous family of operators $\mathcal{D}'(\Sigma;F\otimes \C^m)\to \Ck{\infty}(M;F)$. As such, $[0,1]\ni t\mapsto D_tR_t\in \mathbb{B}(\dom(D_{\max}), \SobH{s}(M;F))$ is norm-continuous for all $s\in \R$, and in particular $[0,1]\ni t\mapsto D_tR_t\in \mathbb{B}(\dom(D_{\max}), \Lp{2}(M;F))$ is norm-continuous. 

Since $\tilde{\mathcal{K}}_t-\mathcal{K}_{D_t}:\checkH(D)\to \Ck{\infty}(M;E)$, we have that $(1-R_t):\dom(D_{\rm max})\to \SobH{m}(M;E)$. Indeed, we can write $\dom(D_{\rm max})=\SobH{m}(M;E)+\ker(D_{t,\max})$ by Example \ref{funnyboudnarydadreturn} and $1-R_t$ is smoothing on $\ker(D_{t,\max})$ because of $\tilde{\mathcal{K}}_t-\mathcal{K}_{D_t}:\checkH(D)\to \Ck{\infty}(M;E)$. Using the semi-classical nature of the arguments in \cite{seeley65}, it follows that $[0,1]\ni t\mapsto 1-R_t\in \mathbb{B}(\dom(D_{\rm max}), \SobH{m}(M;E))$ is norm-continuous. It is clear that $[0,1]\ni t\mapsto D_t\in \mathbb{B}(\SobH{m}(M;E),\Lp{2}(M,E))$ is continuous. Therefore $[0,1]\ni t\mapsto D_t(1-R_t)\in \mathbb{B}(\dom(D_{\rm max}),\Lp{2}(M,E))$. 

Since the map $t\mapsto D_t=D_tR_t+D_t(1-R_t)$ is a sum of continuous maps $[0,1]\to \mathbb{B}(\dom(D_{\rm max}),\Lp{2}(M,E))$, the map $t\mapsto D_{t,\max}$ defines a norm-continuous mapping $[0,1]\to \mathbb{B}(\dom(D_{0,\max}),\Lp{2}(M;F))$.
\end{proof}

\begin{lemma}
\label{rigidlem2} 
Assume that $(D_t)_{t\in [0,1]}$ is a family of elliptic differential operators of order $m>0$ as in Theorem \ref{rigidityformaula} and define the family of self-adjoint operators
$$\tilde{D}_t:=\begin{pmatrix}
0& D_{t,{\rm min}}\\
D^\dagger_{t,{\rm max}}& 0
\end{pmatrix},$$
for $t\in [0,1]$. Then for each $\lambda\in \R\setminus\{0\}$, the mapping $t\mapsto (\lambda i+\tilde{D}_t)^{-1}$ is a norm-continuous mapping $[0,1]\to \mathbb{B}(\Lp{2}(M;F\oplus E))$.
\end{lemma}

\begin{proof}
By the assumptions of Theorem \ref{rigidityformaula}, the domain of $\tilde{D}_t$ is constant in $t$. We have for $t_1,t_2\in [0,1]$ that 
$$(\lambda i+\tilde{D}_{t_1})^{-1}-(\lambda i+\tilde{D}_{t_2})^{-1}=(\lambda i+\tilde{D}_{t_1})^{-1}(\tilde{D}_{t_2}-\tilde{D}_{t_1})(\lambda i+\tilde{D}_{t_2})^{-1}$$
Since $(\lambda i+\tilde{D}_{t})^{-1}:\Lp{2}(M;F\oplus E)\to \dom(\tilde{D}_{t})=\dom(D_{t,\min })\oplus \dom(D_{t,\max}^\dagger)$ with the uniform norm bound $|\lambda|^{-1}$ as $t$ varies, the lemma follows from Lemma \ref{continingaoaod} (applied to $(D_t^\dagger)_{t\in [0,1]}$) and the fact that $t\mapsto D_{t,\min}$ clearly defines a norm-continuous mapping $[0,1]\to \mathbb{B}(\SobH[0]{m}(M;E),\Lp{2}(M;F))$.
\end{proof}

\begin{proof}[Proof of Theorem \ref{rigidityformaula}]
The ellipticity of each $D_t$ ensures that $\dom(D_{t,\min})=\SobH[0]{m}(M;E)$ is constant in $t$. We remark that item \ref{rigid:2} ensures that $\checkH(D_t)\cong \dom(D_{t,\max})/\dom(D_{t,\min})=\dom(D_{t,\max})/\SobH[0]{m}(M;E)$ is independent of $t$, justifying the notation $\checkH(D)$ for the Cauchy data space. 

The only remaining property to verify in order to invoke Theorem \ref{rigidityformaulaABS} below is that the mappings 
$$[0,1]\ni t\mapsto 
\begin{pmatrix}
0& D_{t,{\rm min}}\\
D^\dagger_{t,{\rm max}}& 0
\end{pmatrix}
\quad\mbox{and}\quad [0,1]\ni t\mapsto 
\begin{pmatrix}
0& D^\dagger_{t,{\rm min}}\\
D_{t,{\rm max}}& 0
\end{pmatrix},$$
are continuous in the gap topology. 
Lemma \ref{rigidlem2} applied to the families $(D_t)_{t\in [0,1]}$ and $(D_t^\dagger)_{t\in [0,1]}$ furnishes us with this fact. 
\end{proof}

\begin{corollary}
Let $D$ be an elliptic differential operator and $B$ a regular boundary condition. Assume that $B$ is graphically decomposed with respect to the Calderon projector by $(V_+,W_+,V_-,W_-,g)$ as in Definition \ref{def:elldecomspsHO}. Then it holds that 
$$\indx(D_B)=\dim(W_+)-\dim(W_-)+\dim(\ker(D_{\min})-\dim(\ker(D_{\min}^\dagger).$$
\end{corollary}

\begin{proof}
We define the continuous path of regular boundary conditions $(B_t)_{t\in [0,1]}$ from the continuous path of data for a graphical decomposition $(V_+,W_+,V_-,W_-,tg)_{t\in [0,1]}$. Note that 
\begin{equation}
\label{bzeroandzxolaad}
B_0=W_+\oplus V_-\quad\mbox{and} \quad B_0^*=\scalebox{1.5}{a}_\dagger^{-1}(\adj{V}_+\cap \SobHH{\frac12}(\Sigma;E\otimes \C^m_{\rm op}) \oplus \adj{W}_-).
\end{equation}
We also note that $V_+\oplus W_+=\mathcal{C}_D\cap \SobHH{m-1/2}(\Sigma;E\otimes \C^m)$ by the definition of a graphical decomposition with respect to the Calderon projection. By Theorem \ref{rigidityformaula}, we have that 
$$\indx(D_B)=\indx(D_{B_0})=\indx(D_{B_1}).$$
The proof now follows from Theorem \ref{Thm:FredChar} upon proving that $\indx(B_0,\mathcal{C}_D)=\dim(W_+)-\dim(W_-)$. This follows from the fact that $B_0\cap \mathcal{C}_D=W_+$ and $B_0^*\cap \mathcal{C}_{D^\dagger}=\scalebox{1.5}{a}_\dagger^{-1}(\adj{W}_-)$ by Equation \eqref{bzeroandzxolaad}.
\end{proof}

\begin{theorem}
\label{cortorigidityformaula}
Assume that $(D_t)_{t\in [0,1]}$ is a family of elliptic differential operators of order $m>0$ which is continuous in the topology on the space of order $m$ differential operators $\Ck{\infty}(M;E)\to \Ck{\infty}(M;F)$. Let $(P_t)_{t\in [0,1]}\subseteq \mathbb{B}(\SobHH{{m-{\frac12}}}(\Sigma;E\otimes \C^m))$ be a norm continuous family of projectors such that 
\begin{enumerate}
\item For all $t\in [0,1]$, $P_t$ extends to a bounded operator on $\SobHH{{-{\frac12}}}(\Sigma;E\otimes \C^m)$.
\item For all $t\in [0,1]$, $\alpha\in \{-1/2,m-1/2\}$ the operator
\begin{equation} 
\label{cortori:1}
P_{\mathcal{C}_t}-(1-P_t):\SobHH{\alpha}(\Sigma;E\otimes \C^m)\to \SobHH{\alpha}(\Sigma;E\otimes \C^m),
\end{equation}
is a Fredholm operator, where $P_{\mathcal{C}_t}$ denotes the Calderon projector for $D_t$. 
\end{enumerate} 
Then for all $t\in [0,1]$, $B_t:=(1-P_t)\SobHH{{m-{\frac12}}}(\Sigma;E\otimes \C^m)$ is a regular boundary condition for $D_t$ and the index $\indx(D_{t,B_t})$ is independent of $t$. In particular,
$$\indx(D_{0,B_0})=\indx(D_{1,B_1}).$$
\end{theorem}

The proof goes along similar lines as that of Theorem \ref{rigidityformaulaABS} but adapted to regular realisations with domains contained in Sobolev spaces. We note that Theorem \ref{cortorigidityformaula} in particular holds for a continuous family $(P_t)_{t\in [0,1]}\subseteq \Psi^{\pmb 0}_{\rm cl}(\Sigma;E\otimes \C^m)$ of pseudodifferential projectors such that for all $t\in [0,1]$ 
$$p_+(D_t)-(1-\sigma_{\pmb 0}(P_t))\in C^\infty(S^*\Sigma;\mathrm{Aut}(\pi^*E)).$$

\begin{proof}
We note that each $B_t$ is a regular boundary condition by Theorem \ref{thm:sufficienfeidntofell}. Since $(D_t)_{t\in [0,1]}$ is a continuous family of elliptic differential operators, it defines a bounded adjointable map of $C[0,1]$-Hilbert $C^*$-modules
$$\mathbb{D}:\SobH{m}(M;E)\otimes C[0,1]\to \Lp{2}(M;F)\otimes C[0,1].$$
We write $\check{C}_t:=P_t\SobHH{{m-{\frac12}}}(\Sigma;E\otimes \C^m)$ and let $\check{C}_{C([0,1])}$ denote the $C[0,1]$-Hilbert $C^*$-submodule of $\SobHH{{m-{\frac12}}}(\Sigma;E\otimes \C^m)\otimes C([0,1])$ with fibre $\check{C}_t$, it is well-defined by norm continuity of the family $(P_t)_{t\in [0,1]}$. We write $P\gamma:\SobH{m}(M;E)\otimes C([0,1])\to \check{C}_{C([0,1])}$ for the bounded adjointable mapping obtained from composing the trace mapping with the family $(P_t)_{t\in [0,1]}$.

Following Proposition \ref{Prop:CompExistsABS} and Theorem \ref{rigidityformaulaABS}, we define the operator 
$$\check{L}: \SobH{m}(M;E)\otimes C[0,1] \to \begin{matrix}\Lp{2}(M;F)\otimes C[0,1]\\ \oplus\\ \check{C}\end{matrix}, \quad \check{L}u = \begin{pmatrix} \mathbb{D} u \\ P\gamma u\end{pmatrix}$$
This operator is an adjointable $C[0,1]$-linear map because $\mathbb{D}$ and $P\gamma$ are adjointable. We note that $\check{L}$ if Fredholm in each fibre by Proposition \ref{Prop:CompExistsABS} (cf. \cite[Proposition A.1]{BB}). By an argument along the same lines as in Theorem \ref{rigidityformaulaABS}, the adjointable $C[0,1]$-linear map $\check{L}:\SobH{m}(M;E)\otimes C[0,1] \to \Lp{2}(M;F)\otimes C[0,1] \oplus \check{C}$ is Fredholm with a well defined index $\indx(\check{L})\in K^0([0,1])$. Again arguing as in Theorem \ref{rigidityformaulaABS}, homotopy invariance of $K$-theory implies that
$$\indx(D_{t,B_t})=\indx(\check{L}_t)=r_t^*\indx(\check{L}),$$
where $r_t$ denotes the inclusion $\{t\}\subseteq [0,1]$, is independent of $t$.
\end{proof}

We end this section with an outlook towards index theory. It poses an interesting problem to compute $\indx(D_{{\rm B}_P})$ where $D$ is an elliptic differential operator $D$ of order $m>0$ and $B_P:=(1-P)\SobHH{ {m-{\frac12}}}(\Sigma;E\otimes \C^m)$ is a regular pseudo-local boundary condition on $D$ defined from $P\in \Psi^{\pmb 0}_{\rm cl}(\Sigma;E\otimes \C^m)$. While Theorem \ref{Thm:FredChar} and Corollary \ref{simplexindex} give formulas, a more worthwhile goal would be an explicit formula for the index of $D_{\rm B}$ in terms of the interior principal symbol of $D$, the germ of $D$ on $\Sigma$, the differential operator $\scalebox{1.5}{a}$ defining the boundary pairing and the projection $P$. The Atiyah-Patodi-Singer index theorem and the index theorem in the extended Boutet de Monvel calculus of \cite{schulzeseiler} indicates that such an explicit computation is possible. Under some mild assumptions on $P$, the special case of first order elliptic operators is susceptible to heat kernel techniques by the results of \cite{G99}.

There is an interesting special case that the authors of this paper have not found an explicit solution to in the literature. A classical situation is that $\slashed{D}$ is a Dirac operator and $P$ is a spectral projection of the boundary operator of $\slashed{D}$ constructed from a normal vector field to the boundary. Atiyah-Patodi-Singer \cite{APS} computed the index in the case of product structure at the boundary and \cite{grubb92} considered the general case. For more details, see also \cite{bosswojc,melroseAPS}. Assume that we leave this classical realm and instead use a general vector field transversal to the boundary to define the adapted boundary operator $A$. Then $A$ is not necessarily self-adjoint, but only bisectorial (cf. \cite{BBan}). Is there an explicit formula for the index $\indx(D_{{\rm B}_P})$ for $P=\chi^+(A)$? Preferably, such a formula only uses local data and spectral data of the bisectorial operator $A$. The results of \cite{G99} indicates that there might be a positive answer. The importance of this question is seen in that similar questions can be asked on manifolds with Lipschitz boundary, where a smooth vector field transversal to the boundary (almost everywhere) makes sense.

\section{Symbol computations with Calderón projectors in the first order case}
\label{subsec:symbcominorderone}

In this section, we study the way in which the Calderón projection of a first order elliptic operator relates to the the positive spectral projection of an adapted boundary operator. We show that the difference of these projectors are an operator of order $-1$ and continuous, but in general not compact on the Cauchy data space.

Near the boundary, 
\begin{equation}
\label{decomposlna}
D=\sigma \partial_{x_n}+A_1,
\end{equation}
where $A_1=(A_1(x_n))_{x_n\in [0,1]}$ is a family of first order differential operators acting from $E|_{\Sigma}\to \Sigma$ to $F|_{\Sigma}\to \Sigma$. By ellipticity of $D$, $\sigma$ is a bundle isomorphism and $A_1$ is a family of elliptic differential operators. We note that an adapted boundary operator for $D$ is precisely a differential operator $A$ on $E|_{\Sigma}$ such that $A-\sigma^{-1}A_1|_{x_n=0}$ is of order zero. 

Recall the construction of $E_+(D)$ and $p_+(D)$ from the subsection \ref{subsec:calderonproj}.

\begin{lemma}
\label{specprojvscalderon}
Let $D$ be a first order elliptic differential operator and $A$ an adapted boundary operator on $E|_{\Sigma}$. Then $\chi^+(A)$ is a classical order zero pseudo-differential operator whose principal symbol is $p_+(D)$. 
\end{lemma}

\begin{proof}
Since $\chi^+(A)$ is constructed as a sectorial projection, it is a classical order zero pseudo-differential operator, see \cite{grubbsec}. Its principal symbol is the same as $\chi^+(a)$ where $a$ is the principal symbol of $A$. By construction, $a=\sigma^{-1}a_1$ where $a_1$ is the principal symbol of $A_1|_{x_n=0}$ appearing in the collar decomposition $D=\sigma \partial_{x_n}+A_1$. We conclude that $\sigma_\partial(D)=\sigma \partial_{x_n}+a_1=\sigma (\partial_{x_n}+a)$. It is clear that for $(x',\xi')\in S^*\Sigma$ and $v_0\in E_{x}$, Equation \eqref{boundaryeq}, that in this case takes the form
$$
\begin{cases}
\sigma (\partial_{t}+a(x',\xi'))v(t)=0,\; t>0,\\
v(0)=v_0,\end{cases}
$$
has an exponentially decaying solution if and only if $v_0\in \chi^+(a(x',\xi'))E_{x}$, namely $v(t)=\mathrm{e}^{-a(x',\xi')t}v_0$. Similarly, there is an exponentially increasing solution if and only if $v_0\in \chi^-(a(x',\xi'))E_{x}$. So $E_+(D)=\chi^+(a)E$ and $E_-(D)=\chi^-(a)E$ and we can conclude that the principal symbol of $\chi^+(A)$ equals $p_+(D)$. 
\end{proof}

\begin{remark}
By Proposition \ref{nontrivialityofeplussd}, we know that $E_+(D)\neq 0$ and $E_-(D)\neq 0$. In fact, each fibre is non-zero if the dimension exceeds $1$. For first order operators, this can be proven more directly using the symmetry $a(x',\xi')=-a(x',-\xi')$ which implies that $p_+(D)(x',\xi')=1-p_+(D)(x',-\xi')$ so $p_+(D)\neq 0$ and a similar argument implies $p_-(D)\neq 0$.
\end{remark}

\begin{lemma}
Let $D$ be a Dirac type operator and $A$ an adapted boundary operator on $E|_{\Sigma}$ such that 
$$D=\sigma (\partial_{x_n}+A),$$
near the boundary. Then $\chi^+(A)-P_{\Ca_D}\in \Psi^{-\infty}(\Sigma;E)$ is a smoothing operator.
\end{lemma}

\begin{proof}
It suffices to prove that $\chi^+(A)-Q\in \Psi^{-\infty}(\Sigma;E)$ is a smoothing operator for $Q$ being the approximate Calderón projection from Theorem \ref{horsats}. Since the construction in Theorem \ref{horsats} is microlocal at the boundary we can assume that $M=\Sigma\times [0,\infty)$. In this case, any $f\in \ker D_{\rm max}$ can be written in the form 
$$f(t)=\mathrm{e}^{-tA}f(0).$$
As such, $\mathcal{K}g(t)=\mathrm{e}^{-tA}\chi^+(A)g$ and the Calderón projection $P_{\Ca_D}=\chi^+(A)$ is an approximate Calderón projection. 
\end{proof}

\begin{lemma}
\label{caldeorem}
Let $D$ be a first order elliptic differential operator with adapted boundary operator $A$. The first term in the Calderón projection is given by
\begin{align*}
\sigma_0(P_{\Ca_D})=&p_+(D)=\chi^+(\sigma_1(A)).
\end{align*}
In coordinates near a point $x_0\in \Sigma$, we have that
\begin{align*}
\sigma_{-1}(P_{\Ca_D})(x',\xi')=&\sum_{z: \mathrm{Re}(z)>0}\mathrm{Res}_{z=\lambda} t_1(x',\xi',i\lambda),
\end{align*}
where 
\small
\begin{align*}
t_{1}(x',\xi',i\lambda)=&-(\lambda-c_A(x',\xi'))^{-1}(c_{A,0}(x')+c_{R_0,0}(x',0))(\lambda-c_A(x',\xi'))^{-1}-\\
&-(\lambda-c_A(x',\xi'))^{-1}c_A\left(x',(\lambda-c_A(x',\xi'))^{-1}\mathrm{d}_{x'}c_A(x',\xi'))(\lambda-c_A(x',\xi'))^{-1}\right)-\\
&-(\lambda-c_A(x',\xi'))^{-2}\frac{\partial c_{R_0}}{\partial x_n}(x',0,\xi')(\lambda-c_A(x',\xi'))^{-1}.
\end{align*}
\normalsize
Here, we have written $D=\sigma (\partial_{x_n}+A+R_0)$ for an $x_n$-dependent first order differential operator $R_0$ on $\Sigma$, $c_A$ denotes the principal symbol of $A$, $c_{A,0}$ the sub-leading term (and similarly for $A$), and where we use the notation $c_A(x',\nu\otimes M):=c_A(x',\nu) M$ for an endomorphism valued covector $\nu\otimes M$.
\end{lemma}

\begin{proof}
By the formula for the full symbol of the approximate Calderón projection from Theorem \ref{horsats}, we have that
$$Qf=\gamma_0\circ T_0[\sigma f\otimes \delta_{t=0}],$$
where $T_0$ is a parametrix for $D$. In particular, we note that 
$$Qf=\gamma_0\circ T[ f\otimes \delta_{t=0}],$$
where $T$ is a parametrix for $\sigma^{-1}D$. The operator $T$ is a classical pseudodifferential operator of order $-1$ in a neighbourhood of $M$ in an extended manifold. In coordinates near a point $x_0\in \Sigma$, a full symbol $t$ of $T$ has an asymptotic expansion $t\sim \sum_{j=0}^\infty t_j$ and each $t_j=t_j(x,\xi)$ is smooth in $x$ and homogeneous of degree $-1-j$ in $\xi$. Using the methodology of \cite[Chapter XX]{horIII}, cf. Remark \ref{remarkonredisuydd}, the full symbol of $Q$ can be computed by 
$$q_j(x',\xi')=\sum_{z: \mathrm{Im}(z)>0} \mathrm{Res}_{\xi_n=z} t_j(x',0,\xi',\xi_n),$$
where $(x',\xi')$ denote the associated coordinates on $\Sigma$. As such, we shall need to compute the two terms in the asymptotic expansion of $T$. By an abuse of notation, we will write $t_j(x',\xi',\xi_n)$ for $t_j(x',0,\xi',\xi_n)$.

To compute the asymptotic expansion of $T$, we write the full symbol of $D$ as $c_D+c_{D,0}$ where $c_D$ is the principal symbol (i.e. a homogeneous degree 1 polynomial) and $c_{D,0}$ is $\xi$-independent. We can write $D=\sigma (D_t+A+R_0)$ for an $x_n$-dependent first order differential operator $R_0$ on $\Sigma$ with $R_0|_{x_n=0}$ being a multiplication operator. We use similar notation for $A$ and $R_0$ and note that 
$$\begin{cases}
c_D(x',x_n,\xi',\xi_n)&=\sigma(x',x_n) (i\xi_n+c_A(x',\xi')+c_{R_0}(x',x_n,\xi')),\;\mbox{and}\\
c_{D,0}(x',x_n)&=\sigma(x',x_n) (c_{A,0}(x')+c_{R_0,0}(x',x_n)).
\end{cases}$$ 
The assumption that $A$ is adapted ensures that $c_{R_0}(x',0,\xi')=0$. To compute the parametrix of $\sigma^{-1}D$, we use the symbols
$$\begin{cases}
\tilde{c}_D(x',x_n,\xi',\xi_n)&:=\sigma(x',x_n)^{-1}c_D(x',x_n,\xi',\xi_n)= i\xi_n+c_A(x',\xi')+c_{R_0}(x',x_n,\xi',\xi_n),\\
\tilde{c}_{D,0}(x',x_n)&:=\sigma(x',x_n)^{-1}c_{D,0}(x',x_n)=c_{A,0}(x')+c_{R_0,0}(x',x_n).
\end{cases}$$ 
From the well known parametrix construction, the full symbol $t\sim \sum_{j=0}^\infty t_j$ of $T$ is computed inductively from that $t_0(x,\xi)=c_D(x,\xi)^{-1}$ and 
\begin{align*}
t_{j+1}=&-t_0\tilde{c}_{D,0} t_j-t_0\sum_{k+|\alpha|=j+1, \, k\leq j}\frac{1}{\alpha!} \partial_\xi^\alpha \tilde{c}_{D}D^\alpha_x t_k.
\end{align*}
In particular, 
\begin{align*}
t_{1}=&-t_0\tilde{c}_{D,0} t_0-t_0\sum_{|\alpha|=1}\partial_\xi^\alpha \tilde{c}_{D}D^\alpha_x t_0.
\end{align*}
At $x_n=0$, we have that 
\begin{align*}
t_{0}(x',\xi',\xi_n)=&(i\xi_n+c_A(x',\xi'))^{-1},
\end{align*}
and 
\small
\begin{align*}
t_{1}&(x',\xi',\xi_n)=\\
=&-(i\xi_n+c_A(x',\xi'))^{-1}(c_{A,0}(x')+c_{R_0,0}(x',0))(i\xi_n+c_A(x',\xi'))^{-1}+\\
&+(i\xi_n+c_A(x',\xi'))^{-1}\sum_{j=1}^{n-1}c_A(x',\mathrm{d}x_j)(i\xi_n+c_A(x',\xi'))^{-1}\frac{\partial c_A}{\partial x_j}(x',\xi'))(i\xi_n+c_A(x',\xi'))^{-1}+\\
&+(i\xi_n+c_A(x',\xi'))^{-1}\tilde{c}_D((x',0),\mathrm{d}x_n) (i\xi_n+c_A(x',\xi'))^{-1}\frac{\partial c_{R_0}}{\partial x_n}(x',0,\xi')(i\xi_n+c_A(x',\xi'))^{-1}=\\
=&-(i\xi_n+c_A(x',\xi'))^{-1}(c_{A,0}(x')+c_{R_0,0}(x',0))(i\xi_n+c_A(x',\xi'))^{-1}+\\
&+(i\xi_n+c_A(x',\xi'))^{-1}c_A\left(x',(i\xi_n+c_A(x',\xi'))^{-1}\mathrm{d}_{x'}c_A(x',\xi'))(i\xi_n+c_A(x',\xi'))^{-1}\right)+\\
&+(i\xi_n+c_A(x',\xi'))^{-2} \frac{\partial c_{R_0}}{\partial x_n}(x',0,\xi')(i\xi_n+c_A(x',\xi'))^{-1}.\qedhere
\end{align*}
\normalsize
\end{proof}

\begin{lemma}
\label{symbolofpositivproj}
Let  $A$ be an elliptic first order differential operator on $\Sigma$. The first term in the positive spectral projection is given by 
\begin{align*}
\sigma_0(\chi^+(A))=&\chi^+(\sigma_1(A)).
\end{align*}
In coordinates near a point $x_0\in \Sigma$, we have that
\begin{align*}
\sigma_{-1}(\chi^+(A))=&\sum_{z: \mathrm{Re}(z)>0}\mathrm{Res}_{\lambda=z} b_1(x',\xi',\lambda),
\end{align*}
where in the notations of Lemma \ref{caldeorem}
\small 
\begin{align}
\label{boneformria}
b_{1}(x',\xi',\lambda)=&-(\lambda-c_A(x',\xi'))^{-1}c_{A,0}(x')(\lambda-c_A(x',\xi'))^{-1}+\\
\nonumber 
&-(\lambda-c_A(x',\xi'))^{-1}c_A\left(x',(\lambda-c_A(x',\xi'))^{-1}\mathrm{d}_{x'}c_A(x',\xi'))(\lambda-c_A(x',\xi'))^{-1}\right).
\end{align}
\normalsize
\end{lemma}

\begin{proof}
Consider $B:=(\lambda-A)^{-1}$ as a pseudodifferential with a parameter. Picking coordinates near a point $x_0\in \Sigma$, we write the full symbol with parameter of $B$ as $b\sim \sum_{j=0}^\infty b_j$, where $b_j=b_j(x',\xi',\lambda)$ is smooth in $x'$ and homogeneous of degree $-1-j$ in $(\xi',\lambda)$. By \cite[Theorem 3.3]{grubbsec}, the full symbol of $\chi^+(A)$ has the asymptotic expansion $\sum_{j=0}^\infty \chi_j$ where 
$$\chi_j(x',\xi')=\sum_{z: \mathrm{Re}(z)>0} \mathrm{Res}_{\lambda=z} b_j(x',\xi',\lambda).$$
As such, we shall need to compute the first two terms in the asymptotic expansion of $B$. 

Using the notations of Lemma \ref{caldeorem} and its proof, we can compute $b_j$ inductively from that $b_0(x',\xi',\lambda)=(\lambda-c_A(x',\xi'))^{-1}$ and 
\begin{align*}
b_{j+1}=&-b_0c_{A,0} b_j-b_0\sum_{k+|\alpha|=j+1, \, k\leq j}\frac{1}{\alpha!} \partial_\xi^\alpha c_{A}D^\alpha_x b_k.
\end{align*}
In particular, 
\begin{align*}
b_{1}=&-b_0c_{A,0} b_0-b_0\sum_{|\alpha|=1}\partial_\xi^\alpha c_{A}D^\alpha_x b_0,
\end{align*}
which is readily computed to reproduce the formula in Equation \eqref{boneformria}.
\end{proof}

\begin{definition}
Let $D$ be an elliptic first order differential operator with adapted boundary operator $A$. We say that $A$ is a well-adapted boundary operator if $A=\sigma^{-1}A_1|_{x_n=0}$, where $A_1$ is as in Equation \eqref{decomposlna}.
\end{definition}

We note that upon fixing a choice of normal to the boundary, there is a uniquely determined well adapted boundary operator. In the notations of Lemma \ref{caldeorem}, $A$ is well adapted if and only if $c_{R_0,0}(x',0)=0$ which holds if and only if $R_0|_{x_n=0}=0$. The following Proposition is immediate from Lemma \ref{caldeorem} and \ref{symbolofpositivproj}.

\begin{proposition}
\label{compofdiffofchiplandpc}
Let $D$ be a first order elliptic differential operator with adapted boundary operator $A$. Then $\chi^+(A)-P_{\Ca_D}$ is of order $-1$ and its principal symbol is in the notation of Lemma \ref{caldeorem} given by
\begin{align*}
\sigma_{-1}(\chi^+(A)-P_{\Ca_D})=&\sum_{z: \mathrm{Re}(z)>0}\mathrm{Res}_{\lambda=z} p(x',\xi',\lambda),
\end{align*}
where 
\begin{align*}
p(x',\xi',\lambda)=&(\lambda-c_A(x',\xi'))^{-1}c_{R_0,0}(x',0)(\lambda-c_A(x',\xi'))^{-1}+\\
&+(\lambda-c_A(x',\xi'))^{-2}\frac{\partial c_{R_0}}{\partial x_n}(x',0,\xi')(\lambda-c_A(x',\xi'))^{-1},
\end{align*}

If $A$ is well adapted, then $\sigma_{-1}(\chi^+(A)-P_{\Ca_D})$ is constructed from residues of the simpler symbol 
\begin{align*}
p(x',\xi',\lambda)=&(\lambda-c_A(x',\xi'))^{-2}\frac{\partial c_{R_0}}{\partial x_n}(x',0,\xi')(\lambda-c_A(x',\xi'))^{-1}.
\end{align*}
\end{proposition}

\begin{corollary}
\label{somecompwithcald}
Let $D$ be a first order elliptic differential operator with well adapted boundary operator $A$.
\begin{itemize}
\item If $\frac{\partial c_{R_0}}{\partial x_n}|_\Sigma$ commutes with $c_A$, then $\chi^+(A)-P_{\Ca_D}\in \Psi^{-2}_{\rm cl}(M;E)$ and acts compactly on $\checkH_A(D)$.
\item If $\frac{\partial c_{R_0}}{\partial x_n}|_\Sigma$ anti-commutes with $c_A$, then 
$$\sigma_{-1}(\chi^+(A)-P_{\Ca_D})=\frac{1}{4}c_A^{-2}\chi^-(c_A)\frac{\partial c_{R_0}}{\partial x_n}|_{\Sigma}.$$
\end{itemize}
\end{corollary}

\begin{proof}
The first item follows from Proposition \ref{compofdiffofchiplandpc} if we can prove that $\mathrm{Res}_{\lambda=z} p(x',\xi',\lambda)=0$ for all $z$. Since $A$ is well adapted, and $c_A$ commutes with $\frac{\partial c_{R_0}}{\partial x_n}|_\Sigma$, Proposition \ref{compofdiffofchiplandpc} ensures that 
\begin{align*}
p(x',\xi',\lambda)=&(\lambda-c_A(x',\xi'))^{-3}\frac{\partial c_{R_0}}{\partial x_n}(x',0,\xi').
\end{align*}
In particular, the residues of $p$ vanish if $\frac{\partial c_{R_0}}{\partial x_n}|_\Sigma$ commutes with $c_A$.

On the other hand, if $\frac{\partial c_{R_0}}{\partial x_n}|_\Sigma$ anti-commutes with $c_A$
\begin{align*}
p(x',\xi',\lambda)=&(\lambda-c_A(x',\xi'))^{-2}(\lambda+c_A(x',\xi'))^{-1}\frac{\partial c_{R_0}}{\partial x_n}(x',0,\xi').
\end{align*}
Using Proposition \ref{compofdiffofchiplandpc} we compute that in this case
\begin{align*}
\sigma_{-1}(\chi^+(A)-P_{\Ca_D})(x',\xi')=&\frac{1}{4}c_A(x',\xi')^{-2}\chi^-(c_A(x',\xi'))\frac{\partial c_{R_0}}{\partial x_n}(x',0,\xi').
\qedhere
\end{align*}
\end{proof}

\begin{proposition}
\label{charcompacssndaa}
Let $D$ be an elliptic first order differential operator with adapted boundary operator $A$. Then $\chi^+(A)-P_{\Ca_D}$ is compact on $\checkH_A(D)$ if and only if 
$$
\sigma_0(\chi^-(A))\sigma_{-1}(\chi^+(A)-P_{\Ca_D})\sigma_0(\chi^+(A))=0,
$$
as elements of $\Ck{\infty}(S^*\Sigma;\End(\pi^*E))$.
\end{proposition}

\begin{proof}
We write 
\begin{align*}
\chi^+(A)-P_{\Ca_D}=&\chi^+(A)(\chi^+(A)-P_{\Ca_D})\chi^+(A)+\chi^-(A)(\chi^+(A)-P_{\Ca_D})\chi^-(A)+\\
&\qquad+\chi^+(A)(\chi^+(A)-P_{\Ca_D})\chi^-(A)+ \chi^-(A)(\chi^+(A)-P_{\Ca_D})\chi^+(A)\\
=&\chi^+(A)(\chi^+(A)-P_{\Ca_D})\chi^+(A)+\chi^-(A)(\chi^+(A)-P_{\Ca_D})\chi^-(A)+\\
&\quad+\chi^+(A)(\chi^+(A)-P_{\Ca_D})\chi^-(A)+ \chi^-(A)(\chi^+(A)-P_{\Ca_D})\chi^+(A)
\end{align*}
Since $\chi^+(A)-P_{\Ca_D}\in \Psi^{-1}_{\rm cl}(\Sigma;E)$ the first two terms are compact for any adapted boundary operator $A$. The third term is a pseudodifferential operator of order $-1$ and it acts compactly on $\checkH_A(D)$ if and only if it acts compactly $\SobH{{\frac12}}(\Sigma;E)\to \SobH{-{\frac12}}(\Sigma;E)$ and as such the third term is compact on  $\checkH_A(D)$ for any adapted boundary operator $A$. The fourth term is a pseudodifferential operator of order $-1$ and it acts compactly on $\checkH_A(D)$ if and only if it acts compactly $\SobH{-{\frac12}}(\Sigma;E)\to \SobH{{\frac12}}(\Sigma;E)$ and as such the fourth term is compact on $\checkH_A(D)$ if and only if its symbol vanish, i.e. $\sigma_0(\chi^-(A))\sigma_{-1}(\chi^+(A)-P_{\Ca_D})\sigma_0(\chi^+(A))=0$.
\end{proof}

\begin{proposition}
\label{discexaorderminell}
Let $D$ be a first order elliptic differential operator with well adapted boundary operator $A$ and assume that $\frac{\partial c_{R_0}}{\partial x_n}|_\Sigma$ anti-commutes with $c_A$. Then $\chi^+(A)-P_{\Ca_D}$ is compact on $\checkH_A(D)$ if and only if 
$$\chi^-(c_A)\frac{\partial c_{R_0}}{\partial x_n}|_{\Sigma}= 0.$$
\end{proposition}

\begin{proof}
By Proposition \ref{charcompacssndaa}, it suffices to prove that $\chi^-(c_A)\sigma_{-1}(\chi^+(A)-P_{\Ca_D})\chi^+(c_A)=0$ if and only if $\chi^-(c_A)\frac{\partial c_{R_0}}{\partial x_n}|_{\Sigma}= 0$ when $\frac{\partial c_{R_0}}{\partial x_n}|_\Sigma$ anti-commutes with $c_A$. But by Corollary \ref{somecompwithcald}, 
\begin{align*}
\chi^-(c_A)\sigma_{-1}(\chi^+(A)-P_{\Ca_D})\chi^+(c_A)=&\frac{1}{4}c_A^{-2}\chi^-(c_A)\frac{\partial c_{R_0}}{\partial x_n}|_{\Sigma}\chi^+(c_A)\\
=&\frac{1}{4}c_A^{-2}\chi^-(c_A)\frac{\partial c_{R_0}}{\partial x_n}|_{\Sigma},
\end{align*}
where we in the last line used that $\frac{\partial c_{R_0}}{\partial x_n}|_\Sigma$ anti-commutes with $c_A$. The proposition follows from that $c_A$ is invertible. 
\end{proof}

\begin{example}
\label{unitdiscexample}
Let us give an easy example showcasing when $\chi^+(A)-P_{\Ca_D}$ fails to be compact on $\checkH_A(D)$. Take $M$ to be the unit disc and $E=F=M\times \C^2$. 

We first consider the rather straight forward example of the Dirac type operator $D_0$ given in polar coordinates by
$$D_0=\begin{pmatrix} 0& \partial_r+\frac{i}{r}\partial_\theta\\ 
-\partial_r+\frac{i}{r}\partial_\theta& 0\end{pmatrix}=\sigma(\partial_r+A+R_{00}),$$
where 
$$\sigma=\begin{pmatrix} 0& 1\\-1&0\end{pmatrix},\quad A=\begin{pmatrix} -i\partial_\theta& 0\\0&i\partial_\theta\end{pmatrix}\quad\mbox{and}\quad R_{00}=(r^{-1}-1)A.$$
It is readily verified that $A$ is a well adapted boundary operator, and in this case it is self-adjoint. A short computation with the Cauchy-Riemann equation in polar coordinates show that $P_{\Ca_{D_0}}=\chi^+(A)$ for the operator $D_0$. 

To jazz up this example a bit, we define 
$$R_0:=R_{00}-i\alpha\sigma\partial_\theta,$$
where $\alpha=\alpha(r)$ is a smooth function $\alpha\in \Ck[c]{\infty}(0,1]$ with $\alpha(1)=0$. We define 
$$D_\alpha:=\sigma(\partial_r+A+R_{0})=\begin{pmatrix} i\alpha\partial_\theta& \partial_r+\frac{i}{r}\partial_\theta\\ 
-\partial_r+\frac{i}{r}\partial_\theta& i\alpha\partial_\theta \end{pmatrix}.$$
The reader should note that $D_\alpha$ is formally self-adjoint if and only if $\alpha$ is real valued and $A$ is well adapted for any choice of $\alpha$. Let us compute $\sigma_{-1}(\chi^+(A)-P_{\Ca_D})$ in this case. For simplicity, we write $\xi$ for the cotangent variable on $\Sigma=S^1$. We have that 
$$c_A(\xi)=\begin{pmatrix} -\xi& 0\\0&\xi\end{pmatrix},  \quad\mbox{and}\quad (\lambda-c_A)^{-1}(\xi)=\begin{pmatrix} (\lambda+\xi)^{-1}& 0\\0&(\lambda-\xi)^{-1}\end{pmatrix}.$$
We have $c_{R_0}(r,\xi)=(r^{-1}-1)c_A-\alpha\sigma \xi$ and compute that 
$$\frac{\partial c_{R_0}}{\partial r}|_{\Sigma}(\xi)=-c_A(\xi)-\alpha'(1)\sigma \xi.$$
Although Corollary \ref{somecompwithcald} does not give an immediate answer because $\frac{\partial c_{R_0}}{\partial r}|_{\Sigma}$ does not commute nor anticommute with $c_A$ in this case, we have a sum of a commuting and anticommuting term and a short computation gives us that 
\begin{align*}
p(\xi,\lambda)=&-(\lambda-c_A(\xi))^{-2}\frac{\partial c_{R_0}}{\partial r}(r,\xi)|_{r=1}(\lambda-c_A(\xi))^{-1}\\
=&(\lambda-c_A(\xi))^{-2}c_A(\xi)(\lambda-c_A(\xi))^{-1}+(\lambda-c_A(\xi))^{-2}\alpha'(1)\sigma \xi(\lambda-c_A(\xi))^{-1}\\
=&(\lambda-c_A(\xi))^{-3}c_A(\xi)+\alpha'(1) \xi(\lambda-c_A(\xi))^{-2}(\lambda+c_A(\xi))^{-1}\sigma.
\end{align*}
The first term has no residues, and the residues of the second term is computed as in Corollary \ref{somecompwithcald} to produce 
\begin{align*}
\sigma_{-1}(\chi^+(A)-P_{\Ca_D})(\xi)=&\frac{\alpha'(1) \xi}{4}c_A^{-2}\chi^-(c_A)\sigma=\frac{\alpha'(1) }{4\xi}\chi^-(c_A)\sigma\\
=&\frac{\alpha'(1) }{4}\begin{pmatrix} 0&\frac{\chi_{(0,\infty)}(\xi)}{\xi}\\-\frac{\chi_{(-\infty,0)}(\xi)}{\xi}&0\end{pmatrix},
\end{align*}
using that
$$\chi^-(c_A)=\sigma_0(\chi^-(A))=\begin{pmatrix} \chi_{(0,\infty)}(\xi)& 0\\0&\chi_{(-\infty,0)}(\xi)\end{pmatrix}.$$
We see that $\sigma_{-1}(\chi^+(A)-P_{\Ca_D})\neq 0$ if and only if $\alpha'(1)\neq 0$. Moreover, since 
\begin{align*}
\chi^-(c_A)\sigma_{-1}(\chi^+(A)-&P_{\Ca_D})\chi^+(c_A)\\
=&\frac{\alpha'(1) }{4}\begin{pmatrix} 0&\frac{\chi_{(0,\infty)}(\xi)}{\xi}\\-\frac{\chi_{(-\infty,0)}(\xi)}{\xi}&0\end{pmatrix}\begin{pmatrix} \chi_{(-\infty,0)}(\xi)& 0\\0&\chi_{(0,\infty)}(\xi)\end{pmatrix}\\
=&\frac{\alpha'(1) }{4}\begin{pmatrix} 0&\frac{\chi_{(0,\infty)}(\xi)}{\xi}\\-\frac{\chi_{(-\infty,0)}(\xi)}{\xi}&0\end{pmatrix},
\end{align*}
we see that $\chi^+(A)-P_{\Ca_D}$ fails to be compact on $\checkH_A(D)$ if and only if $\alpha'(1)\neq 0$ by Proposition \ref{charcompacssndaa}. Let us summarise the salient features of this example in a proposition.
\end{example}

\begin{proposition}
\label{compactnessfailure}
Assume that $M$ is the unit disc, $E=F=M\times \C^2$ and take $D_\alpha$ as in Example \ref{unitdiscexample} for an $\alpha$ with $\alpha'(1)\neq 0$. We let $A$ denote its well adapted boundary operator. Then it holds that 
\begin{enumerate}[(i)]
\item $\sigma_{-1}(\chi^+(A)-P_{\Ca_D})\neq 0$.
\item $\chi^+(A)-P_{\Ca_D}$ fails to be compact on $\checkH_A(D)$.
\item \label{cpctfail:3} 
	The space $\chi^+(A)\SobH{-{\frac12}}(S^1,\C^2)\cap {\Ca_D}$ is a finite-dimensional subspace of $\Ck{\infty}(S^1,\C^2)$.
\end{enumerate} 
\end{proposition}

\begin{proof}
The first two statements were already proven in Example \ref{unitdiscexample}. It remains to prove that last. 

We define the operator $\Lambda:=\frac{\alpha'(1)}{4}(1+\Delta)^{-{\frac12}}\in \Psi^{-1}_{\rm cl}(S^1,\C^2)$. Let us first prove the claim that $L:=\chi^+(A)-P_{\Ca_D}+\sigma\Lambda\chi^-(A)\in \Psi^{-1}_{\rm cl}(S^1,\C^2)$ is elliptic. We have that 
$$\sigma_{-1}(L)=\frac{\alpha'(1) }{4}\begin{pmatrix} 0&\frac{\chi_{(0,\infty)}(\xi)}{\xi}-\frac{\chi_{(-\infty,0)}(\xi)}{\xi}\\-\frac{\chi_{(-\infty,0)}(\xi)}{\xi}-\frac{\chi_{(0,\infty)}(\xi)}{\xi}&0\end{pmatrix}=\frac{\alpha'(1) }{4}\begin{pmatrix} 0&\frac{1}{|\xi|}\\-\frac{1}{\xi}&0\end{pmatrix}.$$
This is invertible for all $\xi\neq 0$, so $L$ is elliptic. 

Consider the space $\chi^+(A)\SobH{-{\frac12}}(S^1,\C^2)\cap {\Ca_D}$. We have the following equivalence 
\begin{align*}
f\in \chi^+(A)\SobH{-{\frac12}}(S^1,\C^2)\cap {\Ca_D}&\iff
\begin{cases} 
\chi^+(A)f=f\\
P_{\Ca_D}f=f\end{cases}\\
&\iff
\begin{cases} 
\chi^+(A)f=f\\
(\chi^+(A)-P_{\Ca_D})f=0
\end{cases}\\
&\iff
\begin{cases} 
\chi^+(A)f=f\\
Lf=0.\end{cases}
\end{align*}
In the last equivalence we used that if $\chi^+(A)f=f$, then $\chi^-(A)f=0$ and $\chi^-(A)f=0$ implies that $(\chi^+(A)-P_{\Ca_D})f=Lf$. We can conclude that $\chi^+(A)\SobH{-{\frac12}}(S^1,\C^2)\cap {\Ca_D}=\chi^+(A)\SobH{-{\frac12}}(S^1,\C^2)\cap \ker(L)$. But since $L$ is elliptic, $\ker(L)\subseteq \Ck{\infty}(S^1,\C^2)$ is a finite-dimensional subspace and item \ref{cpctfail:3} follows.
\end{proof}

Although Proposition \ref{compactnessfailure} shows that $\ker(D_{\chi^+(A)\SobH{-\frac12}})$ may fail to be infinite dimensional in general, we have instead the following.

\begin{corollary}
\label{Cor:FinDimComp}
We have that for any spectral cut $r \in \R$ (i.e., $r \in \R$ such that $A_r := A - r$ is invertible bisectorial) and any pseudo-differential Calderón projection $\Proj{{\Ca_D}}$,  
$$ \faktor{{\Ca_D}}{ \Proj{{\Ca_D}} \chi^+(A_r)\SobH{-\frac12}(\Sigma;E)}
\quad\text{and}\quad
\faktor{\chi^+(A_r) \SobH{-\frac12}(\Sigma;E)}{ \chi^+(A) {\Ca_D}}$$
are finite dimensional.
\end{corollary}

\begin{proof}
Using Corollary \ref{Cor:CalCheck}, we have that $\Ca_D^c = \ker \Proj{{\Ca_D}}$ defines a regular boundary condition. 
Therefore, from Theorem \ref{Thm:Ell} with the choice of $\GProj_{\pm} = \chi^{\pm}(A_r)$  we have that  $(\Ca_D^c, \chi^+(A_r)\SobH{-\frac12}(\Sigma;E))$ is a Fredholm pair in $\checkH_A(D)$.
Therefore, 
$$
\faktor{\checkH_A(D)}{\Ca_D^c + \chi^+(A_r)\SobH{-\frac12}(\Sigma;E)} \cong \faktor{{\Ca_D}}{ \Proj{{\Ca_D}} \chi^+(A_r)\SobH{-\frac12}(\Sigma;E)}
$$
via Lemma A4 in \cite{BBan}.

For the remaining quotient, we invoke Theorem \ref{Thm:Ell} with a choice of $\GProj_{+} = \Proj{{\Ca_D}}$, and since the generalised APS condition $\chi^-(A_r)\SobH{\frac12}(\Sigma;E)$ is a regular boundary condition, we have that $(\chi^-(A_r)\SobH{\frac12}(\Sigma;E), {\Ca_D})$ is a Fredholm pair in $\checkH_A(D)$. 
Therefore, using \cite[Lemma A.4]{BBan}, we obtain that 
\begin{equation*}
\faktor{\checkH_A(D)}{\chi^-(A_r)\SobH{\frac12}(\Sigma;E)+ {\Ca_D} } \cong 
\faktor{\chi^+(A_r) \SobH{-\frac12}(\Sigma;E)}{ \chi^+(A_r) {\Ca_D}}.
\end{equation*}
This finishes the proof. 
\end{proof} 

If a boundary condition $B$ defines a realisation with infinite dimensional kernel, then this is equivalent to saying that $B \cap {\Ca_D}$ is infinite dimensional. 
Although Corollary \ref{Cor:FinDimComp} ``measures'' the amount of $\chi^+(A)\SobH{-\frac12}(\Sigma;E)$ in ${\Ca_D}$ (as seen through the projection $\Proj{{\Ca_D}}$), it does not give us information about such a $B$.
The following gives a more precise description in the way of a necessary condition. 

\begin{proposition} 
Let $B$ be a boundary condition. 
If  $\dim \ker(D_{\rm B}) = \infty$, then $B \not \subset \SobH{\frac12}(\Sigma;E)$ and  $\dim \chi^+(A_r)B  = \infty$.
\end{proposition}
\begin{proof}
If $B \subset \SobH{\frac12}$, then by Proposition \ref{Prop:MaxClosed}, we have that $\ker(D_{\rm B})$ is finite. 
Therefore, the contrapositive yields that for $\ker(D_{\rm B})$ infinite, we have that $B \not \subset \SobH{\frac12}(\Sigma;E)$.

Next, assume that $\dim \chi^+(A_r)B  < \infty$.
For $u \in B$, we have that $u = \chi^-(A_r)u + \chi^+(A_r)u$, and we always have that $\chi^-(A_r)u \in \SobH{\frac12}(\Sigma;E)$. 
Setting 
$$B':=\chi^-(A_r)\SobH{\frac12}(\Sigma;E) \oplus \chi^+(A_r)B,$$
we therefore have an inclusion $B \subset B'$. The space $B'$ is a boundary condition since $\chi^+(A_r)B$ was assumed to be finite dimensional.
Now, letting $B_0 = \chi^-(A_r)\SobH{\frac12}(\Sigma;E)$, the map
$$ \ker(D_{{\rm B}'})/\ker(D_{{\rm B}_0}) \embed \dom(D_{{\rm B}'})/\dom(D_{{\rm B}_0}) \cong \chi^+(A_r)B,$$
is an injection given that $\dom(D_{{\rm B}'}) = \dom(D_{{\rm B}_0}) + \dom(D_{\chi^+(A_r)B})$
This allows us to conclude that 
$$
\dim \ker(D_{{\rm B}'})/\ker(D_{{\rm B}_0}) \leq \dim \chi^+(A_r)B,$$
and since
$$
\ker(D_{{\rm B}'}) \cong \ker(D_{{\rm B}'})/\ker(D_{{\rm B}_0}) \oplus \ker(D_{{\rm B}_0}),$$
we have that 
$$\dim\ker(D_{\rm B}) \leq \dim \ker(D_{{\rm B}'}) \leq \dim \chi^+(A_r)B + \dim \ker(D_{{\rm B}_0}) < \infty.$$
Taking the contrapositive and combining with our earlier deduction that $B \not \subset \SobH{\frac12}(\Sigma;E)$ concludes the proof.
\end{proof}

\appendix

\section{Lemmas on dimensions and subspaces}
\label{appondimensiosnsns}
\begin{lemma}
\label{Lem:SubPair} 
Let $\cB$ and $\cB^\ast$ be Banach spaces and suppose that $\inprod{\cdot,\cdot}: \cB \times \cB^\ast \to \C$ is a perfect pairing.
Suppose that $\cB = X \oplus Y$ and let $P_{X,Y}$ denote the bounded projection $\cB \mapsto X$ and similarly for  $P_{Y,X} = 1 - P_{X,Y}$. 
Let $X^\ast = P_{X,Y}^\ast \cB^\ast$, the dual map to $P_{X,Y}$ and similarly $Y^\ast = P_{Y,X}^\ast \cB^\ast$. 
Then, the following hold: 
\begin{enumerate}[(i)]
\item \label{Lem:SubPair:1}
	$P_{X^\ast, Y^\ast} = P_{X,Y}^\ast$ and $P_{Y^\ast, X^\ast} = P_{Y,X}^\ast$. In particular $\cB^\ast = X^\ast  \oplus Y^\ast$.  
\item \label{Lem:SubPair:2}
	The pairing $\inprod{\cdot,\cdot}$ restricts to perfect pairings $X \times X^\ast \to \C$, as well as $Y \times Y^\ast \to \C$.
\item \label{Lem:SubPair:3}
	The annihilators 
		$$Y^\perp := \set{b^\ast \in \cB^\ast: \inprod{b, y} = 0\ \forall y \in Y} = X^\ast,\ \text{and}\  \quad X^\perp = Y^\ast.$$
\end{enumerate} 
\end{lemma}
\begin{proof}
The item \ref{Lem:SubPair:1} is immediate from the definition of the dual projectors.

We prove \ref{Lem:SubPair:2}. 
Let $x \in X$ and $x^\ast \in X^\ast$. 
Then, there exists $C < \infty$ such that 
$$\modulus{ \inprod{x, x^\ast}} \leq C \norm{x}_{\cB} \norm{x^\ast}_{\cB^\ast} = \norm{P_{X,Y} x}_{\cB} \norm{P_{X^\ast, Y^\ast} x^\ast}_{\cB^\ast} =  \norm{x}_{X} \norm{x^\ast}_{X^\ast}.$$
Again, from the fact that $\inprod{\cdot,\cdot}$ is a perfect paring between $\cB \times \cB^\ast$, we have a $C_1 < \infty$ such that 
\begin{equation} 
\label{Eq:SubPair:1} 
\norm{x} \leq C_1 \sup \set{ \frac{\modulus{\inprod{x, b^\ast}}}{\norm{b^\ast}_{\cB^\ast}}: 0 \neq b^\ast \in \cB^\ast}.
\end{equation} 
Now, note that 
\begin{equation}
\label{Eq:SubPair:2}
\inprod{x, b^\ast} = \inprod{P_{X,Y}x, b^\ast} = \inprod{x, P_{X^\ast, Y^\ast}b^\ast}.
\end{equation}
Moreover, there exists $C_2$ such that 
$$\norm{P_{X^\ast, Y^\ast}b^\ast} + \norm{P_{Y^\ast, X^\ast}b^\ast} \leq  C_2 \norm{b^\ast},$$
and therefore,
\begin{equation}
\label{Eq:SubPair:3}
\frac{1}{\norm{b^\ast}}\leq C_2 \frac{1}{\norm{P_{X^\ast, Y^\ast}b^\ast}} . 
\end{equation}
Combining \eqref{Eq:SubPair:2} and \eqref{Eq:SubPair:3} and using \eqref{Eq:SubPair:1}, we obtain 
$$
\modulus{ \inprod{x, x^\ast}} \leq C_1 C_2 \sup \set{ \frac{\modulus{\inprod{x, b^\ast}}}{\norm{P_{X^\ast,Y^\ast} b^\ast}_{\cB^\ast}}: 0 \neq P_{X^\ast, Y^\ast} b^\ast \in \cB^\ast} 
	= C_1 C_2 \sup_{\substack{x^\ast \in X^\ast \\ 1 = \norm{x^\ast}_{X^\ast}}} \modulus{\inprod{x, x^\ast}}.$$
A symmetric argument yields that 
$$\norm{x^\ast} \lesssim \sup_{\substack{x \in X \\ 1 = \norm{x}_{X}}} \modulus{\inprod{x, x^\ast}}.$$

Lastly, we prove \ref{Lem:SubPair:3}. 
First, suppose that that $b^\ast \in Y^\perp$. 
That is, $0 = \inprod{b^\ast, y}$ for all $y \in Y$. 
Since every $y = P_{Y,X}b$  for some $b \in \cB$, 
$0 = \inprod{P_{Y^\ast, X^\ast} b^\ast,b}$ for all $b \in \cB$. 
This implies that $b^\ast \in \ker P_{Y^\ast, X^\ast} = X^\ast$ and therefore, $Y^\perp \subset X^\ast$.

Conversely, let $x^\ast \in X^\ast$. 
Then, for $y \in Y$, we have that
$$ \inprod{x^\ast, y} = \inprod{P^\ast_{X, Y}x^\ast, y} = \inprod{x^\ast, P_{X,Y}y} = 0.$$
That is, $X^\ast \subset Y^\perp$.
The statement that $X^\perp = Y^\ast$ follows similarly.
\end{proof}

Next we prove two results concerning the dimension of certain subspaces.

\begin{lemma}
\label{Lem:CheckNonZeroalternative1} 
Let $P\in \Psi^{\pmb 0}_{\rm cl}(M;\pmb{E})$ be a classical zeroth order pseudo-differential projection in the Douglis-Nirenberg calculus acting on sections of an $\R$-graded Hermitian vector bundle $\pmb{E}$ on a closed manifold $\Sigma$ with $\dim(\Sigma)>0$. If $\sigma_{\pmb 0}(P)\neq 0,1$, then for any $s\in \R$, the space $P\SobHH{ s}(\Sigma;\pmb{E})$ has infinite dimension and infinite codimension.

In particular, in the setting of Lemma \ref{Lem:CheckNonZero} with $\dim(M)>1$, then for any admissible cut $r \in \R$, the spaces $\chi^{\pm}(A_r)\Lp{2}(\Sigma;E)$ are infinite dimensional. 
Moreover, the space $\checkH_A(D) \neq \set{0}$.
\end{lemma}

\begin{proof}
We have that $P$ acts compactly on $\SobHH{ s}(\Sigma;\pmb{E})$ if and only if $\sigma_{\pmb 0}(P)\neq 0$. We conclude that $P\SobHH{ s}(\Sigma;\pmb{E})$ has infinite dimension if and only if $\sigma_{\pmb 0}(P)\neq 0$. The infinite codimensionality of $P\SobHH{ s}(\Sigma;\pmb{E})$ is by a similar argument equivalent to $\sigma_{\pmb 0}(P)\neq 1$.
\end{proof}

\begin{lemma}
\label{Lem:CheckNonZeroalternative2} 
Let $A$ be any elliptic first order differential operator acting on sections of a vector bundle $E$ on a closed manifold $\Sigma$ such that its resolvent set is non-empty  and in some point the principal symbol of $A$ has an eigenvalue with positive real part. Then there exists two sequences of eigenvalues $(\lambda_j^+)_{j\in \mathbb{N}}$ and $(\lambda_j^-)_{j\in \mathbb{N}}$ of $A$ such that 
$$\mathrm{Re}(\lambda^\pm_j)\to \pm \infty, \quad\mbox{as $j\to \infty$}.$$

In particular, in the setting of Lemma \ref{Lem:CheckNonZero} with $\dim(M)>1$, then for any admissible cut $r \in \R$ and an $s\in \R$, the spaces $\chi^{\pm}(A_r)\SobH{s}(\Sigma;E)$ are infinite dimensional. 
Moreover, the space $\checkH_A(D) \neq \set{0}$.
\end{lemma}

\begin{proof}
Define the zeroth order classical pseudodifferential operator $F := A(1 + A^*A)^{-{\frac12}}$. 
The spectrum of $F$ is contained in the unit disc.
From the fact that
the symbol mapping is a $*$-homomorphism, it follows that the spectrum
of $F$ contains the spectrum of its symbol, which we will denote
by $\sym_F$. In fact, the spectrum of $\sym_F$ coincides with the essential
spectrum of $F$. 

From the fact that $A$ is first order, it follows that 
$\lambda \in \spec(\sym_A(x, \xi))$ if and only if
$-\lambda \in \spec(\sym_A(x, -\xi))$. This implies the spectrum of 
$\sym_F$ is symmetric about the imaginary axis.  Since the principal symbol of $A$ has an eigenvalue with positive real part in some point,  it
follows that $F$ has essential spectrum on both sides of the
imaginary axis.

Now, $A$ is assumed to have non-empty resolvent set, hence its must be discrete by  \cite[]{shubinsbook}. 
The essential spectrum of $F$ consists of
all limits points of the image of the spectrum of $A$ under the 
homeomorphism $z \rightarrow z(1 + |z|^2)^{-{\frac12}}$ of the complex
plane onto the open unit disc. Since $F$ has essential spectrum on both sides of
the imaginary axis, the spectrum of $A$ has limit points on both
sides of the imaginary axis. It follows that the spectrum of $A$ must go
to infinity on both sides of the imaginary axis.
By \cite[]{shubinsbook}, each element of the spectrum is an eigenvalue and the statement follows. 

Now, we have that each eigenspace $\Eig(\lambda_i) \subset \Ck{\infty}(\Sigma;E)$ and therefore, $\chi^{\pm}(A_r)\SobH{s}(\Sigma;E)$ are infinite-dimensional. Therefore $\checkH_A(D) \neq {0}$. 
\end{proof}

\section{Bär-Ballmann's approach from an abstract perspective}
\label{bärballappsection}

Although the main results of this paper concern elliptic differential operators on manifolds with boundary, a large proportion of the main tools can be formulated in a general abstract formalism. We believe several of these results to be known to specialists, but nevertheless have decided to formulate the Bär-Ballmann approach in complete abstraction. Similar ideas to those appearing in this appendix can be found in \cite{bossfury}. In particular, several of our main results have abstract counterparts that hold in potentially non-elliptic settings.

\subsection{Preliminaries}

Throughout this section, we fix a closed densely defined linear operator between separable Hilbert spaces $\Hil_1\to \Hil_2$. 
For reasons that will later become apparent we shall denote this operator by $T_{\rm min}$. We also fix an adjoint $T^\dagger_{\rm min}:\Hil_2\to \Hil_1$ which is closed and densely defined. In the main part of this paper, we were exclusively concerned with $T_{\rm min}$ being the minimal closure of an elliptic differential operator on a manifold with boundary and $T^\dagger_{\rm min}$ the minimal closure of its formal adjoint. We introduce the notation 
$$T_{\rm max}:=(T^\dagger_{\rm min})^*\quad\mbox{and}\quad T_{\rm max}^\dagger:=(T_{\rm min})^*.$$
Note that $T_{\rm min}\subseteq T_{\rm max}$ and $T_{\rm min}^\dagger\subseteq T_{\rm max}^\dagger$. 

\begin{definition}
\label{efonedoenfap}
\begin{itemize}
\item A pair $\mathcal{T}=(T_{\rm min},T^\dagger_{\rm min})$ as above is said to be a formally adjointed pair (FAP). 
\item A realisation $T$ of a FAP is an extension $T$ of $T_{\rm min}$ contained in $T_{\rm max}$. If $T$ is closed as an operator, we say that $T$ is a closed realisation of $\mathcal{T}$.
\item The Cauchy data space of a FAP $\mathcal{T}$ is defined as the Hilbert space 
$$\check{\mathhil{h}}_\mathcal{T}:=\faktor{\dom(T_{\rm max})}{\dom(T_{\rm min})}.$$
The quotient map $\dom(T_{\rm max})\to \check{\mathhil{h}}_\mathcal{T}$ will be denoted by $\gamma$. 
\end{itemize}
\end{definition}

The reader can note that also in this abstract setting, whenever $\checkH(\mathcal{T})$ is a Banach space equipped with a surjective map $\gamma:\dom(T_{\rm max})\to \checkH(\mathcal{T})$ with kernel $\dom(T_{\rm min})$ then $\checkH(\mathcal{T})\cong \check{\mathhil{h}}_\mathcal{T}$ via an isomorphism induced from the boundary mappings. We can therefore speak of concrete Cauchy data spaces also in this abstract setting. 

If $\mathcal{T}=(T_{\rm min},T^\dagger_{\rm min})$ is a FAP, then $\mathcal{T}^\dagger:=(T^\dagger_{\rm min},T_{\rm min})$ is also a FAP. A symmetric closed operator $S$ will, unless otherwise indicated, be identified with the FAP $(S,S)$. In particular, $S_{\rm min}=S=S_{\rm min}^\dagger$ and $S_{\rm max}=S^*=S_{\rm max}^\dagger$. We use the notation $\check{\mathhil{h}}_S:=\check{\mathhil{h}}_{(S,S)}\equiv \faktor{\dom(S^*)}{\dom(S)}$.

\begin{proposition}
\label{doublingup}
Let $\mathcal{T}=(T_{\rm min},T^\dagger_{\rm min})$ be a FAP. Define the densely defined operator $\tilde{T}$ on $\Hil\oplus \Hil$ by
$$\tilde{T}:=\begin{pmatrix} 0& T_{\rm min}\\ T_{\rm min}^\dagger&0\end{pmatrix}.$$
Then, the following hold.
\begin{enumerate}[(i)]
\item The operator $\tilde{T}$ is closed and symmetric, and its adjoint is given by 
$$\tilde{T}^*:=\begin{pmatrix} 0& T_{\rm max}\\ T_{\rm max}^\dagger&0\end{pmatrix}.$$
Both $\tilde{T}$ and $\tilde{T}^*$ are $\mathbb{Z}/2$-graded in the grading of $\Hil\oplus \Hil$ making the first summand even and the second odd.
\item The equality $\dom(\tilde{T}^*)=\dom(T_{\rm max}^\dagger)\oplus \dom(T_{\rm max})$ induces an equality 
$$\check{\mathhil{h}}_{\tilde{T}}=\check{\mathhil{h}}_{\mathcal{T}^\dagger}\oplus \check{\mathhil{h}}_\mathcal{T}.$$
\item The von Neumann decomposition of $\dom(\tilde{T}^*)$ induces an isomorphism of Hilbert spaces 
$$\check{\mathhil{h}}_{\mathcal{T}^\dagger}\oplus \check{\mathhil{h}}_\mathcal{T}\cong \ker(\tilde{T}^*+i)\oplus \ker(\tilde{T}^*-i).$$
\end{enumerate}
\end{proposition}

We omit the proof as it follows from standard techniques for unbounded operators. The reader can find relevant details in \cite{weidmann}.

\begin{lemma}
\label{symnondeg}
Let $S$ be a closed symmetric operator. 
\begin{enumerate}[(i)]
\item \label{symnondeg:1} 
The sesquilinear form 
$$\omega_{S,0}:\dom(S^*)\times \dom(S^*)\to \C, \quad \omega_{S,0}(v_1,v_2):=\langle S^*v_1,v_2\rangle_\Hil-\langle v_1,S^*v_2\rangle_\Hil,$$
descends to a well-defined sesquilinear non-degenerate pairing 
$$\omega_{S}:\check{\mathhil{h}}_S\times \check{\mathhil{h}}_S\to \C,$$
satisfying the symmetry condition 
$$\omega_S(\xi_1,\xi_2)=-\overline{\omega_S(\xi_2,\xi_1)}.$$
\item \label{symnondeg:2} 
 Under the isomorphism $\check{\mathhil{h}}_S\cong \ker(S^*+i)\oplus \ker(S^*-i)$, the pairing $\omega_{S}$ corresponds to the sesquilinear form 
\begin{align*}
\omega_{S,vN}:&(\ker(S^*+i)\oplus \ker(S^*-i))\times (\ker(S^*+i)\oplus \ker(S^*-i))\to \C, \\ 
&\omega_{S,vN}\left(\begin{pmatrix}v_{1,+}\\v_{1,-}\end{pmatrix},\begin{pmatrix}v_{2,+}\\v_{2,-}\end{pmatrix}\right):=2i\left\langle\begin{pmatrix}v_{1,+}\\v_{1,-}\end{pmatrix},\begin{pmatrix} 0&1\\-1&0\end{pmatrix} \begin{pmatrix}v_{2,+}\\v_{2,-}\end{pmatrix}\right\rangle_{\Hil\oplus \Hil}.
\end{align*}
\end{enumerate}
\end{lemma}

\begin{proof}
The analogous symmetry condition $\omega_{S,0}(v_1,v_2)=-\overline{\omega_{S,0}(v_2,v_1)}$ for $\omega_{S,0}$ is readily verified. To prove item \ref{symnondeg:1}, we take $v'\in \dom(S^*)$ and $v\in \dom(S)$ and compute that 
$$\omega_{S,0}(v,v'):=\langle S^*v,v'\rangle-\langle v,S^*v'\rangle=\langle S^*v,v'\rangle-\langle Sv,v'\rangle=0,$$
since $v\in \dom(S)$ and $S^*\supseteq S$. This computation, the symmetry condition on $\omega_{S,0}$ and its sesquilinearity implies that $\omega_S$ is a well-defined sesquilinear form satisfying the stated symmetry condition. To show that $\omega_S$ is non-degenerate, if $v\in \dom(S^*)$ satisfies $\omega_S([v],\xi)=0$ for all $\xi\in \check{\mathhil{h}}_S$ then $\langle S^*v,v'\rangle_\Hil=\langle v,S^*v'\rangle_\Hil$ for all $v'\in \dom(S^*)$ so $v\in \dom(S^{**})=\dom(S)$ so $[v]=0$ in $\check{\mathhil{h}}_S$. 

Let us prove item \ref{symnondeg:2}. We use the convention that the inner product is linear in the second leg. For $v_{1,\pm},v_{2,\pm}\in \ker(S^*\pm i)$,
\begin{align*}
\omega_{S}(v_{1,+}+v_{1,-},&v_{2,+}+v_{2,-}) \\
=&\langle S^*(v_{1,+}+v_{1,-}),v_{2,+}+v_{2,-}\rangle-\langle v_{1,+}+v_{1,-},S^*(v_{2,+}+v_{2,-})\rangle\\
=&\langle -iv_{1,+}+iv_{1,-},v_{2,+}+v_{2,-}\rangle-\langle v_{1,+}+v_{1,-},-iv_{2,+}+iv_{2,-}\rangle\\
=&2i\langle v_{1,+},v_{2,-}\rangle-2i\langle v_{1,-},v_{2,+}\rangle \\
=&\omega_{S,vN}\left(\begin{pmatrix}v_{1,+}\\v_{1,-}\end{pmatrix},\begin{pmatrix}v_{2,+}\\v_{2,-}\end{pmatrix}\right).
\qedhere
\end{align*}
\end{proof}

\begin{proposition}
\label{boundarypairingABS}
Let $\mathcal{T}=(T_{\rm min},T^\dagger_{\rm min})$ be a FAP. The sesquilinear form 
$$\omega_{\mathcal{T},0}:\dom(T_{\rm max}^\dagger)\times \dom(T_{\rm max})\to \C, \quad \omega_{\mathcal{T},0}(v_1,v_2):=\langle T_{\rm max}^\dagger v_1,v_2\rangle-\langle v_1,T_{\rm max}v_2\rangle,$$
descends to a well-defined sesquilinear non-degenerate pairing 
$$\omega_{\mathcal{T}}:\check{\mathhil{h}}_{\mathcal{T}^\dagger}\times \check{\mathhil{h}}_{\mathcal{T}}\to \C,$$
satisfying the symmetry condition 
$$\omega_{\mathcal{T}}(\xi_1,\xi_2)=-\overline{\omega_{\mathcal{T}^\dagger}(\xi_2,\xi_1)}.$$
\end{proposition}

\begin{proof}
That $\omega_\mathcal{T}$ is well-defined and satisfies the symmetry condition follows from Proposition \ref{doublingup} and Lemma \ref{symnondeg}. The proof that $\omega_\mathcal{T}$ is non-degenerate goes along the same lines as in the proof of Lemma \ref{symnondeg} and is omitted.
\end{proof}

\subsection{The Cauchy data space and boundary conditions}
\label{app:bspaceandbc}

For a subspace $B\subseteq \check{\mathhil{h}}_T$, we define the realisation $T_{\rm B}$ by 
$$\dom(T_{\rm B}):=\{v\in \dom(T_{\rm max}): \gamma(v)\in B\}.$$

\begin{proposition}
\label{charbodundalda}
Let $(T_{\rm min},T^\dagger_{\rm min})$ be a FAP. The construction $B\mapsto T_{\rm B}$ sets up a one-to-one correspondence of (closed) realisations of $(T_{\rm min},T^\dagger_{\rm min})$ and (closed) subspaces of $\check{\mathhil{h}}_T$.
\end{proposition}

\begin{proof}
A short argument involving the five lemma shows that given a (closed) subspace $B\subseteq \check{\mathhil{h}}_\mathcal{T}$, $T_{\rm B}$ is the unique (closed) realisation $T$ with $\gamma\dom(T)=B$. A similar argument shows that given a (closed) realisation $T$, the subspace $B:=\gamma(\dom(T))\subseteq \check{\mathhil{h}}_\mathcal{T}$ is the unique (closed) subspace $B_0$ for which $T=T_{{\rm B}_0}$.
\end{proof}

\begin{definition}
A subspace $B\subseteq \check{\mathhil{h}}_\mathcal{T}$ is called a \emph{generalised boundary condition} and a closed subspace $B\subseteq \check{\mathhil{h}}_T$ a \emph{boundary condition}. 

If $B$ is a generalised boundary condition, we define its adjoint boundary condition as
$$B^*:=\{w\in \mathhil{h}_{\mathcal{T}^\dagger}: \omega_\mathcal{T}(w,v)=0\quad \forall v\in B\}.$$

\end{definition}

In this terminology, Proposition \ref{charbodundalda} can be phrased as there being a one-to-one correspondence between closed realisations of a formally adjointed closed operator and its boundary conditions. The terminology adjoint boundary condition is motivated by the next proposition.

\begin{proposition}
\label{adjointofafap}
Let $(T_{\rm min},T^\dagger_{\rm min})$ be a FAP and $B$ a generalised boundary condition. Then it holds that 
$$(T_{\rm B})^*=T^\dagger_{{\rm B}^*}.$$
\end{proposition}

\begin{proof}
We first show that $(T_{\rm B})^*\supseteq T^\dagger_{\rm B^*}$. If $v\in \dom T^\dagger_{\rm B^*}$ then for any $w\in \dom T_{{\rm B}}$ we have that 
$$\langle T_{{\rm B}^*}^\dagger v,w\rangle-\langle v,T_{{\rm B}} w\rangle=\langle T_{\rm max}^\dagger v,w\rangle-\langle v,T_{\rm max} w\rangle=\omega_\mathcal{T}(\gamma(w),\gamma(v))=0.$$
Therefore, $v\in \dom(T_{\rm B})^*$ and $(T_{\rm B})^*v=T_{{\rm B}^*}^\dagger v$. 

We prove $(T_{\rm B})^*\subseteq T^\dagger_{{\rm B}^*}$. As $T_{\rm B}\supseteq T_{\min}$, $(T_{\rm B})^*\subseteq T^\dagger_{\max}$ so consider a $v\in \dom(T_{\max}^\dagger)\setminus \dom (T^\dagger_{{\rm B}^*})$, i.e. a $v\in \dom(T_{\max}^\dagger)$ with $\gamma(v)\notin B^*$. Nondegeneracy of $\omega_\mathcal{T}$ implies that there exists a $w\in \dom (T_{{\rm B}})$ with $\omega_\mathcal{T}(\gamma(v),\gamma(w))\neq 0$, or equivalently that there exists a $w\in \dom (T_{{\rm B}})$ with 
$$\langle T_{\rm max}^\dagger v,w\rangle-\langle v,T_{\rm max} w\rangle=\langle (T_{\rm B})^* v,w\rangle-\langle v,T_{\rm B}w\rangle\neq 0.$$
This shows that $v\in \dom(T_{\rm B})^*$ so we conclude that $(T_{\rm B})^*\subseteq T^\dagger_{{\rm B}^*}$.
\end{proof}

\begin{proposition}
\label{adjointofafapsym}
Let $S$ be a closed symmetric operator. 
\begin{enumerate}[(i)]
\item For any generalised boundary condition $B$, $S_{\rm B}^*=S_{{\rm B}^*}$.
\item A boundary condition $B$ for $S$ defines a self-adjoint realisation $S_{\rm B}$ if and only if $B$ is Lagrangian (i.e. $B^*=B$).
\item A boundary condition $B$ for $S$ defines a symmetric realisation $S_{\rm B}$ if and only if $B$ is isotropic (i.e. $B\subseteq B^*$).
\end{enumerate}
\end{proposition}

\begin{proof}
By definition, $S^\dagger_{\rm min}=S_{\rm min}=S$ for the FAP constructed from a closed symmetric operator and therefore the proposition follows from Proposition \ref{adjointofafap}.
\end{proof}

\begin{proposition}
Let $\mathcal{T}=(T_{\rm min},T^\dagger_{\rm min})$ be a FAP and recall the notation from Proposition \ref{doublingup}.  
\begin{enumerate}[(i)]
\item \label{corres:1} There is a one-to-one correspondence between closed realisations of $\mathcal{T}$ and self-adjoint $\mathbb{Z}/2$-graded realisations of $\tilde{T}$ given by 
$$T_{\rm B}\longleftrightarrow \begin{pmatrix} 0& T_{\rm B}\\T^\dagger_{{\rm B}^*}&0\end{pmatrix}.$$
\item \label{corres:2} There is a one-to-one correspondence between closed subspaces $B\subseteq \check{\mathhil{h}}_\mathcal{T}$ and $\mathbb{Z}/2$-graded Lagrangian subspaces $\tilde{B}\subseteq \check{\mathhil{h}}_{\tilde{T}}=\dom(\tilde{T}^*)/\dom(\tilde{T})$ given by 
$$B\longleftrightarrow B^*\oplus B.$$
\end{enumerate}
\end{proposition}

\begin{proof}
By Proposition \ref{adjointofafap} and \ref{adjointofafapsym}, item \ref{corres:1} follows from item \ref{corres:2}. Let us prove item \ref{corres:2}. We note that 
$$\omega_{\tilde{T}}\left(\begin{pmatrix}v_{1}\\v_{2}\end{pmatrix},\begin{pmatrix}v_{1}'\\v_{2}'\end{pmatrix}\right)=\omega_{\mathcal{T}}(v_1,v_2')+\omega_{\mathcal{T}^\dagger}(v_2,v_1'),$$
for $v_1,v_1'\in \check{\mathhil{h}}_{\mathcal{T}^\dagger}$ and $v_2,v_2'\in \check{\mathhil{h}}_\mathcal{T}$. Therefore, given a $B\subseteq \check{\mathhil{h}}_\mathcal{T}$, $\tilde{B}:=B^*\oplus B$ is a $\mathbb{Z}/2$-graded isotropic subspace (i.e. $\tilde{B}\subseteq \tilde{B}^*$). Given $(v_1,v_2)\in \tilde{B}^*$, the condition 
$$\omega_{\tilde{T}}\left(\begin{pmatrix}v_{1}\\v_{2}\end{pmatrix},\begin{pmatrix}v_{1}'\\0\end{pmatrix}\right)=0,$$ 
for all $v_1'\in B^*$ implies that $v_2\in B$ and the condition 
$$\omega_{\tilde{T}}\left(\begin{pmatrix}v_{1}\\v_{2}\end{pmatrix},\begin{pmatrix}0\\v_2'\end{pmatrix}\right)=0,$$ 
for all $v_2'\in B$ implies that $v_1\in B^*$ which proves that $\tilde{B}^*\subseteq \tilde{B}$. We conclude that $\tilde{B}$ is a $\mathbb{Z}/2$-graded Lagrangian subspace. Conversely, given a $\mathbb{Z}/2$-graded Lagrangian subspaces $\tilde{B}\subseteq \check{\mathhil{h}}_{\tilde{T}}$ the fact that $\tilde{B}$ is graded implies that $\tilde{B}=B'\oplus B$ for closed subspaces $B'\subseteq \check{\mathhil{h}}_{\mathcal{T}\dagger}$ and $B\subseteq \check{\mathhil{h}}_\mathcal{T}$ and the fact that $\tilde{B}$ is Lagrangian implies that $B'=B^*$ by an argument as above. We conclude that $B\leftrightarrow B^*\oplus B$ produces a one-to-one correspondence as stated.
\end{proof}

\begin{proposition}
\label{Prop:DomCompABS}
Let $\mathcal{T}=(T_{\rm min},T^\dagger_{\rm min})$ be a FAP and $B_1\subseteq B_2$ be boundary conditions. Then $T_{{\rm B}_1}\subseteq T_{{\rm B}_2}$ and there is a short exact sequence of Hilbert spaces
$$0\to \dom(T_{{\rm B}_1})\to \dom(T_{{\rm B}_2})\to \faktor{B_2}{B_1}\to 0.$$
In particular, the following special cases hold for a boundary condition $B$:
\begin{enumerate}[(i)] 
\item $\faktor{\dom(T_{\rm B})}{\dom(T_{\rm min})}\cong B$, 
\item $\faktor{\dom(T_{\rm max})}{\dom(T_{\rm B})}\cong \faktor{\check{\mathhil{h}}_\mathcal{T}}{B}$.
\end{enumerate}
\end{proposition}

This follows from standard considerations with the open mapping theorem and the five lemma.

\begin{definition}
\label{def:cauchyspacefap}
Let $\mathcal{T}=(T_{\rm min},T^\dagger_{\rm min})$ be a FAP. We define the Hardy space as 
$$\mathcal{C}_\mathcal{T}:=\gamma\ker T_{\rm max}\subseteq \check{\mathhil{h}}_\mathcal{T}.$$
\end{definition}

It is readily verified that the Hardy space fits into a short exact sequence of vector spaces
$$0\to \ker(T_{\rm min})\to \ker(T_{\rm max})\xrightarrow{\gamma} \mathcal{C}_\mathcal{T}\to 0.$$
This also follows from the next proposition.

\begin{proposition}
\label{sesforkernelslsa}
Let $\mathcal{T}=(T_{\rm min},T^\dagger_{\rm min})$ be a FAP and $B$ be a boundary condition. Then there is a short exact sequence of vector spaces
$$0\to \ker(T_{\rm min})\to \ker(T_{{\rm B}})\to \mathcal{C}_\mathcal{T}\cap B\to 0.$$
\end{proposition}

\begin{proof}
This follows from the fact that $\ker(T_{\rm min})=\ker(T_{\rm max})\cap \dom(T_{\min})$ and $\ker(T_{\rm B})=\ker(T_{\rm max})\cap \dom(T_{\rm B})$.
\end{proof}

Let us briefly comment on the relation to boundary triplets. We compare our result to the setup of \cite{grubb68}. 
For more details, history and context for boundary triplets, see also \cite{malamud,behrndtetal}.
 We remark that if a FAP $\mathcal{T}=(T_{\rm min},T^\dagger_{\rm min})$ admits a realisation $T$ which is invertible, which is equivalent to  $0 \in \res(T) = \res(T^*)^{\conj}$, then $T^*$ is also invertible and therefore, $\ker T_{\rm min}=\ker T_{\rm min}^\dagger=0$.
Thus, $\gamma|: \ker T_{\rm max}\to \mathcal{C}_\mathcal{T}$ and $\gamma|: \ker T_{\rm max}^\dagger\to \mathcal{C}_{\mathcal{T}^\dagger}$ are vector space isomorphisms. 

\begin{proposition}
Let $\mathcal{T}=(T_{\rm min},T^\dagger_{\rm min})$ be a FAP admitting a realisation $T_\beta$ which is invertible. Define $\mathrm{pr}_\beta:=T_\beta^{-1}T_{\rm max}:\dom(T_{\rm max})\to \dom(T_\beta)$ and $\mathrm{pr}_{\beta^\dagger}:=(T_\beta^*)^{-1}T_{\rm max}^\dagger:\dom(T_{\rm max}^\dagger)\to \dom(T_\beta^*)$. 
Then, the  following hold. 
\begin{enumerate}[(i)] 
\item  \label{invfap:1} 
It holds that 
\begin{align*}
&\dom(T_{\rm max})=\dom(T_\beta)+ \ker T_{\rm max},\quad \text{and}  \\
&\dom(T_{\rm max}^\dagger)=\dom(T_\beta^*)+ \ker T_{\rm max}^\dagger,
\end{align*}
and the mappings
\begin{align*}
\dom(T_{\rm max})&\xrightarrow{\mathrm{pr}_\beta\oplus \gamma(1-\mathrm{pr}_\beta)} \dom(T_\beta)\oplus \mathcal{C}_\mathcal{T}\quad \mbox{and}\\
\dom(T_{\rm max}^\dagger)&\xrightarrow{\mathrm{pr}_{\beta^\dagger}\oplus \gamma(1-\mathrm{pr}_{\beta^\dagger})} \dom(T_\beta^*)\oplus \mathcal{C}_{\mathcal{T}^\dagger},
\end{align*}
are isomorphisms. 
\item \label{invfap:2} The isomorphisms from \ref{invfap:1} induce isomorphisms
\begin{align*}
\check{\mathhil{h}}_{\mathcal{T}}\cong \check{\mathhil{h}}_{\mathcal{T},\beta}\oplus \mathcal{C}_{\mathcal{T}}\quad \mbox{and}\quad \check{\mathhil{h}}_{\mathcal{T}^\dagger}\cong \check{\mathhil{h}}_{\mathcal{T}^\dagger,\beta}\oplus \mathcal{C}_{\mathcal{T}^\dagger},
\end{align*}
where $\check{\mathhil{h}}_{\mathcal{T},\beta}:=\gamma\dom(T_\beta)$ and $\check{\mathhil{h}}_{\mathcal{T}^\dagger,\beta}:=\gamma\dom T_\beta^*$ are closed subspaces. 
\item \label{invfap:3} There is a one-to-one correspondence between boundary conditions $B$ of $\mathcal{T}$ and triples $(W,V,\mathfrak{T})$, where $V\subseteq \mathcal{C}_\mathcal{T}$ and $W\subseteq \mathcal{C}_{\mathcal{T}^\dagger}$ are closed subspaces and $\mathfrak{T}:V\dashrightarrow W$ is a closed densely defined operator.
\end{enumerate}
\end{proposition}

\subsection{Characterising Fredholm boundary conditions}

In this section, we show how to characterise Fredholm boundary conditions for a FAP satisfying the extra condition that $T_{\rm min}$ and $T_{\rm min}^\dagger$ have closed range and $\ker(T_{\rm min})$ and $\ker(T_{\rm min}^\dagger)$ being finite-dimensional.  For the sake of terminology, we define this notion as follows.

\begin{definition}
\label{efonedoenfapfredholm}
Let $\mathcal{T}=(T_{\rm min},T^\dagger_{\rm min})$ be a FAP. We say that $\mathcal{T}$ has closed range if $T_{\rm min}$ and $T_{\rm min}^\dagger$ have closed range. We say that $\mathcal{T}$ is a Fredholm FAP if it has closed range and $\ker(T_{\rm min})$ and $\ker(T_{\rm min}^\dagger)$ are finite-dimensional.
\end{definition}

The reader should note that by the closed range theorem, $\mathcal{T}=(T_{\rm min},T^\dagger_{\rm min})$ is a FAP with closed range if and only if $T_{\rm min}$ and $T_{\rm max}$ have closed range. For the purpose of studying elliptic differential operators, the following sufficient condition for closed range and Fredholmness will be useful.

\begin{proposition}
\label{closedragneandofiffni}
Let $\mathcal{T}=(T_{\rm min},T^\dagger_{\rm min})$ be a FAP. Assume that the inclusion $\iota:\dom(T_{\rm min})\hookrightarrow \Hil_1$ is a compact inclusion. Then, the following holds.
\begin{enumerate}[(i)] 
\item $\ker(T_{\rm min})$ is finite-dimensional.
\item $T_{\rm min}$ and $T_{\rm max}^\dagger$ have closed range.
\item $\Hil_1 = \ker(T_{\min}) \oplus^\perp \ran(T_{\max}^\dagger)$ and $\Hil_2 = \ker(T_{\max}^\dagger) \oplus^\perp \ran(T_{\min})$.
\end{enumerate}
In particular, if both $\dom(T_{\rm min})\hookrightarrow \Hil_1$ and $\dom(T_{\rm min}^\dagger)\hookrightarrow \Hil_2$ are compact inclusions, then $\mathcal{T}$ is a Fredholm FAP and we additionally have that $\Hil_2 = \ker(T_{\min}^\dagger) \oplus^\perp \ran(T_{\max})$.
\end{proposition}

\begin{proof}
For $u\in \dom(T_{\rm min})$, we have the estimate
$$ \norm{u}_{T_{\rm min}} \simeq \norm{\iota u}_{\Hil_1} + \norm{T_{\rm min} u}_{\Hil_2}.$$
Then, by Proposition A3 in \cite{BB} (ii), we conclude that $T_{\rm min}$ has finite dimensional kernel and closed image.
The fact that $\ran(T_{\rm min}^*)=\ran(T_{\rm max}^\dagger)$ is closed follows from the closed range theorem.
\end{proof}

\begin{remark}
Note that, even without the assumption  $\dom(T_{\min}^\dagger) \embed \Hil_2$ is a compact embedding it still holds that $\Hil_2 = \ker(T^\dagger_{\min}) \oplus \close{\ran(T_{\max})}$.
The compactness of the embedding  ensures that $\ran(T_{\max}) = \close{\ran(T_{\max})}$ as well as the finite dimensionality of $\ker(T_{min}^\dagger)$.
\end{remark}

\begin{lemma}
\label{Lem:Quotient}
Let $K \subset Z$ be a closed subspace of a Banach space $Z$ and let $X \subset Z$ a subspace.
If $K \intersect X$ is closed, the map $\Phi: \faktor{X}{(K \intersect X)} \to \faktor{Z}{K}$ defined by $\Phi[u]_{\faktor{X}{(K \intersect X)}} = [u]_{\faktor{Z}{K}}$ is  well-defined and injective.    
\end{lemma}
\begin{proof}
To show the map is well-defined, suppose $[x]_{\faktor{X}{(K \intersect X)}} = [y]_{\faktor{X}{(K \intersect X)}}$.
Then, $x - y \in (K \intersect X) \subset K$ and therefore, $[x]_{\faktor{Z}{K}} = [y]_{\faktor{Z}{K}}$.

For injectivity, suppose that $\phi[u]_{\faktor{X}{(K \intersect X)}} = 0$, which is the same as $[u]_{\faktor{Z}{K}} = 0$.
Then, $u = u - 0 \in K$ and since $u \in X$, we obtain that $u \in K \intersect X$.
That is, $u \in (K \intersect X)$ and therefore, $[u]_{\faktor{X}{(K \intersect X)}} = 0$.
\end{proof}

\begin{proposition}
\label{Prop:GBcFacABS}
Let $\mathcal{T}=(T_{\rm min},T^\dagger_{\rm min})$ be a FAP and $B$ a generalised boundary condition. The map 
$$\Phi_{\rm B}: \faktor{\dom(T_{\rm B})}{\ker(T_{\rm B})} \to \ran(T_{\rm B}),\qquad \Phi_{\rm B} [u] := T_{\max}u,$$ 
is a bounded bijection.
Moreover, the map 
$$ \faktor{\dom(T_{\rm B})}{\ker(T_{\rm B})}  \to \faktor{\dom(T_{\max})}{\ker(T_{\max})}$$
given by 
$$[u]_{\faktor{\dom(T_{\rm B})}{\ker(T_{\rm B})}} \mapsto [u]_{\faktor{\dom(T_{\max})}{\ker(T_{\max})}}$$
is well-defined and is a bounded injection. 
\end{proposition}

\begin{proof}
To show that this map is bounded, we note that 
\begin{align*}
\norm{\Phi_{\rm B} [u]} & = \norm{T_{\max}u} \leq \inf_{x \in \ker(T_{\rm B})} \norm{u - x} + \norm{T_{\max} u} \\
	&= \inf_{x \in \ker(T_{\rm B})} \norm{u - x} + \norm{T_{{\rm B}} u} = \norm{[u]}_{\faktor{\dom(T_{\rm B})}{\ker(T_{\rm B})}}.
\end{align*}
Surjectivity is clear.
For injectivity, note that if $\Phi_{\rm B}[u] = 0$, then $T_{\max} u = 0$, so $u \in \ker(T_{\max})$. 
But since $u \in \dom(T_{\rm B})$, we have that $u \in \ker(T_{\rm B}) = \dom(T_{\rm B}) \intersect \ker(T_{max})$.

The map
$$ \faktor{\dom(T_{\rm B})}{\ker(T_{\rm B})}  \to \faktor{\dom(T_{\max})}{\ker(T_{\max})}$$
is a well-defined injection since $\ker(T_{\rm B}) = \dom(T_{\rm B}) \intersect \ker(T_{\max})$ by invoking Lemma \ref{Lem:Quotient}. 
\end{proof}

\begin{lemma}
\label{Lem:AddBcABS} 
Let $\mathcal{T}=(T_{\rm min},T^\dagger_{\rm min})$ be a FAP and $Y \subset \dom(T_{\max})$ a subspace. Consider the generalised boundary condition $B := \gamma(Y)\subseteq \check{\mathhil{h}}_\mathcal{T}$. 
Then, $\dom(T_{{\rm B}}) = \dom(T_{\min}) + Y$.
\end{lemma}

\begin{proof}
Since $Y, \dom(T_{\min}) \subset \dom(T_{{\rm B}})$, we have that $\dom(T_{{\rm B}}) \supset \dom(T_{\min}) + Y$.

To prove the opposite containment, let $u \in \dom(T_{\rm B})$, which is if and only if $\gamma(u) \in B = \gamma(Y)$. 
Therefore, there exists $y \in Y$ such that $\gamma(u) = \gamma(y)$.
So, $u - y \in \dom(T_{\min})$ and $y \in Y$ so $u = (u - y) +y \in \dom(T_{\min})  +Y$.
\end{proof}

\begin{lemma}
\label{Lem:RanABS} 
Let $\mathcal{T}=(T_{\rm min},T^\dagger_{\rm min})$ be a FAP. For a generalised boundary condition $B$, we have that $\ran(T_{\rm B}) = \ran(T_{{\rm B} + \Ca_\mathcal{T}})$ and that $\dom(T_{{\rm B} + \Ca_\mathcal{T}}) = \dom(T_{\rm B}) + \ker(T_{\max})$. 
\end{lemma}

\begin{proof}
It is clear that $\ran(T_{\rm B}) \subset \ran(T_{{\rm B} + \Ca_\mathcal{T}})$. 
Set $Y = \dom(T_{\rm B}) + \ker(T_{\max})$ and apply Lemma \ref{Lem:AddBcABS} to obtain that $\dom(T_{{\rm B} + \Ca_\mathcal{T}}) = \dom(T_{\rm B}) + \ker(T_{\max})$, since $\dom(T_{\min}) \subset \dom(T_{\rm B})$.
Then, for $u \in \ran(T_{{\rm B} + \Ca_\mathcal{T}})$, we can find $v \in \dom(T_{{\rm B} + \Ca_\mathcal{T}})$ such that $u = T_{\max}v$. 
By what we have just proved, we have $v = x_{\rm B} + y \in \dom(T_{\rm B}) + \ker(T_{\max})$ and therefore, $u = T_{\rm max}v = T_{\rm max}x_{{\rm B}}$.
I.e., $u \in \ran(T_{\rm B})$.
\end{proof}

\begin{lemma}
\label{Lem:GbcQuoABS}
Let $\mathcal{T}=(T_{\rm min},T^\dagger_{\rm min})$ be a FAP. For a generalised boundary condition $B'$ such that $\Ca_\mathcal{T} \subset B'$, we have that  the inclusion map
$$\faktor{\dom(T_{{\rm B}'})}{\ker(T_{\max})} \embed \faktor{\dom(T_{\max})}{\ker(T_{\max})}$$ 
is an isometry.
If $\mathcal{T}$ has closed range, the normed space  $\faktor{\dom(T_{{\rm B}'})}{\ker(T_{\max})}$ is closed if and only if $B'$ is a boundary condition, i.e., closed in $\check{\mathhil{h}}_\mathcal{T}$.
\end{lemma}

\begin{proof}
Note that the induced norm on $\faktor{\dom(T_{{\rm B}'})}{\ker(T_{\max})}$ is 
$$ \norm{[u]}^2 = \inf_{x \in \ker(T_{\max})} \norm{u - x}^2 + \norm{T_{{\rm B}'}u}^2 =  \inf_{x \in \ker(T_{\max})} \norm{u - x}^2 + \norm{T_{\max}u}^2,$$
where the norm after the ultimate equality is the norm on $\faktor{\dom(T_{\max})}{\ker(T_{\max})}$.
If $B'$ is closed, then $\dom(T_{{\rm B}'})$ is closed and therefore, $\faktor{\dom(T_{{\rm B}'})}{\ker(T_{\max})}$ is a Banach space, i.e. it is closed in $\faktor{\dom(T_{\max})}{\ker(T_{\max})}$.
On the other hand, if $\faktor{\dom(T_{{\rm B}'})}{\ker(T_{\max})}$ is closed, then $\dom(T_{{\rm B}'})$ is closed since $\dom(T_{{\rm B}'}) \cong \faktor{\dom(T_{{\rm B}'})}{\ker(T_{\max})} \oplus \ker(T_{\max})$.
\end{proof}

\begin{lemma}
\label{Lem:ClosedXiABS} 
Let $\mathcal{T}=(T_{\rm min},T^\dagger_{\rm min})$ be a FAP with closed range. For a generalised boundary condition  $B$, $\ran(T_{\rm B})$ is closed if and only if the image of the map  $\Xi: \dom(T_{\rm B}) \to \faktor{\dom(T_{\max})}{\ker(T_{\max})}$ given by
$$ \Xi: \dom(T_{\rm B}) \to \faktor{\dom(T_{\rm B})}{\ker(T_{\rm B})} \embed \faktor{\dom(T_{\max})}{\ker(T_{\max})}$$
has closed range. 
\end{lemma}  

\begin{proof}
We first start by noting that, by Proposition \ref{Prop:GBcFacABS}, the open mapping theorem yields that the we have that the map  $\Phi_{\max}$ as defined in Proposition \ref{Prop:GBcFacABS} is homeomorphism between $\faktor{\dom(T_{\max})}{\ker(T_{\max})}$ and $\ran(T_{\max})$.

Moreover, by the definition of $\Phi_{\rm B}$, we have the following commuting diagram.
\begin{equation*}
\begin{tikzcd}
 \faktor{\dom(T_{\rm B})}{\ker(T_{\rm B})}\arrow{dd}{\Phi_{\rm B}} \arrow{rr}{\iota_1} &  & \faktor{\dom(T_{\max})}{\ker(T_{\max})} \arrow{dd}{\Phi_{\max}}  \\
	& 	& \\ 
 \ran(T_{\rm B}) \arrow{rr}{\iota_2} 	&    & \ran(T_{\max})
\end{tikzcd}
\end{equation*}
Note that $\ran(T_{\rm B}) \subset \ran(T_{\max})$, so $\iota_2$ is simply the inclusion map.
Since $\Phi_{\max}$ is a Banach space isomorphism, we have that 
$$\Phi_{\max}^{-1}\rest{\ran(T_{\rm B})}: \ran(T_{\rm B}) \to \iota_1  \cbrac{  \faktor{\dom(T_{\rm B})}{\ker(T_{\rm B})} }$$
is bounded.
Therefore, $\ran(T_{\rm B})$ is closed if and only if $\iota_1 \cbrac{ \faktor{\dom(T_{\rm B})}{\ker(T_{\rm B})}}$ is closed.
The latter is, indeed, the image of $\Xi$ so this proves the lemma.
\end{proof}

\begin{theorem}
\label{Thm:ClosedRangeCharABS} 
Let $\mathcal{T}=(T_{\rm min},T^\dagger_{\rm min})$ be a FAP with closed range. For a generalised boundary condition $B$, $\ran(T_{\rm B})$ is closed if and only if $B + \Ca_\mathcal{T}$ is a boundary condition.
\end{theorem}

\begin{proof}
Lemma \ref{Lem:RanABS} yields that $\ran(T_{\rm B}) = \ran(T_{{\rm B} + \Ca_\mathcal{T}})$.
Then, Lemma \ref{Lem:ClosedXiABS} says that this is closed if and only if the range of the map $\Xi$ is closed. 
Since $ \Ca_\mathcal{T} \subset B + \Ca_\mathcal{T}$, it is easy to see that the range of $\Xi$ is precisely $\faktor{\dom(T_{{\rm B} + \Ca_\mathcal{T}})}{\ker(T_{\max})}$.
Therefore, by Lemma \ref{Lem:GbcQuoABS}, we obtain that $\ran(T_{\rm B})$ is closed if and only if $B + \Ca_\mathcal{T}$ is closed, i.e., that $B + \Ca_\mathcal{T}$ is a boundary condition. 
\end{proof}

\begin{corollary}
\label{hardyclosedcor}
Let $\mathcal{T}=(T_{\rm min},T^\dagger_{\rm min})$ be a FAP with closed range and $B$ a generalised boundary condition. Then the Hardy space $\mathcal{C}_\mathcal{T}$ is closed in $\check{\mathhil{h}}_\mathcal{T}$. 
\end{corollary}

\begin{proof}
The corollary follows directly from Theorem \ref{Thm:ClosedRangeCharABS} with $B=0$.
\end{proof}

\begin{lemma}
\label{Lem:L2RangeABS}
Let $\mathcal{T}=(T_{\rm min},T^\dagger_{\rm min})$ be a FAP with closed range. When $\ran(T_{\rm B})$ is closed, 
$$ \faktor{\Hil_2}{\ran(T_{\rm B})} \isomorphic \faktor{\Hil_2}{\ran(T_{\max})}  \oplus \faktor{\check{\mathhil{h}}_\mathcal{T}}{(B + \Ca_\mathcal{T})}\isomorphic \ker(T_{\min}^\dagger) \oplus \faktor{\check{\mathhil{h}}_\mathcal{T}}{(B + \Ca_\mathcal{T})}.$$
\end{lemma}

\begin{proof}
By assumption $\ran(T_{\max})$ is closed and therefore,
\begin{align*} 
\Hil_2
	&\isomorphic  \faktor{\Hil_2}{\ran(T_{\max})} \oplus \ran(T_{\max}) \\
	&\isomorphic   \faktor{\Hil_2}{\ran(T_{\max})} \oplus \faktor{\ran(T_{\max})}{\ran(T_{\rm B})} \oplus \ran(T_{\rm B}).
\end{align*}
From Lemma \ref{Lem:RanABS}, we have that $\ran(T_{\rm B}) = \ran(T_{{\rm B} + \Ca_\mathcal{T}})$ and $T_{{\rm B} + \Ca_\mathcal{T}}$ is a closed operator from Lemma \ref{Lem:GbcQuoABS}.
Thus, note that $\ker(T_{{\rm B} + \Ca_\mathcal{T}}) = \ker(T_{\max})$ and therefore, 
$$
\faktor{\dom(T_{{\rm B} + \Ca_\mathcal{T}})}{\ker(T_{\max})} =  \faktor{\dom(T_{{\rm B} + \Ca_\mathcal{T}})}{\ker(T_{{\rm B} + \Ca_\mathcal{T}})}  \cong \ran(T_{{\rm B} + \Ca_\mathcal{T}}) = \ran(T_{\rm B}).$$
Thus, 
$$ \faktor{\ran(T_{\max})}{\ran(T_{\rm B})} \cong \faktor{\dom(T_{\max})}{\dom(T_{{\rm B} + \Ca_\mathcal{T}})} \cong \faktor{\check{\mathhil{h}}_\mathcal{T}}{(B + \Ca_\mathcal{T})}$$
from Proposition \ref{Prop:DomCompABS}.
\end{proof}

\begin{proposition}
\label{Prop:FinCharABS} 
Let $\mathcal{T}=(T_{\rm min},T^\dagger_{\rm min})$ be a Fredholm FAP. For a general boundary condition $B$, the following are equivalent: 
\begin{enumerate}[(i)] 
\item \label{Prop:FinChar1ABS}
	 $\ran(T_{\rm B})$ has finite algebraic codimension in $\Hil_2$;
\item \label{Prop:FinChar2ABS}
	 $T_{\rm B}$ has closed range and $\ran(T_{\rm B})^\perp$ is finite dimensional;
\item \label{Prop:FinChar3ABS} 
	 $T_{\rm B}$ has closed range and $\ker(T_{\rm B}^\ast)$ is finite dimensional;
\item \label{Prop:FinChar4ABS} 
	 $B + \Ca_\mathcal{T}$ has finite algebraic codimension in $\check{\mathhil{h}}_\mathcal{T}$; 
\item \label{Prop:FinChar5ABS} 
	 $B + \Ca_\mathcal{T}$ is closed  and has finite algebraic codimension in in $\check{\mathhil{h}}_\mathcal{T}$.
\end{enumerate} 
If any of these hold, then 
$$\ker (T_{\rm B}^\ast) \cong {\faktor{\Hil_2}{\ran(T_{\rm B})}} \cong \ker(T_{\min}^\dagger) \oplus {\faktor{\check{\mathhil{h}}_\mathcal{T}}{(B + \Ca_\mathcal{T})}}.$$
\end{proposition}

\begin{proof}
The equivalence of \ref{Prop:FinChar1ABS}-\ref{Prop:FinChar3ABS},
as well as \ref{Prop:FinChar4ABS}-\ref{Prop:FinChar5ABS} are standard.
The statements \ref{Prop:FinChar5ABS} and \ref{Prop:FinChar2ABS} are equivalent because of Lemma \ref{Lem:RanABS} coupled with Lemma \ref{Lem:L2RangeABS}.
The formula follows from the fact that $\ran(T_{\rm B})^\perp \isomorphic \ker(T_{\rm B}^\ast)$ along with Lemma \ref{Lem:L2RangeABS}.  
\end{proof}

Before proceeding with studying Fredholm properties of realisations, we make a remark about the maximal domain in terms of the Hardy space and a suitable realisation.

\begin{corollary}
\label{charcomplements}
Let $\mathcal{T}=(T_{\rm min},T^\dagger_{\rm min})$ be a FAP with closed range and $B$ be a boundary condition. Then $B$ is a complement to $\mathcal{C}_\mathcal{T}$ in $\check{\mathhil{h}}_\mathcal{T}$ if and only if $T_{\rm B}$ has closed range and satisfies that 
$$\ker(T_{\rm min})=\ker (T_{\rm B})\quad \mbox{and}\quad \ker(T_{\rm min}^*)=\ker (T_{{\rm B}^*}^\dagger).$$
In this case, 
$$\dom(T_{\rm max})=\dom(T_{\rm B})\oplus_{\ker T_{\rm min}}\ker (T_{\rm max}).$$
\end{corollary}

The reader should note that closed complements of the Hardy space in $\check{\mathhil{h}}_\mathcal{T}$ always exist for a FAP with closed range since $\check{\mathhil{h}}_\mathcal{T}$ is a Hilbert space and the Hardy space is closed by Corollary \ref{hardyclosedcor}.

\begin{proof}
Assume that $B$ is a complement to $\mathcal{C}_\mathcal{T}$ in $\check{\mathhil{h}}_\mathcal{T}$. The operator $T_{\rm B}$ has closed range by Theorem \ref{Thm:ClosedRangeCharABS}. Let $P:\check{\mathhil{h}}_\mathcal{T}\to \mathcal{C}_\mathcal{T}$ denote the projection along $B$. We then arrive at a short exact sequence
\begin{equation}
\label{sesfortbcomplelem}
0\to \dom(T_{\rm B})\to \dom(T_{\rm max})\xrightarrow{P\circ \gamma} \mathcal{C}_\mathcal{T}\to 0.
\end{equation}
In particular, we have the short exact sequence 
$$0\to \ker(T_{\rm B})\to \ker(T_{\rm max})\xrightarrow{\gamma} \mathcal{C}_\mathcal{T}\to 0,$$
showing that $\ker(T_{\rm min})=\ker (T_{\rm B})$. Moreover, the short exact sequence \eqref{sesfortbcomplelem} implies that the map $\Phi_{\rm B}$ from Proposition \ref{Prop:GBcFacABS} is a bijection, and so $\ran(T_{\rm B})=\ran(T_{\max})$. Since $B$ is a complement to $\mathcal{C}_\mathcal{T}$, $B+\mathcal{C}_\mathcal{T}$ is closed, so $\ran(T_{\rm B})=\ran(T_{\max})$ implies that $\ker(T_{\rm min}^*)=\ker (T_{{\rm B}^*}^\dagger)$.

Assume now that $T_{\rm B}$ has closed range and satisfies that 
$$\ker(T_{\rm min})=\ker (T_{\rm B})\quad \mbox{and}\quad \ker(T_{\rm min}^*)=\ker (T_{{\rm B}^*}^\dagger).$$
The first equality ensures that $B\cap\mathcal{C}_\mathcal{T}=0$. The second equality and Proposition \ref{Prop:FinCharABS} implies that $B+\mathcal{C}_\mathcal{T}=\check{\mathhil{h}}_\mathcal{T}$. Therefore $B$ is a complement to $\mathcal{C}_\mathcal{T}$.

The equality $\dom(T_{\rm max})=\dom(T_{\rm B})\oplus_{\ker T_{\rm min}}\ker (T_{\rm max})$ follows from Equation \eqref{sesfortbcomplelem}.
\end{proof}

Now, we characterise Fredholm boundary conditions via the language of Fredholm pairs. 

\begin{theorem}
\label{Thm:FredCharABS}
Let $\mathcal{T}=(T_{\rm min},T^\dagger_{\rm min})$ be a Fredholm FAP and $B$ be a boundary condition. The following are equivalent: 
\begin{enumerate}[(i)]
\item \label{Thm:FredChar1ABS} 
$(B, \Ca_\mathcal{T})$ is a Fredholm pair in $\check{\mathhil{h}}_\mathcal{T}$; 
\item \label{Thm:FredChar2ABS}  
$T_{\rm B}$ is a Fredholm operator.  
\end{enumerate}
If either of these equivalent conditions hold, we have that
$$ B^\ast \cap \Ca_{\mathcal{T}^\dagger} \cong \faktor{\check{\mathhil{h}}_\mathcal{T}}{(B + \Ca_\mathcal{T})}.$$
Moreover, 
$$\indx(T_{\rm B}) = \indx(B, \Ca_\mathcal{T}) + \dim \ker (T_{\min}) - \dim \ker (T_{\min}^\dagger).$$
\end{theorem}

\begin{proof}
Since $B$ is closed, $T_{\rm B}$ is a closed operator and $\ker(T_{\rm B})$ is a closed subspace of $\Hil_1$. If \ref{Thm:FredChar1ABS} holds, then $B + \Ca_\mathcal{T}$ is closed and by Theorem \ref{Thm:ClosedRangeCharABS}, we have that $\ran(T_{\rm B})$ is closed. Similarly, if \ref{Thm:FredChar1ABS} holds then $B + \Ca_\mathcal{T}$ has finite codimension and $\ran(T_{\rm B})^\perp$ has finite codimension in $\Hil_2$ by Proposition \ref{Prop:FinCharABS}.
Moreover, $B \cap \Ca_\mathcal{T}$ is finite dimensional, so Proposition \ref{sesforkernelslsa} yields that $\ker(T_{\rm B})$ is finite dimensional and therefore, $T_{\rm B}$ is Fredholm.

On the other hand if \ref{Thm:FredChar2ABS} holds, i.e., $T_{\rm B}$ is Fredholm, then $\ran(T_{\rm B}) = \ran(T_{{\rm B} + \Ca})$ is closed, so Theorem \ref{Thm:ClosedRangeCharABS} implies that $B +\Ca_\mathcal{T}$ is closed. Moreover, by Proposition \ref{sesforkernelslsa} we have that $B \cap \Ca_\mathcal{T}  \cong \faktor{\ker(T_{\rm B})}{\ker(T_{\min})}$ is finite dimensional.
Moreover $\ran(T_{\rm B})^\perp$ has finite codimension and so by Proposition \ref{Prop:FinCharABS} \ref{Prop:FinChar5ABS}, $B + \Ca_\mathcal{T}$ has finite codimension in $\check{\mathhil{h}}_\mathcal{T}$.
That is, $(B,\Ca_\mathcal{T})$ is a Fredholm pair in $\check{\mathhil{h}}_\mathcal{T}$.

From a similar calculation, we obtain that 
$$ \ker(T_{\rm B}^\ast) = \ker(T^\dagger_{{\rm B}^\ast}) \cong \ker(T_{\min}^\dagger) \oplus (B^\ast \cap \Ca_{\mathcal{T}^\dagger}).$$
By the formula in Proposition \ref{Prop:FinCharABS}, we obtain the desired formula.
Plugging this into the index, we obtain
\begin{align*}
\indx(T_{\rm B}) &= \dim \ker (T_{\rm B}) - \dim \ker(T_{\rm B}^\ast)  \\
	&= \dim \ker (T_{\min}) + \dim(B \cap \Ca_\mathcal{T}) - \dim \ker (T_{\min}^\dagger) - \dim (B^\ast \cap \Ca_{\mathcal{T}^\dagger}) \\ 
	&= \indx(B,\Ca_\mathcal{T}) + \dim\ker (T_{\min}) - \dim \ker (T_{\min}^\dagger).
\qedhere
\end{align*}
\end{proof}

We say that a boundary condition of a FAP is a Fredholm boundary condition if the associated realisation is a Fredholm operator. 

\begin{corollary}
\label{Cor:FredInd}
Let $\mathcal{T}=(T_{\rm min},T^\dagger_{\rm min})$ be a Fredholm FAP. Let $B_1\subseteq B_2$ be generalised boundary conditions with $B_1$ being a Fredholm boundary condition. Then, $B_2$ is a Fredholm boundary condition if and only if $B_2/B_1$ is finite dimensional. In this case, the following index formula holds:
\begin{equation*}
\label{Eq:IndexABS}
\indx(T_{{\rm B}_2}) = \indx(T_{{\rm B}_1}) + \dim\faktor{B_2}{B_1}.
\end{equation*}
\end{corollary}

\begin{proof}
Assume that $B_1$ has finite codimension in $B_2$, so clearly $B_2$ will be closed. It is easy to see that $B_2 \cap \Ca_\mathcal{T}$ is finite dimensional and since $B_1 \subset B_2$, we have that 
$$\dim \faktor{\check{\mathhil{h}}_\mathcal{T}}{(B_2+\mathcal{C}_\mathcal{T})} \leq \dim \faktor{\check{\mathhil{h}}_\mathcal{T}}{(B_1+\mathcal{C}_\mathcal{T})}.$$
That is, $(B_2, \Ca_\mathcal{T})$ is a Fredholm pair and therefore, we have that $B_2$ is a Fredholm boundary condition.

Assume now that $B_2$ is a Fredholm boundary condition, i.e. that $B_2$ is closed and $(B_2, \Ca_\mathcal{T})$ is a Fredholm pair. Note that 
\begin{align*}
\check{\mathhil{h}}_\mathcal{T} 
	&\cong (B_2 + \Ca_\mathcal{T}) \oplus \faktor{\check{\mathhil{h}}_\mathcal{T}}{(B_2 + \Ca_\mathcal{T})}  \\
	&\cong \faktor{(B_2 + \Ca_\mathcal{T})}{(B_1 + \Ca_\mathcal{T})} \oplus (B_1 + \Ca_\mathcal{T}) \oplus \faktor{\check{\mathhil{h}}_\mathcal{T}}{(B_2 + \Ca_\mathcal{T})}
\end{align*}
and therefore, 
$$\faktor{\check{\mathhil{h}}_\mathcal{T}}{(B_1 + \Ca_\mathcal{T})} \cong  \faktor{(B_2 + \Ca_\mathcal{T})}{(B_1 + \Ca_\mathcal{T})}  \oplus \faktor{\check{\mathhil{h}}_\mathcal{T}}{(B_2 + \Ca_\mathcal{T})}\cong \faktor{B_2}{B_1}\oplus \faktor{\check{\mathhil{h}}_\mathcal{T}}{(B_2 + \Ca_\mathcal{T})}.$$
Since $(B_2, \Ca_\mathcal{T})$ and $(B_1, \Ca_\mathcal{T})$ are Fredholm pairs, $\faktor{\check{\mathhil{h}}_\mathcal{T}}{(B_2 + \Ca_\mathcal{T})}$ and $\faktor{\check{\mathhil{h}}_\mathcal{T}}{(B_1 + \Ca_\mathcal{T})}$ are finite-dimensional and we conclude that $\faktor{B_2}{B_1}$ is finite dimensional. 

For the  index formula, we note that it suffices to prove that 
$$\indx(B_2, \Ca_\mathcal{T}) = \indx(B_1,\Ca_\mathcal{T}) + \dim \faktor{B_2}{B_1}.$$
This identity follows from the computation 
\begin{align*}
\indx(B_2, \Ca_\mathcal{T}) 
	&= \dim (B_2 \cap \Ca_\mathcal{T}) - \dim \faktor {\check{\mathhil{h}}_\mathcal{T}}{(B_2 + \Ca_\mathcal{T})} \\
	&= \dim (B_1 \cap \Ca_\mathcal{T}) + \dim \faktor{(B_2 \cap \Ca_\mathcal{T})}{(B_1 \cap \Ca_\mathcal{T})}  \\
	&\qquad- \dim \faktor{\check{\mathhil{h}}_\mathcal{T}}{(B_1 + \Ca_\mathcal{T})} + \dim\faktor{(B_2 \cap \Ca_\mathcal{T})^{\perp, B_2}}{(B_1 \cap \Ca_\mathcal{T})^{\perp, B_1}} \\
	&= \indx(B_1, \Ca_\mathcal{T}) + \dim \faktor{(B_2 \cap \Ca_\mathcal{T})}{(B_1 \cap \Ca_\mathcal{T})} + \dim\faktor{(B_2 \cap \Ca_\mathcal{T})^{\perp, B_2}}{(B_1 \cap \Ca_\mathcal{T})^{\perp, B_1}}.
\end{align*}
This concludes the proof.
\end{proof}

In addition to Theorem \ref{Thm:FredCharABS}, a useful way to obtain Fredholm boundary conditions is given by the following Proposition which follows from \cite[Proposition A.1]{BB}. 
It is adapted from \cite[Corollary 8.8]{BB}, although \cite{BB} only considered the context of regular boundary conditions. The reader can, in the context of elliptic boundary value problems, compare the construction of the next proposition to standard constructions in the Boutet de Monvel calculus \cite{boutetdemonvel,rempelschulze,gerdsgreenbook}. Again, the provenance of the next result dates back further than \cite{BB} and can be found in for instance \cite[Lemma 4.3.1]{gerdsgreenbook}.

\begin{proposition}
\label{Prop:CompExistsABS} 
Let $\mathcal{T}=(T_{\rm min},T^\dagger_{\rm min})$ be a Fredholm FAP with a boundary condition $B$. Choose a closed complementary subspace $\check{C}$ of $B$ in $\check{\mathhil{h}}_\mathcal{T}$ with projection $\check{P}: \check{\mathhil{h}}_\mathcal{T} \to \check{C}$ defined along $B$ (for instance the orthogonal complement).
The map  
$$\check{L}_{\check{P}}: \dom(T_{\max}) \to \begin{matrix}\Hil_2\\ \oplus\\ \check{C}\end{matrix}, \quad \check{L}_{\check{P}}u = \begin{pmatrix} T_{\max} u \\\check{P} \gamma u\end{pmatrix},$$
 satisfies: 
\begin{enumerate}[(i)] 
\item $T_{\rm B}$ has closed range if and only if $\check{L}_{\check{P}}$ has closed range.
\item The range of $T_{\rm B}$ has finite codimension if and only if the range of $\check{L}_{\check{P}}$ has finite codimension, in which case 
$$\dim \coker(T_{\rm B}) = \dim \coker(\check{L}_{\check{P}}).$$
\item $\ker T_{\rm B}$ has finite dimension if and only if $\ker \check{L}_{\check{P}}$ has finite dimension, in which case 
$$\dim \ker(T_{\rm B}) = \dim \ker(\check{L}_{\check{P}}).$$
\item $\check{L}_{\check{P}}$ is Fredholm if and only if $T_{\rm B}$ is Fredholm.
\end{enumerate} 
\end{proposition}

We also note the following: Fredholm boundary conditions must necessarily be infinite dimensional.

\begin{proposition}
Let $\mathcal{T}=(T_{\rm min},T^\dagger_{\rm min})$ be a Fredholm FAP such that $\mathcal{C}_\mathcal{T}$ has infinite codimension in $\check{\mathhil{h}}_\mathcal{T}$. If $B$ is a Fredholm boundary condition then $\dim B=\dim B^* = \infty$. 
\end{proposition}

\begin{proof}
We prove via the contrapositive.
Let $B \subset \check{\mathhil{h}}_\mathcal{T}$ be a finite dimensional subspace.
Indeed, this is a boundary condition and $B + \Ca_\mathcal{T}$ is closed since $\Ca_\mathcal{T}$ is closed.
Now, 
\begin{align*}
\check{\mathhil{h}}_\mathcal{T} &= (B \cap \Ca_\mathcal{T}) \oplus (B \cap \Ca_\mathcal{T})^{\perp} \\
\Ca_\mathcal{T} &= (B \cap \Ca_\mathcal{T}) \oplus (B \cap \Ca_\mathcal{T})^{\perp,\Ca_\mathcal{T}} \\
B + \Ca_\mathcal{T} &= B \oplus (B \cap \Ca_\mathcal{T})^{\perp,\Ca_\mathcal{T}} = (B \cap \Ca_\mathcal{T}) \oplus (B \cap \Ca_\mathcal{T})^{\perp,B} \oplus (B \cap \Ca_\mathcal{T})^{\perp,\Ca_\mathcal{T}}.
\end{align*}
From this, we obtain 
$$\faktor{\check{\mathhil{h}}_\mathcal{T}}{\Ca_\mathcal{T}} \cong \faktor{(B \cap \Ca_\mathcal{T})^{\perp}}{(B \cap \Ca_\mathcal{T})^{\perp,\Ca_\mathcal{T}}}$$
and 
$$ \faktor{\check{\mathhil{h}}_\mathcal{T}}{B + \Ca_\mathcal{T}} \cong \faktor{(B \cap \Ca_\mathcal{T})^{\perp}}{(B\cap \Ca_\mathcal{T})^{\perp,B} \oplus (B \cap \Ca_\mathcal{T})^{\perp,\Ca_\mathcal{T}}}.$$
Therefore, 
$$ \faktor{\check{\mathhil{h}}_\mathcal{T}}{B + \Ca_\mathcal{T}} \oplus (B\cap \Ca_\mathcal{T})^{\perp,B}
	\cong \faktor{(B \cap \Ca_\mathcal{T})^{\perp}}{(B \cap \Ca_\mathcal{T})^{\perp,\Ca_\mathcal{T}}}
	\cong \faktor{\check{\mathhil{h}}_\mathcal{T}}{\Ca_\mathcal{T}}.$$
Since $(B\cap \Ca_\mathcal{T})^{\perp,B}$ is finite dimensional and $\mathcal{C}_\mathcal{T}$ has infinite codimension in $\check{\mathhil{h}}_\mathcal{T}$, we have that $\faktor{\check{\mathhil{h}}_\mathcal{T}}{B + \Ca_\mathcal{T}}$ is infinite dimensional.
That is, $B$ is not Fredholm. By the contrapositive, if $B$ is Fredholm, then $\dim B = \infty$. Since $B^\ast$ is Fredholm if and only if $B$ is Fredholm, a symmetric argument yields that $\dim B^\ast = \infty$.
\end{proof}

\subsection{Rigidity results for the index}

Let us prove a homotopy invariance result for the index of Fredholm realisations of Fredholm FAPs.

\begin{theorem}
\label{rigidityformaulaABS}
Assume that $\mathcal{T}_t=(T_{t,{\rm min}},T^\dagger_{t,{\rm min}})_{t\in [0,1]}$ is a family of Fredholm FAPs $\Hil_1\to \Hil_2$ such that 
\begin{itemize}
\item[(i)] The mappings
$$[0,1]\ni t\mapsto 
\begin{pmatrix}
0& T_{t,{\rm min}}\\
T^\dagger_{t,{\rm max}}& 0
\end{pmatrix}
\quad\mbox{and}\quad 
[0,1]\ni t\mapsto 
\begin{pmatrix}
0& T^\dagger_{t,{\rm min}}\\
T_{t,{\rm max}}& 0
\end{pmatrix}$$
are continuous in the gap topology on the space of densely defined self-adjoint operators (i.e. the resolvents are norm continuous with respect to $t$).
\item[ii)] $\dom(T_{t,{\rm min}})$ and $\dom (T_{t,{\rm max}})$ are constant in $t$, so $\check{\mathhil{h}}_\mathcal{T}:=\dom (T_{t,{\rm max}})/\dom (T_{t,{\rm min}})$ is independent of $t$ up to a canonical Banach space isomorphism.
\end{itemize}
Let $[0,1]\ni t\mapsto P_t\in \mathbb{B}(\check{\mathhil{h}}_\mathcal{T})$ be a norm continuous family of projectors such that $B_t:=(1-P_t)\check{\mathhil{h}}_\mathcal{T}$ satisfies that $(B_t,\Ca_{\mathcal{T}_t})$ is a Fredholm pair in $\check{\mathhil{h}}_\mathcal{T}$ for all $t\in [0,1]$. Then it holds that 
$$\indx(T_{0,B_0})=\indx(T_{1,B_1}).$$
\end{theorem}

The proof uses standard techniques from $KK$-theory and the index theory of Fredholm operators on Hilbert $C^*$-modules. For an overview, see \cite{knudsenjensenthomsen} and \cite{varillygraciabondia}.

\begin{proof}
We write $\Hil_\mathcal{T}$ for the Hilbert space $\dom (T_{t_0,{\rm max}})$ for some fixed $t_0$ and note that $T_{t,{\rm max}}:\Hil_\mathcal{T}\to \Hil_2$ is continuous for each $t$. By assumption i) above, the family $(T_{t,{\rm max}})_{t\in [0,1]}$ defines an adjointable $C[0,1]$-linear map that we denote by $T_{{\rm max}}:\Hil_{\mathcal{T},C([0,1])}\to \Hil_{2,C([0,1])}$ where $\Hil_{\mathcal{T},C([0,1])}$ and $\Hil_{2,C([0,1])}$ denote the $C[0,1]$-Hilbert $C^*$-module defined from the indicated Hilbert spaces. We also set $\check{C}_t:=P_t\check{\mathhil{h}}_\mathcal{T}$ and let $\check{C}_{C([0,1])}$ denote the $C[0,1]$-Hilbert $C^*$-submodule of $C([0,1])\otimes \check{\mathhil{h}}_\mathcal{T}$ with fibre $\check{C}_t$, it is well-defined by norm continuity of the family $(P_t)_{t\in [0,1]}$. 

Following Proposition \ref{Prop:CompExistsABS}, we define the operator 
$$\check{L}_{t}: \Hil_\mathcal{T} \to \begin{matrix}\Hil_2\\ \oplus\\ \check{C}_t\end{matrix}, \quad \check{L}_{t}u = \begin{pmatrix} T_{t, \max} u \\ P_t \gamma u\end{pmatrix},$$
for $t\in [0,1]$. This operator defines an adjointable $C[0,1]$-linear map $\check{L}:\Hil_{\mathcal{T},C([0,1])}\to \Hil_{2,C([0,1])}\oplus \check{C}_{C([0,1])}$ because $T_{{\rm max}}:\Hil_{\mathcal{T},C([0,1])}\to \Hil_{2,C([0,1])}$ and $\xi\mapsto (t\mapsto P_t\xi)$ is an adjointable map $\checkH(D)_{C([0,1])}\to \check{C}_{C([0,1])}$. 

We claim that the adjointable $C[0,1]$-linear map $\check{L}:\Hil_{\mathcal{T},C([0,1])}\to \Hil_{2,C([0,1])}\oplus \check{C}$ is Fredholm. For each $t\in [0,1]$, $\check{L}_t$ is Fredholm, so we can find finite-dimensional vector bundles $E_{t,1},E_{t,2}\to [0,1]$ with $E_{t,1}$ being a subbundle of $\Hil_{\mathcal{T},C([0,1])}$ such that $E_{t,1}|_{t}=\ker \check{L_t}$ and $E_{t,2}$ being a subbundle of $\Hil_{2,C([0,1])}$ such that the induced map $E_{t,2}|_{t}\to \coker \check{L_t}$ is surjective. In particular, there is a finite-rank perturbation of $\check{L}$ as a mapping $\Hil_{\mathcal{T},C([0,1])}\oplus E_{t,2}\to \Hil_{2,C([0,1])}\oplus \check{C}\oplus E_{t,1}$ which is invertible when restricted to the fibre over $t$. By an openness argument, this finite-rank perturbation of $\check{L}$ is invertible on a small open neighbourhood $U_t$ of $t$. In particular, $\check{L}$ is Fredholm when restricted to $U_t$. By compactness of $[0,1]$, $\check{L}$ is Fredholm. 

Since, $\check{L}$ is a Fredholm operator between $C[0,1]$-Hilbert $C^*$-modules, it has a well-defined index in $K^0([0,1])$. Under the isomorphisms $K^0([0,1])\cong K^0(pt)\cong \mathbb{Z}$, homotopy invariance of $K$-theory implies that
$$\indx(\check{L}_0)=r_0^*\indx_{C([0,1])}(\check{L})=r_1^*\indx_{C([0,1])}(\check{L})=\indx(\check{L}_1),$$
where $r_t$ denotes the inclusion $\{t\}\subseteq [0,1]$.  We can now conclude the theorem from Proposition \ref{Prop:CompExistsABS} which implies that
\begin{equation*} 
\indx(T_{0,B_0})=\indx(\check{L}_0)=\indx(\check{L}_1)=\indx(T_{1,B_1}).
\qedhere
\end{equation*}
\end{proof}

\bibliographystyle{alpha}

\providecommand{\noopsort}[1]{}

\setlength{\parskip}{0pt}
\end{document}